\documentclass[a4paper,10pt]{article}

\usepackage{amsthm}
\usepackage{thmtools}
\usepackage{hyperref}
\usepackage{mathrsfs}
\usepackage{changepage}

\declaretheorem[name=Definition,style=definition,qed=$\dashv$,
numberwithin=section]{dfn}

\declaretheorem[name=Fact,style=plain,sibling=dfn]{fact}
\declaretheorem[name=Theorem,style=plain,sibling=dfn]{tm}
\declaretheorem[name=Expected Theorem,style=plain,sibling=dfn]{etm}
\declaretheorem[name=Lemma,style=plain,sibling=dfn]{lem}
\declaretheorem[name=Proposition,style=plain,sibling=dfn]{prop}
\declaretheorem[name=Corollary,style=plain,sibling=dfn]{cor}
\declaretheorem[name=Remark,style=definition,sibling=dfn]{rem}
\declaretheorem[name=Claim,style=plain]{clm}
\declaretheorem[name=Claim,style=plain,numbered=no]{clm*}
\declaretheorem[name=Sublaim,style=plain,numbered=no]{sclm*}
\declaretheorem[name=Sublaim,style=plain,sibling=dfn]{sclm}
\declaretheorem[name=Case,style=definition]{case}

\declaretheorem[name=Stage,style=definition,numbered=no]{stage*}
\declaretheorem[name=Assumption,style=definition,sibling=dfn]{ass}
\declaretheorem[name=Convention,style=definition,sibling=dfn]{conv}
\declaretheorem[name=Conjecture,style=definition,sibling=dfn]{conj}

\usepackage{amsmath}
\usepackage{amsxtra}
\usepackage{amssymb}
\usepackage{turnstile}
\usepackage{enumitem}
\usepackage{comment}
\usepackage{accents}
\usepackage{xcolor}

\usepackage[retainorgcmds]{IEEEtrantools}

\newcommand{\iso}{\cong}

\newcommand{\CC}{\mathbb C}
\newcommand{\QQ}{\mathbb Q}
\newcommand{\RR}{\mathbb R}

\newcommand{\PP}{\mathbb P}
\newcommand{\BB}{\mathbb B}
\newcommand{\sub}{\subseteq}
\newcommand{\cross}{\times}
\newcommand{\all}{\forall}
\newcommand{\ex}{\exists}

\newcommand{\inter}{\cap}
\renewcommand{\int}{\inter}

\newcommand{\om}{\omega}
\newcommand{\pow}{\mathcal{P}}
\newcommand{\OR}{\mathrm{OR}}

\newcommand{\Hull}{\mathrm{Hull}}

\newcommand{\cut}{\backslash}
\newcommand{\N}{N}

\newcommand{\Tt}{\mathcal{T}}
\newcommand{\Ss}{\mathcal{S}}
\newcommand{\Uu}{\mathcal{U}}
\newcommand{\Vv}{\mathcal{V}}
\newcommand{\Ww}{\mathcal{W}}
\newcommand{\Ll}{\mathcal{L}}

\newcommand{\Nn}{\mathcal{N}}

\newcommand{\rg}{\mathrm{rg}}
\newcommand{\dom}{\mathrm{dom}}

\newcommand{\ins}{\trianglelefteq}
\newcommand{\nins}{\ntrianglelefteq}
\newcommand{\pins}{\triangleleft}
\newcommand{\npins}{\ntriangleleft}
\newcommand{\crit}{\mathrm{cr}}

\newcommand{\union}{\cup}
\newcommand{\rest}{\!\upharpoonright\!}
\newcommand{\com}{\circ}

\newcommand{\lh}{\mathrm{lh}}
\newcommand{\Ult}{\mathrm{Ult}}

\newcommand{\sats}{\models}
\newcommand{\elem}{\preccurlyeq}
\newcommand{\J}{\SS}

\newcommand{\AD}{\mathsf{AD}}
\newcommand{\AC}{\mathsf{AC}}
\newcommand{\DC}{\mathsf{DC}}

\newcommand{\HOD}{\mathrm{HOD}}
\newcommand{\HC}{\mathrm{HC}}
\newcommand{\ZFC}{\mathsf{ZFC}}
\newcommand{\ZF}{\mathsf{ZF}}
\newcommand{\KP}{\mathsf{KP}}

\newcommand{\Coll}{\mathrm{Col}}
\newcommand{\es}{\mathbb{E}}

\newcommand{\pbar}{{\bar{p}}}

\newcommand{\eps}{\varepsilon}

\newcommand{\sigmabar}{{\bar{\sigma}}}

\newcommand{\core}{\mathfrak{C}}

\newcommand{\pred}{\mathrm{pred}}
\newcommand{\tc}{\mathrm{tc}}

\newcommand{\un}{\union}

\newcommand{\id}{\mathrm{id}}

\newcommand{\conc}{\ \widehat{\ }\ }

\newcommand{\forces}{\dststile{}{}}
\newcommand{\bfPi}{\undertilde{\Pi}}
\newcommand{\bfSigma}{\undertilde{\Sigma}}

\newcommand{\rSigma}{\mathrm{r}\Sigma}

\newcommand{\rPi}{\mathrm{r}\Pi}

\newcommand{\rDelta}{\mathrm{r}\Delta}
\newcommand{\mSigma}{\mathrm{m}\Sigma}
\newcommand{\mPi}{\mathrm{m}\Pi}
\newcommand{\mDelta}{\mathrm{m}\Delta}
\newcommand{\bfmSigma}{\mathrm{m}\bfSigma}
\newcommand{\bfmPi}{\mathrm{m}\bfPi}

\newcommand{\muSigma}{\mu\Sigma}
\newcommand{\muPi}{\mu\Pi}
\newcommand{\muDelta}{\mu\Delta}
\newcommand{\bfmuSigma}{\mu\bfSigma}
\newcommand{\bfmuPi}{\mu\bfPi}

\newcommand{\supp}{\mathbb{S}}

\DeclareMathOperator{\Th}{Th}

\DeclareMathOperator{\cof}{cof}

\DeclareMathOperator{\rank}{rank}

\newcommand{\st}{\bigm|}

\newcommand{\xvec}{\vec{x}}
\newcommand{\Dd}{\mathcal{D}}
\newcommand{\orr}{\vee}
\newcommand{\OD}{\mathrm{OD}}
\newcommand{\gammavec}{\vec{\gamma}}
\newcommand{\stem}{\mathrm{stem}}
\newcommand{\bfrSigma}{\undertilde{\rSigma}}

\newcommand{\psub}{\subsetneq}

\newcommand{\Yvec}{\vec{Y}}

\newcommand{\Xx}{\mathcal{X}}

\newcommand{\ydot}{\dot{y}}

\newcommand{\surj}{\twoheadrightarrow}
\newcommand{\alphavec}{\vec{\alpha}}
\newcommand{\cHull}{\mathrm{cHull}}

\newcommand{\Avec}{\vec{A}}

\newcommand{\M}{\mathcal{M}}

\newcommand{\lpole}{\left\lfloor}
\newcommand{\rpole}{\right\rfloor}

\newcommand{\Hh}{\mathcal{H}}
\newcommand{\univ}[1]{\lpole #1\rpole}

\newcommand{\tu}{\textup}

\newcommand{\base}{\mathrm{base}}
\newcommand{\tail}{\mathrm{tail}}
\newcommand{\concatB}{%
  \mathbin{\rotatebox[origin=c]{90}{\scalebox{.7}{(\kern1ex)}}}
}

\renewcommand{\supp}{\mathrm{supp}}
\newcommand{\Nm}{\mathrm{Nm}}

\newcommand{\pvec}{\vec{p}}
\newcommand{\betavec}{\vec{\beta}}
\newcommand{\xdot}{\dot{x}}

\newcommand{\tauvec}{\vec{\tau}}

\newcommand{\Lim}{\mathrm{Lim}}
\newcommand{\lex}{\mathrm{lex}}

\newcommand{\Pp}{\mathcal{P}}

\newcommand{\Q}{\mathcal{Q}}

\newcommand{\sigmavec}{\vec{\sigma}}
\renewcommand{\M}{\mathscr{M}}
\newcommand{\Xvec}{\vec{X}}
\newcommand{\Lp}{\mathrm{Lp}}
\newcommand{\tree}{\mathfrak{T}}
\newcommand{\yvec}{\vec{y}}
\newcommand{\deltavec}{\vec{\delta}}

\newcommand{\Nhat}{\widehat{N}}
\newcommand{\Mhat}{\widehat{\M}}
\newcommand{\Ntilde}{\widetilde{N}}
\newcommand{\dfnemph}{\textbf}
\newcommand{\Hdot}{\dot{H}}

\newcommand{\Tdot}{\dot{T}}
\newcommand{\rdot}{\dot{r}}

\newcommand{\Gtilde}{\widetilde{G}}
\newcommand{\Mtilde}{\widetilde{\M}}
\newcommand{\Ttilde}{\widetilde{T}}

\newcommand{\alphagap}{{\alpha_g}}
\newcommand{\betagap}{{\beta_g}}
\newcommand{\Gammagap}{{\Gamma_g}}

\newcommand{\str}{\oplus}

\newcommand{\sdot}{\dot{s}}
\newcommand{\trcl}{\mathrm{trcl}}

\newcommand{\WO}{\mathrm{WO}}
\newcommand{\LO}{\mathrm{LO}}
\newcommand{\betag}{\beta_{\mathrm{g}}}

\newcommand{\alphag}{\alpha_{\mathrm{g}}}
\newcommand{\Gammag}{\Gamma_{\mathrm{g}}}

\newcommand{\xg}{x_{\mathrm{g}}}
\renewcommand{\bfrSigma}{\undertilde{\rSigma}}
\newcommand{\subforces}[1]{\dststile{#1}{}}
\newcommand{\sforces}[2]{\dststile{#1}{#2}}
\newcommand{\stk}{\mathrm{stk}}

\newcommand{\vareps}{\varepsilon}
\newcommand{\gen}{\mathrm{gen}}
\newcommand{\desc}{\mathrm{desc}}
\newcommand{\yg}{{y_{\mathrm{g}}}}
\newcommand{\Pg}{{P_{\mathrm{g}}}}
 \newcommand{\rename}{\mathrm{rename}}

 \newcommand{\steelurl}{http://math.berkeley.edu/\textasciitilde steel}
 
 \newcommand{\rnm}{\mathrm{rnm}}
 \newcommand{\Jj}{\mathcal{J}}

\newcommand{\loc}{\mathrm{loc}}
\newcommand{\pad}{\mathrm{pad}}
\newcommand{\Type}{\mathrm{Type}}
\renewcommand{\SS}{\mathbb{S}}
\begin{document}
\title{$\Sigma_1$ gaps as derived models and  correctness of mice}
\author{Farmer Schlutzenberg and John Steel}

\maketitle

\begin{abstract}
 Assume ZF + AD + $V=L(\RR)$. Let $[\alpha,\beta]$ be a $\Sigma_1$ gap with $\Jj_\alpha(\RR)$ admissible. We analyze $\Jj_\beta(\RR)$ as a natural form of ``derived model''  of a premouse $P$, where $P$ is found in a generic extension of $V$. In particular,  we will have $\pow(\RR)\cap \Jj_\beta(\RR)=\pow(\RR)\cap D$, and if  $\Jj_\beta(\RR)\models$``$\Theta$ exists'', then  $\Jj_\beta(\RR)$ and $D$ in fact have the same universe. This analysis will be employed in further work, yet to appear,
 toward a resolution of a conjecture of Rudominer and Steel on the nature of $(L(\RR))^M$, for $\om$-small mice $M$.
 We also establish some preliminary work toward this conjecture in the present paper.
\end{abstract}

\tableofcontents

\section{Introduction}

\subsection{Gaps of $L(\RR)$ and derived models}\label{subsec:gaps}

\begin{dfn}
 The \emph{$L(\RR)$ language}
 is the language of set theory
 augmented with a constant symbol $\dot{\RR}$. We will always interpret this language in wellfounded models $M$ with $\omega\in M$  and with
 $M\sats$``the set of all reals exists''. The symbol $\dot{\RR}$
 is interpreted as $\RR\cap M$.
\end{dfn}

\begin{conv}
 When we consider definability over segments $\Jj_\beta(\RR)$ of $L(\RR)$,
 we use by default the $L(\RR)$ language.
  Likewise over segments $\Jj_{\beta}(\RR^M)$ of $L(\RR^M)$ for  models $M$ as above.
\end{conv}

\begin{rem}
Recall from \cite{scales_in_LR} the notion of a \emph{$\Sigma_1$ gap $[\alpha,\beta]$ of $L(\RR)$} (we also say just a \emph{gap} for short). 
Let $[\alpha,\beta]$ be a gap of $L(\RR)$.
We say the gap is \emph{admissible}
iff $\Jj_\alpha(\RR)$ is admissible. Recall that non-admissible gaps are called \emph{projective-like}.
Recall that if $[\alpha,\beta]$ is a projective-like gap then $\alpha=\beta$; some admissible gaps (for example, the first) also have $\alpha=\beta$.

Recall that
\emph{$\bfSigma_n^{\Jj_\beta(\RR)}$ types reflect} iff for each $x\in\Jj_\beta(\RR)$ there is $\beta'<\beta$
and $x'\in\Jj_{\beta'}(\RR)$ such that
$t(x/x')=t'$,
where $t=\Th_{\Sigma_n}^{\Jj_\beta(\RR)}(\{x\})$ and $t'=\Th_{\Sigma_n}^{\Jj_{\beta'}(\RR)}(\{x'\})$, and $t(x/x')$ denotes
the theory which results from $t$ by substituting $x'$ for $x$.
We use the analogous terminology
for $\rSigma_n$ replacing $\Sigma_n$.

Let $[\alpha,\beta]$ be an admissible gap. Recall that $[\alpha,\beta]$ is \emph{strong} iff $\bfSigma_{n+1}^{\Jj_\beta(\RR)}$ types reflect, where $n$ is least such that $\rho_{n+1}^{\Jj_\beta(\RR)}=\RR$; otherwise
$[\alpha,\beta]$ is \emph{weak}.
\end{rem}

In this paper we will make progress toward the Rudominer-Steel Conjecture, described in \S\ref{subsec:rudo-steel_conj}.
A step toward this,
approximately stated, is the realization of $\Jj_\beta(\RR)$,
where, for example, $\beta$ ends a weak gap,
as the derived model of a (generic) premouse.  It is analogous
to the realization of $L(\RR)$
as the derived model of a (generic) iterate of $M_\om$. A realization of this kind,
except that it is executed ``in the codes'',
via term relations for sets of reals comprising a self-justifying system corresponding to $\Jj_\beta(\RR)$, is commonly
considered in the core model induction.
Here we will get a lot closer
to an actual realization of the model $\Jj_\beta(\RR)$, and this realization might be of independent interest. However,  
the kind of derived model we define can still have a universe distinct from $\Jj_\beta(\RR)$, and it is stratified in a somewhat different hierarchy. We will also get a similar kind of realization
of (a version of) $\Jj_{\beta+1}(\RR)$
when $\beta$ ends a strong gap.

Suppose again that $\beta$ ends a weak gap $[\alpha,\beta]$.
So $\Jj_\alpha(\RR)$ is
admissible and $\alpha<\beta$.
We now describe roughly the ``derived model'' construction we use,
and the basic components of its construction and analysis,
and also sketch why a more naive attempt to realize $\Jj_\beta(\RR)$ as a derived model runs into problems.

We will first find an $\om$-small projecting $x$-mouse
$P$, for some $x\in\RR$, with $\om$ Woodin cardinals and which is \emph{not} iterable in $\Jj_\alpha(\RR)$,
but which is in a reasonable sense ``stably minimal'' with respect to this failure of iterability. In case it helps,
for each strong cutpoint $\delta$ of $P$ strictly below the sup $\lambda^P$ of Woodins of $P$, we will have $P|\delta^{+P}=\Lp_{\Gamma}(P|\delta)$ where $\Gamma=\Sigma_1^{\Jj_\alpha(\RR)}$
(equivalently, $P|\delta^{+P}$ will be the stack of all sound premice $Q$
such that $\rho_\om^P\leq\delta$,
 $P\pins Q$, and there is an above-$\delta$, $(\om,\om_1)$-iteration strategy $\Sigma$ for $Q$ with $\Sigma\in\Jj_\alpha(\RR)$).
Moreover, $P$ itself will project to $\om$.
We would then like to realize $\Jj_\beta(\RR)$,
or some related model,
as the ``derived model'' of an $\RR$-genericity iterate of $P$.

Now our assumptions
do not guarantee much about the theory modelled by $\Jj_\beta(\RR)$, so $\Jj_\beta(\RR)$ can be very non-closed --
for example, we might have $\beta=\gamma+1$ for some $\gamma$,
and/or it might be that $\Jj_\beta(\RR)\sats$``$\Theta$ does not exist'', etc. Correspondingly,
(it will follow that)
$P$ might not be very closed above $\lambda^P$: it might
be, for instance, that $\OR^P=\lambda^P+\om$
(and this can indeed occur).
An $\RR$-genericity iterate $P'$ of such a $P$
would have $\OR^{P'}=\om_1+\om$ (as $\lambda^P$ gets sent to $\om_1$),
but since $\Jj_\alpha(\RR)$ is admissible,
 $\omega\beta=\OR^{\Jj_\beta(\RR)}>\OR^{\Jj_\alpha(\RR)}=\alpha>\om_1+\om$.
This gives $\OR^{P'}<\beta\omega$, and so $\Jj_\beta(\RR)$ cannot be the naive ``derived model'' $D(\RR)$ of $P'$ (i.e.~of form $L(\RR^*)^{P'[G]}$, for $G$ being $(P',\Coll(\om,{<\lambda^{P'}})$-generic and $\RR=\RR^*=\bigcup_{\alpha<\lambda^{P'}}\RR\cap P'[G\rest\alpha]$; this model has ordinal height $\omega_1+\omega$).
Worse, note that $D(\RR)$ cannot even have a version of $\Jj_\beta(\RR)$ which is definable in the codes; in particular, for no $n<\om$
is the theory $\Th_{\Sigma_n}^{\Jj_\beta(\RR)}(\RR)$ definable from parameters over $D(\RR)$.

This kind of mismatch is fairly easily dealt with,
by replacing $\Jj_\alpha(\RR)$ with a structure of form $(\HC,T)$, where $T$ is the set of pairs $(x,t)$ such that $x\in\HC$
and $t=\Th_{\Sigma_1}^{\Jj_\alpha(\RR)}(\{x\})$. 

Suppose we do this and construct above $(\HC,T)$ through to some $\Jj_{\beta'}(\HC,T)$ corresponding to $\Jj_\beta(\RR)$; note $\OR(\Jj_{\beta'}(\HC,T))=\omega_1+\omega\beta'$.
In order to see that $\OR^{P'}\leq\omega_1+\omega\beta'$, one would like to obtain that $\Jj_{\eta}(\HC,T)$ is a ``derived model''
of $P'$, where $\lambda^P+\omega\eta=\OR^{P'}$.
Toward this, one would like to have that $(\HC,T)$ is a ``derived model'' of $P'|\lambda^{P'}$;
for this, we will use that $T$ will be encoded into $P'|\lambda^{P'}$ via mice witnessing~ $\Sigma_1$ truths in $L(\RR)$.
Now conversely, to see that
$\omega_1+\omega\beta'\leq\OR^{P'}$,
one would like to have
that for each $\eta\leq\beta'$, at least for appropriate $P'$,
$P'|(\omega_1+\omega\eta)$
is generic over $\Jj_\eta(\HC,T)$ for 
the local variant
of the Martin measure Prikry forcing $\PP$ for forcing a premouse with $\om$ Woodin cardinals over $L(\RR)$.
However, because $\Jj_{\beta'}(\HC,T)$
need not be particularly closed, and in particular,
might satisfy ``$\Theta$ does not exist'',
this forcing can be a proper class of $\Jj_{\beta'}(\HC,T)$.
This will mean that 
we need to analyse the forcing relation of $\PP$
level-by-level over $\Jj_{\beta'}(\HC,T)$,
to a natural extent,
and in particular, 
one would like to have
a reasonable version
of the forcing relation
for $\Sigma_0^{P'|(\omega_1+\omega\eta)}$ truth, which is
definable over $\Jj_{\eta}(\HC,T)$. When one attempts this analysis,
one runs into problems,
because $\Jj_\eta(\HC,T)$
is not in general closed under the iterated Martin measure
$\mu^{<\om}$ on the finite tuples of (Turing) degrees.
That is, let $\mu$ be the Martin measure on degrees ($X\in\mu$ iff there is a degree $x$ such that $y\in X$ for all degrees $y\geq_Tx$) and for $n<\om$,
let $\mu^n$ be the $n$th iterate.
Let $\mu^{<\om}=\bigcup_{n<\om}\mu^n$.
Then
there can be $X\in\Jj_\eta(\HC,T)$
such that $X\cap\mu^{<\om}\notin\Jj_\eta(\HC,T)$
(even though for each $n<\omega$, we do have $X\cap\mu^n\in\Jj_\eta(\HC,T)$ for each $X\in\Jj_\eta(\HC,T)$).
The result of this is that
the analysis of the Prikry forcing relation proceeds too slowly in the usual $\Jj_\eta(\HC,T)$ hierarchy to yield 
the desired level-by-level analysis.

In order to solve this problem,
we speed up the hierarchy,  constructing from $\mu^{<\om}$ (above $(\HC,T)$). This produces
what we call the 
\emph{$\M$-hierarchy} associated to $[\alpha,\beta]$; we define $\M_{\omega_1}=(\HC,T)$,
and $\M_{\omega_1+\xi}=(\Ss')^{\mu^{<\om}}_{\xi}(\HC,T)$, where $(\Ss')^{\mu^{<\om}}$ is a slight variant of the transitive version of Jensen's $\Ss$-operator for $\mu^{<\om}$-rud functions (cf.~\S\ref{subsec:notation}
and \cite[p.~610]{schindler2010fine}).
With this second modification, there is
a limit $\beta^*$
such that $\M_{\beta^*}$  encodes  $\Jj_\beta(\RR)$ (in a manner to be described),
and for appropriate
iterates 
$P'$ of $P$,
we will indeed get that $\beta^*=\OR^{P'}$,
$\M_{\beta^*}$ is a ``derived model''
of $P'$, and $P'$
is an $(\M_{\beta^*},\PP)$-generic premouse.
However, the fact that
we construct
from $\mu^{<\om}$
in the $\M$-hierarchy
needs to be incorporated into the definition of ``derived model'' (along with having $(\HC,T)$ as the ``derived model'' of $P'|\lambda^{P'}$). In order to achieve this,  information encoded into $\es^{P'}$ (and in fact into $P'|\eta$
in general, where $\eta\in[\lambda^{P'},\OR^{P'}]$) needs to
be exploited in order
to define $\M_\eta$
(as a ``derived model''
of $P'|\eta$). Much of the paper is devoted to laying these things out clearly.

In the case that $\Jj_\beta(\RR)$ satisfies ``$\Theta$
exists'', we will have that $\M_{\beta^*}$ and $\Jj_\beta(\RR)$ have the
same
universe,  so in this case, $\Jj_\beta(\RR)$ itself is the derived model of a generic premouse. (Actually in this case one could avoid introducing the $\M$-hierarchy at all.)

\subsection{Notation and terminology}\label{subsec:notation}

Our primary background theory is $\ZF+\AD^{L(\RR)}$ (hence $\DC$ holds by \cite{kechrisADDCLR}).
However, mostly we argue only using  determinacy close to $\Jj_\beta(\RR)$, for some $\beta$ ending a gap ($\Jj_{\beta+\om}(\RR)\sats\AD$ should be more than enough). At some points we make remarks which are trivial under the global determinacy assumption, but intended to be of relevance under lesser determinacy assumptions.

Whenever we refer to an ordering on $\OR\cross\OR$, it is the lexicographic
order.

Given $x\in\RR$, we write $[x]$ for the Turing degree of $x$. For definitions pertaining to the Martin measure $\mu$,
see \S\ref{subsec:gaps}.

We write $\trcl(X)$ for the transitive closure of $X$.

Given a first-order structure $M=(N,R_1,\ldots,R_n)$ with universe $N$ and relations, etc, $R_1,\ldots,R_n$, we write $\univ{M}=N$.
Let $\Ll$ be the corresponding language
(with symbols for the $R_1,\ldots,R_n$).
When there is no confusion, we blur between $M$ and $N$, writing for example
 $x\in M$ for $x\in N$,
and $X\sub M$ for $X\sub N$.
Given $X\sub M$, $\Sigma_n^M(X)$ denotes
the class of relations (of finite arity)
over $M$ definable with a $\Sigma_n$ formula
of $\Ll$, and likewise for other formula classes.
And $\Delta^M_n(X)=\Sigma^M_n(X)\cap\Pi^M_n(X)$.
For lightface definability (that is, without parameters) we write $\Sigma^M_n=\Sigma^M_n(\emptyset)$. For boldface,
$\bfSigma^M_n=\Sigma^M_n(N)$.
We say $M$ is a  \emph{transitive structure}
if $\univ{M}$ is transitive.
So we can essentially
consider $\SS_\beta$
as a denoting a transitive structure
of the form $M=(N,V_{\omega+1})$
where $V_{\omega+1}\sub N$ and $V_{\omega+1}$ is the interpretation of a constant symbol $\dot{\RR}$.

Write $\Lim$ for the class of limit ordinals and $\Lim_0=\Lim\cup\{0\}$.
We define a slight variant of the 
transitive version of Jensen's $\Ss$-hierarchy, introduced in \cite[p.~610]{schindler2010fine}.
Given a set $U$, define $\Ss(U)=\bigcup_{i\leq 14}F_i``U^2$
where $F_0,\ldots,F_{14}$ are as in \cite{schindler2010fine};
note that in \cite{schindler2010fine},
$F_i``(U\cup\{U\})^2$ is used in their definition of $\mathbb{S}^A$, not just $F_i``U^2$. Here we only use $U\cup\{U\}$ to proceed at stage $0$ and limit stages.
That is, for a transitive set or structure $X$, let $\Ss_0(X)=X$, and given $\lambda\in\Lim_0$, let \[\Ss_{\lambda+1}(X)=\Ss(\Ss_\lambda(X)\cup\{\Ss_\lambda(X)\}),\] and given a successor ordinal $\alpha+1$, let \[\Ss_{\alpha+2}(X)=\Ss(\Ss_{\alpha+1}(X)).\]
Note then that for $\om\lambda\in\Lim_0$,
$\Ss_{\om\lambda+\om}(X)=\Jj_{\lambda+1}(X)$ is the rud closure of $\Jj_\lambda(X)\cup\{\Jj_\lambda(X)\}=\Ss_{\om\lambda}(X)\cup\{\Ss_{\om\lambda}(X)\}$.
For a class $A$, define $\Ss^A(U)=\bigcup_{i\leq 15}F_i``U^2$
(so now $F_{15}$ is included),
and then define $\Ss_\alpha^A(X)$
from $\Ss^A$ just like $\Ss_\alpha(X)$ is defined from $\Ss$.
We will generally talk about the $\Ss$- and $\Ss^A$-hierarchies, not $\Jj$- and $\Jj^A$-. Noting that $\RR$ is not transitive (so $\Ss_0(\RR)$ was not defined above), define $\Jj_0(\RR)=\Ss_0(\RR)=V_{\omega+1}$,
and likewise if $M$ is some model with wellfounded $\om$, then $\Jj_0(\RR^M)=\Ss_0(\RR^M)=V_{\omega+1}^M$.
Above this base, we define $\Ss_{\alpha}(\RR)$ and $\Ss_{\alpha}(\RR^M)$ like for transitive sets. So $\Jj_\gamma(\RR)=\Ss_{\om\gamma}(\RR)$.

  For $\gamma\in\Lim_0$,
  working in the $L(\RR)$ language, we write \begin{equation}\label{eqn:SS_gamma}\SS_\gamma=\Ss_{\gamma}(\dot{\RR}),\end{equation} and so when we talk about    ``$\SS_\gamma$'' in the context of some model $M$, it denotes $\Ss_\gamma(\RR^M)$ (and if $M$ is not clear from context, then $\SS_\gamma$ should be $\Ss_\gamma(\RR)$).

An \emph{S-gap} of $L(\RR)$ is either the interval $[0,0]$,
or an interval 
$[\omega\alpha,\omega\beta]$
 such that $[\alpha,\beta]$ is  a gap.

We write 
$\LO$ for the set of reals coding linear orders of 
some $n\leq\om$ and $\WO$ for the set of reals coding wellorders of some 
$n\leq\om$. For $\gamma<\om_1$ we write $\WO_\gamma$ for the
set of reals coding wellorders of length $\gamma$.

  Let $P=(N,\es,F)$ be a premouse;
  here $\es$ denotes the internal extender sequence of $P$, and $F$ its active extender. We write $\es^P=\es$, $F^P=F$,
  $\es_+^P=\es\conc\left<F\right>$.
 We write $P^{\mathrm{pv}}=(N,\es,\emptyset)$, write $P|\alpha$ for the initial segment of $P$ of ordinal height $\alpha$, whose active extender $E$
  is the extender $E\in\es_+^P$ indexed at $\alpha$, and we write $P||\alpha=(P|\alpha)^{\mathrm{pv}}$. For further notation related to premice, see \cite[\S1.1]{iter_for_stacks}

For a transitive set $X$
or a real $X$,
an \emph{$X$-premouse} $P$ is just a premouse
over $X$, i.e. $P|0=\trcl(X\cup\{X\})$,
and all elements $P|0$ get put into all fine structural hulls formed.
For $X$ countable and a sound $(\om,\omega_1+1)$-iterable $X$-premouse $M$ such that $\rho_\om^M=X$, $\Sigma_M$ denotes the unique $(\om,\om_1+1)$-strategy for $M$.
By \cite{fsfni_v4}, for all $n<\om$,
every $(n,\om_1+1)$-iterable $n$-sound premouse $N$ is $(n+1)$-solid and $(n+1)$-universal,
and satisfies $(n+1)$-condensation.

Let $\alpha$ be a limit ordinal and $\Gamma=\Sigma_1^{\SS_\alpha}$.
Given a
transitive set $X\in\HC$, $\Lp_\Gamma(X)$ is the stack of all sound $X$-premice
$P$ such that $P$ projects to $X$ and there is an $(\om,\om_1+1)$-iteration strategy for
$P$ in $\SS_\alpha$.

For notation associated to iteration trees,
see \cite[\S1.1]{iter_for_stacks}.
If $\Tt$ is a normal tree,
then $\nu(\Tt)$ denotes $\sup_{\alpha+1<\lh(\Tt)}\nu(E^\Tt_\alpha)$.

 Work in a premouse $P$ with $\delta\in P\sats$``$\delta$ is Woodin''.
 For $\xi<\delta$,
  $\BB_{\delta,\geq\xi}$ denotes the $\delta$-generator extender
 algebra at $\delta$ determined by extenders $E\in\es$ with 
$\crit(E)\geq\xi$ and $\nu(E)$
 a cardinal. From now on,
 whenever
 we say \emph{extender algebra (at $\delta$)},
 we mean the $\delta$-generator version.
Note however that the $\om$-generator version is isomorphic to the $\delta$-generator version below a certain condition $p$.
For forcing below $p$,
we write $x_\delta$ for the canonical name for
the $\BB_\delta$-generic real.

Let $N$ be an $\om$-small
$X$-premouse.
   We write 
$\delta_{-1}^N=\rank(X)$.
For $k<\om$, if $N$ has $k+1$ Woodins $>\delta_{-1}^N$, we write these as $\delta_0^N<\ldots<\delta_k^N$. If $N$
has $\omega$-many such Woodins, 
 $\Delta^N$ denotes $\{\delta_n^N\bigm|n<\om\}$,  $\lambda^N$ denotes $\sup\Delta^N$,
   and for $j<\om$, $\Delta^N_{>j}$
   denotes $\{\delta_n^N\bigm|j<n<\om\}$, etc.

   Any other unexplained notation
  is likely explained in
  \cite[\S1.1]{iter_for_stacks} or \cite{premouse_inheriting}. But one notational device we want to make clear:
  
  \begin{conv}
   The label ``($M$)'' at the start of a lemma (or corollary, etc) indicates that that lemma (corollary, etc) presumes the context of the Rudominer-Steel conjectures (in which $M$ is the mouse in question), as opposed to the more general context of analyzing admissible gaps of $L(\RR)$.
  \end{conv}

\subsection{Ordinal definability in $L(\RR)$}
In this section we discuss some basics regarding ordinal definability in $L(\RR)$.

\begin{dfn}\label{dfn:OD^beta}
For $x,y\in\RR$, $\beta\in\Lim_0$ and $n\in[1,\om)$,
we say that
\[ y\text{ is }\OD^{\beta n}(x) \]
(or $y$ is $\OD^{\beta n}_x$, or $y\in\OD^{\beta n}(x)=\OD^{\beta n}_x$)
iff there is $\gamma<\om_1$ and a $\Sigma_n$ 
 formula\footnote{Of the 
$L(\RR)$ language.} $\varphi$
such that
for all $w\in\WO_\gamma$  and all $z\in\RR$, we have\footnote{See \S\ref{subsec:notation} line (\ref{eqn:SS_gamma}) for the definition of $\SS_\beta$.}
\begin{equation}\label{eqn:charac_OD^beta} 
z=y\iff\SS_\beta\sats\varphi(z,x,w). \end{equation}
If $\beta\geq\om_1$, this is equivalent to requiring that
\[ \{y\}\text{ is 
}\Sigma_n^{\SS_\beta}(\{x,\gamma\}). \]

Let ${<^{\beta n}}(x)={<^{\beta n}_x}$  be the canonical wellorder of 
$\OD^{\beta n}_x$ 
(cf.~\ref{lem:<^alpha_increasing} below). For $y\in\OD^{\beta n}_x$
let
\[ |y|^{\beta n}_x= \text{ rank of }y\text{ in }<^{\beta n}_x.\]

If $x=\emptyset$, we may drop the subscript $x$ from the above notation.

We also define $\OD^{\beta}(x)=\OD^\beta_x=\bigcup_{n<\om}\OD^{\beta n}_x$,
and if $\beta\in\Lim$ define $\OD^{<\beta}(x)=\OD^{<\beta}_x=\bigcup_{\alpha\in\Lim_{0}\cap\beta}\OD^{\alpha}_x$.
\end{dfn}

\begin{lem}\label{lem:<^alpha_increasing}
For $x\in\RR$ and $(\alpha,m),(\beta,n)\in\Lim_{0}\cross[1,\om)$ 
and
$(\alpha,m)\leq_\lex(\beta,n)$, we have:
 \begin{enumerate}
  \item\label{item:OD^alpha_sub_OD^beta} $\OD^{\alpha m}_x\sub\OD^{\beta 
n}_x$.
 \item\label{item:<^alpha_is_inseg_<^beta} $<^{\alpha m}_x$ is an initial 
segment of $<^{\beta n}_x$.
\item\label{item:OD^alpha_in_J_beta} If $\alpha<\beta$ then 
$\OD^{\alpha m}_x,<^{\alpha m}_x\in\J_{\beta}$.
\item\label{item:OD^alpha_function_Sigma_1_def} The function
 \[ (x_0,\alpha_0,m_0)\mapsto(\OD^{\alpha_0 m_0}_{x_0},<^{\alpha_0 m_0}_{x_0}), 
\]
 with domain the set of all $(x_0,\alpha_0,m_0)$ with $x_0\in\RR$, 
$\alpha_0\in\beta\inter\Lim_{0}$ and $m_0\in[1,\om)$,
 is $\Sigma_1^{\SS_\beta}$, uniformly in $\beta$.
 \end{enumerate}
\end{lem}

\begin{proof}
By induction on $(\beta,n)$.
 For notational simplicity assume $x=\emptyset$.
 The case $n>1$ is immediate by induction, so assume $n=1$ and $\alpha<\beta$.
Parts \ref{item:OD^alpha_in_J_beta} and \ref{item:OD^alpha_function_Sigma_1_def}
are then easy by induction, and part \ref{item:<^alpha_is_inseg_<^beta}
will follow from part \ref{item:OD^alpha_sub_OD^beta}
and the definition of $<^{\beta 1}$ (the precise details of which
were left to the reader), so we just 
need to verify part \ref{item:OD^alpha_sub_OD^beta}.

Let $y\in\OD^{\alpha m}$;
 we want to see that $y\in\OD^{\beta 1}$.
 We have $\OD^{\alpha m}\in\SS_\beta$, etc.
By determinacy in $\J_{\beta}$, $<^{\alpha m}$ has countable 
length. 
Let
\[ \gamma=|y|^{\alpha m}<\om_1. \]

By induction and parts \ref{item:OD^alpha_in_J_beta},
\ref{item:OD^alpha_function_Sigma_1_def}
it is easy to see that for $w\in\WO_{\gamma}$
and $z\in\RR$, we have $z=y$ iff
\[ \SS_\beta\sats\exists\alpha_0,m_0\ [
 z\in\OD^{\alpha_0 m_0}\text{ and }|z|^{\alpha_0 m_0}=|w|].\]
So $y\in\OD^{\beta 1}$, as desired.
\end{proof}

\begin{lem} For $\beta\in\Lim$, we have
\[ \OD^{\beta1}_x=\bigcup_{\delta\in\beta\inter\Lim_0\textup{ and 
}1\leq n<\om}\OD^{\delta 
n}_x.\]\end{lem}

\begin{proof}
Assume  $x=\emptyset$.
We have $\supseteq$ by
\ref{lem:<^alpha_increasing}. We verify $\sub$.
If $\beta\geq\om_1$ this follows 
from the characterization of $\OD^{\beta1}$ given by 
\ref{dfn:OD^beta}(\ref{eqn:charac_OD^beta}) and standard calculations. Suppose 
$\beta<\om_1$. Let $y\in\OD^{\beta 1}$, as witnessed by $\varphi,\gamma$.
Let 
$\beta'\in(\beta+1)\inter\Lim$
be least such that for some $w\in\WO_\gamma$, we have
\[ \J_{\beta'}\sats\varphi(y,w).\]
Note that
for all $w'\in\WO_\gamma$ and all $z\in\RR$, we have
\[ z=y\iff\J_{\beta'}\sats\exists w''\ [w''\in\LO\text{ and }w''\iso 
w'\text{ and }\varphi(z,w'')].\]
By minimality, $\beta'=\delta+\om$,
where $\delta\in\Lim_{0}$.
The symmetry in the real code $w'$ above ensures that there is actually some 
$k<\om$ such that for all $w'\in\WO_\gamma$ and $z\in\RR$, we have
\[ z=y\iff\J_{\delta+k}\sats\exists w''\ [w''\in\LO\text{ and 
}w''\iso w'\text{ and }\varphi(z,w'')].\]
Standard calculations
now show that $y$ is $\OD^{\delta n}$ for some $n<\om$.
\end{proof}

\subsection{Conjectures of Rudominer and Steel}\label{subsec:rudo-steel_conj}

In \cite{twms}, Steel made the following two conjectures:

\begin{conj}[Steel]\label{conj:steel1} Assume $\AD^{L(\RR)}$, and let $\mathcal{M}$ be an inner model operator in $L(\RR)$;
 then for a cone of reals $x$,
 there is a wellorder of $\mathcal{M}(x)$ in $L(\mathcal{M}(x))$.
\end{conj}

\begin{conj}[Steel]\label{conj:steel2} Let $M$
be a countable, $\om$-small,
$(0,\om_1+1)$-iterable premouse;
then there is a wellorder of $\RR\cap M$ in $L(\RR\cap M)$.
\end{conj}

And Rudominer and Steel made the following conjecture  in  \cite{rudo_steel}:

\begin{conj}[Rudominer, Steel, 1999]\label{conj:RS} Assume $\AD^{L(\RR)}$.\footnote{Presumably this is with the background assumption of ZFC in $V$.} Let $M$ be an
iterable
countable  $\om$-small premouse. Then there are $\gamma\in\Lim_0$ and $\beta\in\Lim_0\cup\{\OR\}$, $n<\om$ and $\pi$ such that:
\begin{enumerate}
 \item $\RR^M=\Ss_{\gamma}(\RR^M)\cap\RR$,
 \item $\pi:\Ss_{\gamma}(\RR^M)\to\SS_\beta$ is $\Sigma_n$-elementary, and
 \item there is a wellorder of $\RR^M$ which is $\Delta_{n+1}^{\J_{\gamma}(\RR^M)}(\{x\})$ for some $x\in\RR^M$.
\end{enumerate}
\end{conj}

In this paper we make some progress
toward this conjecture, focusing on a case
of the following weaker variant,
which is just the weaker conjecture stated in \cite{rudo_steel}, except that we add the assumption that $M\sats$``$\om_1$ exists''.
\begin{conj}\label{conj:weak}
Assume ZF + $\AD^{L(\RR)}$.
Let $M$ be a $(0,\om_1)$-iterable
countable $\om$-small premouse satisfying ``$\om_1$ exists''. Then there are $\gamma,\beta,\pi$ as in Conjecture \ref{conj:RS}, except that we only demand
that $\pi$ be $\Sigma_1$-elementary and
the wellorder of $\RR^M$ be definable over $\Ss_\gamma(\RR^M)$ from some $x\in\RR^M$.
\end{conj}

\begin{rem}\label{rem:M,g}
Note that if $\gamma$ witnesses Conjecture \ref{conj:weak}, then $\gamma$ is the least $\gamma'\in\Lim_{0}$ such that
$\RR^M$ is wellordered in $\Ss_{\gamma'+\om}(\RR^M)$,
the largest  $\gamma'\in\Lim_{0}$ such that
$\Ss_{\gamma'}(\RR^M)\sats\AD$,
and also the largest $\gamma'\in\Lim_0$ such that $\Ss_{\gamma'}(\RR^M)$ can be $\Sigma_1$-elementarily embedded
into $\J_{\beta}$ for some $\beta\in\Lim_0$.

Note also that if $\beta$ is taken as large as possible witnessing  Conjecture \ref{conj:weak}, and $\beta\neq\OR$,
then $\beta$ ends an S-gap of $L(\RR)$.
\end{rem}
Rudominer and Steel proved certain instances
of the conjectures above in \cite{rudo_steel}. We will verify some further instances of Conjecture \ref{conj:weak} in this paper. We explain this next.

\begin{dfn}
 Let $M$ be a premouse, $\alpha\in\Lim_{0}$ and $n<\om$.
For $x\in\RR^M$, we say that $M$ is 
\emph{$(\alpha,n+1)$-closed} iff
\[ \OD^{\alpha,n+1}_x\sub M\]
for each $x\in\RR^M$.
We say that $M$ is \emph{strongly $(\alpha,n+1)$-closed} iff
for each $x\in\RR^M$ there is $\xi<\om_1^M$ such that
\[ \OD^{\alpha,n+1}_x\sub M|\xi.\]

The \emph{degree of strong closure} of $M$, if it exists,
is the lexicographically least $(\alpha,n)\in\Lim_0\cross\om$
such that $M$ is not strongly 
$(\alpha,n+1)$-closed.
\end{dfn}

Note that the degree of strong closure of $M$ 
might be $(\alpha,0)$.

\begin{lem}\label{lem:beta_ends_a_gap}
Let $M$ be an $\om$-small mouse.
Then the degree of strong closure $(\beta,n)$ of $M$ exists. Moreover,
$\beta$ ends an S-gap of $L(\RR)$.
\end{lem}
\begin{proof}
Because $M$ is $\om$-small, we have $\RR\inter M\sub\OD^{L(\RR)}$.
So there is $\beta\in\Lim_0$ such that $\RR\inter M\sub\OD^{\beta 1}$,
and so $M$ is not strongly $(\beta,1)$-closed.

Let $(\beta,n)$ be the degree of strong closure of $M$. Let $x\in\RR^M$
be such that for no $\xi<\omega_1^M$
is $\OD^{\beta,n+1}_x\sub M|\xi$.
If $n>0$ then $\OD^{\beta,n}_x\psub\OD^{\beta,n+1}_x$, which yields
that $\beta$ ends an S-gap.
Suppose $n=0$. Let $\left<x_m\right>_{m<\om}$ enumerate the reals of $M$. Then $\beta$ is least
such that $\{x_m\bigm|m<\om\}\sub\OD^{\beta,1}_x$, which also yields that $\beta$ ends an S-gap.
\end{proof}

\begin{dfn}
 Let $M$ be an $\om$-small mouse.
 Then $(\beta^M,n^M)$ denotes
 the degree of strong closure of $M$,
 and $\alpha^M$ is the start of the S-gap which ends at $\beta$.
\end{dfn}

\begin{dfn}
Let $[\alpha,\beta]$ be an S-gap of $L(\RR)$. If $\alpha$ is admissible, we say that $\beta$ is of type
\begin{enumerate}[label=--]
 \item \emph{Weak} iff $[\alpha,\beta]$ is a weak S-gap,
 \item \emph{Strong} iff $[\alpha,\beta]$ is a strong S-gap.
\end{enumerate}

If $\alpha$ is 
projective-like (so $\alpha=\beta$), we say that 
$\beta$ is of type
\begin{enumerate}[label=--]
 \item \emph{Limit-uncountable} iff $\cof^{L(\RR)}(\beta)>\om$,\footnote{
If $V\sats\AC_{\om}(\RR)$ then this is absolute between $V$ and $L(\RR)$.
For $\beta$ ending an S-gap implies that $\rho_\om^{\SS_\beta}=\RR$, and hence there is a 
 surjection $f:\RR\to\SS_\beta$ which is definable from parameters over $\SS_\beta$, so $f\in L(\RR)$. But then if $\cof^V(\beta)=\om$, then we can find a sequence $X=\left<x_n\right>_{n<\omega}$ of reals such that $f``X$ is cofinal in $\beta$, but $X\in L(\RR)$.}
 \item \emph{Limit-countable} iff $\beta$ is a limit of limits and 
$\cof^{L(\RR)}(\beta)=\om$,
 \item \emph{Successor-inadmissible} iff $\beta=\gamma+\om$ and $\gamma$ starts (and ends) an 
inadmissible S-gap,
 \item \emph{Successor-weak} iff $\beta=\gamma+\om$ and $\gamma$ 
ends a weak 
S-gap,
 \item \emph{Successor-strong}  iff $\beta=\gamma+\om$ and $\gamma$ 
ends a strong 
S-gap.
\end{enumerate}

Type \emph{Limit} means either Limit-uncountable or Limit-countable,
and likewise for \emph{Successor}.

If $M$ is an $\om$-small mouse,
then we say that $M$ is type \emph{Limit-uncountable} iff $\beta^M$ is type Limit-uncountable, etc.
\end{dfn}

In this paper, we prove some key lemmas working toward the proof of the following
instances of the weak conjecture;
some remaining details of the proof of \ref{tm:conjectures_hold} are yet to be written down, but we believe that what remains is straightforward:\footnote{We may later add the (expected) proof of \ref{tm:conjectures_hold} to this paper, or will otherwise put it in a separate one.}
\begin{etm}\label{tm:conjectures_hold}
Conjecture \ref{conj:weak} holds under
the assumption that $M$ is of weak, 
strong, successor-weak or successor-strong type;
moreover, given $M$ as there, then $M$ is not of strong type.

Conjecture \ref{conj:steel2},
and its relativization above an arbitrary real, holds
assuming that  $M\sats$``$\omega_1$ exists''
and $M$ is of weak, strong, successor-weak or successor-strong type.
\end{etm}

This implies the corresponding
instances of  Conjecture \ref{conj:steel1}, as explained in \cite{twms}. 

The proof of Expected Theorem \ref{tm:conjectures_hold} will in fact yield something
intermediate between Conjectures \ref{conj:weak} and \ref{conj:RS}
(in the cases mentioned in \ref{tm:conjectures_hold}).

The fact that $M$ is not of strong type
follows from an old argument of Martin.
We will show this in Theorem \ref{tm:Martin}, but will first need to develop some fine structure. However,
we will not actually use Theorem \ref{tm:Martin} in the paper, as in the end we will also give an alternate, inner-model-theoretic proof that $M$ is not of strong type, using methods like those for the other three cases. Thus, in the end we will handle all four cases quite uniformly.

The  approach to the proof of Conjecture \ref{conj:weak} in the case that $M$ is weak, is as follows.
Suppose that $[\alpha,\beta]=[\alpha^M,\beta^M]$ is weak. We first find a mouse $P\in\HC^M$
with the properties sketched in \S\ref{subsec:gaps} with respect to the S-gap $[\alpha,\beta]$. We find  $\M_{\beta^*}$, 
as sketched in that section,
which encodes $\SS_\beta$,
and show that $\M_{\beta^*}$ is a kind of derived model of an $\RR$-genericity iterate $P'$ of $P$.
We also show that we get a similar picture
with $\RR^M$ replacing $\RR$;
we get an $\RR^M$-genericity
iterate $\bar{P}'$ of $P$,
and a model $\bar{\M}_{\bar{\beta}^*}$,
a natural analogue of $\M_{\beta^*}$,
but whose reals are just $\RR^M$,
such that $\bar{\M}_{\bar{\beta}^*}$
is a corresponding derived model of $\bar{P}'$. Given $n<\om$,
we can moreover find a variant $P_n'$ of $P'$, as above,
and such that $P_n'$ is an iterate of $\bar{P}'$,
via a tree $\Tt_n$, which is above $\delta_n^{\bar{P}'}$. Using the resulting iteration maps and the symmetry of the derived model,
we obtain an induced embedding $\sigma:\bar{\M}_{\bar{\beta}^*}\to \M_{\beta^*}$. Also,
$\bar{\M}_{\bar{\beta}^*}$ encodes
a level $\Ss_{\bar{\beta}}(\RR^M)$
just as $\M_{\beta^*}$ encodes $\SS_\beta$, and it will follow
that $\sigma$ induces a $\Sigma_1$-elementary $\pi:\Ss_{\bar{\beta}}(\RR^M)\to\SS_\beta$. We will also show that $M|\om_1^M$ is definable from parameters over $\Ss_{\bar{\beta}}(\RR^M)$, completing the proof in the weak S-gap case.
The other cases inolve these kinds of methods, but also some other things come into play.

We finish this section with a couple of simple observations,
already noted in some form in \cite{rudo_steel}.

\begin{lem}\label{lem:n=0_case} Let $M$ be an $\om$-small mouse. Suppose $n^M=0$ and $\beta=\beta^M>0$. Then 
$\beta$ starts and ends a projective-like S-gap,
$\cof^{L(\RR)}(\beta)=\om$ and $\beta$
is of type Limit-countable or of type Successor.
\end{lem}
\begin{proof}
 $M$ is not strongly $(\beta,1)$-closed,
but is strongly $(\delta,m+1)$-closed for all $\delta<\beta$
and $m<\om$. For each $x\in\RR$, we have
\[ \OD^{\beta 1}_x=\bigcup_{\delta<\beta\text{ and 
}n<\om}\OD^{\delta n}_x.\]
It follows that $M$ is $(\beta,1)$-closed.
Fix $x_0\in\RR^M$ witnessing that $M$ is not strongly
$(\beta,1)$-closed.

Note that for cofinally many $(\delta,n)<_\lex(\beta,0)$,
we have $\OD^{\delta,n+2}_{x_0}\neq\OD^{\delta,n+1}_{x_0}$.
This implies that $\beta$ starts an S-gap.
By Lemma \ref{lem:beta_ends_a_gap}, $\beta$ also ends an S-gap.
And since $\RR^M$ is countable
and $\beta$ is least such that \[\RR^M\sub\bigcup_{\delta<\beta\text{ and }n<\om}\OD^{\delta n}_{x_0},\]
$\cof^{L(\RR)}(\beta)=\om$ and $\beta$ is of the claimed type.
\end{proof}

\begin{lem}\label{lem:betag_ends_a_gap}
Let $M$ be an $\om$-small mouse.
Let $(\beta,n)=(\beta^M,n^M)$.
Then there is $x\in\RR$ such that 
$\OD^{<\beta}_{x}\psub\OD^{\beta}_{x}$,
and if $n>0$ we can take $x\in M$.
\end{lem}
\begin{proof}
If $n=0$ then  use \ref{lem:n=0_case} and standard calculations
(that is, for each real $x$, $\Sigma_1^{\SS_\beta}(\{x\})$ has the scale 
property, and apply the Periodicity Theorems with an $x$ which codes enough 
information).

Suppose $n>0$. So $M$ is strongly $(\beta,n)$-closed
but not strongly $(\beta,n+1)$-closed. Fix $x_0\in \RR^M$
witnessing the latter.
Then
\[  \OD^{<\beta}_{x_0}\sub\OD^{\beta 
n}_{x_0}\psub\OD^{\beta,n+1}_{x_0}\sub\OD^\beta_{x_0}.\qedhere\]
\end{proof}

\subsection{Structure of paper}

The topic and arguments in the paper
are related to the Woodin's Mouse Set Theorem (see \cite{twms}),  Rudominer's work in \cite{rudo_mouse_sets} and \cite{rudo_inner_model_operators_in_LR}, and also particularly 
to the methods in the Steel-Rudominer paper \cite{rudo_steel}. 

In \S\ref{sec:background} we cover some background inner model theory; the main content here is the discussion of mouse witnesses in \S\ref{subsec:mouse_witnesses}, which is mostly standard material, but essential for later arguments. By including it here,  the reader can avoid digging through sources which contain a lot of other material not relevant to this paper. In \S\ref{sec:start_of_gap},
given an admissible S-gap $[\alpha,\beta]$ of $L(\RR)$,
we identify a real $x$
and an $x$-mouse $P$
corresponding either to the end of the S-gap or just beyond it.
In the context of proving Theorem \ref{tm:conjectures_hold},
will be interested in particular values of $\beta$, and $P$ will also relate tightly to $M$.
The arguments in \S\ref{sec:start_of_gap}
are mostly due to the second author,
from work in 2005, written at the time in an email correspondence between him and Ralf Schindler. Some further observations in this section were added by the first author later.
In \S\ref{sec:M-hierarchy},\ref{sec:through_gap} 
we introduce
the $\M$-hierarchy corresponding to $[\alpha,\beta]$, define  the relevant derived model construction and Prikry forcing, and analyze the associated forcing relations.
The arguments in \S\ref{sec:through_gap}
are adaptations of some standard ones,
such as those used in the analysis of $\HOD^{L(\RR)}$ in \cite{HOD_as_core_model}. 
The adaptation of those methods to the present context, done in  \S\S\ref{sec:M-hierarchy},\ref{sec:through_gap}, are due to the first author, with the main idea having been found in 2013, and refined later in  2019--2022.

\subsection{Acknowledgements}
The first author thanks the organizers of the \emph{Workshop in Set Theory},  Oberwolfach, 2022, the \emph{Muenster conference in inner model theory} 2022, and the \emph{Advances in set theory} conference in Jerusalem, 2022, for the opportunity to present some of the ideas from the paper.

The first author was partially
supported by the Deutsche Forschungsgemeinschaft (DFG, German Research Foundation) under Germany's Excellence Strategy EXC 2044--390685587, Mathematics M\"unster: Dynamics--Geometry--Structure.

\section{Inner model theoretic background}\label{sec:background}

\subsection{Tame projecting mice}

\begin{dfn}
 For $\eta\in\OR$,
 an \emph{$\eta$-projecting premouse (of degree $m<\om$)}
 is a premouse $M$ such that
 $\eta$ is a strong cutpoint of $M$,
 $M$ is $m$-sound, and
 $\rho_{m+1}^M\leq\eta<\rho_m^M$. 
 Note $(M,\eta)$ determines $m$.
 If $\eta$ is known from context, we may
 just say \emph{projecting premouse} instead of \emph{$\eta$-projecting}.
\end{dfn}

\begin{rem}
 We will often deal with projecting premice.
 The following lemma tells us that in this context
 and 
 assuming tameness, (i) normal iterability above $\eta$ (more precisely, above-$\eta$, $(m,\om_1+1)$-iterability) automatically yields (ii)
 stacks iterability above $\eta$ (more precisely, above-$\eta$, $(m,\om_1,\om_1+1)^*$-iterability).
 So  iterability hypotheses will often be stated in form (i) as  opposed to (ii). However,
 this fact depends on significant background material (from \cite{fsfni_v4}, \cite{fullnorm} and \cite{iter_for_stacks}), which isn't particularly relevant
 to this paper. One could just strengthen 
 the iterability hypotheses throughout,
 changing form (i) to (ii),
 and avoid the appeal to that background material.
\end{rem}

\begin{lem}\label{lem:tame_projecting_stacks}
Let $M$ be a tame $\eta$-projecting premouse
of degree $m$. Suppose $M$ is above-$\eta$, $(m,\om_1+1)$-iterable. Then:
\begin{enumerate}
 \item\label{item:normal_strat_unique}
  $M$ has a unique above-$\eta$, $(m,\om_1+1)$-iteration strategy $\Sigma$,
\item\label{item:Sigma_extends_to_stacks} $\Sigma$ extends to an above-$\eta$ $(m,\om_1,\om_1+1)^*$-strategy $\Sigma^{\stk}_{\min}$ for $M$
with full normalization, as in \cite{fullnorm}
\item\label{item:fs_above_eta} $M$ is $(m+1)$-solid above $\eta$
and $(m+1)$-universal above $\eta$.
\end{enumerate}
\end{lem}
Note that we do not assume that $M$ is countable here, although in our application in this paper, $M$ will always be countable.
\begin{proof}
Part \ref{item:fs_above_eta} holds by \cite{fsfni_v4}.
And part \ref{item:Sigma_extends_to_stacks}
follows from part \ref{item:normal_strat_unique}
by \cite{fullnorm}.

So it suffices to prove part \ref{item:normal_strat_unique}. So suppose $\Sigma,\Gamma$ are two distinct such strategies,
and let $\Tt$ be a countable limit length
tree via $\Sigma\cap\Gamma$
such that $b=\Sigma(\Tt)\neq c=\Gamma(\Tt)$.
Note that by taking a countable hull containing
these objects,
we may assume that $M$ is countable
(use the pullbacks
of $\Sigma,\Gamma$ to iterate the countable version). So we can successfully compare
the phalanxes $\Phi(\Tt,b)$ and $\Phi(\Tt,c)$.
Because $\rho_{m+1}^M\leq\eta$ and $\eta$ is a strong cutpoint, and by standard
fine structure (using part \ref{item:fs_above_eta}) and  tameness,
there is a Q-structure $Q_b=Q(\Tt,b)\ins M^\Tt_b$
and a Q-structure $Q_c=Q(\Tt,c)\ins M^\Tt_c$ for $M(\Tt)$, $\delta(\Tt)$ is Woodin and a strong
cutpoint of $Q_b,Q_c$, and the comparison just mentioned is equivalent to a comparison of $Q_b$ with $Q_c$, and this is above $\delta(\Tt)$.
By the Zipper Lemma,  $Q_b\neq Q_c$.
Therefore at least one of $Q_b,Q_c$
is non-$\delta(\Tt)$-sound; say it is $b$.
Then $b$ does not drop in model or degree.
So $\rho_{m+1}^{Q_b}=\rho_{m+1}^M\leq\eta<\delta(\Tt)$ and $\deg^\Tt(b)=m$.
The comparison ends with a common iterate $P$
of $Q_b,Q_c$, with no dropping in model or degree above $Q_b,Q_c$, and note then that $Q_c=M^\Tt_c$
and $\rho_{m+1}^{Q_c}=\rho_{m+1}^{Q_b}<\delta(\Tt)<\rho_m^{Q_c}$. (If $\delta(\Tt)=\rho_m^{Q_c}$ then we easily get $Q_b=Q_c$.) So now
\[\Hull^M_{m+1}(\eta\cup\pvec_{m+1}^{M})\iso\Hull^{Q_b}_{m+1}(\eta\cup\pvec_{m+1}^{Q_b})\iso\Hull^P_{m+1}(\eta\cup\pvec_{m+1}^P)
 \iso\Hull^{Q_c}_{m+1}(\eta\cup\pvec_{m+1}^{Q_c}),
\]
which implies that $c$ also does not drop in model or degree. By tameness,
$\delta(\Tt)$ is not a limit of Woodins
of $P$, so let $\gamma<\delta(\Tt)$ bound the Woodins of $P$ which are ${<\delta(\Tt)}$.
Let
\[ \theta=\sup\delta(\Tt)\cap\Hull^P((\gamma+1)\cup\pvec_{m+1}^P).\]
By Zipper Lemma, $\theta<\delta(\Tt)$.
Let
\[ C=\cHull^P(\theta\cup\pvec_{m+1}^P)\]
and $\pi:C\to P$ be the uncollapse.
Since $\delta(\Tt)$ is the least Woodin of $P$
which is $>\gamma$, we get $\delta(\Tt)\in\rg(\pi)$ (if $m>0$ this is clear or $P$ is active this is clear; if $m=0$ and $P$ is passive, then $P$ has a largest cardinal $\kappa$, $\kappa\geq\delta(\Tt)$,
and $p_1^P\not\sub\kappa$, which is easily enough). We have $\crit(\pi)=\theta<\delta(\Tt)$
and $\pi(\theta)=\delta(\Tt)$, $\theta$ is a limit
cardinal of $P$ and of $C$, and $C|\theta=P|\theta$. By tameness,
$\delta(\Tt)$ is a strong cutpoint of $P$,
so $\theta$ is a strong cutpoint of $C$.
Condensation gives that $C||\theta^{+C}=P||\theta^{+C}$. And $\theta$ is Woodin in $C$,
but not in $P$, by choice of $\gamma$.
So $\theta^{+C}<\theta^{+P}$.
Note that $C$ is above-$\theta$ iterable,
via lifting to a continuation of $\Tt\conc b$
with $i_{Q_bP}^{-1}\com\pi$ (this only
uses the normal strategies we have). Letting
$J\pins P$ be such that $\rho_\om^J=\theta$
and $J||\theta^{+J}=C||\theta^{+C}$,
note that also $J$ is above-$\theta$ iterable.
But as both are $\theta$-sound and project
to $\theta$, which is a common strong cutpoint,
it follows that $C=J$. So $C\in P$,
but from $C$ we can obtain $\core_{m+1}(P)$,
a contradiction.\end{proof}

\subsection{P-construction}

\begin{dfn}
 Let $\Tt\in\HC^M$ be an iteration tree on $M|\om_1^M$.
 We say that $\Tt$ is \emph{P-standard} iff:
 \begin{enumerate}
  \item $\Tt$ is according to $\Sigma_{M|\om_1^M}$;
  let $N\pins M|\om_1^M$ with $\Tt$ on $N$ and $\rho_\om^N=\om$.
  \item $\Tt$ has limit length; let $\delta=\delta(\Tt)$.
  \item $M(\Tt)$ is not a Q-structure for itself (that is,  $M(\Tt)\sats\ZFC$ 
and $\J(M(\Tt))\sats$``$\delta$ is Woodin'').
\item $M|\delta$ has largest cardinal $\om$.
\item $N\in M|\delta$ and $\Tt,M(\Tt)$ are definable from parameters
over $M|\delta$.
\item $M|\delta$ is generic for $\BB_{\delta\xi}^{\J(M(\Tt))}$, for some 
$\xi<\delta$.\qedhere
 \end{enumerate}
\end{dfn}

\begin{dfn}
 Let $\Tt\in\HC^M$ be P-standard and $\delta=\delta(\Tt)$.
 The \emph{P-construction} $\mathscr{P}^M(M(\Tt))$ of $M$ over $M(\Tt)$
 is the structure $P$ defined as follows.
 We will have $\OR^P\leq\omega_1^M$.
 Set $M(\Tt)\ins P$. Given $\nu\in(\delta(\Tt),\OR^P]$,
 $P|\nu$ is active iff $M|\nu$ is active.
 And if $P|\nu$ is active then $F^{P|\nu}=F^{M|\nu}\rest(P||\nu)$. 
We define $P$ as the least such stage such that $P$ fails to be a premouse, or $P$ is 
a Q-structure for $M(\Tt)$.
\end{dfn}

Then we have (see for example \cite{sile}):

\begin{lem}
 Let $\Tt,\delta$ be as above. Then $Q=\mathscr{P}^M(M(\Tt))$
 is well-defined, and $Q=Q(\Tt,b)$ where $b=\Sigma_{M|\om_1^M}(\Tt)$
 is the correct branch.
\end{lem}

\begin{rem}\label{rem:OR^Q_leq_OR^R}
 Note that $\OR^Q\leq\OR^R$, where $R\pins M$ is least such that
 $\delta\leq\OR^R$ and $\rho_\om^R=\om$. (Otherwise letting $Q=Q(\Tt,b)$,
 which is the output of the P-construction, we 
get $R\in Q[M|\delta]$, so $\delta$ is countable in $Q[M|\delta]$,
but $\delta$ is regular there, by the $\delta$-cc.)
\end{rem}

We will be modifying the proof of the following lemma, due to Steel. We will
also need to apply the lemma itself. 

 \begin{lem}[Steel]\label{lem:forcing}Let $N$ be an $n$-sound premouse. Let
$\eta$ be a cardinal strong cutpoint of $N$. Let $\QQ\in N|\eta$ be a forcing
and let $G\sub\QQ$ be $N$-generic. Then
 \begin{enumerate}
 \item $N[G]$ can be reorganized as an $n$-sound $(N|\eta,G)$-premouse, whose
extender sequence (which is above $\eta$) is given by the standard method of
extending extenders (on $\es_+^N$) to small forcing extensions.
 \item If $\eta<\rho_n^N$ then $\rho_n^{N[G]} = \rho_n^N$ and $p_n^{N[G]} =
p_n^N$.
 \item If $\rho_n^N\leq\eta$ then $\rho_n^{N[G]}=(N|\eta,G)$ and
$p_n^{N[G]}=p_n^N\cut(\eta+1)$.
 \item If $\eta<\rho_n^N$ then the $\rSigma_{n+1}$ strong forcing
relation\footnote{Let $\varphi$ be $\rSigma_{n+1}$ and let $x\in(V^\QQ)^N$.
Suppose $n>0$ and $\varphi(v)$ has the form
\[ \ex y,t\left[T_n(y,t)\ \&\ \ex w\psi(v,y,t,w)\right]\]
 where $\psi$ is $\rSigma_0$. Then $q\sforces{\QQ}{n+1,\str}\varphi(x)$ iff
there are $\QQ$-names $y,t,w\in N$ such that
\[ q\subforces{\QQ}T_n(y,t)\ \&\ \psi(x,y,t,w); \]
 here $T_n$ is the $\rSigma_n$ theory predicate of the extension $N[G]$. If
$n=0$ and $\varphi(v)$ has the form $\ex w\psi(v,w)$
 where $\psi$ is $\rSigma_0$, then $q\sforces{\QQ}{1,\str}\varphi(x)$ iff there
is a $\QQ$-name $w$ such that
$q\subforces{\QQ}\psi(x,w)$.}
$\sforces{\QQ}{n+1,\str}$
 is $\rSigma_{n+1}^N(\{\delta\})$, and the strong $\rSigma_{n+1}$ forcing
theorem holds. (That is, for $\rSigma_{n+1}$ formulas $\varphi$ and
$x\in(V^\QQ)^N$, we have
$N[G]\sats\varphi(x^G)$
iff $q\sforces{\QQ}{n+1,\str}\varphi(x)$ for some $q\in G$.)
 \item If $\eta\leq\rho_\om^N$ then the forcing theorem holds and for each $k$,
the $\rSigma_k$ forcing relation is definable over $N$.
\end{enumerate}
 \end{lem}
 
 Versions of the this lemma (and its proof) have appeared elsewhere, such as in
 \cite{sile} and
\cite{scalesK(R)}.

\begin{dfn}\label{dfn:P-con}
 Let $N,\eta,R$ be such that:
 \begin{enumerate}[label=--]
  \item $N,R$ are premice,
  \item  $\eta<\OR^N$ is a strong cutpoint of $N$ and $N|\eta$ is passive,
  \item  $R\sub N|\eta$
and $R$ is definable from 
parameters over $N|\eta$,
\item $R\sats\ZFC$ and $\J(R)\sats$``$\eta$ is Woodin''
and $N|\eta$ is generic over $\J(R)$ for the $\eta$-generator extender algebra of $\J(R)$ at $\eta$.
\end{enumerate}
Let $\gamma\in[\eta,\OR^N]$
and $N'=N||\gamma$ or $N'=N|\gamma$.
The \emph{P-construction}
$\mathscr{P}^{N'}(R)$
of $N'$ over $R$,
if it is well-defined, is the premouse $P$ such that:
\begin{enumerate}[label=--]
 \item $\OR^P=\gamma$,
 \item $R\ins P$,
 \item $\es^P_\xi=\es^{N'}_\xi\rest(P||\xi)$ for all $\xi\in[\eta,\gamma]$, and
 \item $P\sats$``$\eta$ is Woodin''.\qedhere
\end{enumerate}
\end{dfn}

\begin{lem}
 Let $N,\eta,R$ be as in Definition \ref{dfn:P-con}.
 There is a largest $\gamma$
 such that  $\mathscr{P}^{N|\gamma}(R)$
 is well-defined.
\end{lem}

\subsection{Mouse set theorem and mouse witnesses}\label{subsec:mouse_witnesses}
In this section we review some mostly standard material on the mouse set theorem and on mice witnessing 
$\Sigma_1^{\SS_\alpha}(\RR)$ facts.

\begin{rem}
 We will usually talk about $(m,\om_1+1)$-iterability
 and \[ (m,\om_1,\om_1+1)^*\]
 -iterability
 in this paper, whereas in the $\AD$ context,
 it is common to talk about $(m,\om_1)$-iterability
 and $(m,\om_1,\om_1)^*$-iterability
 instead. Of course under $\ZF+\AD$, these are equivalent.
 We will also be interested in strategies in arbitrary $\SS_\alpha$,
 however, which of course can model much less than $\ZF$ (but still $\AD$).  But also in these models, the ``$\om_1$''
 is equivalent to the ``$\om_1+1$''. This is because
 every $X\in\pow(\om_1)\cap L(\RR)$
is constructible from a real, and in fact,
there is a sharp $x^\#$ and an iteration $j:x^\#\to N$
and $\bar{\Tt}\in x^\#$ such that $j(\bar{\Tt})=\Tt$,
and hence a further iterate $\Ult(N,F^N)$
containing a $\Tt$-cofinal branch. These things are all low-level projectively definable, so if $\Sigma\in\SS_\alpha$
is an $(m,\om_1)$-strategy, then the extension $\Sigma'$
to an $(m,\om_1+1)$-strategy is also in $\SS_\alpha$,
and likewise if $\Sigma$ is definable from parameters
over $\SS_\alpha$, then $\Sigma'$ is definable
at essentially the same level of complexity as is $\Sigma$.
Likewise for $(m,\om_1,\om_1+1)^*$. Only
if we need to be precise about this level of complexity
might it be relevant to consider $(m,\om_1)$-iterability.
\end{rem}

\begin{dfn}
 Let $\alpha\in\Lim$. Let $\Gamma_\alpha=\Sigma_1^{\SS_\alpha}$.
 Write $\Gammag=\Gamma_{\alphag}$.
 
 Given $x\in\RR$ or a transitive $x\in\HC$, let $\Lp_\alpha(x)=\Lp_{\Gamma_\alpha}(x)$ be the stack
 of all sound $x$-mice $N$ which project to $x$ and are $(\om,\om_1+1)$-iterable 
via a strategy in $\SS_\alpha$.
Such sound projecting mice $N$ have a unique such strategy $\Sigma_N$, and $\Sigma_N$ extends naturally to an 
$(\om,\om_1,\om_1+1)^*$-strategy $\Sigma^{\stk}_N$, and $\Sigma^{\stk}_N\rest\HC$ is projective in $\Sigma\rest\HC$, by \cite{iter_for_stacks}.
(Also under $\AD$, any $(\om,\om_1)$-strategy
extends uniquely to an $(\om,\om_1+1)$-strategy.)
Given $x\in\RR$, let $C_\alpha(x)=C_{\Gamma_\alpha}(x)=\OD^{<\alpha}_x$.
Likewise for transitive $x\in\HC$.
\end{dfn}

\begin{fact}[Mouse set theorem, Woodin]
 Let $\alpha$ be a limit of limits. Then for each $x\in\HC$,
 we have $C_{\Gamma_\alpha}(x)=\pow(x)\inter\Lp_{\Gamma_\alpha}(x)$.
\end{fact}

We now proceed to mouse witnesses,
which is the main content of this section.
Recall that $\Sigma_1^{\J_{\beta+\om}}$
is uniformly equivalent to $\oplus_{n<\om}\Sigma_n^{\J_{\beta}}$, in a natural sense. In connection with this we make the following definition: 

\begin{dfn}
Let $(\varphi,n)\mapsto\gamma_{\varphi,n}$
be a recursive function with domain $\om\cross\om$,
such that for each $(\varphi,n)\in\om\cross\om$,
if $\varphi=\varphi(\vec{x})$ is a $\Sigma_1$ formula of $\Ll_{L(\RR)}$ in free variables $\vec{x}$,
then $\gamma_{\varphi,n}$ is the natural formula
of $\Ll_{L(\RR)}$ in the same free variables,
such that whenever $(M,\RR^M)$ is transitive and $\vec{a}\in M^{<\om}$, then $(\Ss_n(M),\RR^M)\sats\varphi(\vec{a})$
iff $(M,\RR^M)\sats\gamma_{\varphi,n}(\vec{a})$.
Let $k_n\in(0,\om)$ be least such that $\gamma_{\varphi,n}$ is $\Sigma_{k_n}$.
\end{dfn}

\begin{dfn}
Let $\varphi\in\Ll_{L(\RR)}$ be $\Sigma_1$ and $n<\om$.
Let $\psi_{\varphi,n}(\dot{x},\dot{m},\dot{t})$ be the natural
$\Pi^1_{4}[\dot{m},\dot{t}]$ formula (in free variable $\dot{x}$,
representing an element of $\RR$,
and predicates $\dot{m},\dot{t}$, representing  subsets of $\RR$) asserting\footnote{It is appropriate to have $\SS_\gamma$ coded by a set of reals, because by the minimality of $\gamma$, $\SS_\gamma$ must project to $\RR$, and it is sound.}
\[ \text{``}\dot{m}\text{ is a model }\iso\J_\gamma\text{ where }\gamma\in\OR\text{ 
is least such that }\J_\gamma\sats\gamma_{\varphi,n}(\dot{x}), \]
\[ \text{ and }\dot{t}=\Th_{\Sigma_{k_n}}^{\SS_\gamma}\text{''}.\]
That is, $\psi_{\varphi,n}(\dot{x},\dot{m},\dot{t})$
makes the following assertions:
\begin{enumerate}
\item $\dot{m}$ codes a model
in the language of $L(\RR)$
(with binary relations $=^{\dot{m}}$ (an equivalence relation) and $\in^{\dot{m}}$,
and interpretation  $\RR^{\dot{m}}$ of constant $\dot{\RR}$, where $\in^{\dot{m}}$
and $\RR^{\dot{m}}$ both respect ${=^{\dot{m}}}$),\footnote{It is better not to demand
that ${=^{\dot{m}}}$ be actual equality,
because when $\dot{m}$ is defined in the natural manner, it will not be actual equality.}

\item $(V_{\om+1}^{\dot{m}},{\in^{\dot{m}}},{=^{\dot{m}}})$ is isomorphic to $(V_{\omega+1},{\in},{=})$,
\item $\dot{m}\sats$``Extensionality + Pairing + $V=L(\RR^{\dot{m}})$'',
\item for each $\Sigma_{k_n}$ formula $\varphi\in\Ll_{L(\RR)}$
and $x\in\RR$, letting $x'\in\dot{m}$
be  isomorphic to $x$,
we have $\dot{m}\sats\varphi(x')$ iff $(\varphi,x)\in\dot{t}$\footnote{
Note that this item 
does not push the complexity of
$\psi_{\varphi,n}$ up substantially
(in particular, it does not particularly
depend on $n$), because we can use the usual recursive trick to express that $\dot{t}$ is a satisfaction relation; that is, we express that it is correct about atomic formulas, and then simply express
that it satisfies the right recursive properties
up to $\Sigma_{k_n}$ formulas. Of course
we could similarly express that $\dot{t}$
is the entire $\Sigma_\om$ theory, but such  theories
will not be available to us as sets in the proof later.}
\item $(\gamma_{\varphi,n},\dot{x})\in \dot{t}$
and $(\gamma'_{\varphi,n},\dot{x})\notin\dot{t}$,
where $\gamma'_{\varphi,n}$ says ``there is a proper segment of me (in the $L(\RR)$ hierarchy)
which satisfies $\gamma_{\varphi,n}(\dot{x})$'',
\item $\dot{m}$ is wellfounded.\footnote{That this
assertion is
 projective in $\dot{m}$ (or in $(\dot{m},\dot{t})$) seems to make use of $\DC_{\RR}$.}
\end{enumerate}

Write
\[ \psi_{\varphi,n}(\dot{x},\dot{m},\dot{t})\iff\all^\RR x_1\ \exists^\RR x_2\ \all^\RR 
x_{3}\ \exists^\RR x_{4}\ [\varrho_{\varphi,n}(\vec{x},\dot{x},\dot{m},\dot{t})] \]
with $\varrho_{\varphi,n}$ arithmetic,
where $\vec{x}=(x_1,\ldots,x_4)$,
and $(\varphi,n)\mapsto\varrho_{\varphi,n}$  recursive.
\end{dfn}

\begin{dfn}
 Write $\CC_\delta=\Coll(\om,\delta)$ (the forcing).
\end{dfn}

\begin{dfn}
 Let $X\in\HC$ be transitive
 and  $N$ be an $\om$-small $X$-premouse.
 Let $x\in\RR\cap\Ss_\om(X)$. Let $\vec{\delta}=(\delta_4,\ldots,\delta_0),\dot{S},\dot{T}\in N$. Let $\varphi(\dot{x})\in\Ll_{L(\RR)}$ be $\Sigma_1$,
 with free variable $\dot{x}$. Let $n<\om$.
 We say that $(N,\vec{\delta},\dot{S},\dot{T})$ is a \emph{pre-$(\varphi(x),n)$-witness}
 iff  
 \[ \delta_0<\ldots<\delta_4\in\OR^N,\]
 each $\delta_i$ is Woodin in $N$,
 $N\sats\ZF^-$+``$\delta_4^{+}$ exists,
 $\dot{S},\dot{T}$ are $\CC_{\delta_0}$-names,
 and for some $\lambda\in\OR^N$,
\[ \forces_{\CC_{\delta_0}}\ \dot{S},\dot{T}\text{ are }\CC_{\delta_4}\text{-absolutely 
complementing trees on }\om\cross\lambda\]
and writing $\vec{x}=(x_1,\ldots,x_4)$,
\[ \begin{array}{ll}\forces_{\CC_{\delta_0}}\ \forces_{\CC_{\delta_1}}\all^\RR x_1\
\forces_{\CC_{\delta_2}}\exists^\RR x_2\  \forces_{\CC_{\delta_3}}\all^\RR x_{3}\  
\forces_{\CC_{\delta_{4}}}\exists^\RR x_{4}\
\varrho_{\varphi,n}(\vec{x},x,p[\dot{S}]_0,p[\dot{S}]_1)\text{''}.\end{array}\]

A \emph{$(\varphi(x),n)$-witness} is a $(0,\om_1+1)$-iterable 
pre-$(\varphi(x),n)$-witness.

A \emph{minimal $(\varphi(x),n)$-witness} is a $(\varphi(x),n)$-witness
$(N,\vec{\delta},\dot{S},\dot{T})$ such that if $(N',
\vec{\delta}',\dot{S}',\dot{T}')$ is a pre-$(\varphi(x),n)$-witness
and $N'\ins N$ then $N'=N$
and $(\vec{\delta},\dot{S},\dot{T})\leq_N(\vec{\delta}',\dot{S}',\dot{T}')$.

A pre-$(\varphi(x),n)$-witness $(N,\vec{\delta},\dot{S},\dot{T})$ is \emph{above-$\delta$}
iff $\delta$ a strong cutpoint of $N$ and $\delta<\delta_0$ where $\vec{\delta}=(\delta_4,\ldots,\delta_0)$.

We will also just say that $N$ is a \emph{(minimal) (pre-)$(\varphi(x),n)$-witness}, if there is $(\vec{\delta},\dot{S},\dot{T})$ witnessing
that $(N,\vec{\delta},\dot{S},\dot{T})$ is such.

A \emph{(minimal) (pre-)$\varphi(x)$-witness}
is a (minimal) (pre-)$(\varphi(x),n)$-witness for some $n<\om$. (So we don't minimize on $n$,
as it's not necessary, though it would be more natural to do so.)
\end{dfn}
\begin{rem}\label{rem:iterate_pre_varphi-witness}
Let $(N,\dot{S},\dot{T})$ be a pre-$(\varphi(x),n)$-witness.
 Note that non-dropping degree $0$ iteration maps on $N$
are fully elementary, as $N\sats\ZF^-$.
Thus, minimality is preserved by such maps,
as are the witnessing objects and their minimality.
\end{rem}

The first lemma below is proved by comparison, and using that $\delta_0,\delta_0'$ are strong cutpoints of $N,N'$ respectively, by tameness: 

\begin{lem}\label{lem:compare_min_varphi_witnesses}
 Let $(N,\vec{\delta},\dot{S},\dot{T}),(N',\vec{\delta}',\dot{S}',\dot{T}')$ be countable minimal $(\varphi(x),n)$-witnesses
 over the same $X$,
 and $\Sigma,\Sigma'$ be $(0,\om_1+1)$-strategies for $N,N'$.
 Then there is a common non-dropping iterate $P$ of $N,N'$,
 via $\Sigma,\Sigma'$ respectively,
 and letting $j:N\to P$ and $j':N'\to P$ be the iteration maps,
 then $j(\vec{\delta},\dot{S},\dot{T})=j'(\vec{\delta}',\dot{S}',\dot{T}')$.
 If, moreover, $N|\delta_0=N'|\delta_0'$
 where $\vec{\delta}=(\delta_4,\ldots,\delta_0)$ and likewise $\delta_0'$, then  $\delta_0=\delta_0'<\crit(j),\crit(j')$.
\end{lem}

\begin{lem}\label{lem:min_varphi(x)-witness_stacks_iterability}
 Let $N$ be a countable minimal $(\varphi(x),n)$-witness over $X$.
 Then $N$ has a unique $(0,\om_1+1)$-strategy, and its
 unique such strategy extends to a
$(0,\om_1,\om_1+1)^*$-strategy.
\end{lem}
\begin{proof}
Supposing $\Sigma,\Gamma$ are two distinct $(0,\om_1+1)$-strategies
for $N$, we can find a countable limit length tree $\Tt$
via $\Sigma\cap\Gamma$, such that $b=\Sigma(\Tt)\neq\Gamma(\Tt)=c$. Now compare the phalanxes $\Phi(\Tt\conc b)$
and $\Phi(\Tt\conc c)$, producing trees $\Uu$ and $\Vv$ respectively. The minimality and standard
fine structural arguments show that we get a common final model $P$, $P$ is above $M^\Tt_b$ in $\Uu$
and above $M^\Tt_c$ in $\Vv$, and there are no drops along $b\conc b^\Uu$ or along $c\conc b^\Vv$. Therefore \[ \Hull^P(X)=\Hull^{M^\Tt_b}(X)=\Hull^{M^\Tt_c}(X) \]
and this hull is bounded in $\delta(\Tt)$, by the Zipper Lemma.
Also $\delta(\Tt)$ is Woodin in $P$, so $\delta(\Tt)=i^\Tt_b(\varepsilon)=i^\Tt_c(\varepsilon)$
for some Woodin $\varepsilon$ of $N$ (since $N$ is $\om$-small). Since $i^\Tt_b,i^\Tt_c$ are continuous at  $\varepsilon$, it follows that $\Hull^N(X)$ is bounded in $\varepsilon$. Let $\eta=\sup(\Hull^N(X)\cap\varepsilon)$ and let
\[ H=\Hull^N(\eta\cup X).\]
Then a standard argument shows that $\eta=H\cap\varepsilon$,
so letting $C$ be the transitive collapse of $H$
and $\pi:C\to H$ the uncollapse map, we get $\crit(\pi)=\eta$
and $\pi(\eta)=\varepsilon$. By elementarity,
$C$ is also a minimal $(\varphi(x),n)$-witness.
Comparing $C$ with $N$, minimality ensures that
they iterate to a common iterate $D$, and note that the tree on $C$ is above $\eta$. But letting $i<\omega$ be such that $\varepsilon=\delta_i^N$
(the $i$th Woodin of $N$),
we have $\eta=\delta_i^C$, so we get $\delta_i^C=\delta_i^D$, but also $\delta_i^C<\varepsilon=\delta_i^N\leq\delta_i^D$, contradiction.

So $N$ has a unique $(0,\om_1+1)$-strategy.
Therefore by \cite{iter_for_stacks}, it extends
to a $(0,\om_1,\om_1+1)^*$-strategy.
\end{proof}

\begin{lem}
Let $(N,\vec{\delta},\dot{S},\dot{T})$ be a countable minimal $(\varphi(x),n)$-witness, 
and $\Sigma$ the unique $(0,\om_1+1)$-strategy for $N$.
Let $\vec{\delta}=(\delta_4,\ldots,\delta_0)$.
Let $G_0$ be $(N,\CC_{\delta_0})$-generic.
Let $(S,T)=(\dot{S}_{G_0},\dot{T}_{G_0})$.
Let $\Tt,\Tt'$ be successor length normal trees on $N$ via $\Sigma$,
above $\delta_0$,
with $b^\Tt,b^{\Tt'}$ non-dropping. Let $j,j'$ be the iteration maps, and $j^+,(j')^+$ their extensions to $N[G]$.
Then
$p[j^+(S)]\inter p[(j')^+(T)]=\emptyset$.
\end{lem}
\begin{proof}
Fix a $(0,2,\om_1+1)^*$-strategy $\Sigma'$ for $N$,
which must extend $\Sigma$,
(provided) by Lemma \ref{lem:min_varphi(x)-witness_stacks_iterability}.
Suppose $y\in p[j^+(S)]\inter p[(j')^+(T)]$.
Compare $N_0=M^\Tt_\infty$ with 
$N_1=M^{\Tt'}_\infty$,
 using the second round of $\Sigma'$.
By \ref{lem:compare_min_varphi_witnesses}, we get a common iterate $P$, and no drops on main branches.
Let $j_0:N_0\to P$ and $j_1:N_1\to P$ be the iteration maps. Let $j_i^+:N_i[G_0]\to P[G_0]$ be the extension of $j_i$.
By \ref{rem:iterate_pre_varphi-witness} and \ref{lem:compare_min_varphi_witnesses}, we have
\[ j_0^+(j^+(S,T))=(S^*,T^*)= j_1^+((j')^+(S,T)). \]
Shifting elements of the trees pointwise under the various maps, we get
\[ y\in p[j_0^+(j^+(S))]\inter p[j_1^+((j')^+(T))]=p[S^*]\inter p[T^*].\]
By absoluteness,
$P[G_0]\sats\text{``}p[S^*]\inter p[T^*]\neq\emptyset\text{''}$,
so $N[G_0]\sats\text{``}p[S]\inter p[T]\neq\emptyset\text{''}$,
a contradiction.
\end{proof}

\begin{lem}\label{lem:varphi(x)-witness_implies_truth}
Suppose there is a  countable $(\varphi(x),n)$-witness
(over some $X$).
Then $L(\RR)\sats\varphi(x)$.
\end{lem}
\begin{proof}
Since there is a countable $(\varphi(x),n)$-witness over $X$,
we can find a minimal one $(N,\dot{S},\dot{T})$,
as witnessed by Woodin cardinals $\left<\delta_i\right>_{i\leq 4}$ and
 strategy $\Sigma$,
which by Lemma \ref{lem:compare_min_varphi_witnesses} extends to a $(0,\om_1,\om_1+1)^*$-strategy $\Sigma'$.
Let $G_0$ be $(N,\CC_{\delta_0})$-generic, and $S=\dot{S}_{G_0}$. Let $\mathscr{T}$ be the set of all countable successor length
above-$\delta_0$ trees $\Tt$ on $N$ via $\Sigma$, such that $b^\Tt$ does not drop. For $\Tt\in\mathscr{T}$
write $(i^\Tt)^+:N[G_0]\to M^{\Tt}_\infty[G_0]$
for the extension of $i^\Tt$.
Let
\[ t=\bigcup_{\Tt\in\mathscr{T}} p[(i^\Tt)^+(S)].\]

We claim that there is  $\gamma\in\OR$
such that $t=\Th_1^{\SS_\gamma}$,
and hence $\J_\gamma\sats\varphi(x)$,
so $L(\RR)\sats\varphi(x)$.
For given $\Tt\in\mathscr{T}$ and $P=M^\Tt_\infty$,
and given $G$ which is 
$(P[G_0],\CC_{i^\Tt(\delta_4)})$-generic, earlier lemmas give that
\[ p[(i^\Tt)^+(S)]\inter P[G_0,G]=t\inter P[G_0,G].\]
So using genericity iterations, it follows that
\[ \all^\RR x_1\ \exists^\RR x_2\ \all^\RR x_3\ \exists^\RR x_4\ 
[\varrho_\varphi(\vec{x},x,t_0,t_1)]\]
(with $\vec{x}$ as usual), which establishes the claim.\end{proof}

\begin{dfn}\label{dfn:section}
 For a binary relation $R\sub X\cross Y$,
 and for $x\in X$, let $R_x=\{y\in Y\bigm|R(x,y)\}$.
\end{dfn}

 \begin{dfn}\label{dfn:Gamma_strategy}
 Let $M$ be a transitive structure.
 Let $\Gamma$ be a pointclass.
A \emph{$\Gamma$-$(\om_1+1)$-iteration strategy} for $M$
is an iteration strategy $\Sigma$ for $M$
 such that there is a binary $\Gamma$-relation 
 $R$ such that whenever $x\in\RR$ codes $M$,
 then $R_x$ codes $\Sigma\rest\HC$ with respect to $x$. We similarly define
 a \emph{$\Gamma$-$(k,\om_1+1)$-iteration strategy}
 for $k\leq\om$ and $M$ a $k$-sound premouse.
\end{dfn}

\begin{fact}[Mouse witness existence]\label{fact:mouse_witness_existence}
 Let $\alpha$ be a limit ordinal.\footnote{Recall
 here that our indexing of the $\J$-hierarchy
 is not the conventional one; we only index at limit ordinals, whereas usually the indexing uses all ordinals.} Let $x\in\RR$
 and $\varphi$ be $\Sigma_1$, and suppose $\alpha$
 is least such that $\SS_\alpha\sats\varphi(x)$.
 So $\alpha=\gamma+\om$ for some limit $\gamma$.
 Suppose
$\gamma$ does not end a strong S-gap.
 Let $n<\omega$ be such that $\J_{\gamma}\sats\gamma_{\varphi,n}(x)$.
 Let $X\in\HC$ with
 $x\in\J(X)$. 
 Then $\SS_\alpha\sats$``there is a $(\varphi(x),n)$-witness
 over $X$''.
 \end{fact}
 \begin{rem}
Once we have proved Theorem \ref{tm:Martin},
 we will actually be able to improve this result,
 showing the necessity of the assumption
 that $\gamma$ not end a strong S-gap.
 \end{rem}
\begin{proof}
 The proof will follow very much the methods and notions of \cite{twms}, to which the reader should refer
 as needed. In particular
 Definitions 1.2 ($C_\Gamma$), 3.1 (\emph{good pointclass}), and 3.10 (\emph{(coarse) $\Gamma$-Woodin}) are important.

Let $k=k_n$.
 Since $\J_\gamma=\Hull_{\Sigma_1}^{\SS_\gamma}(\RR\cup\gamma)$,
 and by the minimality of $\gamma$,
 it is easy to see that $\J_\gamma=\Hull_{\Sigma_{k+3}}^{\SS_\gamma}(\RR)$.
(In fact one can state a much more optimal result,
using the fine structure of \cite{scales_in_LR},
but we don't need to be that careful here.)
But certainly making use of \cite{scales_in_LR},  since $\gamma$ does not end a strong S-gap,
we can find good pointclasses
 $\left<\Gamma_i,\Gamma_i'\right>_{i\leq 5}$ such that
 each $\Gamma_i,\Gamma_i'\in\SS_\alpha$,
 $\Th_{\Sigma_{k+3}}^{\SS_\gamma}(\RR)\in\Gamma_0$,
and $\Gamma_i\sub\Delta_{\Gamma_i'}$ and $\Gamma_{i'}\sub\Delta_{\Gamma_{i+1}}$ for each $i<5$.
 Let $T_i$ be the tree of a  $\Gamma_i$-scale on a a universal $\Gamma_i$-set.
 
By \cite[Lemma 4.1]{twms}, the operator $z\mapsto C_{\Gamma_0}(z)$ is fine structural;
let $y_0$ be at the base of a cone witnessing this
(cf.~\cite[Definition 2.1]{twms}),
with $x\leq_T y_0$ and $X$ coded into $y_0$.
By  Woodin \cite{LCFD}, we can fix $y\geq_T y_0$ such that
 $\theta=\omega_2^{L(T_5,y)}$ is Woodin in $H=\HOD^{L(T_5,y)}_{T_5,y_0}$.

By the proof of \cite[Lemma 3.11]{twms}, for each $i<5$ there is a club of $\delta<\theta$ such that
  $V_\delta^H$ is   $\Gamma_i$-Woodin. Moreover, for each $i<4$ and each $\delta<\theta$
 such that $V_\delta^H$ is $\Gamma_{i+1}$-Woodin, there is a club of $\xi<\delta$ such that
  $V_\xi^H$ is $\Gamma_i$-Woodin.
   Let $\delta_0<\delta_1<\delta_2<\delta_3<\delta_4$ be determined by: $\delta_0$ is the least $\delta$
   such that $V_\delta^H$ is $\Gamma_4$-Woodin,
   and given $\delta_i$, where $i<4$,
   $\delta_{i+1}$ is the least $\delta>\delta_i$
   such that $V_\delta^H$ is $\Gamma_{3-i}$-Woodin.
   In particular, $V_{\delta_4}^H$ is $\Gamma_0$-Woodin.

\begin{clm}\label{clm:V_lambda^H_iterable}
There is a $\Gamma_4'$-$(\om_1+1)$-strategy
$\Sigma$ for $V_{\delta_4+2}^H$ (a coarse structure);
so $\Sigma\in\SS_\alpha$.
\end{clm}
\begin{proof}
 $\Sigma\rest\HC$ is the strategy $\Psi$ determined
as follows: given a countable limit length
tree $\Tt$ on $V_{\lambda+2}^H$,
$\Psi(\Tt)$ is the unique $\Tt$-cofinal
branch $b$ such that there is $A\in
C_{\Gamma_4'}(M(\Tt))\cap M^\Tt_b$
such that $A\sub\delta(\Tt)$ and
$M(\Tt)$ is not Woodin with respect to $A$.
This will be appropriately
definable, by \cite[Lemma 3.5]{twms}.

If $\Psi$ is indeed an $\om_1$-strategy,
then by its definability, $\Psi\in\SS_\alpha\sats\AD$, so $\Psi$ extends to an $(\om_1+1)$-strategy.
So suppose
$\Psi$ is not an $\om_1$-strategy.
Note that $\theta=\om_2^{L(T_5,y)}$ is countable (in $V$).
Let $G\sub\Coll(\om,V_{\delta_4+2}^H)$ be $H$-generic.
Because we have $T_5\in H$, $H[G]\sats\psi(T_5,V_{\delta_4+2}^H)$,
where $\psi$ asserts ``there is a countable length putative tree $\Tt$ on $V_{\delta_4+2}^H$,
according to $\Psi$, and either (i) $\Tt$ has illfounded last model, or (ii) $\Tt$ has limit length
and there are two distinct $\Tt$-cofinal branches
$b_0,b_1$ and sets $A_i\sub\delta(\Tt)$
such that $A_i\in C_{\Gamma_4'}(M(\Tt))\cap M^{\Tt}_{b_i}$
and $M(\Tt)$ is not Woodin with respect to $A_i$
(for $i=0,1$),
or (iii) there is no $\Tt$-cofinal branch
$b$ and set $A\sub\delta(\Tt)$ as required'',
with ``$\Psi$'' and ``$A\in C_{\Gamma_4'}(B)$'' expressed via $T_5$ (and cf.~
\cite[***categoryquant]{twms}).
By homogeneity of the collapse, this is forced by the empty condition. Working in $H$, let $\bar{H}$
be countable transitive and $\pi:\bar{H}\to V_\gamma^H$
be elementary, with $\gamma$ sufficiently large
and everything relevant in $\rg(\pi)$.
Write $\pi(\bar{\delta_4})=\delta_4$ etc.
Let $g\in H$
be $(\bar{H},\Coll(\om,V_{\bar{\delta_4}+2}^{\bar{H}}))$-generic. Let $\Tt'\in\bar{H}[g]$ witness that
$\bar{H}[g]\sats\psi(\bar{T}_9,V_{\bar{\lambda}+1}^{\bar{H}})$ in $\bar{H}[g]$. Since $\pi(\bar{T}_9)=T_9$,
$\bar{H}[g]$ is correct about this.

Let $\Tt''$ be the tree on $\bar{H}$
which is equivalent to $\Tt'$ (so $M^{\Tt''}_0=\bar{H}$,
whereas $M^{\Tt'}_0=V_{\bar{\delta_4}+2}^{\bar{H}}$,
but the trees use the same extenders and have the same structure). 
By Martin-Steel \cite[Theorems 3.12, 4.3]{itertrees}
applied in $H$, 
we can fix a  $\Tt''$-maximal $\pi$-realizable branch $b\in H$. Let $\sigma:M^{\Tt''}_b\to V_\gamma^H$
be a $\pi$-realization, so $\sigma\com i^{\Tt''}_b=\pi$.
Note that $H\sats\varrho(T_5,\delta_4)$,
where $\varrho(T_5,\delta_4)$ asserts ``For every $\xi\leq\delta_4$,
we have $C_{\Gamma_4'}(V_{\xi})\sub V_{\xi+1}$,
and $\xi$ is not Woodin in 
 $C_{\Gamma_4'}(V_\xi)$''. So $\bar{H}\sats\varrho(\bar{T}_5,\bar{\delta}_4)$, and $M^{\Tt''}_b\sats\psi(i^{\Tt''}_b(\bar{T}_5),i^{\Tt''}_b(\bar{\delta_4}))$.
 Since $\sigma(i^{\Tt''}_b(\bar{T}_5))=T_5$,
 $M^{\Tt''}_b$ is correct about this,
 and it applies in particular
 to $V_{\xi}^{M^{\Tt''}_b}=M(\Tt'')$.
 Now since $\Tt'$ is via $\Psi$
 (which by definition is only defined
 when there is a unique cofinal branch with the right property), $\Tt'$ must have limit length 
 and $b$ must be $\Tt'$-cofinal,
and there is an appropriate witness $A\in C_{\Gamma_4'}(M(\Tt'))$. It follows
 that there is also another $\Tt'$-cofinal
 branch $b_1$ and a set $A_1\in C_{\Gamma_4'}(M(\Tt'))$
 such that $M(\Tt')$ is not Woodin with respect to $A_1$.
 But then $A_1\in M^{\Tt'}_b\cap M^{\Tt'}_{b_1}$
 (since we in fact had $C_{\Gamma_4'}(M(\Tt'))\sub M^{\Tt'}_b$), but since $b\neq b_1$, this contradicts the  Zipper Lemma \cite[Theorem 6.10]{outline}.
\end{proof}

 Let $\delta_{-1}=0$.  Essentially the same proof as for the previous claim gives the following, which we leave to the reader:
\begin{clm}\label{clm:simple_strategy_between_Woodins}
Let $i\leq 4$
 and $\delta_{i-1}<\eta<\delta_i$. Then 
 there is a $\Gamma_{4-i}$-$(\om_1+1)$-strategy
 for above-$\delta_{i-1}$ trees
 on $V_{\eta}^H$.
\end{clm}

Now let $\left<N_\alpha\right>_{\alpha\leq\delta_4}$ be the
models of the fully backgrounded $L[\es,y_0]$-construction $\CC$
of $V_{\delta_4}^H$, where for all $i\leq 4$
and all $\alpha\in(\delta_{i-1},\delta_i)$,
we impose the restriction that
if $N_\alpha$ is active then $\crit(F^{N_\alpha})>\delta_{i-1}$.
By Claim \ref{clm:V_lambda^H_iterable}, this
construction does not break down, and so reaches a model $N_{\delta_4}$ of height $\delta_4$;
and moreover for each $i\leq 4$, $N_{\delta_i}$
has height $\delta_i$ and is definable over $V_{\delta_i}^H$. Since the iteration strategies
are in $L(\RR)$, $\CC$ only reaches tame (in fact $\om$-small) models.

\begin{clm} For each $i<4$, we have:
\begin{enumerate}
 \item\label{item:no_project_across_delta_0}  There is no $\alpha\in[\delta_i,\delta_{i+1}]$
 such that $N_\alpha$ projects ${<\delta_i}$.
 \item $\delta_i$ is Woodin in $N_{\delta_{i+1}}$
 \tu{(}and hence Woodin in $N_{\delta_4}$\tu{)}.
\end{enumerate}
\end{clm}
\begin{proof}
We assume $i=0$, but otherwise
it is likewise.
 Suppose otherwise and let $\alpha<\delta_{i+1}$ be least such that $N_\alpha$ is a Q-structure for $\delta_i$ (this includes the possibility that $N_\alpha$ projects ${<\delta_i}$, by definition).
 Let $\alpha'<\delta_{i+1}$ be such that
 $\CC\rest(\alpha+1)\in V_{\alpha'}^H$.
Fix some $k<\omega$ such that  $(p[T_{3}])_k$
codes a 
 $\Gamma_{3}$-$(\om_1+1)$-strategy
for $V_{\alpha'}^H$
(that is, letting $X'=\{(x,y)\in\RR^2\bigm|x\oplus y\in X\}$, then $(X')_z$ codes the strategy whenever $z$ codes $V_{\alpha'}^H$, as in Definition \ref{dfn:Gamma_strategy};
recall $R_a$ denotes the section of $R$ at $a$ (Definition \ref{dfn:section})).  Note that (in $V$) 
it is a $\Gamma_4$ assertion about reals $z$ coding the parameter $V_{\delta_0}^H$
that ``there is a countable transitive
set $M$ such $V_{\delta_0}^H=V_{\delta_0}^M$
and the fiber of $T_{3}$ at $(k,z)$
codes an above-$\delta_0$, $\om_1$-strategy for $M$, and there is a fully backgrounded construction $\left<N'_\beta\right>_{\beta\leq\alpha''}$
of $M$, extending $\CC\rest\delta_0$, with $\crit(F^{N'}_\beta)>\delta_0$ for all $\beta\in[\delta_0,\alpha'']$, such that  $N'_{\alpha''}$ is a Q-structure  for ${\delta_0}$''. Since $\Gamma_4$ is good,
it follows that we can fix such an $M\in C_{\Gamma_4}(V_{\delta_0}^H)$. Let $\Sigma$ be the witnessing $\Gamma_3$-$(\om_1+1)$-strategy for $M$. Let $N$ be the witness $N_{\alpha''}'\in M$,
with $\alpha''$ minimal,
so $N|\delta_0=N_{\delta_0}$,
and let $n<\omega$ be least such that 
either $\rho_{n+1}^{N}<\delta_0$
or there is an $\bfrSigma_{n+1}^N$ failure of Woodinness of $\delta_0$.

Work in $L(T_4,V_{\delta_0}^H)$,
where $\delta_0$ is Woodin
(and recall that $C_{\Gamma_4}(V_{\delta_0}^H)=\pow(V_{\delta_0}^H))\cap L(V_{\delta_0}^H,T_4)$, and more generally  $C_{\Gamma_4}^m(V_{\delta_0}^H)=\pow^m(V_{\delta_0}^H)\cap L(V_{\delta_0}^H,T_4)$ for all $m\in[1,\om)$).
Note that $M$ is above-$\delta_0$, $\delta_0^+$-iterable (in $L(T_4,V_{\delta_0}^H)$ now), via the restriction $\Sigma'$ of $\Sigma$.

\begin{sclm}$\rho_{n+1}^N=\delta_0$.\end{sclm}
\begin{proof}
Suppose  $\rho_{n+1}^N<\delta_0$.

Continue to work in $L(V_{\delta_0}^H,T_4)$.
Fix a transitive $\bar{V}$, an ordinal $\gamma\gg\delta_0$,
an elementary $\pi:\bar{V}\to V_\gamma$,
and $\bar{\delta}_0<\delta_0$
such that $\crit(\pi)=\bar{\delta}_0$ and $\pi(\bar{\delta}_0)=\delta_0$,
 with $T_4,M,N\in\rg(\pi)$.
 We may also assume we have a surjection $\sigma:V_{\bar{\delta}_0}\to\bar{V}$.
 Let $\pi(\bar{T}_4,\bar{M},\bar{N})=(T_4,M,N)$.
 
 Back in $V$, 
  $\bar{M}$ is above-$\bar{\delta}_0$, $\Gamma_3$-$(\om_1+1)$-iterable, via strategy $\bar{\Sigma}$, given by lifting to $\Sigma$ under $\pi\rest\bar{M}$.  Therefore $L(T_4,V_{\delta_0}^H)\sats$``$\bar{M}$ is $\delta_0^+$-iterable'',
  witnessed by the restriction $\bar{\Sigma}'$ of $\bar{\Sigma}$.
  
Work again in $L(T_4,V_{\delta_0}^H)$.
Let $\widetilde{\bar{N}}$ be the $\bar{\delta}_0$-core
of $\bar{N}$. 
We have $\rho_{n+1}^{\widetilde{\bar{N}}}=\rho_{n+1}^{N}<\bar{\delta}_0$. 
Note that $\bar{\delta}_0$ is a cutpoint of $\widetilde{\bar{N}}$,
and $\widetilde{\bar{N}}$ is above-$\bar{\delta}_0$,
$(n,\delta_0^+)$-iterable, via the strategy $\Psi$ given by uncoring and lifting to the background universe $\bar{M}$, using $\bar{\Sigma}'$ for $\bar{M}$. 
Also $N|\bar{\delta}_0=N_{\bar{\delta}_0}=N'_{\bar{\delta_0}}$. 

Now $\widetilde{\bar{N}}\in V_{\delta_0}^H$.
Working still in $L(T_4,V_{\delta_0}^H)$, we get that $\widetilde{\bar{N}}$ iterates to the background construction $\CC\rest[\bar{\delta}_0,\delta_0]$, for example much as in
\cite[Theorem 6.26]{premouse_inheriting}
or \cite{a_comparison_process_for_mouse_pairs}. To verify  the hypotheses required for this: the fact that $\bar{\delta}_0$ is Woodin
in $\widetilde{\bar{N}}$
inductively prevents $\CC$ from using background
extenders $E^*$ with $\crit(E^*)\leq\bar{\delta}_0$ after stage $\bar{\delta_0}$
(by tameness, until reaching an non-dropping
iterate of $\widetilde{\bar{N}}$, if it ever does),
and the background extenders $E^*$
used for $\CC\rest\delta_0$ (which are also extenders of $H$) cohere
$\Psi$ appropriately, because of the reduction
of $\Psi$
to $\bar{\Sigma}'$, and that $E^*$ coheres $\bar{\Sigma}'$.
(The latter coherence is easy enough to verify using the countable completeness of $E^*$ in $H$,
that $\bar{M}$ is small relative to $\crit(E^*)$ in $H$,
and that $\bar{\Sigma'}$ is $\Delta_{\Gamma_5}$.
If $E^*$ does not cohere $\bar{\Sigma}'$,
then working in $H[g]$ where $g\sub\Coll(\om,\bar{M})$ is $H$-generic,
take a countable hull $\bar{H}[g]$ of $V_\eta^H[g]$, where $\eta$ is sufficiently large and with
 $E^*,\bar{M},\bar{\Sigma}',T_5$
 in the range of the uncollapse map,
and realize the ultrapower $\Ult(\bar{H}[g],\bar{E^*})$
back into $H[g]$ for a contradiction.)
So iteration to background applies.
But $\CC\rest[\bar{\delta}_0,\delta_0]$ does not reach a non-dropping iterate of $\widetilde{\bar{N}}$, since $\widetilde{\bar{N}}|\bar{\delta_0}$
is a cardinal segment of $N_{\delta_0}$.
So we get a tree $\Tt$ on $\widetilde{\bar{N}}$
via $\Psi$, such that $N_{\delta_0}\pins M^\Tt_{\delta_0}$. But now using the Woodinness of $\delta_0$ and since $\CC$ does not reach a  superstrong mouse, we can run the usual proof for a contradiction.\end{proof}

So  $\rho_{n+1}^{N}=\delta_0$,
but here is an $\bfrSigma_{n+1}^N$ failure of Woodinness of $\delta_0$. But  $\delta_0$ is Woodin in $L(V_{\delta_0}^H,T_4)$, so we can now run the usual
proof that Woodinness is absorbed by $L[\es,y_0]$-constructions there,
contradicting the fact that $N\in L(V_{\delta_0}^H,T_4)$. This completes the proof of the claim.
\end{proof}

Let $\lambda$ be least such that 
$V_{\lambda}^H$ is $\Gamma_1$-Woodin
and $\lambda>\delta_4$.
Let $\CC^+=\left<N_\alpha\right>_{\alpha\leq\lambda}$ be the $L[\es,y_0]$-construction
of $V_\lambda^H$
extending $\CC$, with $\crit(F^{N_\alpha})>\delta_4$ for all $\alpha>\delta_4$.
As in \cite{twms}, using tameness,
there is some $\eta\in[\delta_4,\lambda)$
such that $V_\eta^H$ is $\Gamma_0$-Woodin
and there is some $\alpha>\eta$ such that
either $N_\alpha$ projects $<\eta$
or $\core_\om(N_\alpha)$ is a Q-structure
for $\eta$. Let $(\eta,\alpha)$ be lexicographically least such.
Then $\delta_0,\delta_1,\delta_2,\delta_3,\delta_4,\eta$ are each Woodin in $N_\alpha$ (and $\delta_4\leq\eta$).

\begin{clm} $C_{\Gamma_0}^2(N_\eta)\sub N_\alpha||\eta^{++N_\alpha}$, and in fact
there is $P\pins N_\alpha$
such that the universe of $C_{\Gamma_0}^2(N_\eta)=\pow^2(N_\eta)\cap P$.\end{clm}
\begin{proof}
 This is like in \cite{twms}: By choice of $y_0$,
 $C_{\Gamma_0}(N_\eta)$ is a mouse set,
 and tameness ensures that $\eta$ is a strong cutpoint of $N_\alpha$,
 so the claim follows from comparison.
\end{proof}

Let $P$ be as in the claim. Then $P\sats\ZF^-+$``$\delta_0<\delta_1<\delta_2<\delta_3<\eta$ are Woodin cardinals and $\eta^+$ exists''.
Moreover, by the claim, $\pow(\eta)\cap P=\pow(\eta)\cap L(T_0,P|\eta)$ ($\es^{P|\eta}$ is itself
definable from parameters over $V_\eta^P$,
by \cite{V=HODX_pub}), and so it is straightforward
to see that $P$ is a $(\varphi(x),n)$-witness
(recall $y\geq_T x$, so $x\in P$),
as witnessed by trees $S,T\in P$
which embed into fibers of $T_0$
(we find $S,T\in L(T_0,P|\eta)$,
of cardinality $\eta^{+L(T_0,P|\eta)}$ there,
and it follows that they are in $P$;
here we have the trees themselves directly
in $P$, as opposed to forcing names for trees).

Now $P$ is a $y$-premouse, so we are done in the case that $X=y$. More generally, we still have $X\in\HC^P$. Consider the Q-local $L[\es,X]$-construction of $P$ (see \cite{scales_from_mice}).
This produces final model $R$, an $X$-mouse,
and $P$ has universe that of $R[P|\delta_0^P]$,
where $\delta_0^P$ is the least Woodin of $P$
(which is $\leq\delta_0$), $\delta_0^P$ is the least Woodin of $R$, and $P|\delta_0^P$
is generic over $R$ for the extender algebra at $\delta_0^R$, and $R$ is the output of the P-construction of $P$ above $R|\delta_0^R$.
So we can find a $\Coll(\om,\delta_0^R)$-generic
$G$ such that $P\sub R[G]$, and it follows
that $R$ is a $(\varphi(x),n)$-witness.
The strategy for $R$ is derived from that for $P$, 
which is in turn in $\SS_\alpha$ (like in Claim \ref{clm:simple_strategy_between_Woodins}), completing the proof.
\end{proof}

\begin{fact}[Mouse witnesses]\label{fact:mouse_witnesses}
 Let $\alpha$ be a limit ordinal
 and suppose $\alpha$ is not of the form $\gamma+\om$ where $\gamma$ ends a strong S-gap. Let $x\in\RR$ and $X\in\HC$ with $x\in\Ss_\om(X)$,
 and $\varphi$
 be $\Sigma_1$. Then the following are equivalent:
 \begin{enumerate}
  \item\label{item:J_alpha_sats_varphi} $\SS_\alpha\sats\varphi(x)$
 \item\label{item:J_alpha_sats_witness_exists} $\SS_\alpha\sats$``there is a 
$\varphi(x)$-witness over $X$''
 \item\label{item:J_alpha_sats_min_witness_exists} $\SS_\alpha\sats$``there is a 
minimal $\varphi(x)$-witness $N$ over $X$
such that $N=\Hull^N(X)$ and $\J(N)$ is a sound $(0,\om_1+1)$-iterable
$X$-premouse with $\rho_1^{\J(N)}=X$'',
so 
$\Th^N(X)\in\OD^{<\alpha}(X)$ if $N$ witnesses this.
 \end{enumerate}
\end{fact}

\begin{proof}[Proof Sketch]
 \ref{item:J_alpha_sats_varphi} $\Rightarrow$ 
\ref{item:J_alpha_sats_witness_exists} is by
Fact \ref{fact:mouse_witness_existence}.

\ref{item:J_alpha_sats_witness_exists} $\Rightarrow$ 
\ref{item:J_alpha_sats_min_witness_exists}:
Let $(N,\dot{S},\dot{T})$ be such that
$\SS_\alpha\sats$``$(N,\dot{S},\dot{T})$
is a  $\varphi(x)$-witness over $X$''. Then by minimizing,
we can find a minimal one $(N',\dot{S}',\dot{T}')$ with $N'\ins N$. Note then that $\dot{S}',\dot{T}'$
are definable over $N'$ (in the language of $X$-premice, which has a constant referring to $X$).
Let $\bar{N}=\cHull_{\om}^{N'}(X)$, let 
$\pi:\bar{N}\to N'$ be the uncollapse map,
and $\pi(\bar{\dot{S}},\bar{\dot{T}})=(\dot{S}',\dot{T}')$. Then $\bar{N}=\Hull^{\bar{N}}(X)$
and $\SS_\alpha\sats$``$(\bar{N},\bar{\dot{S}},\bar{\dot{T}})$ is a minimal $\varphi(x)$-witness over $X$''.
Moreover, just like  in \cite[Lemma 2.4]{hsstm},
$\J(\bar{N})$ is sound with $\rho_1^{\bar{N}}=X$
and $p_1^{\bar{N}}=\OR^{\bar{N}}$,
and is iterable in $\SS_\alpha$ (and $0$-maximal trees on $\J(\bar{N})$ correspond very simply to $0$-maximal (equivalently, $\om$-maximal) trees on $\bar{N}$). The fact that $\Th^N(X)\in\OD^{<\alpha}(X)$ now follows from the compatibility
of sound projecting $X$-mice.

\ref{item:J_alpha_sats_min_witness_exists} $\Rightarrow$ 
\ref{item:J_alpha_sats_varphi}:  By Lemma \ref{lem:varphi(x)-witness_implies_truth}
in $L(\RR)$ (where the $\varphi(x)$-witness $N$ is  $(0,\om_1+1)$-iterable (recall all subsets of $\om_1$ in $L(\RR)$
are projectively definable)),
we have $L(\RR)\sats\varphi(x)$.
So fix $n<\om$ and $N\in\SS_\alpha$
such that $\SS_\alpha\sats$``$N$ is a  minimal
$(\varphi(x),n)$-witness''.
Let $\Sigma$ be a $(0,\om_1+1)$-strategy
 for $N$ in $\SS_\alpha$. Then $t_0,t_1\in\SS_\alpha$,
where $t_0,t_1$ are the model and theory defined as in the proof of Lemma \ref{lem:varphi(x)-witness_implies_truth}.
Since $t_0\iso\SS_\gamma$ where $\gamma$ is least
such that
$\J_{\gamma+\om}\sats\varphi(x)$, $t_0$ yields a surjection $\RR\to\pow(\RR)\cap\SS_\gamma$, so
a diagonalization gives $\gamma<\alpha$.
\end{proof}

\begin{cor}\label{cor:local_MST}
 Let $\alpha$ be a limit ordinal
 and suppose $\alpha$ is not of the form $\gamma+\om$ where $\gamma$ ends a strong S-gap.
 For every $y\in\OD^{<\alpha}(x)$
 there is a sound $x$-premouse
 $M$ such that $y\in M$ and $M$
 has an $(\om,\om_1+1)$-strategy
 in $\SS_\alpha$.
\end{cor}
\begin{proof}
 Let $\varphi(u,v)$ be the $\Sigma_1$ assertion
 that $v\in\OD^{<\OR}(u)$.
 So $\SS_\alpha\sats\varphi(x,y)$,
 so by Fact \ref{fact:mouse_witnesses},
 $\SS_\alpha\sats$``there is a minimal $\varphi(x,y)$-witness'', so take $N$ such, and an iteration strategy $\Sigma$ for $N$ in $\SS_\alpha$.

 Let $P$ be the output of the Q-local $L[\es,x]$-construction of $N$. So $P$ is an $x$-mouse,
 and $P$ is also iterable in $\SS_\alpha$,
 so it suffices to see that $y\in P$.
 Suppose not. Let $\delta_0<\ldots<\delta_4$
 be Woodins of $N$ and $\dot{S},\dot{T}\in N$ be
 $\Coll(\om,\delta_0)$-names witnessing that $N$ is a $\varphi(x,y)$-witness. Then $\delta_0<\ldots<\delta_4$ are also Woodin in $P$.
 Fix some extender algebra names 
 $y',\dot{S}',\dot{T}'\in P$ evaluating to $y,\dot{S},\dot{T}$. Fix a condition $p\in\BB^P_{\delta_0}$ forcing
 that the extension is a pre-$\varphi(x,y')$-witness.
 Then since $P$ is countable and $y\notin P$,
 we can construct a perfect set $\mathscr{P}$ of $P$-generics
 for the extender algebra below $p$,
 arranging that $y'_{G_1}\neq y'_{G_2}$
 whenever $G_1,G_2\in\mathscr{P}$ are distinct.
 But then by the proof of Lemma \ref{lem:varphi(x)-witness_implies_truth},
 we get $y'_G\in\OD^{<\alpha}(x)$ for each $G\in\mathscr{P}$, so this set is uncountable, a contradiction. (It's not quite directly
 by Lemma \ref{lem:varphi(x)-witness_implies_truth} itself,
 because $P[G]$ need not be iterable below $\delta_0$ for arbitrary $G\in\mathscr{P}$,
 but note that we only need the iterability
 above $\delta_0$,
 which just comes from $\Sigma$.)
\end{proof}

\begin{dfn}
 Let $\alpha$ be a limit ordinal which is not of the form $\gamma+\om$ where $\gamma$ ends a strong S-gap.
Write $\Gamma_\alpha=\Sigma_1^{\SS_\alpha}$. For transitive $X\in\HC$,
 $\Lp_{\Gamma_\alpha}(X)$ denotes
 the stack of all $\om$-sound
 $X$-premice $N$ such that $\rho_\om^N=X$
 and $\SS_\alpha\sats$``$N$ is $(\om,\om_1+1)$-iterable''.
\end{dfn}

\begin{fact}
 Let $\alpha$ be a limit ordinal
 such that $\alpha$ is not of the form $\gamma+\om$ where $\gamma$ ends a strong S-gap.
 Let $X\in\HC$ be transitive.
 Then
 \[ \OD^{<\alpha}(X)=\Lp_{\Gamma_\alpha}(X)\cap\pow(X).\]
\end{fact}
\begin{proof}
 The fact that $\Lp_{\Gamma_\alpha}(X)\cap\pow(X)\sub\OD^{<\alpha}(X)$ follows
 directly from the definitions, using the compatibility of lower part mice in $\SS_\alpha$,
 and that we can refer directly to the existence of an iteration strategy with $\Sigma_1^{\SS_\alpha}$.
 The fact that $\OD^{<\alpha}(X)\sub\Lp_{\Gamma_\alpha}(X)$
 is a direct consequence of Corollary \ref{cor:local_MST}.
\end{proof}

\section{The start of a limit gap}\label{sec:start_of_gap}

\subsection{Embedding into the start of a limit gap}

\begin{lem}\label{lem:Sigma_1_Hull_of_J_alpha}
Let $\alpha$ be a limit of limits which starts an S-gap of $L(\RR)$.
Let 
$N$ be a countable $\om$-small premouse such that $N|\om_1^N$ is $(0,\om_1+1)$-iterable
and $N$ is
$(\alpha,0)$-closed.
Let
\[ H=\Hull_{1}^{\Ss_\alpha}(\RR^N)
\text{ and }t=\Th_1^{\Ss_\alpha}(\RR^N).\]
Let $\Ss_{\bar{\alpha}}(\RR^N)$ be the transitive collapse of $H$ \tu{(}see below\tu{)}.
Then:
\begin{enumerate}
\item \label{item:RR^N_Sigma_1-closed} 
$\RR^N=H\inter\RR$ and $H\elem_1\SS_\alpha$ \tu{(}hence $H$ is extensional\tu{)}.

\item\label{item:alpha-bar<OR^N}
$t$ is $\bfSigma^{N|\om_1^N}_1$-definable.
Hence if $\om_1^N<\OR^N$
then $t\in N$,
and so if $N$ is also admissible then
 $\bar{\alpha}<\OR^N$.
\end{enumerate}
\end{lem}
\begin{proof}
Part \ref{item:RR^N_Sigma_1-closed}:
If $x\in H\inter\RR$ then $x\in\OD^{<\alpha}_y$
for some
$y\in\RR^N$,
hence $x\in\RR^N$.

Let $\varphi$ be $\Sigma_1$
and $x_0\in\RR^N$ and suppose that $\SS_\alpha\sats\varphi(x_0)$.
Let $\varphi(v)$
assert $\exists z\psi(v,z)$, where $\psi$ is $\Sigma_0$.
We want to see that there is $z\in H$ such that $\SS_\alpha\sats\varphi(x_0,z)$.

Let $\gamma<\alpha$
be least such that $\J_{\gamma+\om}\sats\varphi(x_0)$. Then $\gamma\in H$, $[\gamma+\om,\gamma+\om]$
is a projective-like S-gap,
 and so by \cite{scales_in_LR}
and the first periodicity theorem,
the pointclasses
$\Sigma_1^{\J_{\gamma+\om}}(\{x_0\})$
and
$\all^\RR\Sigma_1^{\J_{\gamma+\om}}(\{x_0\})$
have the scale property,
and the latter also has the uniformization property.

Let $\varrho(\dot{x},\dot{y})$ assert ``$\dot{x},\dot{y}\in\RR$ 
and there are 
$\gamma'\in\Lim$
and $n'<\om$
and $z\in\J_{\gamma'+n'}$ such that
$\J_{\gamma'+n'}\sats\psi(\dot{x},z)$ and
$z$ is definable over $\J_{\gamma'+n'}$
from ordinals and $\dot{y}$''.
So $\varrho(\dot{x},\dot{y})$ is $\Sigma_1$ and $\J_{\gamma+\om}\sats\exists 
y\in\RR[\varrho(x_0,y)]$. By
the uniformization property
for $\all^\RR\Sigma_1^{\J_{\alpha+\om}}(\{x_0\})$, note that there is $y$ such that
$\J_{\gamma+\om}\sats\varrho(x_0,y)$ and $y$ is definable from $x_0$ over $\J_{\gamma+\om}$, and (as $\gamma+\om<\alpha$) therefore $y\in H$.
It follows
that there is $z\in H$ such that
$\J_{\gamma+\om}\sats\varphi(x_0,z)$,
as desired.

Part \ref{item:alpha-bar<OR^N}: 
Let $t^+$ be the set of pairs
$(\varphi,x)$ such that $\varphi\in\Ll_{L(\RR)}$ is $\Sigma_1$,
$x\in\RR^N$, and there are $\eta_0<\eta<\om_1^N$
such that $x\in N|\eta_0$, $\rho_\om^{N|\eta_0}=\om$,  
 $N|\eta$ is (equivalent to) a minimal pre-$\varphi'_x(N|\eta_0)$-witness,\footnote{That is, recall that for a transitive set $P$,
a pre-$\varphi'_x(P)$-witness
is a premouse over $P$. But $N$ is a premouse (over $\emptyset$),
so $N|\eta$ cannot literally be a pre-$\varphi'_x(N|\eta_0)$-witness.
But since $\rho_\om^{N|\eta_0}=\om$, $N|\eta$ is equivalent
to a premouse over $N|\eta_0$, and we want that to be the pre-$\varphi'_x(N|\eta_0)$-witness.}
where $\varphi'_x\in\Ll_{L(\RR)}$ is a $\Sigma_1$ formula
in one free variable, such that $\varphi'_x(N|\eta_0)$
asserts $\varphi(x)$ in a natural manner
(referring to $x$ via the parameter $N|\eta_0$).
So $t^+$ is $\bfSigma_1^{N|\om_1^N}$-definable.

By iterability and Lemma \ref{lem:varphi(x)-witness_implies_truth},
$L(\RR)\sats t^+$.
And $t\sub t^+$ because $N$ is $(\alpha,0)$-closed and by Fact \ref{fact:mouse_witnesses}
(apply its part \ref{item:J_alpha_sats_min_witness_exists}).
Let $\alpha'\in\Lim$ be least such that $\J_{\alpha'}\sats t^+$, so
$t\sub t^+\sub\Th_1^{\SS_{\alpha'}}(\RR^N)$.

If $t=t^+$  we are done,
so suppose $t\psub t^+$.
Given $(\varphi,x)\in t^+$, let $\beta_{\varphi,x}$ be the least
$\beta\in\Lim$ such that $\SS_\beta\sats\varphi(x)$.
Let $(\varphi_0,x_0)\in t^+$ be such that $\beta_0=\beta_{\varphi_0,x_0}>\alpha$,
taking $\beta_0$ minimal possible.
So $\beta_0=\gamma_0+\om$ for some limit $\gamma_0\geq\alpha$.
Let $n_0<\om$ be such that $\J_{\gamma_0}\sats\gamma_{\varphi_0,n_0}(x_0)$.
Then for $(\varphi,x)\in t^+$,
we have
\[ (\varphi,x)\in t\iff(\varphi',(x,x_0))\in t^+,\]
where $\varphi'(x,x_0)$ asserts 
\[\text{``}\exists\gamma\in\Lim\ 
\Big[\J_\gamma\sats\varphi(x)\ \&\ \all\xi\in\Lim\cap(\gamma+1)\Big(\J_\xi\sats\neg\gamma_{\varphi_0,n_0}(x_0)\Big)\Big]\text{''}.\]
(If $\SS_\alpha\sats\varphi(x)$
then since $\alpha$ is a limit of limits, there is a limit $\gamma<\alpha$
such that $\J_\gamma\sats\varphi(x)$,
and clearly this witnesses $\varphi'(x,x_0)$
in $\SS_\alpha$, so $(\varphi',(x,x_0))\in t\sub t^+$.
Conversely, if $(\varphi',(x,x_0))\in t^+$,
then $L(\RR)\sats\varphi'(x,x_0)$,
and if $\gamma$ witnesses this,
note that $\gamma<\gamma_0$, so $\gamma+\om<\gamma_0+\om=\beta_0$, so by the minimality of $\beta_0$,
actually $\SS_\alpha\sats\varphi'(x,x_0)$,
so $\SS_\alpha\sats\varphi(x)$.)
So $t$ is $\bfSigma_1^{N|\om_1^N}(\{x_0\})$, as desired.
\end{proof}

\begin{rem}
Even if $\SS_\alpha$ is admissible,
this needn't transfer to
$\Ss_{\bar{\alpha}}(\RR^N)$:
we can have $\Ss_{\bar{\alpha}}(\RR^N)\sats$
``For all reals $x$ there is $y$ such that $\varphi(x,y)$'' where $\varphi$ is $\Sigma_1$ in $\Ll_{L(\RR)}$, while $\SS_\alpha\sats$``There is a real $x$ such that for all $y$, $\neg\varphi(x,y)$''. In fact, if $N|\om_1^N=\Lp_{\Gamma_\alpha}(\emptyset)$ (see below) then $N$ is $(\alpha,0)$-closed, and
all reals of $\RR^N$  belong to iterable mice in $\SS_\alpha$, hence also in $\Ss_{\bar{\alpha}}(\RR^N)$, by $\Sigma_1$-elementarity, but of course not all reals in $V$ belong to iterable mice in $\SS_\alpha$, and note that therefore in this case,
$\Ss_{\bar{\alpha}}(\RR^N)$ is not admissible.
\end{rem}

\begin{dfn}
 Let $\alpha$ be a limit of limits which starts an S-gap. Let $\Gamma=\Gamma_\alpha$.
 Let $N$ be a premouse and $\Tt$ a normal tree on $N$.
We say that $\Tt$ is \emph{$\Gamma$-guided} iff for every
limit $\eta<\lh(\Tt)$, we have
\[ Q=\Q(\Tt\rest\eta,[0,\eta)_\Tt)\text{
exists and }Q\pins \Lp_{\Gamma_\alpha}(M(\Tt)).\]

Suppose that $N|\om_1^N$ is $\om_1$-iterable and let
$\Psi=\Sigma_N$ (its unique $(\om,\om_1+1)$-strategy).
Given $\xi<\om_1^N$,
let $\Psi_{\geq\xi}$ be its restriction to
above-$\xi$ trees (which is the unique above-$\xi$, $(\om,\om_1+1)$-strategy for 
$N|\om_1^N$). We say that $\Psi$ (or $\Psi_{\geq\xi}$) is 
\emph{$\Gamma$-guided}
iff every tree $\Tt$ of countable  length via $\Psi$ (or 
$\Psi_{\geq\xi}$) is $\Gamma$-guided. We say that $\Psi$ (or $\Psi_{\geq\xi}$)
is \emph{$N$-$\Gamma$-guided} iff for every (above-$\xi$) limit length
$\Tt\in\HC^N$ via $\Psi$, $\Tt\conc\Psi(\Tt)$ is $\Gamma$-guided.\footnote{Note 
that this definition is ostensibly stronger than just requiring
that every $\Tt\in\HC^N$ via $\Psi$ is $\Gamma$-guided.}
\end{dfn}

The ultimate lemma in this section verifies 
 Conjecture \ref{conj:RS} (the strong conjecture)
in some cases:

\begin{lem}\label{lem:when_Psi>xi_0_Gamma_guided}
Let $\alpha,N,H,\bar{\alpha}$ be as in \ref{lem:Sigma_1_Hull_of_J_alpha}.
 For $\xi<\om_1^N$  let
\begin{equation}\label{eqn:S} S_{\xi}=\Big\{Q\Bigm|N|\xi\pins Q\pins N|\om_1^N\text{ and }\rho_\om^{Q}=\om\Big\}. \end{equation}
Then:
\begin{enumerate}[label=\tu{(}\alph*\tu{)}]
 \item\label{item:firsta} if there is $\xi<\om_1^N$
 such that $\rho_\om^{N|\xi}=\om$ and 
$N|\om_1^N=\Lp_{\Gamma_\alpha}(N|\xi)$,  then $S_\xi$
 is $\Sigma_1^{\J_{\bar{\alpha}}(\RR^N)}(\{N|\xi\})$-definable,
 and hence $<^N\rest\RR^N$ is $\Delta_1^{\J_{\bar{\alpha}}(\RR^N)}(\{N|\xi\})$-definable;
 \item\label{item:firstb} if there is no $\xi<\om_1^N$ as in part \ref{item:firsta},
 but there is $\xi<\om_1^N$
 such that $\rho_\om^{N|\xi}=\om$ and for each $\xi'\in(\xi,\om_1^N)$, we have $\xi'\neq\om_1^{\Lp_{\Gamma_\alpha}(N|\xi')}$,
 then $S_\xi$
 is $\all^\om\Sigma_1^{\Ss_{\bar{\alpha}}(\RR^N)}(\{N|\xi\})$-definable,
 and hence $<^N\rest\RR^N$ is $\Delta_2^{\Ss_{\bar{\alpha}}(\RR^N)}(\{N|\xi\})$-definable;
  \item\label{item:second} if there is no $\xi$
 as in part \ref{item:firsta}
 or part \ref{item:firstb},
but $\xi<\om_1^N$ is such 
that $\rho_{\om}^{N|\xi}=\om$ and
 $\Psi_{\geq\xi}$ is 
$N$-$\Gamma_\alpha$-guided,
 then $S_{\xi}$ is $\Pi_1^{\Ss_{\bar{\alpha}}(\RR^N)}(\{N|\xi\})$-definable,
 and hence $<^N\rest\RR^N$
 is $\Delta_2^{\Ss_{\bar{\alpha}}(\RR^N)}(\{N|\xi\})$-definable.
\end{enumerate}
\end{lem}
\begin{proof}
Part \ref{item:firsta}
 is an easy consequence of Lemma \ref{lem:Sigma_1_Hull_of_J_alpha}.
 
 Part \ref{item:firstb}:
$S_\xi$ is $\all^\om\Sigma_1^{\Ss_{\bar{\alpha}}(\RR^N)}(\{N|\xi\})$-definable,
as given a sound premouse $R\in\HC^M$
such that $N|\xi\pins R$
and $\rho_\om^R=\om$,
we have that $R\pins N$
iff for every $\gamma\in(\xi,\OR^R)$,
if $R|\gamma\sats$``$\om$ is the largest cardinal''
then there is $R'\ins R$
such that $\gamma\leq\OR^{R'}$ and $\rho_\om^{R'}=\om$
and there is an an above-$\gamma$, $(\om,\om_1)$-iteration strategy for $R'$ in $\Ss_{\bar{\alpha}}(\RR^N)$. 

Part \ref{item:second}:
Fix $\xi$ as  hypothesized.
 Let $R\in\HC^N$
 be a sound premouse such that $\rho_\om^R=\om$ and $N|\xi\pins R$.
 We claim that $R\pins N$ iff
 for every $\alpha'<\bar{\alpha}$,
  $\Ss_{\alpha'}(\RR^N)\sats$``For   every countable length, above-$\xi$, $\om$-maximal putative tree $\Tt$ on $R$,
  and every $Q\in\HC$,
if 
for every limit $\zeta<\lh(\Tt)$, we have that $Q(\Tt\rest\zeta,[0,\zeta)_\Tt)$ is above-$\delta(\Tt\rest\zeta)$, $(\om,\om_1)$-iterable,
  then
  \begin{enumerate}[label=--]
   \item  $\Tt$ is an iteration tree (has wellfounded models), and
  \item if $\Tt$ has limit length
  and $Q$ is a $\delta(\Tt)$-sound Q-structure for $M(\Tt)$
  and $Q$ is above-$\delta(\Tt)$, $(\om,\om_1)$-iterable,
  then there is a $\Tt$-cofinal branch $b$ with
 $Q\ins M^\Tt_b$.''
  \end{enumerate}

  For clearly our assumptions yield the forward implication. So suppose the above statement holds of $R$,
  but $R\npins N$. Then by minimizing the height of such an $R$, we may assume that we can find $R'\pins N$
  such that $\rho_\om^{R'}=\om$
  and $R|\om_1^R=R'|\om_1^{R'}$ but  $R\neq R'$. Let $\xi<\om_1^M$ be such that
  $R,R'\in N|\xi$
  and $\xi=\om_1^S$
  where $S={\Lp_{\Gamma_\alpha}(N|\xi)}$;
  such a $\xi$ exists since the 
 hypotheses of \ref{item:firsta} and 
 \ref{item:firstb} fail. 
 Working inside $S$, we can form a simultaneous comparison
  with $S|\xi$-genericity iteration of $(R,R')$, through length $\leq(\xi+1)$, using $S$ to build Q-structures via P-construction at limit stages. (See \cite{odle_v2} for more details on such comparisons.)
  This comparison cannot terminate,
  so it runs to length $\xi+1$.
  But the process just described
  also yields a Q-structure $Q$ at stage $\xi$ (with $Q\in S$), since $\Psi_{\geq\xi}$ is 
$N$-$\Gamma_\alpha$-guided.
By the assumptions on $R$,
$Q$ must yield a wellfounded branches at stage $\xi$, but these are inside $S$, which satisfies ``$\xi=\om_1$''.
  This contradicts the termination of simultaneous comparison with genericity iteration in $S$.\qedhere
\end{proof}

\subsection{Minimally transcendent mice}\label{subsec:mtr}

Fix an S-gap $[\alphag,\betag]$ of $L(\RR)$ such that $\alphag$ is a limit of limits,
and in the context of the conjectures, such that
$[\alphag,\betag]$
is the unique S-gap
such that $\alphag$ is a limit
of limits and
$\betag\leq\beta^M<\betag+\om^2$.\footnote{***Note to self:
This ignores the case that $\beta^M<\om^2$.}
Let $\Gammag=\Gamma_{\alphag}$.

Let $\xg,\yg\in\RR$
be such that
\[ \yg\in\OD^{<\betag+\om^2}_{\xg}\cut
\OD^{<\alphag}_{\xg},
\]
and in the context of the conjectures, 
such that if $M$ is strongly $(\alphag,1)$-closed then $\xg,\yg\in\RR^M$.

By Lemma \ref{lem:when_Psi>xi_0_Gamma_guided}, in the context of the conjectures,
we may and do make the following assumption for the remainder of the paper:

\begin{ass} \label{ass:Sigma_M_not_Gamma-guided}
($M$) There is no $\xi<\om_1^M$
as in hypotheses \ref{item:firsta}--\ref{item:second} of Lemma \ref{lem:when_Psi>xi_0_Gamma_guided}. That is:
\begin{enumerate}[label=(\alph*)]
\item $M$ is strongly $(\alphag,1)$-closed,
\item for every $\xi<\om_1^M$
there is $\xi'\in(\xi,\om_1^M)$
such that  $\xi'=\om_1^{\Lp_{\Gammag}(M|\xi')}$, and
\item there is no 
$\xi<\om_1^M$ such that
$\Sigma_{M,\geq\xi}$ is $M$-$\Gammag$-guided.
\end{enumerate}
\end{ass}

\begin{dfn}\label{dfn:mtr}Let $X\in\HC$ be transitive.

We say that $X$ is \emph{high} iff
 $\xg,\yg\in X$.

 Let $N$ be an $X$-premouse and $\delta<\OR^N$ with $\rank(X)\leq\delta$.
 
 We say that $N$ is \emph{$\delta$-bounded} or \emph{bounded at $\delta$}
 iff for all $\xi<\OR^N$ such that $\delta$ is a strong cutpoint of $N|\xi$,
 we have
 \[ N|\xi\ins\Lp_{\Gammag}(N|\delta).\]
 Let $\zeta<\eta\leq\OR^N$.
We say that $N$ is \emph{$[\zeta,\eta)$-bounded}
iff $N$ is $\delta$-bounded for all $\delta\in[\zeta,\eta)$, and \emph{${<\eta}$-bounded} iff $[0,\eta)$-bounded.

 Suppose $\delta$ is a strong cutpoint of $N$.
 We say that $N$ is \emph{$\delta$-full} or \emph{full at $\delta$}
 iff
 \[ \Lp_{\Gammag}(N|\delta)\ins N.\]
We say that $N$ is \emph{$\delta$-exact} or \emph{exact at $\delta$}
iff
\[ \Lp_{\Gammag}(N|\delta)=N|(\delta^+)^N.\]
We say that $N$ is a \emph{$\delta$-mtr} (for \emph{minimally transcendent})
or \emph{mtr at $\delta$}
iff $N$ is $\delta$-exact and for some $n<\om$,
$N$ is $(n,\om_1,\om_1+1)^*$-iterable and $\rho_{n+1}^N\leq\delta<\rho_n^N$.
Note that $n$ is determined by $(N,\delta)$; we say that $n$ is the 
\emph{degree} of $(N,\delta)$.
\end{dfn}

Note that a $\delta$-mtr has $\delta$ as a strong cutpoint, by assumption.
Note that a $\delta$-mtr isn't required to be $\delta$-sound.

\begin{lem}\label{lem:overspill}
 Let $X$ be high and $N$ be an $X$-premouse.
 Let $\delta<\OR^N$ and suppose that
 $\delta$ is a strong cutpoint of $N$ and $N=\Lp_{\Gammag}(N|\delta)$.
 Let $R$ be an $X$-premouse and $j:N\to R$ be such that
\begin{enumerate}[label=--]
 \item  $j$ is either $\Sigma_2$-elementary or cofinal $\Sigma_1$-elementary, and
 \item $R$
 is above-$j(\delta)$, $(0,\om_1+1)$-iterable.
 \end{enumerate}
 Then:
 \begin{enumerate}
  \item\label{item:R_is_j(delta)-bounded} $R$ is $j(\delta)$-bounded,
 and
 \item\label{item:R_is_<j(delta)-bounded} if $N$ is $[\zeta,\delta)$-bounded then $R$ is $[j(\zeta),j(\delta))$-bounded.
 \end{enumerate}
\end{lem}
\begin{proof}
Part \ref{item:R_is_j(delta)-bounded}:
We use an overspill style argument.
Let $\psi(\dot{n},\dot{S},\dot{\delta},\dot{x},\dot{y})$ be the $\Sigma_1$ formula of $\Ll_{L(\RR)}$
asserting ``$\dot{S}$
is a sound premouse, $\dot{\delta}<\OR^{\dot{S}}$, and there is 
$\gamma\in\Lim$
such that 
$\big[\J_\gamma\sats$``there is an above-$\dot{\delta}$,
$(\om,\om_1)$-strategy for $\dot{S}$'', $\dot{n}<\om$,
$(\gamma+(\dot{n}+1)\om)\in\OR$ and
$\dot{y}\notin\OD^{\gamma+n\om}_{\dot{x}}\big]$''.

Note that since $\alphag$ is a limit of limits,
\begin{equation}\label{eqn:N_thinks_segs_bounded_it} 
\begin{split}
N\sats&\all\eta\in(\delta,\OR)\
 \all n<\om\\
 &\ \ \ \ \exists W\pins 
L[\es]\ \Big[W\text{ is a 
pre-}\psi(n,L[\es]|\eta,\delta,\xg,\yg)\text{-witness}\Big].\end{split}\end{equation}

Now suppose that $R$ is not $j(\delta)$-bounded. Let 
$\eta\in(j(\delta),\OR^{R})$
be such that $\rho_\om^{R|\eta}=j(\delta)$
and $R|\eta\npins\Lp_{\Gammag}(R|j(\delta))$.
So $\SS_{\alphag}\sats$``there is no above-$j(\delta)$,
$(\om,\om_1)$-strategy for $R|\eta$''. 
The statement in (\ref{eqn:N_thinks_segs_bounded_it})
lifts to $R,j(\delta)$ (including the case
that $j$ is not $\Sigma_2$-elementary, but is cofinal $\Sigma_1$-elementary). Let
$n<\om$ be such that $\yg\in\OD^{\alphag+n\om}_{\xg}$.
Let $W\pins R$ be a minimal 
pre-$\psi(n,R|\eta,j(\delta),\xg,\yg)$-witness. Then $W$ is iterable.
It follows that there is some $\gamma\in\Lim$ such that
$\J_\gamma\sats$``there is an above-$j(\delta)$, $(\om,\om_1)$-strategy for $R|\eta$''
and  $\yg\notin\OD^{\gamma+n\om}_{\xg}$.
Since $\yg\in\OD^{\alphag+n\om}_{\xg}$,
we have $\gamma<\alphag$,
so in fact $\SS_{\alphag}\sats$``there is an above-$j(\delta)$, $(\om,\om_1)$-strategy for $R|\eta$'', a contradiction.

Part \ref{item:R_is_<j(delta)-bounded}
Suppose also that $N$ is $[\zeta,\delta)$-bounded;
we deduce that $R$ is $[j(\zeta),j(\delta))$-bounded. Let $\psi(\dot{S},\dot{\eta})$ be the $\Sigma_1$  formula of $\Ll_{L(\RR)}$
asserting
``$\dot{S}$ is a sound premouse, $\dot{\eta}<\OR^{\dot{S}}$ 
and there is an above-$\dot{\eta}$, $(\om,\om_1)$-strategy
for $\dot{S}$''. Then
\[ \begin{split}
    N\sats&\all\eta\in[\zeta,\delta)\ \all\xi\in(\eta,\OR)\\
   &\ \  [\text{if }\eta\text{ is a strong 
cutpoint of }N|\xi\text{ and }\rho_\om^{N|\xi}\leq\eta\text{ then  }\\
&\ \  \text{ there is 
an above-}\delta\text{ pre-}\psi(N|\xi,\eta)\text{-witness 
}W\pins L[\es]|\delta^+].
\end{split}\]

This lifts to $R,j(\delta)$,
which by part \ref{item:R_is_j(delta)-bounded} implies $R$ is $[j(\zeta),j(\delta))$-bounded.
\end{proof}

\begin{lem}[MTR preservation]\label{lem:MTR_preservation}
 Let $X$ be high and $N$ be an $X$-premouse
 which is a $\delta$-mtr of degree $n$.
 Let $\Sigma$ be an $(n,\om_1,\om_1+1)^*$-strategy for $N$.
 Let  $\Tt$ be a successor length $n$-maximal tree on $N$ via $\Sigma$
 such that $b^\Tt$ does not drop in model or degree.
 Let $R=M^\Tt_\infty$. Then:
 \begin{enumerate}
  \item  
$R$ is an $i^\Tt(\delta)$-mtr of degree $n$,
and
\item\label{item:again_<j(delta)-bounded} if $N$ is ${<\delta}$-bounded
then $R$ is ${<j(\delta)}$-bounded.
\end{enumerate}
\end{lem}
\begin{proof}
Let $i=i^\Tt$. By Lemma \ref{lem:overspill},
$R$ is $i(\delta)$-bounded
and part \ref{item:again_<j(delta)-bounded} holds.
And the fact that $\rho_{n+1}^R\leq i(\delta)$ is pretty standard fine 
structure; see e.g.~\cite[\S3]{fsfni_v4}.
So we just need to 
 see that $R$ is $i(\delta)$-full; that is, that
\[ \Lp_{\Gammag}(R|i(\delta))\ins R|                                                                                                                                                                                                                                                                                                                                                                                                                                                                                                                                                                                            i(\delta)^{+R}.\]
Suppose not and let $P\pins\Lp_{\Gammag}(R|i(\delta))$
be such that $\rho_\om^P\leq i(\delta)$ and
\[ P\npins 
R|i(\delta)^{+R}\text{ and }P|i(\delta)^{+P}=R|i(\delta)^{+R}.\]
We can successfully compare $P$ with $R$, with trees $\Uu,\Vv$ respectively.
Note here that $i(\delta)$ is a strong cutpoint of $P$ and $R$.
We form $\Uu$ with the unique above-$i(\delta)$,
$(\om,\omega_1+1)$ strategy $\Psi$ for $P$,
and form $\Vv$ with the tail of $\Sigma$.
We get common last model $R'=M^\Uu_\infty=M^\Vv_\infty$, $b^\Uu,b^\Vv$ do not 
drop in model,
and $\deg^\Uu(\infty)=\deg^\Vv(\infty)=n$.
Now $\Psi\in\J_{\alphag}$.
By
Lemma \ref{lem:tame_projecting_stacks},
it follows that $R'$ is above-$i^\Vv(i(\delta))$, $(n,\om_1,\om_1+1)^*$-iterable
in $\SS_{\alphag}$, so $R$ is above-$i(\delta)$, 
$(n,\om_1,\om_1+1)^*$-iterable there,
so $N$ is above-$\delta$, $(n,\om_1,\om_1+1)^*$-iterable there,
and
therefore so is the $\delta$-core of $N$. But $\rho_{n+1}^N\leq\delta$,
and $\Lp_{\Gammag}(N|\delta)\ins N$. This easily gives a contradiction, possibly excluding the case that
$\crit(\pi)=\delta$
where $C$ is the $\delta$-core of $N$ and $\pi:C\to N$ the $\delta$-core map, and $\delta$ is not a strong cutpoint of $C$
(because in this case, $C\not\ins\Lp_{\Gammag}(N|\delta)$ by definition).
But in this situation, note that $C$ is active type 1 and $\delta$ is a strong cutpoint of $C^{\mathrm{pv}}$ and $\rho_1^C=\delta$ (otherwise
condensation gives
partial measures in $\es^C$
with critical point $\delta$). 
Now let $C'=\Ult_0(C,F^C)$,
and note that $C'\ins\Lp_{\Gammag}(N|\delta)$,
a contradiction.
\end{proof}

Note that in the foregoing proof it is important that $b^\Tt$ is non-dropping.

We
now use what we have done so far
to establish something toward the conjectures.

\begin{lem}\label{lem:bounded_Lp_ZF^-}
 Let $\alpha$ be a limit of limits,
 $X\in\HC$ be transitive and
 $R=\Lp_{\Gamma_\alpha}(X)$. Then:
 \begin{enumerate}
  \item\label{item:admissibility_gives_Hull_1_bounded} If $\SS_\alpha$ is admissible then
 $\Hull_1^{\SS_\alpha}(\{X\})$
  is bounded in $\alpha$.
\item\label{item:Hull_1_bounded_gives_ZF^-} If $\Hull^{\SS_\alpha}_1(\{X\})$
 is bounded in $\alpha$
 then $R\sats\ZF^-$,
 and in fact, $\OR^R$ is a cardinal in $L(R)$.
 \end{enumerate}
\end{lem}
\begin{proof}
Part \ref{item:admissibility_gives_Hull_1_bounded}: Let $t=\Th_1^{\SS_\alpha}(\{X\})$. Since $t\in\SS_\alpha$,
an application of admissibility
using the parameter $t$
gives the desired conclusion.

Part \ref{item:Hull_1_bounded_gives_ZF^-}:
Because $\Hull^{\SS_\alpha}(\{X\})$
is bounded in $\alpha$,
there is a limit $\alpha'<\alpha$
such that $R=\Lp_{\Gamma_{\alpha'}}(\{X\})$.
So we get an iteration strategy for $R$ in $\SS_\alpha$, by simply unioning the witnessing strategies in $\J_{\alpha'}$ for the projecting proper segments of $R$.
Therefore $\beta<\om_1$
then 
there is an iteration strategy for $\J_\beta(R)$
in $\SS_\alpha$.
It follows that $\J_\beta(R)$
cannot project ${<\OR^R}$, which suffices.
\end{proof}

\begin{lem}\label{lem:J_alpha-bar_admissible} ($M$)
Suppose  there is no $x\in\RR^M$
such that $\Hull_1^{\SS_\alpha}(\{x\})$ is cofinal in $\alpha$.
Let $\pi:\Ss_{\bar{\alpha}}(\RR^M)\to\SS_\alpha$ be the uncollapse map of Lemma \ref{lem:Sigma_1_Hull_of_J_alpha}, applied with $\alpha=\alphag$.
Then:
\begin{enumerate}
 \item\label{item:all^R_Sigma_1-elem} $\pi$ is $\all^\RR\Sigma_1$-elementary, and
 \item\label{item:admissibility_goes_down} if $\SS_\alpha$ is admissible
 then $\Ss_{\bar{\alpha}}(\RR^M)$
 is admissible.
\end{enumerate}
\end{lem}
\begin{proof}
Part \ref{item:admissibility_goes_down}
is an easy consequence of part \ref{item:all^R_Sigma_1-elem}.
For
part \ref{item:all^R_Sigma_1-elem},
let $x_1\in\RR^M$ and $\varphi$ be a $\Sigma_1$ formula of the $\Ll_{L(\RR)}$ language
such that
\begin{equation}\label{eqn:adm_test} \Ss_{\bar{\alpha}}(\RR^M)\sats\all 
x\in\RR^M\ \varphi(x,x_1). \end{equation}
We will show that
$\SS_\alpha\sats\all x\in\RR\ \varphi(x,x_1)$.

Recall here 
(see Assumption \ref{ass:Sigma_M_not_Gamma-guided})
that
$M$ is strongly $(\alphag,1)$-closed, $\betag\leq\beta^M<\betag+\om^2$, and there is no $\xi<\om_1^M$
 such that
$\Sigma_{M,\geq\xi}$ is $M$-$\Gammag$-guided.
Let $\xi_0<\om_1^M$ be such that $\xg,\yg,x_1\in M|\xi_0$
and $\rho_\om^{M|\xi_0}=\om$,
so $M|\xi_0$ is high.
Fix $\Tt\in\HC^M$ which is
an above-$\xi_0$, $\Gammag$-guided normal tree on 
$M|\om_1^M$
of limit length such that
\[ \Q(\Tt,b)\npins R=\Lp_{\Gammag}(M(\Tt)),\]
where $b=\Sigma_{M|\om_1^M}(\Tt)$. Let $\delta=\delta(\Tt)$. So
by Lemma \ref{lem:bounded_Lp_ZF^-},
$R\sats\ZF^-$ and
$R\pins\Q(\Tt,b)$,
and in particular, $R\sats$``$\delta$ is Woodin''.
We have
$M|\xi_0\pins R\in\HC^M$.

Given an $(R,\Coll(\om,\delta))$-generic $G$ and
$\eta<\OR^R$ such that $\rho_\om^{R|\eta}=\delta$,
let $y_{\eta,G}$ be the canonical real coding $(R|\eta,G)$.
Then $R$, considered as a mouse
over $R|\eta$,
translates into a
 $y_{\eta,G}$-mouse
 $R_{\eta,G}$, and
 $R_{\eta,G}=\Lp_{\Gammag}(y_{\eta,G})$.
So by \ref{fact:mouse_witnesses}, for any $\Sigma_1$ formula $\psi$ and real $y\in (R|\eta)[G]$, 
\[ \SS_{\alphag}\sats\psi(y,x_1)\iff\text{ there is a 
pre-}\psi(y,x_1)\text{-witness }N\pins R_{\eta,G}.\]

Now we may take $G\in M$,
so $R[G]\in\HC^M$ and in particular each real $y\in (R|\eta)[G]$ is in $M$, and hence
by (\ref{eqn:adm_test}) and $\Sigma_1$-elementarity, $\SS_{\alphag}\sats\varphi(y,x_1)$.
So writing $\dot{G}$
for the canonical name for $G$,
we have
$R\sats\ \forces_{\CC_{\delta}}$``For all $\eta\in\OR$
with $\rho_\om^{R|\eta}=\delta$, 
for all reals $y\in(R|\eta)[\dot{G}]$,
there is a  
pre-$\varphi(y,x_1)$-witness $N\pins 
R_{\eta,\dot{G}}$''.

Now this statement is preserved by non-dropping iteration maps
on $R$, and working in $V$, we can make any real generic
over an image of $R$. But by Lemma \ref{lem:overspill},
if $i:R\to R'$ is a correct iteration map, then
\[ R'\ins \Lp_{\Gammag}(R'|i(\delta))
,\]which suffices.
\end{proof}

\begin{dfn}
Let $X$ be high and $N$ be an $\om$-small $X$-premouse.
Let $\delta\in\OR^N$. We say that $N$ is a \emph{$\delta$-mGW} (for 
\emph{minimal Gamma-Woodin}) or \emph{mGW at $\delta$}
iff $N$ is a $<\delta$-bounded $\delta$-mtr of degree $n$, 
$N\sats$``$\delta$ is 
Woodin'' and for some $\chi<\delta$ we have
\[ \Hull_{n+1}^N(\chi\cup\pvec_{n+1}^N)\text{ is unbounded in }\delta.\qedhere\]
\end{dfn}
\begin{rem}
 Let $N$ be a $\delta$-mGW of degree $n$.
So (by the definition of \emph{$\delta$-mtr}) $N$ is $(n,\om_1,\om_1+1)^*$-iterable.
 Let $\eta<\delta$ be such that $N\sats$``$\eta$ is not Woodin''.
 Then the Q-structure $Q\pins N$ for $\eta$
 is such that 
$Q\pins\Lp_{\Gammag}(N|\eta)$. (This is an immediate consequence of 
${<\delta}$-boundedness.)

The last condition in the definition of $\delta$-mGW
(the unboundedness of the hull) already follows
from the rest if also $\rho_{n+1}^N<\delta$, and then in fact 
$\chi=\rho_{n+1}^N$ works (still assuming iterability).\footnote{This argument seems
to use $\om$-smallness,
to get that $\delta\in\rg(\pi)$;
but only assuming
tameness there's still some $\chi<\delta$;
 just take $\chi$ to strictly bound the Woodins
${<\delta}$.}
For suppose otherwise and let
\[ \theta=\sup(\delta\inter\Hull_{n+1}^N(\rho_{n+1}^N\cup\pvec_{n+1}^N)).\]
Let
$C=\cHull_{n+1}^N(\theta\cup\pvec_{n+1}^N)$
and $\pi:C\to N$ be the uncollapse. Then $C$ is $\theta$-sound. By $\om$-smallness, $\delta\in\rg(\pi)$, and as usual we therefore have
$\pi(\theta)=\delta$, $\theta$ is Woodin in $C$
and a strong cutpoint of $C$, and is a limit cardinal of $N$. Now $\theta$ is not Woodin in $N$, because otherwise it becomes a Woodin limit of Woodins in $N$, contradicting tameness.
But by condensation, $C||\theta^{+C}=N||\theta^{+C}$. Since $\theta$ is Woodin in $C$,
therefore $\theta^{+C}<\theta^{+N}$,
and therefore by $(n+1)$-universality,
$\rho_{n+1}^N<\theta$.
So $\theta$ is definably singularized over $C$, so $C$ is the (iterable, $\theta$-sound) Q-structure
for $\theta$. Therefore
 $C\pins N$.
But $\core_{n+1}(N)$ is definable over $C$,
so $\core_{n+1}(N)\in N$, a contradiction.

\end{rem}

\begin{cor}\label{cor:it_pres_mGW}
 Let $X,N,\delta,n,\Sigma,\Tt,R$ be as in \ref{lem:MTR_preservation},
 with $N$ a $\delta$-mGW.
 Then $R$ is an $i^\Tt(\delta)$-mGW. 
\end{cor}

Standard fine structure gives:

\begin{lem}\label{lem:core_of_mGW}
 Let $D$ be a $\zeta$-mGW of degree $d$. Then the $\zeta$-core
 of $D$ is a $\zeta$-sound $\zeta$-mGW of degree $d$.
\end{lem}

\begin{dfn}
 Let $X\in\HC$. Say that $X$ is \emph{sufficient}
 iff $X$ is transitive and there are $D\in X$ and $\zeta\in D$
 such that $D$ is a $\zeta$-sound $\zeta$-mGW.
\end{dfn}

Note that if $X$ is sufficient,
as witnessed by $D$,
then $D$ is a $Y$-premouse
for some high $Y$,
so $X$ is also high.

The following argument, due to the second author,
comes from the Steel-Schindler email exchange \cite{emails_2005} of 2005:

\begin{lem}\label{lem:coring_preserves_mtr}
 Let $X$ be sufficient, as witnessed by $D,\zeta$,
 of degree $d$.
 Let $H,N$ be $n$-sound $\om$-small $X$-premice
 and $\pi:H\to N$ be an $n$-embedding such that $\pi(p_{n+1}^H)=p_{n+1}^N$.
Let $\delta\in H$ be such that $\rho_{n+1}^H\leq\delta$
and suppose $N$ is a 
$\pi(\delta)$-mtr of degree $n$. Then $H$ is a $\delta$-mtr of degree $n$, 
and moreover, $n=d$.
\end{lem}
\begin{rem}
 The ``moreover'' clause, that $n=d$,
 does not convey the extent of agreement
 between $D$ and $H$,
 nor between $D$ and $N$, which becomes clear
 in the proof: assuming $\delta^{+H}$-soundness and $\pi(\delta)^{+N}$-soundness
 of $H$ and $N$ respectively, they are, modulo genericity iterations and generic extensions,
 essentially equivalent above the iteration images
 of $\zeta$, which are $\delta^{+H}$ and $\pi(\delta)^{+N}$.
\end{rem}

\begin{proof}
By replacing $H$ with the $\delta^{+H}$-core of $H$,
we may assume that $H$ is $\delta^{+H}$-sound.
Similarly, we may assume that $N$ is
$\pi(\delta)^{+N}$-sound.

Let $\chi<\zeta$ with
$p_{d+1}^D\inter\zeta\sub\chi$ and
$\Hull_{d+1}^D(\chi\cup\pvec_{d+1}^D)\text{ unbounded in }\zeta$,
and such that $D$ 
has no Woodins in $[\chi,\zeta)$.

Consider the $N|\pi(\delta)^{+N}$-pseudo-genericity-iteration $\Tt$ of
$D$ for the extender algebra $\BB_{\zeta,\geq\chi}^D$, of length $\leq\pi(\delta)^{+N}$, after 
first linearly iterating the least measurable of $D$ which is $\geq\chi$
out to $\delta$. Let $b=\Sigma_D(\Tt)$.
Then $\Tt$ is 
definable from parameters over $N|\pi(\delta)^{+N}$,
with Q-structures at limit stages determined by P-construction.
Since $\Tt$ uses only total extenders,
\ref{cor:it_pres_mGW} applies to it. Let $\eta=\lh(\Tt)$.

Suppose the process terminates in $\eta<\pi(\delta)^{+N}$ stages.
That is,
\[ i^\Tt_b(\zeta)=\delta(\Tt)=\eta<\pi(\delta)^{+N}.\] Then by 
\ref{cor:it_pres_mGW},
$M^\Tt_b$ is a mGW at $\eta$,
of degree $d$, so is not above-$\eta$,
$(d,\om_1+1)$-iterable in $\SS_{\alphag}$.
But because $\zeta<\pi(\delta)^{+N}$, we get a Q-structure
$Q$ for $M(\Tt)$, given by the P-construction of some $R\pins 
N|\pi(\delta)^{+N}$ above $M(\Tt)$,
and $R$ is above-$\eta$, $(\om,\om_1+1)^*$-iterable
in $\SS_{\alphag}$. But then a comparison leads to contradiction.

So $\eta=\pi(\delta)^{+N}$. Now $i^\Tt_b(\zeta)=\eta$.
For suppose $i^\Tt_b(\zeta)>\eta$. Then $M^\Tt_b\sats$``$\eta$ is not 
Woodin'', so letting $Q\pins M^\Tt_b$ be the Q-structure,
by \ref{cor:it_pres_mGW}, we have
\[ Q\pins \Lp_{\Gammag}(M(\Tt)).\]
But because $N$ is a $\pi(\delta)$-mtr, it is easy to see that
$\Lp_{\Gammag}(N|\eta)\ins N$.
But then $Q$ is reached by the P-construction of $N$ over $M(\Tt)$,
so $b\in N$, which contradicts termination of genericity iteration as usual.

Now $M^\Tt_b$ is the output of the P-construction $P$ of $N$ over $M(\Tt)$.
For $M^\Tt_b\sats$``$\eta$ is Woodin'',
so $\eta$ is a strong cutpoint of $M^\Tt_b$,
and $M^\Tt_b$ is $\eta$-sound; and likewise for $N$.
Therefore $M^\Tt_b\ins P$ or $P\ins M^\Tt_b$.
But also, $M^\Tt_b$ is an $\eta$-mGW, hence an $\eta$-mtr,
and $N$ is an $\eta$-mtr, and it follows
that $M^\Tt_b=P$.

Now $i^\Tt_b$ is continuous at $\zeta$, so
\[ H'=\Hull_{d+1}^{M^\Tt_b}(\chi\cup\pvec_{d+1}^{M^\Tt_b})\text{ is unbounded 
in }\eta\]
and $H'\sub\rg(i^\Tt_b)$.

Now $d\leq n$. For
$\rho_{n+1}^N\leq\pi(\delta)<\rho_n^N$
and
$\rho_{d+1}^{M^\Tt_b}\leq\eta<\rho_d^{M^\Tt_b}$,
and since $P=M^\Tt_b$, the fine structure of P-constructions
gives immediately that $M^\Tt_b$ is 
$\rDelta_1^N(\{M(\Tt)\})$, and $\rho_i^{M^\Tt_b}=\rho_i^N$, etc, for 
$i\leq\max(d,n)$, and that $d\leq n$.
But if $d<n$ then there 
is an $\bfrSigma_n^N$ singularization of $\eta=\pi(\delta)^{+N}$,
which contradicts the fact that $\pi(\delta)<\rho_n^N$.
So $d=n$.\footnote{Note that therefore 
$\pi(\delta)^{+N}<\rho_n^N$.}

Now we assume at this point, for notational simplicity,
that $n=0$, $D$ is passive and $\OR^D$ is a limit of limits; the general 
case is then a straightforward
adaptation using standard fine structural techniques.
By these assumptions, $i^\Tt_b$ is cofinal.

Given $\alpha<\OR^D$ with $\max(p_1^D)<\alpha$ let
\[ D_\alpha=\Hull^{D|\alpha}_1(\chi\cup\{p_1^D\}), \]
\[ M_\alpha=\Hull^{M^\Tt_b|i^\Tt_b(\alpha)}_1(\chi\cup\{p_1^{M^\Tt_b}\}),\]
so $D_\alpha\iso M_\alpha$. Let $\zeta_\alpha=\sup(\zeta\inter 
D_\alpha)$. So $\zeta_\alpha<\zeta$ and 
$\left<\zeta_\alpha\right>_{\alpha<\OR^D}$
is cofinal in $\zeta$.
Let
\[ j_\alpha:\zeta\inter D_\alpha\to\eta\inter M_\alpha \]
be the isomorphism. So $j_\alpha\sub i^\Tt_b\rest\zeta_\alpha$ and note $j_\alpha\in N$.

Let $\gamma_\alpha$ be the least 
$\gamma\in b$ such that
$\crit(i^\Tt_{\gamma b})\geq\sup i^\Tt_{0\gamma}``\zeta_\alpha$.
Then\footnote{For clearly $\gamma_\alpha$
has this property. Suppose $\gamma<\gamma_\alpha$ and $\gamma$
has the stated property. Let $\xi=\max(b\inter(\gamma+1))$
and 
$\beta+1\in b$ with $\pred^\Tt(\beta+1)=\xi$.
Because $\xi<\gamma_\alpha$,
we have $\crit(E^\Tt_\beta)\in\sup i^\Tt_{0\xi}``\zeta_\alpha$.
And 
$\rg(j_\alpha)\sub\nu(E^\Tt_\gamma)\leq\lambda(E^\Tt_\beta)$, 
but then
\[ \sup\rg(j_\alpha)\sub\lambda(E^\Tt_\beta)<\sup 
i^\Tt_{0,\beta+1}``\zeta_\alpha\leq\sup 
i^\Tt_b``\zeta_\alpha=\sup\rg(j_\alpha),\]
contradiction.}
\begin{equation}\label{eqn:gamma_alpha_def}\gamma_\alpha=
\text{ least }\gamma<\lh(\Tt)\text{ such that }
j_\alpha\sub i^\Tt_{0\gamma}\text{ and } 
\rg(j_\alpha)\sub\nu(E^\Tt_\gamma).\end{equation}

Now the map
\[ \sigma:\alpha\mapsto(D_\alpha,\zeta_\alpha,M_\alpha,j_\alpha,\gamma_\alpha), 
\]
where
\[ 
\dom(\sigma)=\{\beta\in\Hull_1^D(\chi\cup\{p_1^D\})\mid 
\max(p_1^D)<\beta<\OR^D\},\]
is $\Sigma_1^N(\{p_1^N,\vec{x},\pi(\delta)\})$ for some $\vec{x}\in X^{<\om}$.

(We can easily pass from $\alpha\in\dom(\sigma)$ to $D_\alpha,\zeta_\alpha$.
To pass to $M_\alpha$, use that $\Tt$ is definable over 
$N|\pi(\delta)^{+N}$
from parameters in $X$, so $M^\Tt_b$
is $\Delta_1^N(\{\vec{x},\pi(\delta)^{+N}\})$
for some $\vec{x}\in X^{<\om}$,
and $\pi(\delta)^{+N}\leq\max(p_1^N)$ since $N$ is passive.
This easily gives $j_\alpha$,
and from here we get $\gamma_\alpha$ via (\ref{eqn:gamma_alpha_def}).)

Since $\Hull_1^D(\chi\cup\{p_1^D\})$ is cofinal $\zeta$,
it is also cofinal in $\OR^D$, so $\dom(\sigma)$ is cofinal in $\OR^D$.

So $N\sats\all\alpha,\beta\in\dom(\sigma)$,
we have $\gamma_\alpha<\lh(\Tt)$, and if $\alpha<\beta$ then 
\[ \gamma_\alpha\leq_\Tt\gamma_\beta\text{ and } 
\crit(i^\Tt_{\gamma_\alpha\gamma_\beta})\geq\sup 
i^\Tt_{0\gamma_\alpha}``\zeta_\alpha.\]

Now all parameters used to define these things are in $\rg(\pi)$,
so they have preimages in $H$. Write $\pi(\bar{\Tt})=\Tt$ etc.
So $\bar{\Tt}$ is a genericity iteration of $D$,
via correct strategy, and note that the  $\bar{\gamma}_\alpha$
yield a $\bar{\Tt}$-cofinal branch $\bar{b}$ such that
 $i^{\bar{\Tt}}_{\bar{b}}(\zeta)=\bar{\eta}$.
Moreover, the direct limit
of the maps $\pi\rest M^{\bar{\Tt}}_{\bar{\gamma}_\alpha}:M^{\bar{\Tt}}_{\bar{\gamma}_\alpha}\to M^\Tt_{\gamma_\alpha}$ (under the iteration maps of $\bar{\Tt}$ along $\bar{b}$;
note that each $M^{\bar{\Tt}}_{\bar{\gamma}_\alpha}\in H$) is 
a $0$-embedding $\sigma:M^{\bar{\Tt}}_{\bar{b}}\to M^\Tt_b$, and it follows that $\bar{b}=\Sigma_D(\bar{\Tt})$.
But then since $i^{\bar{\Tt}}_{\bar{b}}(\zeta)=\bar{\eta}$,
it follows
that $H|\bar{\eta}=\Lp_{\Gammag}(H|\delta)$.
(We have $H|\beta{\eta}\ins\Lp_{\Gammag}(H|\delta)$ because $N|\eta=\Lp_{\Gammag}(N|\pi(\delta))$. For the other direction, consider the forcing extension $M^{\bar{\Tt}}_{\bar{b}}[H|\bar{\eta}]$, where $\bar{\eta}$
is regular, and which, by Lemma \ref{lem:MTR_preservation}, contains $\Lp_{\Gammag}(H|\bar{\eta})$, hence contains  $\Lp_{\Gammag}(H|\delta)$.) This completes the proof.
\end{proof}

\begin{dfn}\label{dfn:descent}
Let $N\in\HC$ be a premouse.
Say that $(N,n,\eta)$ is \emph{pre-appropriate}
iff $\eta<\OR^N$ is a strong cutpoint of $N$,
$N$ is $n$-sound and $\rho_{n+1}^N\leq\eta<\rho_n^N$ (we do not assume $\eta$-soundness).
This $n$ is the \emph{degree} of $(N,\eta)$.
Say that $(N,n,\eta)$ is \emph{appropriate}
iff it is pre-appropriate and  
$N$ is above-$\eta$, $(n,\om_1+1)$-iterable. 
Say that $(N,n,\eta)$ is \emph{$\Gammag$-\tu{(}pre-\tu{)}appropriate} iff $N$ is (pre-)appropriate and $N|\eta^{+N}=\Lp_{\Gammag}(N|\eta)$.

Fix an appropriate $(N,n,\eta)$.
Then $\Sigma_{N,\eta}$ denotes the unique
above-$\eta$, $(n,\om_1+1)$-strategy for $N$ (uniqueness is by Lemma \ref{lem:tame_projecting_stacks}). Let $\Sigma=\Sigma_{N,\eta}$.
An $(N,\eta)$-\emph{descent} is a pair
\[ (\left<N_i,n_i,\eta_i\right>_{i\leq m},\left<\Tt_i\right>_{i<m})\in\HC,\]
where $m\leq\om$, such that there are $\left<b_i\right>_{i<m}$ such that:
\begin{enumerate}
 \item $N_0=N$ and $n_0=n$ and $\eta_0=\eta$.
 \item $(N_i,n_i,\eta_i)$ is pre-appropriate for each $i\leq m$.
\item $\Tt_i$ is a non-trivial, above-$(\eta_i^+)^{N_i}$, $n_i$-maximal tree on 
$N_i$, for each $i<m$.
\item For $i<m$, if $\Tt_i$ has limit length, let $\Uu_i=\Tt_i\conc b_i$,
and otherwise let $\Uu_i=\Tt_i$. Then $\Uu_i$ is an $n_i$-maximal tree.
 Moreover,
 \[ \Uu_0\conc\Uu_1\conc\ldots\conc\Uu_{m-1} \]
 is essentially\footnote{For example,
it could be that $\Uu_1$ is literally a tree on $N_1\pins M^{\Uu_0}_\infty$,
but with some trivial changes, we literally obtain a tree via $\Sigma$.} a tree via $\Sigma$ (hence $n$-maximal).
 \item\label{item:soundness_i>0} $\nu(\Tt_i)\leq\eta_{i+1}$ for each $i<m$.
 \item For each $i<m$, we have
$(N_{i+1},n_{i+1})\ins(M^{\Uu_i}_\infty,\deg^{\Uu_i}(\infty))$
and either
\[ (N_{i+1},n_{i+1})\pins(M^{\Uu_i}_\infty,\deg^{\Uu_i}(\infty))
\text{ or }b^{\Uu_i}\text{ drops in model or degree;}\]
moreover,
if $\Tt_i$ has limit length then $N_{i+1}\pins M^{\Uu_i}_\infty$.
\end{enumerate}

\begin{rem}\label{rem:soundness_for_later_ones}
Note that it follows that $N_{i+1}$ is $\eta_{i+1}$-sound for each $i<m$ (although we have not assumed
that $N_0$ is $\eta_0$-sound).
\end{rem}

The above is a \emph{$\Gammag$-descent}
iff $(N_i,n_i,\eta_i)$ is $\Gammag$-pre-appropriate (hence $\Gammag$-appropriate)
for each $i\leq m$.

A $\Gammag$-descent as above is \emph{$\Gammag$-maximal} iff it has finite length and
there is no proper extension which is also a $\Gammag$-descent.
(That is, there is no $\Gammag$-descent
\[ (\left<N'_i,n'_i,\eta_i\right>_{i\leq 
m+1},\left<\Tt'_i\right>_{i<m+1})\in\HC\]
with $N'_i=N_i$ and $n'_i=n_i$ for $i\leq m$
and $\Tt'_i=\Tt_i$ for $i<m$.)

Say that $(N,n,\eta)$ is \emph{$\Gammag$-stable}
iff it is the end node of a $\Gammag$-maximal
$\Gammag$-descent. 

In the context of the conjectures, 
a $\Gammag$-descent which is in  $\HC^M$ is called
 \emph{$M$-$\Gammag$-maximal} iff
it has finite length and there is no proper extension
in $\HC^M$ which is also a $\Gammag$-descent.
Say that $(N,n,\eta)$ is \emph{$M$-$\Gammag$-stable}
iff it is the end node of an $M$-$\Gammag$-maximal
$\Gammag$-descent.
\end{dfn}

An $\om$-descent would easily yield a normal tree with a unique
branch which drops infinitely often, so:

\begin{lem}\label{lem:no_om-descent}
If $(N,n,\eta)$ is appropriate,
then there is no $(N,\eta)$-descent of length $\om$.
\end{lem}

\begin{lem}\label{lem:get_stable} Let
$\mathscr{D}=(\left<(N_i,n_i,\eta_i)\right>_{i\leq 
m},\left<\Tt_i\right>_{i<m})$
be a $\Gammag$-descent.
Then:
\begin{enumerate}
 \item\label{item:max_descent_ex} There is a $\Gammag$-maximal $\Gammag$-descent $\mathscr{D}'$
extending $\mathscr{D}$.
\item\label{item:sound_stable} For each $x\in\HC$ there is a transitive $X\in\HC$ and a $\Gammag$-stable $\Gammag$-appropriate
tuple $(N,n,\eta)$, such that $x\in X$ and $N$ an
$\eta$-sound $X$-premouse.
\end{enumerate}
\end{lem}
\begin{proof}
Part \ref{item:max_descent_ex}
is an immediate consequence of the definitions and 
Lemma \ref{lem:no_om-descent}, using some $\DC_{\RR}$.
Note here that $\Gammag$-stability
for the tuple $(N,n,\eta)$
just refers to iterability above $\eta$.

Part \ref{item:sound_stable}
is an easy consequence
of the previous part and earlier lemmas.
(We get $\eta$-soundness because
we can start with $(N_0,n_0,\eta_0)$ such that $N_0$ is $\eta_0$-sound,
by replacing the given $N_0$
by its $\eta_0$-core if needed;
and cf.~Remark \ref{rem:soundness_for_later_ones}.)
\end{proof}

\begin{lem}\label{lem:get_M-stable}($M$)
Let $\mathscr{D}$ be as in Lemma \ref{lem:get_stable}.
\begin{enumerate} 
 \item\label{item:M-max_dexcent_ex} If  $\mathscr{D}\in \HC^M$
then there is an $M$-$\Gammag$-maximal $\Gammag$-descent $\mathscr{D}'\in\HC^M$ extending $\mathscr{D}$.
\item\label{item:sound_M-stable} For each $x\in\HC^M$ there is a transitive $X\in\HC^M$
and  an $M$-$\Gammag$-stable $\Gammag$-appropriate tuple $(N,n,\eta)\in\HC^M$, such that $x\in X$ and $N$ is an $\eta$-sound $X$-premouse.
\end{enumerate}
\end{lem}
\begin{proof}
This is like for Lemma \ref{lem:get_stable}, but restricting
attention to elements of $\HC^M$.
(Note that most of the argument,
including following the relevant iteration strategy, need not take place in $M$.)
In Part \ref{item:sound_M-stable}
we can start with $N_0\pins M$.
\end{proof}
\subsection{Mtr-suitable mice
and admissible gaps}
\begin{dfn}\label{dfn:mtr-suitable}
 Let $X\in\HC$ be transitive,
 and $N$ be an $X$-premouse.
 We say that $N$ is \emph{almost mtr-suitable}
 iff for some $n<\om$:
 \begin{enumerate}
 \item $X$ is high.
\item $N$ is $\om$-small.
 \item $N$ is $n$-sound, with
$\rho_{n+1}^N=X$.
  \item $N$ has $\om$ Woodin cardinals with supremum $\lambda$, where 
$\rank(X)<\lambda\leq\rho_n^N$.
  \item\label{item:power_set_C_Gammag} For each strong cutpoint 
$\eta<\lambda$ of $N$, $N$ is an $\eta$-mtr (of degree $n$).
   \item\label{item:Gammag-boundedness} ($\Gammag$-boundedness) Let 
$\theta,\xi$ be such that $\rank(X)\leq\theta<\xi<\lambda$,
  $\theta$ is a strong cutpoint of $N|\xi$ and $\rho_\om^{N|\xi}\leq\theta$.
  Then $N|\xi\pins \Lp_{\Gammag}(N|\theta)$.
  \item $N$ is $(n,\om_1+1)$-iterable.
 \end{enumerate}
 Clearly $n$ is determined by $N$; say $n$ is the \emph{degree} of $N$.
 
An almost mtr-suitable premouse $N$  is \emph{mtr-suitable} iff, with notation as above,
\begin{enumerate}[resume*]
 \item $X$ is sufficient, and
 \item  $p^N_{n+1}\cap\lambda=\emptyset$.\qedhere
\end{enumerate}
\end{dfn}

Note that each almost mtr-suitable premouse $N$
has a unique $(n,\om_1+1)$-strategy $\Sigma_N$,
and so is in fact $(n,\om_1,\om_1+1)^*$-iterable;
see the proof of Lemma \ref{lem:tame_projecting_stacks}.

The next lemma  is  easy
to see:
\begin{lem}\label{lem:mtr-suitable_o_prop_seg_mtr}
 Let  $X\in\HC$ be transitive,
 and $N$ be an almost mtr-suitable $X$-premouse. Let $X'=X$ or $X'=N|\eta$
 where $\eta$ is a strong cutpoint of $N$. Let $N_0\pins N$ with $X'\in N_0$, and let $N'$ be the reorganization of $N_0$ as an $X'$-premouse.
 Then $N'$ is not an $\eta'$-mtr at any $\eta'$. Therefore
 $N'$ is not almost mtr-suitable.
\end{lem}

\begin{lem}\label{lem:cofinal_mtr}
Suppose $N$ satisfies all requirements
of \tu{(}almost\tu{)} mtr-suitability, except that we only know
condition \ref{item:power_set_C_Gammag} holds
for cofinally many strong cutpoints $\eta<\lambda$. Then $N$ is \tu{(}almost\tu{)} mtr-suitable.
\end{lem}
\begin{proof}
If $\eta$ is a strong cutpoint of $N$,
compare $N$ with $R$ where $R\pins\Lp_{\Gammag}(N|\eta)$ projects to $\eta$ and $\eta^{+R}=\eta^{+N}$.
Then $R$ and $N$ coiterate to a common model,
with no drops in model or degree on either side.
But then by the above-$\eta$ iterability
of $R$ in $\SS_{\alphag}$,
and since the normal strategy for $R$ extends to one for stacks,
$N$ is above-$\eta'$ iterable  in
$\SS_{\alphag}$,
for some $\eta'>\eta$ where condition \ref{item:power_set_C_Gammag} holds, and this is a contradiction.
\end{proof}

\begin{lem}\label{lem:mtr-suitable_mouse}Suppose $\SS_{\alphag}$ is admissible. Let $X$ be high
 and $(N,n,\eta)$ be  $\Gammag$-stable
 $\Gammag$-appropriate,
 with $N$ over $X$.
 Then there is a countable successor length tree
 $\Tt$ on $N$, via $\Sigma_{N,\eta}$ \tu{(}Definition \ref{dfn:descent}\tu{)},
 such that $b^\Tt$ does not drop
 in model or degree, and there is $\delta<\OR^{M^\Tt_\infty}$ such that $\delta$ is Woodin in
 $M^\Tt_\infty$ \tu{(}hence a strong cutpoint of $M^\Tt_\infty$\tu{)},
 and an $(M^\Tt_\infty,\Coll(\om,\delta))$-generic
 $G$ 
 such that $M'$ is  mtr-suitable 
 and $\Gammag$-stable,
 where $M'$ is the reorganization
 of $M^\Tt_\infty[G]$ as a premouse over $(M^\Tt_\infty|\delta,G)$. Moreover, if $N$ is $\eta$-sound then we can
 arrange that $M'$ is sound.
\end{lem}

\begin{proof}
 We may assume $X=N|\eta$. We will first find an almost mtr-suitable iterate of $N$. Let $\Psi_{\Gammag}$ be the ``sound Q-structure $\Gammag$-short tree strategy''; that is,
 given a limit length tree $\Uu$ on some premouse, $\Psi_{\Gammag}$ is the unique
 $\Uu$-cofinal branch $b$ such that
 $Q\ins M^{\Uu}_b$, where $Q\pins\Lp_{\Gammag}(M(\Uu))$ is the Q-structure for $M(\Uu)$,
 if such $(Q,b)$ exists, and otherwise $b$ is undefined. Note that any tree on $N$
 which is via $\Psi_{\Gammag}$
 is also via $\Sigma_N$.
 
Now there is a limit length
 tree $\Tt$ on $N$, via $\Psi_{\Gammag}$,
 such that letting $b=\Sigma_N(\Tt)$,
 then $\Tt\conc b$ is not via $\Psi_{\Gammag}$.
 For otherwise using the admissibility
 of $\SS_{\alphag}$, $N$ would be $(\om,\om_1)$-iterable, hence $(\om,\om_1+1)$-iterable,
 in $\SS_{\alphag}$. (Consider
 the statement that ``for every putative $n$-maximal
 tree $\Tt$ on $N$, either $\Tt$ has wellfounded models or there is a $\Tt$-maximal branch $b$ such that,
 letting $\eta=\sup(b)$, then $(\Tt\rest\eta)\conc b$ is via $\Psi_{\Gammag}$'',
 noting that this can written
 in the form ``$\all x\in\RR\ \varphi(x,N)$'',
 with some $\Sigma_1$ formula $\varphi$.)

 Let $\Tt_0$ witness this, and $b_0=\Sigma_N(\Tt_0)$. Note then that,
 using $\Gammag$-stability, we have
 \[ \Gammag(M(\Tt_0))= M^{\Tt_0}_{b_0}|\delta(\Tt_0)^{+M^{\Tt_0}_{b_0}} \]
and  $M^{\Tt_0}_{b_0}\sats$``$\delta(\Tt_0)$ is Woodin'' and $b_0$ does not drop in model or degree, and further, $\delta(\Tt_0)<\rho_n^{M^{\Tt_0}_{b_0}}$.
It follows that $(M^{\Tt_0}_{b_0},n,\delta(\Tt_0))$ is also $\Gammag$-stable
$\Gammag$-appropriate.

Repeating the proceeding process
$\om$ many times,
we get an $n$-maximal tree of the form $\Tt=\Tt_0\conc\Tt_1\conc\Tt_2\conc\ldots$,
and letting $b=\Sigma_N(\Tt)$
(note $b$ is actually trivial
as $\Tt$ is equivalent to a stack of length $\om$), then $b$ does not drop in model or degree,
and $\lambda=\delta(\Tt)$
is a limit of Woodins,
and in fact each $\delta(\Tt_i)$ is Woodin in $M^\Tt_b$.

 Since $N$ is $\om$-small, we have $M^\Tt_\infty=\Ss_\xi(M(\Tt))$ for some $\xi\in\Lim_0$. Then $M^\Tt_\infty$ is almost mtr-suitable
 (recall again that we assumed
 that $X=N|\eta$;
 the original $N$ need not be iterable below $\eta$); the fact that $(M^\Tt_\infty,n,\eta)$
 is $\Gammag$-stable follows
 from the $\Gammag$-stability
 of $(N,n,\eta)$ together with (full) normalization.
 
In particular, $M^\Tt_\infty$ is mGW,
as witnessed by each of the $\delta(\Tt_n)$.

Now form a correct tree $\Uu$ on $M^\Tt_\infty$,
with last model $P$,
iterating to make $M^\Tt_\infty$
generic for the extender algebra of $P$
at $\delta=j(\delta(\Tt_0))$, where $j:M^\Tt_\infty\to P$ is the iteration map.
Let $G$ be $(P,\Coll(\om,\delta))$-generic.
Let $n<\om$ be such that $(p_{n+1}^P\cap j(\lambda))\sub \delta'= j(\delta(\Tt_n))$. Let $M'$ be the reorganization
of $P[G]$ as a premouse over $(P|\delta',G)$. 

It is now straightforward to see that $M'$ 
is mtr-suitable, and clearly (using normalization,
see Remark \cite{fullnorm}) it is $\Gammag$-stable.

Finally suppose that $N$ is $\eta$-sound,
and let $M''$
be the $\delta'$-core of $M'$.
Then $M''$ is $\delta'$-sound,
and note that by Lemma \ref{lem:coring_preserves_mtr}
it is also mtr-suitable.
To see that this can be achieved
with an iteration tree, just use the normalization of $\Tt_0\conc\ldots\conc\Tt_n\conc\Uu$,
and note that this works.
\end{proof}

\begin{lem}\label{item:iterate_pres_mtr-suitability_stability}
 Let $N$ be  mtr-suitable of degree $n$. Let $\Tt$ be a successor length tree on $N$ via $\Sigma_N$. Then:
 \begin{enumerate}
\item\label{item:stability_non-dropping} If $b^\Tt$ does not drop then 
$M^\Tt_\infty$ is mtr-suitable; hence no $P\pins M^\Tt_\infty$ is mtr-suitable.
\item\label{item:stability_dropping} If $N$ is $\Gammag$-stable and $b^\Tt$ drops then no $P\ins 
M^\Tt_\infty$ is  mtr-suitable.
\end{enumerate}
\end{lem}
\begin{proof}
Part \ref{item:stability_non-dropping}: The mtr-suitability of $M^\Tt_\infty$
follows easily from 
MTR preservation (\ref{lem:MTR_preservation})
and Lemmas
\ref{lem:mtr-suitable_o_prop_seg_mtr} and \ref{lem:cofinal_mtr}.

Part \ref{item:stability_dropping}: This
follows easily from $\Gammag$-stability.\qedhere
\end{proof}

By  \cite{sile}\footnote{Or  its proof,
if the authors of \cite{sile}
implicitly assumed more than $N\sats$``$\om_1$ exists''.}
we have:
\begin{fact}\label{fact:self-it}
If $N$ is a tame $(0,\om_1+1)$-iterable premouse
satisfying ``$\om_1$ exists''
then
 there is $\xi<\om_1^M$ such that
 $\Sigma_{N|\om_1^N,\geq\xi}\rest\HC^N$
 is definable from parameters over $N|\om_1^N$.
\end{fact}

\begin{rem}\label{rem:if_om_1_exists_easier}
Part \ref{item:get_mtr-suitable_in_M_with_strategy_in_M}
of the following lemma
can be arranged just by starting
above some $\xi$ witnessing
Fact \ref{fact:self-it},
assuming that $M\sats$``$\om_1$ exists''. However,
we are not assuming this,
so we give a direct proof in our present context (which is very related
to the original proof of Fact \ref{fact:self-it}).\end{rem}

\begin{lem} ($M$)
Suppose $\SS_{\alphag}$
is admissible. Then for
 each $x\in \HC^M$, there are $X,N\in\HC^M$ such that
 \begin{enumerate}[label=\tu{(}\roman*\tu{)}]
 \item\label{item:get_mtr-suitable_in_M}  
$X$ is sufficient,
 $x\in X$, 
$N$ is  a sound, mtr-suitable, $\Gammag$-stable
 premouse  over $X$, and 
 \item\label{item:get_mtr-suitable_in_M_with_strategy_in_M}  $\HC^M$ is closed under $\Sigma_N$
 and $\Sigma_N\rest\HC^M$ is definable from parameters over $M|\om_1^M$.
\end{enumerate}
 \end{lem}

Note that the lemma gives a $\Gammag$-stable $N$,
not just $M$-$\Gammag$-stable.
Note that since $\Sigma_N$ simply determines
$(\Sigma_N)^{\stk}$, it also follows
that $\HC^M$ is closed under $(\Sigma_N)^{\stk}$
and $(\Sigma_N)^{\stk}$ is definable from parameters over $M|\om_1^M$. (Note here also that even if $M=M|\om_1^M\not\sats\ZFC^-$,
we do have that $M|\om_1^M$ is a limit of proper segments which model $\ZFC^-$, because of Assumption \ref{ass:Sigma_M_not_Gamma-guided}.)

\begin{proof}
Part \ref{item:get_mtr-suitable_in_M}:
By Lemmas \ref{lem:bounded_Lp_ZF^-}and \ref{lem:J_alpha-bar_admissible},  $\Ss_{\bar{\alpha}}(\RR^M)$
is admissible in $M$, so the first
part of the proof 
of Lemma \ref{lem:mtr-suitable_mouse} can also be executed in $M$, but starting with
an  $M$-$\Gammag$-stable tuple
$(N,n,\eta)$,
instead of $\Gammag$-stable,
and with $N$ being $\eta$-sound.
This produces a tree $\Tt=\Tt_0\conc\Tt_1\conc\ldots\sub\HC^M$ (so $\Tt_0\conc\ldots\conc\Tt_n\in\HC^M$
 for each $n<\om$; it doesn't matter whether $\Tt\in M$)
such that $M^\Tt_\infty$
has $\om$ Woodins.

Let $P$ be the $\delta=\delta(\Tt_i)$-core of $M^\Tt_\infty$, where $i<\om$ is sufficiently large,
like in the earlier proof; so actually $P=M^{\Tt_i}_\infty$
is a $\delta$-sound $\delta$-mGW, and
 $P\in\HC^M$.

We claim that $P$ is mtr-suitable and $\Gammag$-stable (in $V$)
as a mouse over $P|\delta_i^P$.
To see this, fix a sound mtr-suitable $\Gammag$-stable mouse
$Q$ in $V$. Iterate $Q$ to $Q'$, making $P$ extender algebra generic
at some Woodin cardinal $\varepsilon$ of $Q'$,
with $Q'$ being $\varepsilon$-sound.
Then working in $Q'[P]$, iterate $P$,
to make $Q'[P]|\varepsilon^{+Q'[P]}$
generic at the image of  
$\delta_i^P$
 (iterating below $\delta_i^P$ as usual), producing iterate $P'$,
sound above $\delta_i^{P'}$. By Lemma \ref{item:iterate_pres_mtr-suitability_stability} part \ref{item:stability_non-dropping} (***and another lemma to add?), we have
 $\delta_0^{P'}=\varepsilon^{+Q'[P]}$, and 
note then that \[ P'[Q',P]=^*_{\delta_0^{P'}}Q'[P]=^*_{\delta}Q'.\]

Now suppose that $\Vv$ is a correct above-$\delta_i^P$
tree on $P$ witnessing that $P$ is not
$\Gammag$-stable,
as further witnessed by some $(R,r,\xi)$
where $R\ins M^\Vv_\infty$. Then $\Vv$ lifts to a tree
$\Vv'=i_{PP'}\Vv$ on $P'$, which is equivalent
to an above-$\varepsilon^{+Q'}$
tree $\Ww$ on $Q'$. Let $(R',r,\xi')$ be the resulting
translated image of $(R,r,\xi)$,
with $R'\ins M^\Ww_\infty$.
(The copying works and respects the iteration strategies,  by uniqueness of the normal strategies.)
By $\Gammag$-stability for $Q$, $R'$ is above-$\xi'$, $(r,\om_1+1)$-iterable in $\SS_{\alphag}$.
But then by lifting to this strategy,
it follows that $R$ is above-$\xi$,
$(r,\om_1+1)$-iterable in $\SS_{\alphag}$,
a contradiction.

Part \ref{item:get_mtr-suitable_in_M_with_strategy_in_M}:
Cf.~Remark \ref{rem:if_om_1_exists_easier}. Using Assumption \ref{ass:Sigma_M_not_Gamma-guided},
we give a local version
of the proof of Fact \ref{fact:self-it}, without
having to assume that $M\sats$``$\om_1$ exists''.
(If $M=M|\om_1^M\sats\ZF^-$,
then we can just replace $M$ with $\J(M)$, which then satisfies ``$\om_1$ exists'' and has $\HC^{\J(M)}=\HC^M$,
and then run the preceding construction but starting above
some $\xi$ witnessing
Fact \ref{fact:self-it}.
So the interesting case here
is when $M=M|\om_1^M\not\sats\ZF^-$, which for example if $\rho_1^{M}=\om$, can cause difficulties with the reflection arguments that usually work at $\om_1^M$.)

Now let $P,\delta$ be as constructed
above; we claim that $M$ is closed under $\Sigma_{P,\geq\delta}$,
and $\Sigma_{P,\geq\delta}$ is definable from parameters over $M|\om_1^M$, which easily suffices.
For let $\Tt\in\HC^M$ be a limit length tree on $P$ which is via $\Sigma_{P,\geq\delta}$;
we compute $b=\Sigma_{P,\geq\delta}(\Tt)$ working in $M$.
By Lemma \ref{lem:Sigma_1_Hull_of_J_alpha},
$t=\Th_{\Sigma_1}^{\Ss_{\bar{\alpha}}(\RR^M)}(\RR^M)$
is definable from parameters
over $M|\om_1^M$, so we can
refer to $t$ in computing $b$.
Using $t$,
we can compute $\Lp_{\Gammag}(M(\Tt))$.
So we may assume there is no Q-structure
$Q$ for $M(\Tt)$
such that $Q\pins\Lp_{\Gammag}(M(\Tt))$,
and in particular,
$\Lp_{\Gammag}(M(\Tt))\sats$``$\delta(\Tt)$ is Woodin''.
Using Asssumption \ref{ass:Sigma_M_not_Gamma-guided},
let $\xi<\om_1^M$ be such that $P,\Tt\in M|\xi$
and $\xi=\om_1^{\Lp_{\Gammag}(M|\xi)}$. Let $R\pins M$
be such that $\xi=\om_1^R$
and $\rho_\om^R=\om$.
Let $\varepsilon$ be the supremum of the Woodin cardinals of $R$
which are ${<\delta(\Tt)}$;
so $\delta\leq\varepsilon<\delta(\Tt)$
and $\varepsilon$ is a strong cutpoint of $M(\Tt)$
and of $M^\Tt_b$.
Definably over $R$ from the parameter $\Tt$,
we can build the above-$\varepsilon$, minimal
$R|\om_1^R$-genericity inflation $\Xx$ of $\Tt$ (see \cite{fullnorm}).
If this terminates
with a tree $\Xx$ of length $\xi'+1<\xi$,
then $b$ is recovered from the pair $(\Tt,\Xx)$.
Otherwise 
it reaches a tree $\Xx$ of length $\xi+1$,
with the Q-structure $Q=Q(\Xx\rest\xi,[0,\xi)_\Xx)$ equal
to the P-construction of $R$
computed
over $M(\Xx\rest\xi)$. But then since $\Lp_{\Gammag}(R|\om_1^R)\ins R$,
we have $\Lp_{\Gammag}(M(\Xx))\ins Q$. But all of these objects are in $\HC^M$, and since $P$ is $M$-$\Gammag$-stable (as a mouse over $P|\delta$), it follows that 
$[0,\xi)_\Xx$ does not drop in model or degree,
and letting $b'$ be the (possibly $\Tt$-cofinal) branch of $\Tt$
determined by $\Xx$
and $\sigma:M^{\Tt}_{b'}\to M^\Xx_\xi$
the minimal inflation map,
then $b'=b$ and $\sigma(\delta(\Tt))=\xi$.
In particular, $\Xx$ determines $b$, which suffices.
\end{proof}

\begin{tm}
 $\OD^{<\alpha}$ is a mouse set, as witnessed by a premouse $N$ such that for some $n<\om$,
$N$ is $(n,\om_1+1)$-iterable, $N$ has $\om$ Woodins, $\rho_{n+1}^N\leq\delta_0^N<\lambda^N\leq\rho_n^N$ and
 $N$ is $(n+1)$-sound.
\end{tm}
\begin{rem}
 Note that regarding $\rho_{n+1}^N$, we only know that $\om\leq\rho_{n+1}^N\leq\lambda^N$.
 There must be instances where $N$ does not project to $\om$.
\end{rem}

\begin{proof}
 We have already shown that for a cone of reals $x$, the corresponding statement holds for $\OD^{<\alpha}(x)$ and $x$-mice $N_x$ (and in fact we can take $N_x$ to be fully sound and with $\rho_{n+1}^{N_x}=\om<\lambda^N\leq\rho_n^N$, and hence $N_x$ is uniquely determined by $x$). Fix such an $x,N_x$.
Let $N$ be the output of the full Q-local $L[\es]$-construction $\left<N_\alpha\right>_{\alpha\leq\OR^{N_x}}$ of $N_x$ (the lightface version).
Let $N'=\core_{n+1}(N)$. We claim that $N'$ satisfies all the requirements. The main thing here is to verify that  $\OD^{<\alpha}=\RR^{N'}$, which we now do.

Since  $\RR^{N'}=\RR^N$,
 we just consider $N$.
We have $\RR^N\sub\OD^{<\alpha}$
because for each $y\in\RR^N$,
there is $\eta<\delta_0^{N_x}$ such that $y\in(N_\eta)^{N_x}$. Then by condensation,
there is some $\bar{\eta}<\om_1^{N_x}$ such that
$y\in(N_{\bar{\eta}})^{H}$ for some $H\pins N_x|\om_1^{N_x}$. But $H$ is iterable in $\SS_\alpha$,
so $N_{\bar{\eta}}^{H}$ is also, which shows that $y\in\OD^{<\alpha}$.
Now let $y\in\OD^{<\alpha}$ and suppose $y\notin N$. We have $y\in N_x$, since $y\in\OD^{<\alpha}(x)$.
And $N_x|\delta_0^{N_x}$ is extender algebra generic over $N$ at $\delta_0^N=\delta_0^{N_x}$, and $N$ is just the P-construction of $N_x$ above $N|\delta_0^N$. Let $\sigma,\tau\in N|\delta_0^{+N}$
be extender algebra names such that $\sigma_G=x$ and $\tau_G=y$, where $G$ is the extender algebra generic determined by $N_x|\delta_0^{N_x}$. Let $\xi<\delta_0^{+N}$ be such that $\tau\in N|\xi$ and $\rho_\om^{N|\xi}=\delta_0^N$.
There is $P\pins N_x|\delta_0^{+N_x}$ with $\xi<\OR^P$
and $P$ is a minimal $\varphi(y)$-witness above $\xi$,
where $\varphi(y)$ asserts ``there is an ordinal $\nu$ such that $y\in\OD_\nu$''. Let $\gamma=\OR^P$.
Then we can fix $p\in\BB^N_{\delta_0^N}$
forcing the above facts about $\sigma,\tau,\gamma,\dot{\es}$, where $\dot{\es}$ is the name for the extender algebra generic (so $\dot{\es}_G=N_x|\delta_0^{N_x}$), and forcing ``$\tau\notin V$''.
We can now build a perfect set of $N$-generics $H$
with $p\in H$,
and a perfect set of pairwise distinct reals $z_H$, such that $z_H=\tau_H$.
Letting $P_H=N[H]|\gamma$,
then note that $P_H$ is iterable above $\xi$
(since iterating it is equivalent to iterating $N|\gamma$ above $\xi$), and therefore
$\varphi(z_H)$ is true. So we get a perfect set of reals in $\OD^{<\alpha}$, a contradiction.
\end{proof}

\section{The $\M$-hierarchy of an admissible gap}\label{sec:M-hierarchy}

\begin{conv}
In  \S\S\ref{subsec:the_hierarchy},\ref{subsec:mu-fs},
the variable ``$\gamma$'' is only ever interpreted as a limit ordinal
$\geq\om_1$,
and
the symbol
$\cdot$, appearing as $\dot{x}$, denotes either a constant symbol or predicate
symbol or free variable.
\end{conv}

\subsection{The hierarchy}\label{subsec:the_hierarchy}

\begin{dfn}[$\mu_n,\mu$]\label{dfn:mu}
Let $\Dd$ denote the set of Turing degrees. For $n<\om$ let $\mu_n$ denote the iterated Martin measure;
that is, for $\mu_n\sub\pow(\Dd^n)$
and for $A\sub\Dd^N$, we have
\[ A\in\mu_n\iff\exists^{\Dd}s_0\all^{\Dd}t_0\ldots\exists^{\Dd}s_{n-1}\all^{\Dd}t_{n-1}\ \Big[\Big(\bigwedge_{i<n}s_i\leq_Tt_i\Big)\wedge(t_0,\ldots,t_{n-1})\in A\Big].\]
Let $\mu$
denote $\bigcup_{n<\om}\mu_n$.
\end{dfn}

\begin{dfn}[$\M^\alpha_\delta(\RR)$]\label{dfn:MgammaRR}
Fix $\alpha\in\OR$ starting an S-gap of $L(\RR)$ such that $\SS_\alpha$ is
admissible. We define transitive structures $\M^\alpha_\delta=\M^\alpha_\delta(\RR)\in L(\RR)$,
for ordinals
$\delta\geq\om_1$. Usually
$\alpha$ will be fixed,
and we will drop the 
superscript ``$\alpha$'', as we do now.

Let $T$ be
the set of pairs $(x,t)$ such that $x\in\HC$ and
$t=\Th_{\Sigma_1}^{\SS_\alpha}(\{x\})$. Then we define
\[ \M_{\om_1}=(\HC,T).\]
Clearly this structure is amenable. By \cite[Lemmas 1.3, 1.8]{jensen_fs} or \cite[p.~610, 611]{schindler2010fine}, there is a finite basis for the
$\mu$-rud functions. We define the $\Ss^\mu$-hierarchy (the analogue of Jensen's $\Ss$-hierarchy) as in \cite{schindler2010fine} (in particular, using the finite basis there);  the $\Ss^\mu$-hierarchy then consists of transitive models.
For ordinals $\delta=\om_1+\beta$, define
\[ \M_\delta = ((\Ss^\mu)_\beta(\M_{\om_1}),\M_{\om_1}). \]
In other words,
$\M_{\delta+1}=(S^\mu(\lpole\M_\delta\rpole),\M_{\om_1})$, and the sequence of
universes $\lpole\M_\delta\rpole$ is continuous at
limits.
\end{dfn}

We will apply the preceding definition sometimes
working with the true $\HC$
(as the universe of $\M_{\omega_1}$) and the true $\mu$,
but at other times with other sets $\HC'$ replacing $\HC$
(and $\RR'=\RR^{\HC'}$ replacing $\RR$
and $\Dd'=\Dd^{\HC'}$ replacing $\Dd$),
and $\mu'$ defined relative to $\HC'$ just as $\mu$ is defined relative to $\HC$
(that is, $\mu_n'\sub\pow((\Dd')^n)$,
and
is defined like $\mu_n$,
but with the degree quantifiers
restricted to $\Dd'$).
In the context of the conjectures,
we will be particularly
interested in the case that
$\HC'=\HC^M$. When starting with $\HC'$, we will assume $\alpha'$
starts an admissible S-gap relative to $\HC'$,
and $\HC'=\HC^{\Ss_{\alpha'}(\HC')}$,
and we will only be interested
in a resulting $\M_\delta'$ assuming that $\HC'=\HC^{\M_{\delta'}}$
and $\M_{\delta'}\sats\AD+$``Turing determinacy''.
For the general development of fine structure
and so forth in this section,
we will just explicitly
write $\M_\delta$ and assume $\HC\sub\M_\delta$, but the reader will happily see that everything
works in the same manner for the more general case just mentioned.

Now fix  $\alpha$ as in Definition \ref{dfn:MgammaRR}, and supress
it from the notation. From now on we blur the distinction between
$\M_\delta$ and its universe. So for
each limit $\gamma\geq\om_1$, $\M_\gamma$ is amenable to $\mu$
and has ordinal height $\gamma$; and,
$\M_{\delta+\om}$ is the closure of
$\M_\delta\un\{\M_\delta\}$ under $\mu$-rud functions. Whenever
we talk about
 $\M_\delta$, we will assume that
 $\M_\delta\sats\AD+$``Turing determinacy''.

\begin{dfn}[$\Hdot$, $\Tdot$, $\Ll$, $\all^*_k$, $\all^*$, $\Ll^\mu$,
$\sats$, $\Sigma_n^{\M_\gamma}$]\label{dfn:language}
Let $\Ll$ be the language of set theory augmented with constant symbols
$\Hdot$, $\Tdot$, and unary predicates $\Hdot'$, $\Tdot'$. Let $\Ll^\mu$ be $\Ll$ augmented with 
the quantifier $\all^*$,
and for each $k<\om$, the quantifier $\all^*_k$.

Formulas of $\Ll$ and $\Ll^\mu$ are interpreted over $\M=\M_\gamma$, for limits $\gamma\geq\omega_1$, as
follows.

If $\gamma=\om_1$, we set $(\Hdot',\Tdot')^\M=\M_{\om_1}$ and $(\Hdot,\Tdot)^\M=\emptyset$.
If $\gamma>\om_1$, we set $(\Hdot,\Tdot)^\M=\M_{\om_1}$ and $(\Hdot',\Tdot')^\M=\emptyset$.

We have introduced the 4 symbols $\dot{H},\dot{T},\dot{H}',\dot{T}'$ instead of just $\dot{H},\dot{T}$, since the correct interpretations are 
predicates for $\M_{\omega_1}$ but constants for $\M_\delta$ when $\delta>\om_1$.
From now on we will actually  ignore the symbolic distinction and write only $\dot{H},\dot{T}$.

For $k<\om$ the quantifier ``$\all^*_k s$'' means ``For $\mu$-cofinally
many tuples $s$ in $\Dd^k$''. That is, given a formula $\varphi$ and
$k<\om$, the formula $\text{``}\all^*_k s\ \varphi(s,\vec{u})\text{''}$
means the following (where the quantifier $Q^{\Dd} x$ means $Q x\in\Dd$):\footnote{Given sufficient Turing
determinacy, $\all^*_k s\varphi(s)$ is
equivalent to ``There is $A\in\mu_k$ such that for all $s\in A$'', but we will
also need to use this quantifier without assuming such determinacy.}
\[ \all^{\Dd} r_0\ \exists^{\Dd} s_0\ \ldots\ \all^{\Dd} r_{k-1}\ \exists^{\Dd} s_{k-1}\ \Big[
\all i\ [r_i\leq_T s_i]\text{ and }
\varphi(\left<s_0,\ldots,s_{k-1}\right>,\vec{u})\Big].\]
 The dual quantifier ``$\exists^*_ks$'' is the ``$\mu$-measure one many'' variant; that is,
\[ \exists^\Dd r_0\ \all^\Dd s_0\ \ldots\ 
 \exists^\Dd r_{k-1}\ \all^\Dd s_{k-1}\ \Big[\all i\ [r_i\leq_T s_i]\implies\varphi(\left<s_0,\ldots,s_{k-1}\right>)]\Big].
\]
In general $\exists^*_k$ can be stronger than $\all^*_k$, but of course within sets for which Turing determinacy holds, $\all^*_k$ is  equivalent to $\exists^*_k$.
The quantifier $\all^*$ means ``$\ex k<\om\ \all^*_k$'', and $\exists^*$ means ``$\ex k<\om\ \exists^*_k$''.

Let $\Sigma_n^{\M}$, etc, be the usual
first order classes, defined using $\Ll$.

We say a set $X\sub\M$ is \dfnemph{$\mu$-definable over $\M$}
iff
\[ X=\{x\in\M\st\M\sats\varphi(x)\} \]
 for some $\varphi\in\Ll^\mu$. Likewise
\dfnemph{$\mu$-definable over $\M$ from parameters}.
\end{dfn}

Each $\all^*_k$ is definable with ordinary (i.e. first-order) quantifiers, but
in general,
$\all^*$ need not be (when $\M_\gamma$ projects to
$\om_1$).

\begin{dfn}
Let $\Ll^\mu_0$ be the set of formulas
in $\Ll^\mu$ in which all ordinary quantifiers are bounded. Say a class function
$f$ is \dfnemph{$\mu$-simple} iff for every formula
$\varphi(\dot{x},\dot{\yvec})\in\Ll^\mu_0$, there is an $\Ll^\mu_0$ formula
$\varphi_f$ such that for all $\xvec,\yvec$,
we have $\varphi(f(\xvec),\yvec)$ iff $\varphi_f(\xvec,\yvec)$. A \dfnemph{relatively-rud function scheme}
is a finite description of an ``$A$''-rud function,
in the usual sense, either
in terms of a recursion via the original
definition in \cite{jensen_fs}
or as a composition of functions from a finite basis; in particular, the scheme is formally an element of $V_\om$. Here the ``$A$'' is just a symbol and can be interpreted via any set/class. Given a scheme $g$, and a class $C$,
$f^C_g$ is the $C$-rud function defined by $g$.
Write $f_g=f^\mu_g$.
\end{dfn}

\begin{lem}\label{lem:mu-rud_implies_mu-simple}
 Every $\mu$-rud function is $\mu$-simple, and in fact, there is a recursive function
$(\varphi,g)\mapsto\psi_{\varphi,g}$ sending pairs $(\varphi,g)$ such that $\varphi$
is an $\Ll^\mu_0$ formula
and $g$ is relatively-rud function schemes $g$ to $\Ll^\mu_0$ formulas 
such that $\psi_{\varphi,g}$ suffices as a formula as desired for
$\varphi_{f}$, where $f=f^\mu_g$.
\end{lem}
\begin{proof}
 The proof is essentially that in \cite{jensen_fs}. The one main difference is as
follows. Let $f_\mu:V\to V$ be defined $f_\mu(x)=\mu\inter x$. We have to check
that for every formula
$\varphi(\dot{z},\dot{x},\dot{\yvec})\in\Ll^\mu_0$, the relation ``$\ex z\in
f_\mu(x)\ [\varphi(z,x,\yvec)]$'', is given by an $\Ll^\mu_0$ formula. But this relation is equivalent to ``$\exists z\in x\ [(\all^*s[s\in z])\wedge\varphi(z,x,\vec{y})]$''. (Because the quantifier $\all^*s$ is always interpreted over
$\Dd^{<\om}$, there is no analogous computation needed for it.)\end{proof}

\begin{lem}For limits $\gamma\geq\omega_1$ and  $X\sub\M_\gamma$, we have
$X\in\M_{\gamma+\om}$ iff $X$ is $\mu$-definable over $\M_\gamma$ from
parameters.\end{lem}
\begin{proof}This is a straightforward generalization of the proof 
for $L$ in Jensen \cite{jensen_fs};
in particular, the fact that every element of $\pow(\M_\gamma)\inter\M_{\gamma+\om}$
is $\mu$-definable over $\M_\gamma$ from parameters follows
from Lemma \ref{lem:mu-rud_implies_mu-simple} as in \cite{jensen_fs}. We leave the remaining details to the reader.\end{proof}

\begin{dfn} Let $X\sub\Dd^{<\om}$. Then $X$ is a \dfnemph{tree} iff $X$ is closed under initial segment. For $n<\om$, we say $X$ is a \dfnemph{measure one tree of height $n$} iff $X$ is a tree, $\emptyset\neq X\sub\Dd^{\leq n}$ and for every $s\in X$,
there is $A\in\mu_{n-\lh(s)}$ such that $s\conc A\sub X$, where $s\conc A=\{s\conc t\bigm|t\in A\}$.  We say $X$ is a \dfnemph{measure one tree of height $\om$}
if $X$ is a tree and
 $X\cap\Dd^{\leq n}$ is a measure one tree of height $n$, for each $n<\om$.
A \dfnemph{measure one tree}
is one of height $n$ for some $n\leq\omega$.
\end{dfn}

For $X\sub\Dd^{\leq n}$, $X$ is a measure one tree
of height $n$ iff $X\neq\emptyset$ and for each
$s\in X$ with $\lh(s)<n$, there are $\mu_1$-measure one many $x\in\Dd$ such that
 $s\conc\left<x\right>\in X$. So
there is a fixed sentence $\varphi$
such that for each $X\sub\HC$,
$X$ is a measure one
tree iff $(\HC,X)\sats\varphi$. Given a set $Y\in\mu_n$ where  $n<\om$, we
can easily pass to a measure one tree denoted $\tree(Y)$, with $\tree(Y)\sub Y$:
Let $s\in\tree(Y)$ iff
\[ \all m\leq\lh(s)\ \all^*_{n-m}t\ [s\conc t\in Y]
\]
 (note the complexity of the definition depends on $n$). Also for
$Y\sub\Dd^n$ such that $\Dd^n\cut Y\in\mu$, let $\tree(Y)=\tree(\Dd^n\cut Y)$.
Note that $\tree$ is $\mu$-rud.

\begin{prop}\label{prop:musigma1}Let $\mathscr{M}$ be transitive and $\mu$-rud closed,
with $\RR\un\{\RR\}\sub\mathscr{M}$. Then the function $f:\mathscr{M}\to \mathscr{M}$, defined
$f(x)=x\inter\mu$, is $\Sigma_1^{\mathscr{M}}(\{\RR\})$ in the usual language of set theory,
uniformly.\end{prop}
\begin{proof}
 Let $x\in\mathscr{M}$, with $x\sub\bigcup_{n<\om}\pow(\Dd^n)$. Then $t=\tree``x\in \mathscr{M}$.
Moreover, $t$ is a set of measure one trees, and for each
$X\in x$, there is $T\in t$ such that either $T\sub X$ or $T\inter X=\emptyset$.
In light of the remarks above, 
this easily gives a $\Sigma_1(\{\RR\})$ definition of $f$.\end{proof}

\begin{rem}
In the following definition, recall that $\Ll$ is the standard
first order language with symbols $\in,\dot{H},\dot{T},\dot{H}',\dot{T}'$  (with no second order quantifiers).
\end{rem}

\begin{dfn}
 A formula $\varphi(\vec{z})$ of $\Ll$ is a \dfnemph{Q-formula}
 if it has the form
 \[ \all x\ \exists y\ [x\sub y\wedge\psi(y,\vec{z})],\]
 for some $\Sigma_1$ formula $\psi$,
 whose only free variables are $(y,\vec{z})$
 (and this does not include $x$).
 
 An embedding $\pi:\M\to\N$ between $\Ll$-structures is a \dfnemph{weak $0$-embedding}
 iff $\pi$ is $\Sigma_0$-elementary
 and there is a $\sub$-cofinal set $X\sub\M$
 such that $\pi$ is $\Sigma_1$-elementary
 on all statements using only parameters in $X$.
\end{dfn}

\begin{lem}[$\Sigma_1$ condensation]\label{lem:Pi2charac} Let
$\M=\M_{\om_1}$. There is a Q-formula $\varphi$ in $\Ll$ such that for all $\Ll$-structures
$\Nn=(\lpole\Nn\rpole,\M)$ with
$\Nn$ transitive and $\M\in\Nn$, we
have: $\Nn\sats\varphi$ iff there is
$\gamma\in\Lim\cut(\om_1+1)$ such that $\Nn=\M_\gamma$.

Therefore, if $\Nn=(\lpole\Nn\rpole,\M)$ with $\Nn$ transitive and
$\M\in\Nn$ and
$\pi:\Nn\to\M_\gamma$ is a weak $0$-embedding, then $\Nn=\M_\beta$ for some
$\beta\in\Lim\cap(\gamma+1)$.
\end{lem}
\begin{proof}
Each $\M_\gamma$ satisfies the following Q-formula $\psi$: ``$\all
x\  \ex y\ \Big[x\sub y\wedge\exists z\ \all X\in y\ \all n<\om$, if $X\sub\Dd^{\leq n}$,
then there is a measure one tree $T\in z$ such that either $T\sub X$ or $T\inter
X=\emptyset\Big]$''. Moreover, if $\Nn$ is as hypothesized and $\Nn\sats\psi$, then
$\Nn$
is correct about ``measure one tree''. Using these observations, the usual
proof goes through.
\end{proof}

\begin{dfn}[$h^\M_0$]\label{dfn:h0} Let $\gamma\geq\omega_1$ be a limit and $\M=\M_\gamma$. We define the
surjection
\[ h^\M_0:\om\cross(\Lim\inter\gamma)^{<\om}\cross\HC\surj\M. \]
If $\alphavec\sub\om_1$ then
$h^\M_0(n,\alphavec,z)=y$, where if $n>0$ and $z$
is a function and $n-1\in\dom(z)$ then $y=z(n-1)$, and otherwise $y=z$. If
$\max(\alphavec)\geq\om_1$ then $h^\M_0(n,\alphavec,z)$ is defined in the usual
manner by iteratively evaluating $\mu$-rud functions coded by $n$, at levels
$\M_\alpha$ for various $\alpha\in\alphavec$,
feeding in parameters evaluated at lower levels.

We also write $h^\gamma_0=h^\M_0$.\end{dfn}

\begin{lem}\label{lem:h_0^M_is_Delta_1}
The graph of $h_0^\gamma$ is
$\Delta_1^{\M_\gamma}$, uniformly in $\gamma$.\end{lem}
\begin{proof}This is just as for $\SS_\gamma$, using \ref{prop:musigma1}.\end{proof}

\subsection{Fine structure}\label{subsec:mu-fs}
We now develop the fine structure for the $\M$-hierarchy. Toward this, we define a hierarchy through the $\Ll^\mu$ relations, defined over 
$\M=\M_\gamma$, for some limit $\gamma\geq\omega_1$, along with projecta and standard parameters.  Although this  is quite
routine, we will need to carefully
compare the fine structure of the $\M$-hierarchy
with that of corresponding generic premice, for which a precise development of the fine structure is in order. For premice,
we use Mitchell-Steel fine structure,
but omitting the parameters $u_n$ of \cite{fsit},
as described in detail in 
\cite[\S5]{V=HODX_pub}.
As explained there, this has no impact on soundness, fine structural parameters and or projecta.
The fine structure of $\M$ will be analogous. In particular,
we define the $\mSigma_n$ definability
hierarchy in Definition \ref{dfn:mSigma_n} below, by analogy with the
 $\rSigma_n$ hierarchy for premice.

\begin{rem}\label{rem:rSigma_n^J_beta}
The fine structure developed in \cite{scales_in_LR} for the
$\J$-hierarchy of $L(\RR)$
uses the $\Sigma_n^{\SS_\beta}$ hierarchy (where $\beta\in\Lim_0$),
as opposed to the $\rSigma_n^{\SS_\beta}$ hierarchy
(defined analogously to $\rSigma_n^M$ for premice $M$). One has $\Sigma_n^{\SS_\beta}\sub\rSigma_n^{\SS_\beta}$,
but the inclusion can be strict. However,
the $\bfrSigma_n^{\SS_\beta}$ sets
are exactly the $\bfSigma_n^{\SS_\beta}$ sets.
Moreover, there is an $\rSigma_{n+1}^{\SS_\beta}(\{\pvec_n^{\SS_\beta}\})$-definable Skolem function for $\rSigma_{n+1}^{\SS_\beta}$, uniformly in $\beta,n$, but for $\Sigma_{n+1}^{\SS_\beta}$, more parameters are required to define the relevant Skolem functions. This complicates the fine structure somewhat; the $\rSigma_n$ hierarchy is smoother. The fine structure
for the $\rSigma_n$ hierarchy of $\SS_\beta$ is developed
almost exactly like that for the $\mSigma_n$ hierarchy for $\M_\gamma$ (for limits $\gamma\geq\omega_1$)
which we are about to develop; in fact it is a simplification thereof. So without further mention, we implicitly automatically adapt all the  fine structural definitions introduced below for the $\mSigma_n$ hierarchy for $\M_\gamma$, to the $\rSigma_n$  hierarchy for $\SS_\gamma$; the straightforward details will be left to the reader.
\end{rem}

\begin{dfn}[$\mSigma_n^\M,\muSigma_n^\M,\rho_n^\M,p_n^\M$]\label{dfn:mSigma_n}
Let $\gamma\geq\omega_1$ be a limit and $\mathscr{M}=\mathscr{M}_{\gamma}$.
Let $\mSigma_1=\Sigma_1$ (using $\Ll$, which had no $\all^*$-quantifiers), $\rho_0^\M=\OR^\M$ and $p^\M_0=\emptyset$.
Let $n\geq 1$ and suppose we have
defined $\mSigma_n$ and $p^\M_{n-1}$ . Let
$\muSigma_n$ be the class of relations of the form
``$\all^*s\ \varphi$'', where $\varphi$ is $\mSigma_n$. Let
$\mPi_n,\muPi_n,\mDelta_n$ and $\muDelta_n$ be the corresponding dual and
ambiguous classes.

For $X\sub\M$, let
$\Th_{\mSigma_n}^{\M}(X)$
denote the $\mSigma_n$-theory of $\M$ in parameters in
$X$. Let
$\Hull_{\mSigma_n}^{\M}(X)$
be the set of all $z\in\M$ such that for some
$\mSigma_n$ formula $\varphi$ and some
$\xvec\in X^{<\om}$, $z$
is the unique $z'\in\M$ such that
$\M\sats\varphi(z,\xvec)$. Also let
$\cHull_{\mSigma_{n}}^{\M}(X)$
be the transitive collapse of that hull, if it  is extensional. (Normally we will have $\HC\sub X$, in which case it is extensional, by Lemma \ref{lem:h_0^M_is_Delta_1}.)
We define $\Th_{\muSigma_n}$, $\Hull_{\muSigma_n}$ and $\cHull_{\muSigma_n}$
analogously.

Let $\rho_n=\rho_n^{\M}$ be the least
$\rho\geq\om_1$ such that for some $p\in\gamma^{<\om}$,
\[ \Th_{\mSigma_n}^{\M}(\M_\rho\un\{p\})\notin\M.\]
Let $p^\M_n$ be the least
$p\in\OR^{<\om}$ such that
\[ \Th_{\mSigma_n}^\M(\M_{\rho_n}\un\{p,p^\M_{n-1}\}
)\notin\M.\]
Also write  $\pvec^\M_n=(p_1^\M,\ldots,p_n^\M)$.

Let $T_n\sub\M\cross\M$
be defined as follows. If $\rho_n>\om_1$ then
\[
T_n(r,t)\iff\ex
q\ \ex\alpha<\rho_n\left[r=(q,\alpha)\ \&\ t=\Th_{\muSigma_n}^{\M}
(\M_\alpha\un\{q\})\right]. \]
If $\rho_n=\om_1$ then
\[ T_n(r,t)\iff t=\Th_{\muSigma_n}^{\M}(\{r\}). \]
For each $r$ of the appropriate
form, there is a unique $t\in\M$ such
that $T_n(r,t)$; this follows \ref{prop:muSigma_n_theories} below in the case that
$\rho_n>\om_1$,\footnote{Note that $\rho_n$ is defined using $\mSigma_n$
but $T_n$ uses $\muSigma_n$.}
and if $\rho_n=\om_1$, it is just because $\HC\sub\M$.

Let $\mSigma_{n+1}$ be the class of relations of the form
\[\varphi(\xvec)\iff\ex r, t\ \Big[T_n(r,t)\wedge \psi(r,t,\xvec)\Big], \]
where $\psi$ is $\Sigma_1$.

This completes the recursion. We also let $\mSigma_0=\Sigma_0$ (in $\Ll$) and
let
$\muSigma_0$ denote the class of
relations of the form ``$\all^*s\ \varphi$'', where
$\varphi$ is $\Sigma_0$.
\end{dfn}

\begin{lem}\label{prop:muSigma_n_theories}
 Let $\gamma>\om_1$ be a limit, $\M=\M_\gamma$, $n\in[1,\om)$, and suppose
$\rho_n^{\M}>\om_1$. Then for all
$\alpha<\rho_n^\M$ and $q\in\M$, we have
$
\Th_{\muSigma_n}^{\M}(\M_\alpha\un\{q\}
)\in\M$.
\end{lem}
\begin{proof}
 The theory in question is in
$\Ss_{\om^2}(t)$ where $t=\Th_{\mSigma_n}^{\M}(\M_\alpha\un\{q\})$,\footnote{In case the reader doesn't notice, the first theory
consists of $\muSigma_n$ formulas,
whereas the second $\mSigma_n$.} and since $t\in\M_{\rho_n^{\M}}$, also
$\Ss_{\om^2}(t)\in\M_{\rho_n^\M}$.
\end{proof}

\begin{lem} Let $\M=\M_\gamma$ where $\gamma\geq\omega_1$ is a limit. Then:
\begin{enumerate}
\item\label{item:mu-def_equiv_bfmSigma_om}  A subset of $\M$ is
$\mu$-definable from parameters iff it is $\bfmSigma_n^\M$ for some $n<\om$.
\item\label{item:mSigma_n_sub_muSigma_n} $\mSigma_n^\M\sub\muSigma^\M_n$
for each $n<\om$.
\item \label{item:Bool_comb_musigma_n}
 Boolean combinations
of $\muSigma_n^\M$ relations are $\mDelta^\M_{n+1}$, for each $n<\om$, and uniformly so.
\end{enumerate}\end{lem}
\begin{proof}
Parts \ref{item:mu-def_equiv_bfmSigma_om} and 
\ref{item:mSigma_n_sub_muSigma_n} are clear, as is \ref{item:Bool_comb_musigma_n} when $n\geq 1$.
Let us observe that
$\muSigma_0\sub\mDelta_1^{\M}$; this fact easily
extends to Boolean combinations. If $\gamma>\om_1$, this is because
$\M$ is $\mu$-closed, and by the proof of
\ref{prop:musigma1}.

Suppose $\gamma=\om_1$. Let $\varphi(\xdot,\sdot)$ be
$\mSigma_0$ and let $x\in\HC$. If $\varphi$ does not use the predicate $T$,
then the result is immediate, since the question of whether
$\all^*_s\varphi(x,s)$ is easily computed over $\SS_\delta$ where
$\delta$ is large enough that $x\in\SS_\delta$. 
But the general case is an easy
generalization, since if $k>0$ then $\all_k^* s$, if
$(r,t)\in\tc(\{s,x\})$ and $T^\M(r,t)$, then $(r,t)\in\tc\{x\}$. For
$r\in\tc\{t\}$ and $t\in\tc\{x\}$ because $s$
is a tuple of sets of reals (not theories).
\end{proof}

\begin{rem}
The foregoing argument is cheap. If we had defined $T^{\M_{\om_1}}$ instead as the function
$f:\HC\to\HC$ where $f(r)=\Th_{\Sigma_1}^{\J_{\alphag}}(\{r\})$, then we would actually
have to do something. Let us consider this variant.
By Turing completeness,
for each $n<\om$ 
there is a fixed theory $t_n$
such that 
\[ \all^*_ns\ \Big[t_n=\Th_{\Sigma_1}^{\J_{\alphag}}(\{(x,s)\})\Big].\]
But from $t_0=\Th_{\Sigma_1}^{\J_{\alphag}}(\{x\})$,
we can recover the $t_n$'s for $n>0$.
For $\psi\in t_1$ iff
\[ \all^\Dd d\ \exists^\Dd e\ \Big[\SS_{\alphag}\sats\psi(x,e)\Big],\] which by
admissibility
 is equivalent to
\[ \exists\alpha'<\alphag\ \Big[\J_{\alpha'}\sats\all^\Dd d\ \exists^\Dd e\ \psi(x,e)\Big],\]
which is determined by $t_0$. And $t_{n+1}$ is recovered similarly from $t_n$ (cf.~the proof of Lemma \ref{lem:T_com_Sigma_1^J_alpha}).

It easily follows that Boolean combinations
of $\muSigma_n^{\M}$ relations
are $\mDelta_1^{\M}$, uniformly in $n$.
\end{rem}

\begin{dfn}
Fix $\gamma\geq\om_1$ and let $\M=\M_\gamma$. Let $k<\om$. We define
$k$-solidity and $k$-soundness by induction on $k$. We say that $\M$ is
\dfnemph{$0$-solid} and \dfnemph{$0$-sound}. Suppose $0<k$ and $\M$ is
$(k-1)$-sound. We say
that $\M$ is
\dfnemph{$k$-solid} iff
\[
\cHull_{\mSigma_k}^\M(\M_\alpha\un\{p^\M_k\cut(\alpha+1),p_{k-1}^\M\}
)\in\M \]
for each $\alpha\in p_k^\M$. We say that $\M$ is \dfnemph{$k$-sound} iff $\M$ is
$k$-solid and
\[
\M=\Hull_{\mSigma_k}^{\M}(\M_{\rho_k}\un\{p^\M_k,p^\M_{k-1}\}
). \]
We also say that $\M$ is \dfnemph{$\om$-sound} iff $\M$ is $k$-sound for all
$k<\om$.
\end{dfn}

In \ref{lem:Msound} we will show that $\M_\gamma$ is $\om$-sound.

We now define $\mSigma_n$ Skolem functions for $\M=\M_\gamma$. This is a
straightforward adaptation of methods used for premice and
for the $\SS_\gamma$ analogue in \cite{scales_in_LR}. We will
closely follow the methods of \cite[\S5]{V=HODX_pub} (which stem from  \cite[\S2]{fsit}),
and the reader might want to review that material first (it is independent of the earlier sections of that paper). Recall that in
\ref{dfn:h0} we
defined the surjection $h_0=h_0^\M$.

\begin{dfn}\label{dfn:nice_skolem_M}
Let $\gamma>\om_1$ be a limit and $n\in\om$. Let $\M=\M_\gamma$, and assume
$\om_1<\rho_n^\M$ and $\M$ is $n$-sound. A \dfnemph{nice $\mSigma_{n+1}$ Skolem
function for
$\M$} is a partial function
\[ h:_{\mathrm{p}}\om\cross\M\cross\HC\to\M \]
such that
\begin{enumerate}[label=\alph*),ref=(\alph*)]
 \item The graph of $h$ is $\mSigma_{n+1}(\{p_n^\M\})$.
 \item For each $\mSigma_{n+1}$ formula $\varphi(\dot{x},\dot{y})$ and
$x\in\M$, if
\[ \M\sats\ex y\varphi(x,y) \]
then for some $i<\om$ and $z\in\HC$, $h(i,x,z)$ is defined and
\[ \M\sats\varphi(x,h(i,x,z)). \]
 \item\label{item:h_Hull} For each $x\in\M$ and $z\in\HC$,
\[ h``(\om\cross\{x\}\cross\{z\})=\Hull_{n+1}^\M(\{x,z,p_n^\M\}).\qedhere \]
\end{enumerate}
\end{dfn}

By \ref{dfn:nice_skolem_M}\ref{item:h_Hull}, if
$\M$ is $(n+1)$-sound and $h$ is a nice $\mSigma_{n+1}$ Skolem function for $\M$
then \[ \M=h``(\om\cross(\rho_{n+1}^\M\cross\{p_{n+1}^\M\})\cross\HC). \]

If $\om_1<\rho_n^\M$ we will define a nice $\mSigma_{n+1}$
Skolem function $h_{n+1}^\M$ for $\M$. For property \ref{dfn:nice_skolem_M}(a)
it suffices to ensure that $h$ is $\mSigma_{n+1}(\{\pvec_n^\M\})$, since
$\pvec_{n-1}^\M$ is $\mSigma_{n+1}^\M(\{p_n^\M\})$, by an easy modification of
the proof of
\cite[Theorem 5.8]{V=HODX_pub}. (However, the resulting definition will be
\emph{uniformly} $\rSigma_{n+1}(\{\pvec_n^\M\})$, but seemingly not uniformly
$\rSigma_{n+1}(\{p_n^\M\})$.)

\begin{dfn}
A formula $\varphi$ of $\Ll^\mu$ is \dfnemph{$\Sigma_1$-over-$\mSigma_{n+1}$}
iff there is an $\mSigma_{n+1}$ formula $\psi$ such that $\varphi=\ex y\psi$.
\end{dfn}

\begin{dfn}[codes a putative witness]\label{dfn:codes_a_witness_M} We describe canoncial forms of witnesses to $\mSigma_{n+1}$ formulas over $\M_\gamma$.

Let
$\gamma\geq\om_1$ be a limit and $\M=\M_\gamma$. Let $n<\om$. Suppose that $\M$ is
$n$-sound and for each $i\in(0,n]$, if $\om_1<\rho_{i-1}^\M$ then we have
defined $h_i=h_i^\M$
and it is a nice $\mSigma_i$ Skolem function for $\M$. Let $\rho_n=\rho_n^\M$
and $p_i=p_i^\M$ and $\pvec=\pvec_n^\M$.\footnote{Here we might have used $(p_{n-1}^\M,p_n^\M)$
in place of $\pvec_n^\M$, but we opted for the latter for notational
simplicity.}
Let $\varphi$ be a
$\Sigma_1$-over-$\mSigma_{n+1}$
formula in free variables $\dot{x}$
such that:
\begin{enumerate}[label=--]
 \item 
If $\gamma>\om_1$ and $n=0$
then
$\varphi(\dot{x})$ has form $\ex y,w\ [\varrho(\dot{x},y,w)]$,
where $\varrho$ is $\mSigma_0$.
\item Otherwise 
$\varphi(\dot{x})$ has form $\ex y,r,t\ \Big[T_n(r,t)\wedge\psi(\dot{x},y,r,t)\Big]$
where:
\begin{enumerate}[label=--]
\item If $\gamma>\om_1$ (so $n>0$)
then  $\psi$ is $\mSigma_1$.
\item If $\gamma=\om_1$ 
then $T_0=T^{\M_{\om_1}}$ and (independent of  $n$) $\psi$ is $\Sigma_1$ in the usual language of
set theory (in particular, $\psi$ does not use $\Tdot$).
\end{enumerate}
\end{enumerate}
\begin{case}\label{case:om_1<rho_n} $\om_1<\rho_n$.

Let $u\in\M$, $z,z'\in\HC$,
$\betavec,\gammavec\in[\gamma]^{<\om}$ and $i<\om$. Working in $\M$, we say that
\dfnemph{$u$ codes a putative witness to
$(\varphi(\dot{x}),(i,\vec{\beta},z))$ at $(\gammavec,z')$} iff either
\begin{enumerate}[label=\arabic*),ref=(\arabic*)]
\item $n=0$ and there are 
$U,x,y,w,\lambda,m,m',\betavec_1,\betavec_2,i_1,i_2$ such that:
\begin{enumerate}[label=\alph*),ref=(\alph*)]
\item $\gammavec=(\lambda,m',\betavec_1,i_1,\betavec_2,i_2)$,
\item $u=(U,\M_{\om_1})$, $U$ is transitive and
$\M_{\om_1},x,y,w\in U$,
\item $\lambda\in\Lim\inter[\om_1,\OR^U)$ 
and
$\betavec,\betavec_1,\betavec_2\in(\Lim\inter(\lambda+1))^{<\om}$
and $m,m',i_1,i_2<\om$
and $m'\leq m$,
\item  setting
$(\Hdot,\Tdot)^u=\M_{\om_1}$, we have
\[ \begin{split}u\sats&\text{``}\M_{\lambda+m}\text{ exists }\wedge x=h_0(i,\betavec,z)\\&\wedge y=h_0(i_1,\betavec_1,z')\wedge w=h_0(i_2,\betavec_2,z')\wedge\varrho(x,y,w)\text{''},\end{split} \]
and
\item  the definition of the outermost relatively-rud schemes coded by $i_1$ and $i_2$ (respectively, $i$) have
rank ${\leq m'}$ (respectively, $m$) in the relatively-rud scheme hierarchy\footnote{That is, say a relatively-rud
scheme $h$ has rank $1$ if is one of the schemes in the standard finite basis for relatively-rud schemes, and rank $k+1$ if it is of the form $h(\vec{x})=g(f_1(\vec{x}),\ldots,f_m(\vec{x}))$,
for some $g$ of rank $1$ and  $f_i$s of rank $k$.
Then $i_1,i_2$ should specify schemes of rank $\leq m$. ***Have to modify the $\Ss$-hierarchy for this***Note that for $m>0$, the elements of $\M_{\lambda+m}$ are precisely those
of the form $f(\M_\lambda,\vec{x})$ for some $f$ of rank $m$ and $\vec{x}\in\M_\lambda$.}
\end{enumerate}
\end{enumerate}
; or
\begin{enumerate}[resume*]
\item\label{item:n>0} $n>0$ and there are $\alpha,\beta,\betavec_1,\betavec_2,\betavec_3,i_1,i_2,i_3$ such that:
\begin{enumerate}[label=\alph*),ref=(\alph*)]
  \item $\gammavec=(\alpha,\beta,\betavec_1,i_1,\betavec_2,i_2,
\betavec_3,i_3)$,
 \item $\alpha\in[\om_1,\gamma)$,
 \item $u$ is a
set of
$\muSigma_n$ formulas in parameters in
$\M_\alpha\un\{\pvec\}$,
\item  $\beta<\alpha$
and $\betavec,\betavec_1,\betavec_2,\betavec_3\in\alpha^{<\om}$ and
$i_1,i_2,i_3<\om$,

 \item $u$ contains the following assertions (expressed with the help of  the parameter $\pvec$):
 \begin{enumerate}[label=\roman*),ref=(\roman*)]
 \item ``$x:=h_n(i,(\betavec,p_n),z)$ is defined, as are
$t:=h_n(i_3,(\betavec_3,p_n),z')$ and
$q:=h_n(i_2,(\betavec_2,p_n),z')$ and $y:=h_n(i_1,(\betavec_1,p_n),z')$'',\footnote{Here the notation ``$:=$''
means that we define the symbol on the left by the expression on the right. Literally,
the symbols $x,y,t,y$ do not themselves
show up in formulas in $u$.}
 \item ``$t$
is a set of $\muSigma_n$ formulas in parameters in
$\M_\beta\un\{q\}$'',
 \item ``$\psi(x,y,(\beta,q),t)$''.
 \end{enumerate}
 \item The assertions made by $u$ about elements of $t$ are precisely
those induced by corresponding elements of $u$,\footnote{This is as in the line
immediately preceding the \emph{Remark} on page 26 of \cite{fsit}.}
\end{enumerate}
\end{enumerate}

We also say $u$ \dfnemph{codes a putative witness to $(\varphi(\dot{x}),(i,\betavec,z))$} iff $u$
codes a
witness
to $(\varphi(\dot{x}),(i,\betavec,z))$ at some $(\gammavec,z')$.

Let $x\in\M$. We also say that $u$ \dfnemph{codes a putative witness to $\varphi(x)$ at $(\vec{\gamma},z')$} iff there are $i,\betavec,z$\footnote{Note that $i,\betavec$ are uniquely determined by $\gammavec$, but $z$ need not be.} such that:
\begin{enumerate}[label=--]
 \item $u$ codes a putative witness to $(\varphi(\dot{x}),(i,\betavec,z))$
 at $(\gammavec,z')$, and
\item $x=h_n^\M(i,(\betavec,p_n^\M),z)$.\end{enumerate}
We say that $u$ \dfnemph{codes a putative witness to $\varphi(x)$} if $u$ codes a putative witness to $\varphi(x)$ at some $(\gammavec,z')$.
\end{case}

\begin{case} $\om_1=\rho_n$.

Let
$m\leq n$ be least such that $\rho_m=\om_1$.

Let $u\in\M$, $z\in\HC$ and $i<\om$.
In
$\M$, we say that
\dfnemph{$u$ codes a putative witness to
$(\varphi(\dot{x}),(i,z))$} iff either
\begin{enumerate}[label=\arabic*),ref=(\arabic*)]
\item $n=0=i$ (so $\gamma=\om_1$ and $u\in\HC$) and $u=(U,r,t)$ for some transitive $U$ such
that  $r,t\in U$ and $t$ is a 
set of $\Sigma_1$ formulas
of the $L(\RR)$
language in parameter $r$, and $U\sats\ex y\ \psi(z,y,r,t)$.
\end{enumerate}
; or
\begin{enumerate}[resume*]
 \item $n>0$ and for some $z'\in\HC$ and $i_1,i_2<\om$,
 we have:
\begin{enumerate}[label=\alph*),ref=(\alph*)]
 \item $u$ is a  set of $\muSigma_n$
formulas in the parameter $(z,z',\pvec)$, 
and $u$ contains the following assertions:
 \begin{enumerate}[label=\roman*),ref=(\roman*)]
 \item[(i)] ``$x:=h_m(i,p_m,z)$,
$t:=h_m(i_1,p_m,z')$, and
$q:=h_m(i_2,p_m,z')$ are defined'',
 \item[(ii)] ``$t$
is a set of $\muSigma_n$ formulas in the parameter $q$'',
 \item[(iii)] ``$\ex y\ \psi(x,y,q,t)$''.
 \end{enumerate}
 \item[(c)] The assertions made by $u$ about elements of $t$ are precisely
those induced by corresponding elements of $u$.
\end{enumerate}
\end{enumerate}

Let $x\in\M$. We also say that $u$ \dfnemph{codes a putative witness to $\varphi(x)$}
iff there are $i,z$ such that:
\begin{enumerate}[label=--]
\item  $u$ codes a putative witness to $(\varphi(\dot{x}),(i,z))$, and
\item either:
\begin{enumerate}[label=\arabic*),ref=(\arabic*)]
 \item $n=0=i$ and $x=z$, or
 \item $n>0$ and $x=h_m(i,p_m,z)$.\qedhere
\end{enumerate}
\end{enumerate}
\end{case}
\end{dfn}

Now the key fact is that, for example, if $\om_1<\rho_n$
and $n>0$ then
\[ \M\sats\varphi(x)\iff\exists
\alpha<\rho_n\ \Big[
\Th_{\muSigma_n}^\M(\M_\alpha\un\{\pvec_n\})\text{ codes a
witness to }\varphi(x)\Big], \]
and likewise in the other cases.

\begin{dfn}[$h^\M_{n+1}$]\label{dfn:h_n+1^M}
Let $\gamma\geq\omega_1$ be a limit. Suppose $\M=\M_\gamma$ is $n$-sound and $\om_1<\rho_n^\M$. We define
$h_{n+1}^\M$.
Let $\tau(\xdot,\ydot)$ be $\mSigma_{n+1}$ and let $x\in\M$.

Suppose for now that $\M\sats\ex y\ \tau(x,y)$.
If $n>0$, let $\gammavec_{\tau,x}$ denote the least $\gammavec$
such that for some $\alpha<\rho_n$ and $z'\in\HC$,
\[ \Th_{\muSigma_n}^\M(\M_\alpha\un\{\pvec_n^\M\}) \]
codes a putative witness to
$\ex y\ \tau(x,y)$ at $(\gammavec,z')$. If $n=0$, define $\gammavec_{\tau,x}$
similarly, but with $u$ witnessed by an $\M_\alpha$ (allowing successor $\alpha$). Let
$Z'_{\tau,x}$ be the set of all $z'\in\HC$ witnessing the choice of
$\gammavec_{\tau,x}$.

Now drop the assumption that $\M\sats\exists y\ \tau(x,y)$. Let $y\in\M$ and $z'\in\HC$.
We define
\[ h_{n+1}^\M(\tau,x,z')=y \]
iff:
\begin{enumerate}[label=--]
 \item 
$\M\sats\exists y'\ \tau((x,p_n^\M,z'),y')$,
 \item if $n=0$ then
 $y=h_0^\M(i_1,\betavec_1,z')$ where
 $\gammavec_{\tau,(x,p_n^\M,z')}=(\lambda,m',\betavec_1,i_1,\betavec_2,i_2)$, and
\item if $n>0$ then
$y=h_n^\M(i_1,(\betavec_1,p_n),z')$
(in particular, we have $(i_1,(\betavec_1,p_n),z')\in\dom(h_n^\M)$),
where
$\gammavec_{\tau,(x,p_n^\M,z')}=(\alpha,\beta,\betavec_1,i_1,\betavec_2,i_2,
\betavec_3,i_3)$.\qedhere
\end{enumerate}
\end{dfn}

\begin{lem}$h^\M_{n+1}$ is a nice $\mSigma_{n+1}$ Skolem function for $\M$.\end{lem}
\begin{proof} The fact that $h^\M_{n+1}$ satisfies the requirements of \ref{dfn:nice_skolem_M} follows readily from the following observations:
\begin{enumerate}[label=\alph*),ref=(\alph*)]
 \item $h^\M_{n+1}$ is a partial function
 which is $\mSigma_{n+1}^\M(\{\pvec_n^\M\})$, and so
is in fact $\mSigma_{n+1}^\M(\{p_n^\M\})$.
 \item Suppose $\M\sats\ex y\ \tau(x,y)$
 and fix $z'\in Z'_{\tau,x}$.
 Let
$\tau'(\rdot,\ydot)$ be the natural $\mSigma_{n+1}$ formula asserting
``$\tau((\rdot)_0,\ydot)$'', where $(a,b,c)_0=a$.
If $n>0$ then there are $\alpha,\beta,\ldots,\betavec_3,i_3$ such that
\[ \gammavec_{\tau',(x,p_n^\M,z')}=\gammavec_{\tau,x}=(\alpha,\beta,\betavec_1,i_1,\betavec_2,i_2,
\betavec_3,i_3) \text{ and }z'\in Z'_{\tau',(x,p_n^\M,z')}, \]
and if $n=0$ then similarly
there are $\lambda,m',\betavec_1,i_1,\betavec_2,i_2$ such that
\[ \gammavec_{\tau',(x,p_n^\M,z')}=\gammavec_{\tau,x}=(\lambda,m',\betavec_1,i_1,\betavec_2,i_2) \text{ and }z'\in Z'_{\tau',(x,p_n^\M,z')}. \]
(If $x=h^\M_n(i,(\betavec,p_n),z)$, then
$(x,p_n,z)=h^\M_n(i',(\betavec,p_n),z)$ for some $i'$, by
\ref{dfn:nice_skolem_M}\ref{item:h_Hull} or the construction of $h_0^\M$. But
 $i$ is
 not incorporated into $\vec{\gamma}_{\tau,x}$,
 so the change from $i$ to $i'$ does not matter.
If $n=0$,
we might need to increase the ``$m$''
so as to ensure that
$(x,p_n^\M,z)\in\M_{\alpha+m}$. But this does not change $\gammavec$ either,
particularly in light of the fact that just $m'$, not $m$, is recorded in $\gammavec$.)
 
 It follows that $y=h^\M_{n+1}(\tau',x,z)$ is
defined, and $\M\sats\tau(x,y)$.
\item If $z\in\HC$ and $\M\sats\ex !y\ \tau((x,p_n,z),y)$ then by
\ref{dfn:h_n+1^M}, $h^\M_{n+1}(\tau,x,z)$ equals that $y$.\qedhere
\end{enumerate}
\end{proof}

We also use the notion of coding witnesses in the context of premice:

\begin{dfn}\label{dfn:codes_a_witness_premouse_and_skolem_premouse}
Let $N$ be an $n$-sound premouse such that $\om<\rho_n^N$. Let $\varphi$ be
$\Sigma_1$-over-$\rSigma_{n+1}$. Let $x,i,\betavec,u,\gammavec\in N$. Working in $N$, we
define $h_0^N$, the assertion that \dfnemph{$u$ codes a putative witness to
$(\varphi(\dot{x}),(i,\betavec))$} (or to \dfnemph{$\varphi(x)$}) \dfnemph{(at $\gammavec$)}, and
$h_{n+1}^N$, analogously to
\ref{dfn:codes_a_witness_M} and \ref{dfn:h_n+1^M}.
\end{dfn}

\begin{lem}\label{lem:Msound} For each limit $\gamma\geq\om_1$, $\M_\gamma$ is
$\om$-sound.\end{lem}
\begin{proof} The is much as in \cite{scales_in_LR}, using $\Sigma_1$ condensation for the $\M$-hierarchy, Lemma \ref{lem:Pi2charac}. More literally (because of the small changes in
the fine structural notions) use the existence of nice Skolem
functions, and adapt the arguments of \cite[\S5]{V=HODX_pub}.\end{proof}

\begin{lem}\label{lem:Msubstructure} Let $\gamma\geq\om_1$ be a limit and $n<\om$. Suppose
$\om_1<\rho=\rho_{n+1}(\M_\gamma)$. Let $\alpha\in[\om_1,\rho)$ and let
\[  \Hh = \cHull_{\mSigma_{n+1}}^\M(\M_\alpha\un\{\pvec_{n}^\M\}).\]
Then there is $\beta<\rho$ such that $\Hh=\M_\beta$, and letting
$\pi:\Hh\to\M$ be the uncollapse, then $\pi$ is a near
$n$-embedding\footnote{That is, $\pi$ has the same preservation properties as do
near $n$-embeddings between premice.}, and $\pi$ is $\muSigma_{n+1}$-elementary.
\end{lem}
\begin{proof}
The proof that $\pi$ is a near $n$-embedding is similar to the proof of the
previous lemma. The $\muSigma_{n+1}$-elementarity (not just $\mSigma_{n+1}$) of
$\pi$ is an immediate consequence, since $\RR\sub\rg(\pi)$.
\end{proof}

\begin{prop}\label{prop:JMhierarchy} Let $\gamma\geq\omega_1$ be a limit. Then:
\begin{enumerate}
 \item \label{item:new_set_first-order_def}
  Suppose there is
$X\sub\RR$ which is first-order definable \tu{(}not just $\mu$-definable\tu{)} from parameters over $\M_\gamma$ but
$X\notin\M_\gamma$. Then
$\pow(\RR)^{\M_{\gamma+\om}}=\pow(\RR)^{\Ss_{\om^2}(\M_\gamma)}$.

\item\label{item:no_new_set_first-order_def} If there is no $X$
as in part \ref{item:new_set_first-order_def} then $\M_\gamma$ and
$\SS_\gamma$ have the same universe, as do $\M_{\gamma+\om}$ and
$\SS_{\gamma+\om}$.
\end{enumerate}
\end{prop}

\begin{proof}
Let $\M=\M_\gamma$. Consider part \ref{item:new_set_first-order_def}. It's easy enough to see that
$\pow(\RR)^{\M_{\gamma+\om}}\sub\Ss_{\om^2}(\M)$. For the
converse, let $n\in[1,\om)$ be such that
$\rho_n^{\M}=\om_1$. Let
$\pvec=\pvec_n^\M$ and let $t$ be the natural
coding of $\Th_{\mSigma_n}^{\M}(\RR\un\pvec)$ as a set of reals. Since
$\M$ is sound, every set of reals in
$\Ss_\om(\M)$ is
$\Sigma^1_m(t,z)$ for some $z\in\RR$ and $m<\om$. But
we claim that every such set is $\muSigma^{\M}_{n+1}(\pvec,z)$ (and note $n$ is fixed here). For let
\[ \varphi(x)\iff\ex^\RR x_0\ \all^\RR x_1\ \ldots\ \ex^\RR x_{m-2}\all^\RR
x_{m-1}\ \left[\psi(x_0,\ldots,x_{m-1},x,z)\right] \]
where $\psi$ is arithmetic in $t$. Then for $x\in\RR$, we have $\varphi(x)$
iff $\all^*_ms\ \varphi'(s,x,z)$,
where $\varphi'(s,x,z)$ is the formula:
\[ \ex x_0\in s_0\ \all x_1\in s_1
 \ldots\ \ex x_{m-2}\in
s_{m-2}\ \all x_{m-1}\in s_{m-1}\left[\psi(x_0,\ldots,x_{m-1},x,z)\right]. \]
But the truth of $\varphi'(s,x,z)$ is determined by
$t'_w=\Th_{\mSigma_n}^{\M}(\{\pvec,s,x,z,w\})$, where $w$ is any real coding all
reals $\leq_T s_0\oplus\ldots\oplus s_{m-1}$. This easily gives that
$\varphi(x)$ is $\muSigma_{n+1}^{\M}(\pvec,z)$.

The uniformity of the preceding argument in fact gives that
 $\Th_{\Sigma_\om}^\M(\RR\cup\{\pvec\})$ is $\muSigma_{n+1}^{\M}(\pvec)$. It follows that letting $t_1\sub\RR$ naturally code
$\Th_{\Sigma_1}^{\Ss_\om(\M)}(\RR\un\{\M\})$, then $t_1$ is
$\muSigma_{n+1}^{\M}(\pvec)$.
(That is, letting $\varphi$ be a $\Sigma_1$ formula and $y\in\RR$, then $\Ss_\om(\M)\sats\varphi(y,\M)$
iff there is $m<\om$, $z\in\RR$
and a $\Sigma^1_m(t,z)$ set of reals
which codes some transitive
set $X\in\Ss_\om(\M)$
such that $\M\in X$ and $X\sats\varphi(z,\M)$. Thus, the uniformity in the calculations above yield that $t_1$
is $\muSigma_{n+1}^{\M}(\pvec)$.)

Also, $\Ss_\om(\M)=\Hull_{\Sigma_1}^{\Ss_\om(\M)}(\RR\un\{
\M\})$. Therefore every set of reals in
$\Ss_{\om+\om}(\M)$ is $\Sigma^1_m(t_1,z)$ for some $z\in\RR$ and $m<\om$, and it is
therefore
$\muSigma_{n+2}^{\M}(\pvec,z)$, like before.
And etc.

Now consider \ref{item:no_new_set_first-order_def}; suppose there is no $X$ as in \ref{item:new_set_first-order_def}. Then
$\rho_\om^\M>\om_1$, so $\M_{\gamma+\om}\sats$``$\Theta$
exists''. Let $\theta=\Theta^{\M_{\gamma+\om}}$. Using part \ref{item:new_set_first-order_def} it
follows that $\M_\theta$ and $\SS_\theta$ have the same universe.
But then $\M_{\gamma'}$ and $\SS_{\gamma'}$ have the same universe
for all limits $\gamma'\in[\theta,\OR^\M+\om]$, because
$\rho_\om(\M_{\gamma'})>\om_1$ for all limits $\gamma'\in[\theta,\OR^M]$.
\end{proof}

\begin{dfn} Let $\gamma>\om_1$ be a  limit,  $\M=\M_\gamma$ and $n<\om$.
Suppose
$\rho_{n+1}^{\M}=\om_1<\rho_n^{\M}$. We say that
$\mSigma_{n+1}$ is \dfnemph{$\mu$-reflecting at $\gamma$} iff whenever $\varphi$
is an $\mSigma_{n+1}$ formula, $x\in\M$ and $m<\om$, if
\[\M\sats\all^*_m s\ \varphi(x,s), \]
then there is $\gamma'<\rho_n^\M$ such that:
\begin{enumerate}[label=--]
 \item 
if $n=0$ then
$\M_{\gamma'}\sats\all^*_m s\ \varphi(x,s)$, and
\item if $n>0$ then letting
$t=\Th_{\muSigma_n}^{\M}(\M_{\gamma'}\un
\{\pvec_n^\M\})$, we have
\[ \M\sats\all^*_m s\ \Big[t\mathrm{\ codes\ a\ witness\ to\
}\varphi(x,s)\Big].\qedhere \] 
\end{enumerate}
\end{dfn}

\begin{lem}
 Let $\gamma>\om_1$ be a limit,
  $\M=\M_\gamma$ and $n<\om$. Suppose
$\rho_{n+1}^{\M}=\om_1<\rho_n^{\M}$. Then:

\begin{enumerate}[label=\textup{(\alph*)}]
 \item\label{item:mSigma_n+1=muSigma_n+1} Suppose $\mSigma_{n+1}$ is $\mu$-reflecting at $\gamma$. Then
$\muSigma_{n+1}^{\M}=\mSigma_{n+1}^{\M}$.

 \item\label{item:when_not_mu-reflect} Suppose $\mSigma_{n+1}$ is not $\mu$-reflecting at $\gamma$. Fix
$\varphi,x,m$ witnessing this. Then:
\begin{enumerate}[label=\textup{\roman*)}]
 \item\label{item:cof(rho_n)>om} $\cof(\rho_n^{\M})>\om$.\footnote{Note this seems to assume
 $\AC_{\om,\RR}$.}
 \item\label{item:define_the_theories} There is an $\mSigma_{n+1}$
formula $\psi$ such that for all $y,t\in\M$, we have
$t=\Th_{\mSigma_{n+1}}^{\M}(\{y\})$ iff $\all^*_m
s\ \psi(y,t,s,x,p_n^\M)$.
 \item\label{item:define_the_J_theories}
Let
$k\in[1,\om)$. Then $\Th_{\Sigma_1}^{\Ss_{k\om}(\M)}(\RR\un\{\M\})$, coded naturally as
a set of reals, is $\muSigma_{n+k}^{\M}(\pvec_{n+1}^\M,x)$.\footnote{
Note that this is one step lower in the $\muSigma$-hierarchy than what is 
given by the proof of Proposition \ref{prop:JMhierarchy};
in \ref{prop:JMhierarchy} we had $\rho_n^\M=\om_1$,
whereas here we have $\rho_{n+1}^\M=\om_1<\rho_n^\M$.}
\end{enumerate}
\end{enumerate}
\end{lem}
\begin{proof}
 Parts \ref{item:mSigma_n+1=muSigma_n+1}
 and \ref{item:when_not_mu-reflect}\ref{item:cof(rho_n)>om} are clear. (Moreover, this
doesn't require Turing determinacy beyond $\M$.) 

Consider \ref{item:when_not_mu-reflect}\ref{item:define_the_theories}. We assume $n=0$ for simplicity, but the general proof is
similar.
Let $\psi(s,x,y,t)$ assert ``There is $\gamma'\in\OR$ such that
$\M_{\gamma'}\sats\varphi(x,s)$ and $\Th_1^{\M_{\gamma'}}(\{y\})=t$''.
This works because by part \ref{item:cof(rho_n)>om},
$\Th_{\mSigma_1}^{\M}(\{y\})=\Th_{\mSigma_1}^{\M_
{\gamma'}}(\{y\})$ for sufficiently large $\gamma'<\gamma$.

Consider \ref{item:when_not_mu-reflect}\ref{item:define_the_J_theories}. 
This can be proven similarly
to Proposition \ref{prop:JMhierarchy}. However, because we now have one less $\mu$-quantifier to work with, we need to adjust the argument
for $k=1$ (and then the rest proceeds inductively as before). For this, note that if $\varphi_0$ is arithmetic and $0<\ell<\om$, then the relation $\varphi'_0(u,v)$ asserting
\begin{adjustwidth}{2em}{2em}
 $\text{``}u\in\Dd^{2\ell}$ and letting $C$ be the
set of reals of degree $\leq\max(u)$
and 
\[ t=\Th_{\mSigma_{n+1}}^{\M}(C\cup\{\pvec_{n+1}^\M,v\}),\]
then 
\[
\all^{u_0} x_0\ \exists^{u_1} x_1\ \ldots\ \all^{u_{2\ell-2}} x_{2\ell-2}\ \exists^{u_{2\ell-1}} x_{2\ell-1}\
 \varphi_0(x_0,\ldots,x_{2\ell-1},t)\text{''}\]
\end{adjustwidth}
is of the form $\all^*_{m}
s\ \varrho(u,v,x,\pvec_{n+1}^\M)$, for some $\mSigma_{n+1}$ formula $\varrho$ (here $m,x$ were fixed in the statement of the lemma; note that we can correctly specify $t$ by using the method used to prove part \ref{item:define_the_theories}). 
\end{proof}

\begin{dfn} Let $\gamma\geq\om_1$ be a limit and let $n\in[1,\om)$.

\dfnemph{$\gamma+n$-Turing determinacy} is
the assertion that for each $k<\om$, the $\bfmSigma_n^{\M_\gamma}$
subsets of $\Dd^k$ satisfy Turing determinacy; that is, if $X\sub\Dd^k$ is $\bfmSigma_n^M$ then
\[ (\all^*_k s\ [s\in X])\implies(\exists^*_k s\ [s\in X]).\]

\dfnemph{$\gamma+n$-Turing completeness} is the assertion that
$\gamma+n$-Turing determinacy holds, and for each $k<\om$, given a sequence
$\Avec=\left<A_n\right>_{n<\om}$ such that each $A_n\in\Dd^k$ and $\{(n,s)\bigm|
s\in A_n\}$ is either $\bfmSigma_n^{\M_\gamma}$ or
$\bfmPi_n^{\M_\gamma}$, then
$\bigcap_{n<\om}\Avec\in\Dd^k$.
\end{dfn}

Of course $\gamma+n$-Turing completeness follows $\gamma+n$-Turing
determinacy if we have $\AC_{\om,\RR}$. For the remainder of
the paper we assume $\AC_{\om,\RR}$.

We write $\muPi_n$ for $\neg\muSigma_n$.
Assuming $\gamma+n$-Turing
completeness, (a) $\muPi_n^{\M_\gamma}$-definable relations are just those
definable in the form ``$\all k<\om\all^*_k
s\neg\varphi$, where $\varphi$ is $\mSigma_n$, and (b) $\muSigma_n$ is
closed under ``$\&$'', ``$\orr$'',
``$\ex i<\om$'', and ``$\all i<m$'' for $m<\om$. However, (it seems) it may not
be closed under ``$\neg$'' or ``$\all i<\om$'' or ``$\ex y$''.

\begin{dfn} For $x\in\RR$, limit ordinals $\gamma\geq\omega_1$, and $1\leq n<\om$, let
$\OD^{\mu,\gamma+n}(x)$ denote the set of all $y\in\RR$ such that for some
$\muSigma_n$ formula $\varphi$ and $\gammavec\in\gamma^{<\om}$, for all
$m,k<\om$, we
have $y(m)=k$ iff $\M_\gamma\sats\varphi(m,k,x,\gammavec)$.\end{dfn}

\begin{dfn}[$\beta^*,n^*$]\label{dfn:beta*} Let $(\beta^*,n^*)\in\Lim\cross\om$ be
least such that either $\beta^*+n^*+1$-Turing completeness
fails, or for some $x\in\RR$,
$\OD^{\mu,\beta^*+n^*+1}(x)\neq\OD^\alpha(x)$.\end{dfn}

We consider $\beta^*$ as the end of the S-gap in the $\mu$
hierarchy. By \ref{lem:Pi2charac} and \ref{lem:Msubstructure},
$\rho_{n^*+1}(\M_{\beta^*})=\om_1$.

Let us now observe the larger scale correspondence
between
the $\mu$ hierarchy and the standard $L(\RR)$ hierarchy. 

\begin{dfn}\label{dfn:between_hierarchies} We define a function
$f:\Lim\inter[\om_1,\beta^*]\to\OR$, such that for all
$\gamma\in\dom(f)$, $\M_\gamma$ corresponds to
$\SS_{f(\gamma)}$. Let
$f(\om_1)=\alpha$. Let $f$ be continuous at limits of limits. Given
$f(\gamma)$, if $\rho_\om^{\M_\gamma}=\om_1$ then let
$f(\gamma+\om)=f(\gamma)+\om^2$. Otherwise let $f(\gamma+\om)=f(\gamma)+\om$.
\end{dfn}
\begin{rem}\label{rem:between_hierarchies}
For all limits
$\gamma>\om_1$, we have $f(\gamma)\geq\gamma$,
and repeated application of Proposition \ref{prop:JMhierarchy}
gives that
\begin{equation}\label{eqn:pow(RR)correspond}
\pow(\RR)\inter\M_\gamma=\pow(\RR)\inter\SS_{f(\gamma)}. \end{equation}
Note further that:
\begin{enumerate}
 \item 
$\M_\gamma\sub\SS_{f(\gamma)}$ and $\M_\gamma$ is $\Sigma_1^{\SS_{f(\gamma)}}(\{\alpha\})$,
uniformly in $\gamma$.
\item
$\M_\gamma\sats$``$\Theta$ exists'' iff $\SS_{f(\gamma)}\sats$``$\Theta$ exists''.
\item $f(\gamma)=\gamma$ iff either (i) $\M_\gamma\sats$``$\Theta$ exists''  or (ii) $\gamma=\om\alpha+\xi\om^\om$ (ordinal
exponentiation) for some ordinal $\xi$.
\item If $f(\gamma)=\gamma$ then
$\M_\gamma$ and $\SS_\gamma$ have the same universe, and 
$\SS_\gamma$ is $\Sigma_1^{\M_\gamma}$.
\item 
By (\ref{eqn:pow(RR)correspond}), and since
$\gamma\leq\beta^*$, we have $f(\gamma)\leq\beta$.
\item
Suppose $\gamma>\om_1$ and $f(\gamma)>\gamma$. 
Then $\rho_1^{\M_\gamma}=\om_1$ and $\rho_1^{\SS_{f(\gamma)}}=\RR$. Let $p^\M=p_1^{\M_\gamma}$
and $p^\J=p_1^{\SS_{f(\gamma)}}$. If $p^\M=\emptyset$ then
$p^\J=\{\alpha\}$ (since $\alpha<f(\gamma)\leq\beta$); otherwise
$p^\M=\{\xi\}=p^\J$ for some $\xi>\alpha$. In either case,
$\Th_1^{\M_\gamma}(\RR\un\{p^\M\})$ and
$\Th_1^{\SS_{f(\gamma)}}(\RR\un\{p^\J\})$ are recursively inter-translatable (after substitution of $p^\M$ for $p^\J$ if they differ,
which is just the case that $p^\M=\emptyset$ and $p^\J=\{\alpha\}$).

\item Finally, $\Th_1^{\M_{\om_1}}(\RR)$ and $\Th_2^{\SS_\alpha}(\RR)$ are
also recursively inter-translatable.
\item Therefore $f(\beta^*)\leq\beta<f(\beta^*)+\om^2$, because for all reals
$x$, $\OD^\alpha(x)=\OD^{<\beta}(x)$.
\end{enumerate}

\end{rem}

\subsection{Ordinal definability at the end of a strong gap}

The material in this subsection won't actually be used elsewhere in the paper. We will end up needing
Theorem \ref{tm:Martin}, but will
actually give a second proof of it later. But the proof here is actually easier and more standard, so one should be aware of it. We make use of the $\rSigma_n$ version of the fine structure of $L(\RR)$ (see Remark \ref{rem:rSigma_n^J_beta}).

\begin{lem}\label{lem:rSigma_types_reflect}
 Let $\beta\in\OR$
 and $k<\om$ with $\rho_{k+1}^{\SS_\beta}=\RR<\rho_k^{\SS_\beta}$. If
  $\bfSigma_{k+1}^{\SS_\beta}$ types reflect, then $\bfrSigma_{k+1}^{\SS_\beta}$ types reflect.
\end{lem}
\begin{proof}
 If $k=0$ this is immediate,
 as $\rSigma_1^{\SS_\beta}=\Sigma_1^{\SS_\beta}$ for all $\beta$, by definition.
 
For $k>0$, we just consider the case that $k=1$.
 
 \begin{case} There is $\gamma<\rho_1^{\SS_\beta}$ such that
  $\Hull_{\rSigma_1}^{\SS_\beta}(\RR\cup\gamma\cup\{p\})$ is cofinal in $\beta$, where $p=p_1^{\SS_\beta}$.
  
  Let $x\in\RR$ and $\xi<\beta$.
  We want to show that $\Th_{\rSigma_2}^{\SS_\beta}(\{(x,\xi)\})$ reflects.
  
  Let $\rho=\rho_1^{\SS_\beta}$. Let $w$ be the set of $1$-solidity
  witnesses for $(\J_\beta,p)$.
  Let $T=\Th_{\rSigma_1}^{\SS_\beta}(\RR\cup\gamma\cup\{p\})$ where $\gamma<\rho$ is as in the case hypothesis and  is large enough that $w\in\Hull_1^{\SS_\beta}(\RR\cup\gamma\cup\{p\})$. Let \[ t=\Th_{\Sigma_2}^{\SS_\beta}(\{z\}) \]
  where $z=(x,\xi,\rho,p,w,\gamma,T)$.
  It is enough to see that $\Th_{\rSigma_2}^{\SS_\beta}(\{z\})$ reflects.
  
  Let $\beta'<\beta$ and $z'\in\J_{\beta'}$ be such that $t(z/z')=t'$ where
  \[ t'=\Th_{\Sigma_2}^{\SS_{\beta'}}(\{z'\}).\]
Let $z'=(x,\xi',\rho',p',w',\gamma',T')$. Then it is easy to see that 
\begin{enumerate}[label=--]
 \item $T'=\Th_{\rSigma_1}^{\SS_{\beta'}}(\RR\cup\gamma'\cup\{p'\})$
and
\item $\Hull_1^{\SS_{\beta'}}(\RR\cup\gamma'\cup\{p'\})$
is cofinal in $\beta'$,
\item $w'$ is the set of $1$-solidity
witnesses for $(\J_{\beta'},p')$,
\item $w'\in\Hull_{\rSigma_1}^{\SS_{\beta'}}(\RR\cup\gamma'\cup\{p'\})$.
\end{enumerate}

Let $\varphi$ be an $\rSigma_2$ formula.
Then the following are equivalent:
\begin{enumerate}
 \item 
$\SS_\beta\sats\varphi(z)$
\item $\SS_\beta\sats$``there is $\delta\in(\gamma,\rho)$
such that  $\Th_{\rSigma_1}^{\SS_\beta}(\RR\cup\delta\cup\{p\})$
codes a putative witness to $\varphi(z)$''
\item\label{item:Sigma_2(T,rho)_expresses_rSigma_2} $\SS_\beta\sats$``there is $\delta\in(\gamma,\rho)$
and a theory $U$ in parameters in $\RR\cup\delta\cup\{p\}$ such that $U$ 
codes a putative witness to $\varphi(z)$ and:
\begin{enumerate}
\item $U$
is closed under logical deduction,
\item $U$ contains no formula of the form ``$\psi\wedge\neg\psi$'', 
\item\label{item:all_Sigma_1_truths_in} $\Th_{\rSigma_1}^{\SS_\beta}(\RR\cup\delta\cup\{p\})\sub U$,
\item $T=U\rest(\RR\cup\gamma\cup\{p\})$, and
\item for every  $\psi\in U$
there is $\psi'\in T$ such that 
\[\text{``}\exists \tau\ [S_\tau(\RR)\sats\psi\wedge\neg\psi']\text{''}\in U. \]
\end{enumerate}
\end{enumerate}
The statement in condition \ref{item:Sigma_2(T,rho)_expresses_rSigma_2}
is moreover $\Sigma_2(\{z\})$
(the main complexity comes from clause
\ref{item:all_Sigma_1_truths_in},
which is $\Pi_1(\{p,\delta\})$). Writing
$\varphi^*$ for the $\Sigma_2$ formula used in condition \ref{item:Sigma_2(T,rho)_expresses_rSigma_2}, it follows that
$\SS_\beta\sats\varphi(z)$ iff
$\J_{\beta'}\sats\varphi^*(z')$.
For each $\rSigma_2$ formula $\varphi$ such that $\SS_\beta\sats\varphi(z)$, let $\delta'_\varphi$
be the least $\delta\in(\gamma',\rho')$
witnessing that $\J_{\beta'}\sats\varphi^*(z')$. Let $\rho''=\sup_\varphi\delta'_\varphi$. So $\rho''\leq\rho'$. It is easy to see that for each $\rSigma_2$ formula $\varphi$ such that $\SS_\beta\sats\varphi(z)$,
the theory $\Th_{\rSigma_1}^{\SS_{\beta'}}(\RR\cup\delta_\varphi\cup\{p'\})$
is in $\J_{\rho''}$.
So letting \[\J_{\beta''}=\cHull_{\rSigma_1}^{\SS_{\beta'}}(\RR\cup\rho''\cup\{p'\}) \]
and $\pi:\J_{\beta''}\to\J_{\beta'}$ the uncollapse map,
it follows that $\rho_1^{\J_{\beta''}}=\rho''$ and $p_1^{\J_{\beta''}}=\pi^{-1}(p')$
(note that we have the $1$-solidity witnesses in $\rg(\pi)$).

It is now straightforward to see that $\Th_{\rSigma_2}^{\J_{\beta''}}(\{\pi^{-1}(z')\})=\Th_{\rSigma_2}^{\SS_\beta}(\{z\})$ (after exchanging $\pi^{-1}(z')$ for $z$), completing
the proof in this case.
 \end{case}

 \begin{case} Otherwise.

 This case is dealt with fairly similarly.
 However, there is no theory $T$.
 We take $\gamma<\rho_1^{\SS_\beta}$ large enough that $w\in\Hull_1^{\SS_\beta}(\RR\cup\gamma\cup\{p\})$.
Condition \ref{item:Sigma_2(T,rho)_expresses_rSigma_2} in the previous case is modified to assert that $\SS_\beta\sats$``there is $\delta\in(\gamma,\rho)$ and $U$ and $\tau\in\OR$ such that $U=\Th_{\rSigma_1}^{S_\tau(\RR)}(\RR\cup\delta\cup\{p\})$
and $\Th_{\rSigma_1}^{\SS_\beta}(\RR\cup\delta\cup\{p\})\sub U$ and $U$ codes a putative witness to $\varphi(z)$''.
Defining things otherwise  as before,
we get again that $\rho_1^{\J_{\beta''}}=\rho''$ and $p_1^{\J_{\beta''}}=\pi^{-1}(p')$ and for every $\alpha<\rho''$, $\Hull_{\rSigma_1}^{\J_{\beta''}}(\RR\cup\alpha\cup\{\pi^{-1}(p')\})$ is bounded in $\beta''$,
which ensures that $\Th_{\rSigma_2}^{\J_{\beta''}}(\{\pi^{-1}(z')\})=\Th_{\rSigma_2}^{\SS_\beta}(\{z\})$ modulo  exchange of parameters.
 \end{case}
\end{proof}

The following result is mentioned in \cite[p.~2]{rudo_steel},
where it is stated that ``The proof of Theorem 3.3 from [\emph{Scales in $L(\RR)$}, Steel]
shows that if $\alpha$ ends a \emph{strong} S-gap, then for a cone of reals $x$, $\OD^\alpha(x)=\OD^{<\alpha}(x)$.'' 
The full result (removing the restriction of the cone) is then credited in \cite{rudo_steel} to Woodin.
We prove the result below using essentially Martin's proof
(Lemma \ref{lem:rSigma_types_reflect} is used, though this is somewhat incidental). \footnote{It seems that if one was using exclusively the fine structure of $L(\RR)$ as presented in \cite{scales_in_LR},
and hence not considering the $\rSigma_n$ hierarchy, one might have been led into wanting access to a  parameter in order to define the necessary $\Sigma_n$ Skolem functions,
and this might have led to the cone version of Theorem \ref{tm:Martin} being mentioned in the manner it was in \cite{rudo_steel}. This issue disappears after passing from $\Sigma_n$ to $\rSigma_n$. It does seem that in
order to justify this passage, one needs Lemma \ref{lem:rSigma_types_reflect}
(or alternatively one could   modify
the definition of weak/strong S-gap
to use $\rSigma_n$ instead of $\Sigma_n$, and redo the analysis of weak and strong S-gaps of \cite{scales_in_LR} under this modification).
However, the methods involved in the $\rSigma_n$ version of the fine structure of $L(\RR)$,
as well as the proof of Lemma \ref{lem:rSigma_types_reflect},
were in fact already well known already in 1999 (when \cite{rudo_steel} was written),
and contained in published form for example in the union of \cite{scales_in_LR} and \cite[\S2]{fsit}.}
For the result we drop our global assumption of $\AD^{L(\RR)}$.

\begin{tm}[essentially Martin]\label{tm:Martin} Assume $\ZF+\DC$.
Let $[\alpha,\beta]$ be a strong S-gap of $L(\RR)$,
and suppose $\AD^{\SS_\alpha}$ holds.
 Then 
$\OD^{<\alpha}_x=\OD^{\beta}_x$
for all reals $x$.
\end{tm}
\begin{proof}
 Suppose not. Let $x,y\in\RR$ with $y\in\OD^\beta_x\cut\OD^{<\alpha}_x$.
 By minimizing ordinal parameters
 and taking $m$ sufficiently large, we can 
 find an even $m\in(0,\om)$ and $\Sigma_0$ formula $\varphi$  of the $L(\RR)$ language such that for all $m<\om$, we have
 \[ m\in y\iff \SS_\beta\sats\all X_0\exists X_1\ldots\all X_{n-2}\exists X_{n-1}\ \varphi(x,X_0,\ldots,X_{n-1},m).\]
 Let $k<\om$ be least such that $\rho_{k+1}^{\SS_\beta}=\RR$. 
  Let $\pvec=\pvec_{k+1}^{\SS_\beta}$.
 Let $h:_{\mathrm{p}}\RR\to\SS_\beta$
 be a surjective partial function which is $\rSigma_{k+1}(\{\pvec\})$-definable. (Note that because we use $\rSigma_{k+1}$, not $\Sigma_{k+1}$, such an $h$ exists.)
Then for all $m<\om$, we have
 \[ m\in y\iff \all^\RR x_0\exists^\RR x_1\ldots\all^\RR x_{n-2}\exists^\RR x_{n-1}\ \SS_\beta\sats\varphi'(\pvec,x,\vec{x},m),\]
 where $\vec{x}=(x_0,\ldots,x_{n-1})$ and $\varphi'(\pvec,x,\vec{x},m)$ asserts ``if $x_0\in\dom(h)$ then [$x_1\in\dom(h)$ and if $x_2\in\dom(h)$ then [\ldots[$x_{n-1}\in\dom(h)$ and $\varphi(x,h(x_0),\ldots,h(x_{n-1}),m)$]\ldots]]''; note that $\varphi'(\pvec,x,\vec{x},m)$
 just makes a simple assertion about $\Th_{\rSigma_{k+1}}^{\SS_\beta}(\{(x,\vec{x}),\pvec\})$.
 
 Given a Turing degree $t$,
 write $\all^tx$ for ``$\all x\leq_T t$'',
 and $\exists^tx$ for ``$\exists x\leq_T t$''.
 It follows that
 \[ m\in y\iff \all^*_n s\ \all^{s_0}x_0\exists^{s_1}x_1\ldots\all^{s_{n-2}}x_{n-2}\exists^{s_{n-1}}x_{n-1}\ \SS_\beta\sats\varphi'(\pvec,x,\vec{x},m),\]
where $s=(s_0,\ldots,s_{n-1})$.
In fact we also have

\[\all^*_n s\ \all^\om m\ \Big[ m\in y\iff \all^{s_0}x_0\exists^{s_1}x_1\ldots\all^{s_{n-2}}x_{n-2}\exists^{s_{n-1}}x_{n-1}\ \SS_\beta\sats\varphi'(\pvec,x,\vec{x},m)\Big]. \]
Let $X$ be the set of all such $s\in\Dd^n$
(that is, after we remove the ``$\all^*_ns$'' quantifier at the front,
the resulting statement holds of $s$).

Since $[\alpha,\beta]$ is strong
and by Lemma \ref{lem:rSigma_types_reflect},
for each $s\in\Dd^n$ 
there is $\widetilde{\beta}<\beta$ and some $\widetilde{\pvec}\in[\widetilde{\beta}]^{<\om}$ such that
\[ \Th_{\rSigma_{k+1}}^{\J_{\widetilde{\beta}}}(\{(x,s),\widetilde{\pvec}\})=\Th_{\rSigma_{k+1}}^{\SS_\beta}(\{(x,s),\pvec\})\]
(modulo exchange of parameters).

Let $\ell=\lh(\pvec)$.
It follows that if $s\in X$ then
there is  $\widetilde{\beta}<\beta$
and $\widetilde{\pvec}\in[\widetilde{\beta}]^\ell$
such that  
\begin{equation}\label{eqn:define_y_locally}\all^\om m\ \Big[m\in y\iff \all^{s_0}x_0\exists^{s_1}x_1\ldots\all^{s_{n-2}}x_{n-2}\exists^{s_{n-1}}x_{n-1}\ \J_{\widetilde{\beta}}\sats\varphi'(\widetilde{\vec{p}},x,\vec{x},m)\Big].\end{equation}
(This might also hold for some $s\in\Dd^n\cut X$.)
Moreover, since $\widetilde{\beta}<\beta$,
in fact the least such $\widetilde{\beta}$ is ${<\alpha}$,
since $[\alpha,\beta]$ is an S-gap.
Let $X'$ be the set of all $s\in\Dd^n$ such that there is
$(\widetilde{\beta},\widetilde{\pvec})$
with $\widetilde{\beta}<\alpha$ and $\widetilde{\pvec}\in[\widetilde{\beta}]^\ell$
satisfying line (\ref{eqn:define_y_locally}),
and for $s\in X'$ let $(\widetilde{\beta}_s,\widetilde{\pvec}_s)$ be the lexicographically least witness. Note that
$X'\sub\Dd^n$ and 
$\all^*_ns\ [s\in X']$, and both $X'$ and the function are $\Sigma_1^{\SS_\alpha}(\{x,y\})$.
Since $\SS_\alpha$ is admissible,
it follows that $\alpha'=\big(\sup_{s\in X'}\widetilde{\beta}_s\big)<\alpha$,
and $X'$ is measure one.
Note that in fact, the function and $X'$ are
 $\Sigma_1^{\SS_{\alpha'}}(\{x,y\})$.
 For $s\in X'$, let $\eta_s\in\OR$
 be the  ordinal at the rank of 
  $(\widetilde{\beta}_s,\widetilde{\pvec}_s)$ in the lexicographic ordering
  of $[\alpha']^{1+\ell}$.
  Let $\alpha''=\sup_{s\in X'}\eta_s$.
Letting $f:\Dd^n\to\OR$ be
  $f(s)=\eta_s$, then $f\in\J_{\alpha''+\om}$.
  
  Now let $U$ be the ultrapower
  of $\alpha''$ modulo the $n$th iterate $\mu^n$ of the Martin measure $\mu$,
  using only functions $g:\Dd^n\to\alpha''$
 with $b\in\J_{\alpha''+\om}$.
Then $U$ is a wellorder, and
  has ordertype ${<\alpha}$, since $\SS_\alpha$ is admissible and the ordertype of $U$ is definable over $\J_{\alpha''+\om}$. We may take $U\in\alpha$. Let $\eta=[f]_{\mu^n}$ be the ordinal represented by $f$ with respect to this ultrapower. By taking $\alpha'''<\alpha$ large enough, the set
  \[ \mathscr{F}=\Big\{g\in\J_{\alpha''+\om}\Bigm|g:\Dd^n\to\alpha''\text{ and }[g]_{\mu^n}=\eta\Big\} \]
    is definable over $\J_{\alpha'''}$ from the parameter $(\eta,\alpha'')$.
  But this easily results in a definition of $y$ from $(x,\eta,\alpha',\alpha'')$ over $\J_{\alpha'''}$: we have $m\in y$
  iff for all $g\in\mathscr{F}$, there is $s\in\Dd^n$
  such that
  \[ \all^{s_0}x_0\exists^{s_1}x_1\ldots\all^{s_{n-2}}x_{n-2}\exists^{s_{n-1}}x_{n-1}\ \J_{\beta^*}\sats\varphi'(\pvec^*,x,\vec{x},m),\]
  where $g(s)$
  is the rank of $(\beta^*,\pvec^*)$
in the lexicographic ordering of $[\alpha']^{1+\ell}$. So $y\in\OD^{<\alpha}_x$, a contradiction.
\end{proof}

We will actually give a second proof of the result above later, an
inner model theoretic proof, which avoids the 
(and Martin's) key trick of taking an ultrapower via the Martin measure. So although
we have just given a proof of the result,
the remainder of the paper will ignore it.

\subsection{Without global determinacy assumptions}

In this section we drop the global determinacy assumption,
just assuming determinacy in $\SS_{\alphagap}\sats\KP$.

\begin{lem}\label{lem:T_com_Sigma_1^J_alpha}
For each $k<\om$, we have Turing completeness for
$\bfSigma_1^{\J_\alphagap}$ 
subsets of $\Dd^k$.\end{lem}
\begin{proof}Proof for $k=1$: Let $\varphi$ be a $\Sigma_1$
formula and $x\in\RR$. Then by admissibility, (i) $\all^{\Dd}d\ \ex^{\Dd} e$ such
that $e\geq d$ and $\J_{\alphagap}\sats\varphi(e,x)$ iff (ii)
$\ex\alpha'<\alphagap$ such that $\all^{\Dd}
d\ \ex^{\Dd} e$ such that $e\geq d$ and
$\J_{\alpha'}\sats\varphi(e,x)$. But if (ii) then by Turing completeness
in $\J_\alphagap$, there is a cone of degrees $e$ such that
$\J_{\alpha'}\sats\varphi(x,e)$. Now for $k=2$: By the $k=1$ case, (i)
$\all^{\Dd}d_1\ \ex^{\Dd}e_1\ \all^{\Dd} d_2\ \ex^{\Dd}e_2$ such that
$e_1\geq d_1$ and $e_2\geq d_2$ and $\J_\alphagap\sats\varphi(x,e_1,e_2)$
iff (ii) $\all^{\Dd}
d_1\ \ex^{\Dd} e_1$ such that $e_1\geq d_1$ and there is $\alpha'<\alphagap$
and
a cone of degrees $e_2$ such that $\J_{\alpha'}\sats\varphi(x,e_1,e_2)$.
By admissibility, (ii) reflects to some $\alpha''<\alphagap$, so we again get
Turing determinacy. Etc.
\end{proof}

\section{Through an admissible gap}\label{sec:through_gap}

We now carry on with the $\alphag,\betag,\Gammag,\xg,\yg$  fixed
at the start of  \S\ref{subsec:mtr}.
Recall that we  adopted
there 
Assumption \ref{ass:Sigma_M_not_Gamma-guided} (in the context of the conjectures). We also assume in this section  that $\SS_{\alphag}$ is  admissible,
so everything in \S\ref{sec:M-hierarchy} applies with start of S-gap $\SS_{\alphag}$.
Write $\M_\gamma=\M^{\alphag}_\gamma$. Define $\beta^*$
relative to $\alphag$
as in Definition \ref{dfn:beta*}.

Given the results in \S\ref{sec:start_of_gap},
we may assume that $\xg$ is sufficient and $\Pg$
is a sound $\Gammag$-stable
mtr-suitable $\xg$-premouse
of degree $d=\deg(\Pg)<\om$,
and $\Pg$ is $(d,\om_1+1)$-iterable,
and such that in the context of the conjectures,
we have $\Pg\in\HC^M$, $\HC^M$
is closed under $\Sigma_{\Pg}$
and $\Sigma_{\Pg}\rest\HC^M$
is definable from parameters in $\HC^M$.
From now on an unqualified ``premouse'' will  mean
``$\xg$-premouse'', unless
specified otherwise.

The plan is now to realize $\M_{\beta^*}$
as a kind of derived model of an $\RR$-genericity iterate $N$ of  $\Pg$.

\subsection{The generic $\M(\RR^*)$}
\label{subsec:the_generic_M(R*)}

\begin{dfn}
Let $\PP$ be a poset and
$G$ be $(V,\PP)$-generic.

Let $\Tt\in V[G]$.
Let $d=\deg(\Pg)$.
We say that $\Tt$ is an \emph{almost-relevant generic $\Sigma_{\Pg}$-tree}
iff $\Tt$ is a  $d$-maximal tree on $\Pg$  of  length ${\leq\omega_1^V}$,
 $\Tt\rest\alpha$ is via $\Sigma_{\Pg}$
for each $\alpha<\omega_1^V$,
if $\lh(\Tt)<\om_1^V$
then $\Tt$ has successor length
and $b^\Tt$ does not drop,
and
 if $\lh(\Tt)=\om_1^V$
then there is no $n<\om$
such that $\Tt$ is based on $\Pg|\delta_n^{\Pg}$. 

Let $\Tt'\in V[G]$.
We say that $\Tt'$ is a \emph{relevant generic $\Sigma_{\Pg}$-tree} iff
there is an almost-relevant generic
 $\Sigma_{\Pg}$-tree $\Tt$
 such that either
$\lh(\Tt)<\om_1^V$ and $\Tt'=\Tt$, or $\lh(\Tt)=\om_1^V$,
 $\Tt'=\Tt\conc b$
where $b$ is some $\Tt$-cofinal branch, and $M^\Tt_b$ is wellfounded.

Let $N\in V[G]$.
 We say that $N$ is a \emph{generic non-dropping $\Sigma_{\Pg}$-iterate}
 iff there is a relevant generic $\Sigma_{\Pg}$-tree
 $\Tt\in V[G]$
such that $N=M^\Tt_\infty$.
\end{dfn}

\begin{rem}
 Note that every countable non-dropping $\Sigma_{\Pg}$-iterate $N$ of $\Pg$ (that is, with $N\in\HC^V$)
 is a generic non-dropping $\Sigma_{\Pg}$-iterate of $\Pg$.
\end{rem}

\begin{lem}\label{lem:generic_iterates_wfd}
Let $\Tt\in V[G]$ be an almost-relevant generic $\Sigma_{\Pg}$-tree of limit length \tu{(}hence of length $\om_1^V$\tu{)}.
Then there is a unique $\Tt$-cofinal branch $b\in V[G]$, and moreover, $M^\Tt_b$ is wellfounded and does not drop in model or degree.
\end{lem}
\begin{proof}
The existence and uniqueness of $b$ is just because the Woodin cardinals of $\Pg$ are strong cutpoints, and the fact that $b$ does not drop in model or degree
is because $\lambda^{\Pg}\leq\rho_d^{\Pg}$
where $d=\deg(\Pg)$.
The wellfoundedness of $M^\Tt_b$ is by an easy absoluteness argument between $V$ and $V[G]$.\end{proof}

\begin{dfn}\label{dfn:RR-genericity_iteration}
Let $\PP$ be a poset
and $G$ be $(V,\PP)$-generic.
Let $\Tt\in V[G]$.
We say that $\Tt$ is an
\emph{$\RR$-genericity $\Sigma_{\Pg}$-tree}
iff $\Tt$ is a relevant generic $\Sigma_{\Pg}$-tree,
$\lh(\Tt)=\om_1^V+1$, and letting $N=M^\Tt_{\infty}$, 
there is an $(N,\CC^N)$-generic $g\in V[G]$
such that \[\RR^V=\bigcup_{\alpha<\om_1^V}(\RR^{V[G]}\cap N[g\rest\alpha]).\]
We say that $N\in V[G]$
is an \emph{$\RR$-genericity $\Sigma_{\Pg}$-iterate}
iff $N=M^\Tt_\infty$ for some $\RR$-genericity $\Sigma_{\Pg}$-tree $\Tt\in V[G]$.\end{dfn}

\begin{rem}\label{rem:get_RR-genericity_generic}
 A standard forcing
 construction\footnote{Starting
 with $\left<x_n,g_n\right>_{n<\om}$,
 modify each $g_n$ only on finitely much
 of its support, producing $g'_n$,
 in a manner such that $g=\bigcup_{n<\om}g'_n$
 is as desired.}
 shows that if $\Tt\in V[G]$
 is a relevant generic  $\Sigma_{\Pg}$-tree,
 then letting $N=M^\Tt_\infty$,
the following are equivalent:
\begin{enumerate}[label=--]\item $\lh(\Tt)=\om_1^V+1$ and there is $g\in V[G]$ as in Definition \ref{dfn:RR-genericity_iteration},
 \item 
there is $\left<x_n,g_n\right>_{n<\om}\in V[G]$
such that $\RR^V=\{x_n\}_{n<\om}$ and for each $n$, $g_n$ is $(N,\Coll(\om,\delta_n^{N}))$-generic
and $x_n\in N[g_n]$. 
\end{enumerate}
\end{rem}

\begin{dfn}
 Let $N\in\HC$ be a $\Sigma_{\Pg}$-iterate of $\Pg$, via successor length tree $\Tt$. Let $k=\deg^\Tt_\infty$.
Then $\Sigma_{\Pg N}$ denotes the $(k,\omega_1+1)$-strategy for $N$
given by (full) normalization,
as in \cite{fullnorm}. 
\end{dfn}

\begin{dfn}\label{dfn:standard_decomposition}
 Let $d=\deg(\Pg)$
 and either let $N=\Pg$
 or let $N$ be a relevant generic $\Sigma_{\Pg}$-iterate (appearing in some generic extension of $V$). Let $\Uu$
 be a $d$-maximal tree on $N$.
 Then the \dfnemph{standard decomposition}
 is given by decomposing $\Uu$ into
 its segments in the intervals between Woodins; that is, it is the unique sequence $\left<\Uu_i\right>_{0\leq i\leq n}$
 or $\left<\Uu_i\right>_{i<\om}$
 such that:
 \begin{enumerate}[label=--]
\item $\Uu_0$ is based on $N|\delta_0^{N}$, 
\item for each $i<\om$,
$\Uu_{i+1}$ is defined
iff $b^{\Uu_0\conc\ldots\conc\Uu_i}$
does not drop in model (hence nor degree)
and $\Uu\neq\Uu_0\conc\ldots\conc\Uu_i$,
and
\item if $\Uu_{i+1}$ is defined
then it is based on $M^{\Uu_i}_\infty|[\delta_i^{M^{\Uu_i}_\infty},\delta_{i+1}^{M^{\Uu_i}_\infty})$
\end{enumerate}
(note that some $\Uu_i$'s might be trivial, but cofinally many are non-trivial). We say that $\Uu$ is \emph{$\lambda^N$-unbounded} iff  $\Uu_i$ is defined for each $i<\om$, and otherwise \emph{$\lambda^N$-bounded}.
\end{dfn}

\begin{dfn}\label{dfn:Sigma_Pg,N_for_generic_N}
Let $\PP$ be a poset, let $G$ be $(V,\PP)$-generic, and work in $V[G]$.
 
 Let $N$ be a relevant generic $\Sigma_{\Pg}$-iterate, via tree $\Tt\in V[G]$, with $\lh(\Tt)=\omega_1^V+1$.
Let $d=\deg(\Pg)$.
Then $\Sigma_{\Pg N}$ denotes the
(putative) partial iteration strategy $\Sigma$ for $N$ such that:
\begin{enumerate}
 \item 
the domain of $\Sigma$
consists of all limit length $d$-maximal trees
$\Uu$ on $N$ which are via $\Sigma$, $\lh(\Uu)\leq\omega_1^V$,
and if $\Uu$ is $\lambda^N$-bounded
and based on $N|\delta_i^N$
then $\Uu'\in V$,
where $\Uu'$ is the equivalent
tree literally on $N|\delta_i^N$, and
\item  letting $\left<\Tt_i\right>_{i<\om}$
be the standard decomposition of $\Tt$,
 then either:
\begin{enumerate}[label=--]
 \item $\Uu$ is $\lambda^N$-bounded,
 based on $N|\delta_i^N$,
 and $\Uu\conc\Sigma(\Uu)$
 is via $\Sigma_{\Pg N_i}$,
 where $N_i=M^{\Tt_0\conc\ldots\conc\Tt_i}_\infty$, or
\item $\Uu$ is $\lambda^N$-unbounded,
and $\Sigma_N(\Uu)$ is the unique $\Uu$-cofinal branch.\qedhere
\end{enumerate}
\end{enumerate}
\end{dfn}

\begin{lem}\label{lem:Sigma_Pg,N_for_generic_N_is_it_strat}
 Let $\Pp,G,N$ be as in Definition \ref{dfn:Sigma_Pg,N_for_generic_N}.
 Then $\Sigma_{\Pg N}$
 is total on its 
\tu{(}putative\tu{)} domain,
and it produces only wellfounded models. 
Moreover, if $\Uu$ is via $\Sigma_{\Pg N}$, of successor length, and $b^\Uu$
does not drop, then $M^\Uu_\infty$
is a relevant generic $\Sigma_{\Pg}$-iterate, via tree $\Xx$ with $\lh(\Xx)=\omega_1^V+1$.
\end{lem}
\begin{proof}
Normalization converts trees via $\Sigma_{\Pg N}$ to either trees $\Xx\in V$ via $\Sigma_{\Pg}$ or (putative) relevant generic $\Sigma_{\Pg}$-trees,
so wellfoundedness
is by Lemma \ref{lem:generic_iterates_wfd}.
\end{proof}

\begin{dfn}\label{dfn:finite_stacks_strat_for_generic_iterate}
 Let $\PP$ be a poset, let $G$ be $(V,\PP)$-generic, and work in $V[G]$.
 
 Let $N$ be a relevant $\Sigma_{\Pg}$-iterate, via tree $\Tt\in V[G]$,
 with $\lh(\Tt)=\omega_1^V+1$.
 Let $d=\deg(\Pg)$. Then $\Sigma^{<\om}_{\Pg N}$
 denotes the (putative) partial iteration strategy for finite stacks of trees
 $(\Uu_0,\ldots,\Uu_n)$
 iteratively as in Definition \ref{dfn:Sigma_Pg,N_for_generic_N}; that is:
 \begin{enumerate}
\item $\Uu_0$ is via $\Sigma_{\Pg N}$,
 \item if $\Uu_{i+1}$ is defined
 then $b^{(\Uu_0,\ldots,\Uu_i)}$ does not drop,
 and $\Uu_{i+1}$ is via $\Sigma_{\Pg M^{(\Uu_0,\ldots,\Uu_i)}_\infty}$ (see Lemma \ref{lem:Sigma_Pg,N_for_generic_N_is_it_strat}).\qedhere
 \end{enumerate}
\end{dfn}
\begin{lem}
 Let $\PP,G,N$ be as in Definition \ref{dfn:finite_stacks_strat_for_generic_iterate}. Then:
 \begin{enumerate}
  \item $\Sigma^{<\om}_{\Pg N}$ is total
 on its \tu{(}putative\tu{)} domain, and produces only wellfounded models, and
 \item\label{item:stack_wfd_direct_limit} if $\left<\Uu_i\right>_{i<\om}$
 is a stack of length $\om$ all of whose proper segments are via $\Sigma^{<\om}$,
 then the direct limit of the stack is wellfounded.
 \end{enumerate}
\end{lem}
\begin{proof}
Part \ref{item:stack_wfd_direct_limit}
is a slight embellishment on the proof of Lemma \ref{lem:generic_iterates_wfd}.
\end{proof}

\begin{lem}
Let $N,Q$ be relevant
generic $\Sigma_{\Pg}$-iterates,
with $\lambda^N=\omega_1^V=\lambda^Q$.
Then there is a successful comparison
$(\Tt,\Uu)$ of $(N,Q)$
via $(\Sigma_{\Pg N},\Sigma_{\Pg Q})$
\tu{(}hence $\Tt,\Uu$ are $d$-maximal
where $d=\deg(\Pg)$\tu{)},
$b^\Tt,b^\Uu$ do not drop,
$M^\Tt_\infty=M^\Uu_\infty$,
and $\lambda^{M^\Tt_\infty}=\omega_1$,
so $M^\Tt_\infty$ is also a relevant generic $\Sigma_{\Pg}$-iterate.
\end{lem}
\begin{proof}
This is straightforward;
the main point
is that for each $n<\om$
the comparison of $N|\delta_n^N$
and $Q|\delta_n^Q$ can be done in $V$,
and hence has only countable length.
\end{proof}

  \begin{dfn}\label{dfn:basic_CC-forcing}
   Let $N$ be an $\om$-small premouse
   with $\om$ Woodins. 
  We write $\CC^N$ for the finite support product $\prod_{n<\om}\Coll(\om,\delta_n^N)$.
  (This is isomorphic to $\Coll(\om,\lambda)$, but we will often want to consider factoring the forcing with initial segments,
  and this is convenient notation for this.)
   Given $m<\om$, we write
  $\CC_{\delta_m^N}=\prod_{n\leq m}\Coll(\om,\delta_n^N)$
  and $\CC_{<\delta_m^N}=\prod_{n<m}\Coll(\om,\delta_n^N)$, and if $m\leq n<\om$
  then $\CC_{[\delta_m^N,\delta_n^N]}$,
  or $\CC_{[\delta_m^N,\lambda^N)}$, etc,
  are defined in the obvious way.
  We also write $\delta_{-1}^N=0$ and $\CC_0^N$ is the trivial forcing.
  In forcing expressions,
  where we factor $\CC$ in some way,
  for example as $\CC_{\delta_0}\cross\CC_{(\delta_0,\delta_1]}\cross\CC_{(\delta_1,\lambda)}$,
  we write $\CC_{\mathrm{tail}}$
  to refer to the final factor $\CC_{(\delta_1,\lambda)}$.
  Write $\Nm_\lambda^N$ for the set of all $x\in N|\lambda$ such that
$x$ has bounded support;
that is, $x$ is a $\CC_{\delta}$-name for some $\delta\in\Delta^N\cup\{0\}$. (So $x$ is also a $\CC^N$-name.)  Let the \emph{support} $\supp(x)$ of $x$ be the least such $\delta$. (We use the minimal support to define standard names,
such as $\check{y}$ for $y\in N|\lambda$; so $\supp(\check{y})=0$.) Also define the \emph{base} $\base(x)$ of $x$ to be the least $\delta\in\Delta\cup\{0\}$ such that $\supp(x)\leq\delta$ and $x\in N|\delta^{+N}$. For $p\in\CC^N$,
  let the \emph{base} $\base(p)$ of $p$
  be the least $\delta\in\Delta^N\cup\{0\}$
  such that $p\in\CC_\delta^N$ (where $\CC_0^N=\{\emptyset\}$).

Let $\Gtilde$ be the standard name for the generic filter
 $\sub\CC^N$. (Note $\Gtilde\notin\Nm_\lambda^N$;
 it has unbounded support.) Let the following $\CC^N$-names,  either elements of $N$ or classes of $N$,
be the natural choices.  Let $\widetilde{\RR}$ be the name for
\[ \bigcup_{\alpha<\lambda}\RR\inter N[\Gtilde\inter(N|\alpha)]. \]
Let $\widetilde{\HC}$ be the natural variant
with $\HC$ replacing $\RR$, and likewise
for other such notions.

For the names $\tau$ mentioned above, $\tau_G$ is the usual interpretation
of $\tau$ via $G$;
so if $\tau\in\Nm_\lambda$ then $\tau_G\in\widetilde{\HC}_G$.

Let $\delta\in\{0\}\cup\Delta^N$.
Let $G\sub\CC_\delta^N$ be $N$-generic and
$x\in\HC^{N[G]}$. Then $T(x,\delta)^{N[x]}$ denotes the set of all $\Sigma_1$ formulas
of $\Ll_{L(\RR)}$ such that for some
strong cutpoint
$\xi\in[\delta,\delta^{+N})$
of $N$
with $x\in (N|\xi)[G]$,
there is $P\pins N|\delta^{+N}$ such that $P[G]$
is a 
pre-$\varphi(x)$-witness, when considered as a $(P|\xi,G)$-premouse. Note that $T(x,\delta)^{N[x]}$
depends only on $N|\delta^{+N}$
and $x$, and the foregoing
definition is made over $(N|\delta^{+N})[x]$,
uniformly in $N,\delta,x$.

Now let $\Ttilde\sub
\CC^N\cross\Nm_\lambda$ be the following $\CC^N$-name. We put $(p,\tau)\in\Ttilde$ iff, letting $\delta=\max(\base(p),\base(\tau))$, we have
\[ N|\delta^{+N}\sats
p\subforces{\CC_\delta}\text{``}\tau=(\tau_0,\tau_1)\text{ where }\tau_1=T(\tau_0,\delta)^{N[\tau_0]}\text{''}.
\]
Note that $\Ttilde$ is $\rDelta_2^{N|\lambda}$.

Let $\Mtilde_\lambda$ be the natural name
for the structure $(\widetilde{\HC},\widetilde{T})$.
\end{dfn}

The structure $(\widetilde{\mathscr{M}}_\lambda)_G$ has signature
that of $\M_{\omega_1}$,
so the $\Sigma_0^{\widetilde{\mathscr{M}}_\lambda^{{\Pg}}}$-forcing
relation $\forces_{\lambda0}$
of ${\Pg}$
 regards forcing (with $\CC^{{\Pg}}$)
for $\Sigma_0$ formulas
in that language (that is,
with $\in$ and the predicate $\dot{T}$). Likewise for the
higher complexity forcing relations for $\widetilde{\mathscr{M}_\lambda}^{{\Pg}}$.

It will also be convenient to define a local version of $\Mtilde_\lambda$ over premice such as $N'=N|\delta_n^{+N}$ (for $N$ as above):

\begin{dfn}\label{dfn:local_Mtilde_delta}
 Let $N'$ be an $\om$-small premouse
 and $\delta\in\OR^{N'}$ be such that
 $N'\sats\ZF^-+$``$\delta$ is the largest cardinal and is a strong cutpoint, and there are only finitely many Woodin cardinals''.
 Then $\Mtilde_{\delta}^{N'}$
 denotes the natural $N'$-proper class $\Coll(\om,\delta)$-name for the structure $(H',T')$ in the language of $\M_{\om_1}$,
 defined like the $\Mtilde_\lambda$ of Definition \ref{dfn:basic_CC-forcing},
 except that for $(N',\Coll(\om,\delta))$-generics $G$,
 $H'$ is the universe of $N'[G]$,
 and for $x\in N'[G]$,
 $T'(x,t)$ holds iff $t=T(x,\delta)^{N'[x]}$, where $T(x,\delta)^{N'[x]}$ is just as in Definition \ref{dfn:basic_CC-forcing}. 
\end{dfn}

\label{lem:varphi(x)-witness_exists}
 \begin{dfn}\label{dfn:R^N_epsvec}
  Let $N$ be an  $\om$-small premouse
  with $\om$ Woodins. 
  Let 
  $\vec{\vareps}=\{\vareps_0<\ldots<\vareps_{k-1}\}\in[\Delta^N]^{<\om}$.
  Let $\delta_{-1}^N=0$.
  For $i<k$, let $\vareps_i^-=\delta_n^N$ where $\vareps_i=\delta_{n+1}^N$.
  Assuming that $k>0$ and $\vec{\vareps}\neq\{\delta_0^N,\ldots,\delta_{k-1}^N\}$,
  let $i_0$ be the least $i< k$
  such that  $\vareps_i>\delta_i^N$.
  The \emph{$L[\es]$-$\mathscr{P}$-construction
  of $N$ at $\vec{\varepsilon}$}
  is the following construction:
  \begin{enumerate}
   \item\label{item:L_E_con_stage}
   We begin with the $L[\es]$-construction $\left<N_\alpha\right>_{\alpha\leq\vareps_{k-1}}$
   of $N|\vareps_{k-1}$ starting with  first model
   $N_0=N|\vareps_{i_0-1}$ (where $\vareps_{-1}=0$), where background extenders
   are required (amongst the usual requirements) to be $E\in\es^N$ such that
   $\nu(E)$ is an $N$-cardinal
   in $(\vareps_i^-,\vareps_i)$
   for some $i\in[i_0,k)$ (hence $\crit(E),\lh(E)\in(\vareps_i^-,\vareps_i)$ also);
   this produces a final model $N_{\vareps_{k-1}}$ of height $\vareps_{k-1}$.
  \item We then form the P-construction
  $\mathscr{P}^N(N_{\vareps_{k-1}})$
  of $N$ over $N_{\vareps_{k-1}}$,
  producing a final model $R$ of height $\OR^N$.
  \end{enumerate}
  Assuming $R$ above is a well-defined premouse,
   $R$ is the \emph{output}
  or \emph{last model} of the construction.
  Write $R^N_{\vec{\vareps}}=R$.
  
  If instead $k=0$ or $\vec{\vareps}=\{\delta_0^N,\ldots,\delta_{k-1}^N\}$,
  then we define $R_{\vec{\vareps}}^N=N$.
\end{dfn}

\begin{rem}\label{rem:R_construction_basic_props}Continue with the notation introduced
in Definition \ref{dfn:R^N_epsvec}.
Suppose $N={\Pg}$ and let $n=\deg(\Pg)$.
Suppose $k>0$ and $\vec{\vareps}\neq\{\delta_0^N,\ldots,\delta_{k-1}^N\}$.
Then
$R$ is well-defined,
 $\OR^R=\OR^{\Pg}$,
and letting $\ell_0<\om$ be least such that
 $\varepsilon_{k-1}<\delta_{\ell_0}^{\Pg}$, then $R$ is $\om$-small, has $\om$ Woodins and
 \[\Delta^R=\{\varepsilon_i\}_{i<k}\cup\{\delta_\ell^{\Pg}\bigm|\ell\in[\ell_0,\om)\}. \]
 By the fine structure of P-construction, $R$ is $n$-sound
 and $\rho_{n+1}^R<\lambda=\lambda^R\leq\rho_n^R$, $\pvec_{n}^R=\pvec_{n}^{\Pg}$,
 \begin{equation}\label{eqn:R_is_eps_k-1-sound}R\text{ is }\varepsilon_{k-1}\text{-sound}\end{equation}
   and $p_{n+1}^R\cut\varepsilon_{k-1}=p_{n+1}^{\Pg}\cut\varepsilon_{k-1}=p_{n+1}^{\Pg}\cut\lambda$.
 
For each $i<k$,
 $\Pg|\varepsilon_{i}$ is $R$-generic
  for the $\varepsilon_i$-generator extender algebra of $R$ at $\varepsilon_i$,
  and $R|\varepsilon_i$ is definable over $\Pg|\varepsilon_i$.
 Also, $\varepsilon_i^{+R}=\varepsilon_i^{+\Pg}$ and $R|\varepsilon_i^{+R}=\Lp_{\Gammag}(R|\varepsilon_i)$ 
 is the result of the $\mathscr{P}$-construction
 $\mathscr{P}^{\Pg|\varepsilon_i^{+\Pg}}(R|\varepsilon_i)$. (Note these things are trivially so for $i<i_0$.) Therefore
 $(R|\varepsilon_i^{+R})[\Pg|\varepsilon_i]$ has universe that of $\Pg|\varepsilon_i^{+\Pg}=\Lp_{\Gammag}(\Pg|\varepsilon_i)$. It follows that there are $\CC_{\varepsilon_i}$-generics $g_i,h_i$
 over $R|\varepsilon_i^{+R}$, $\Pg|\varepsilon_i^{+\Pg}$ respectively,
 with $g_i=g_{i+1}\rest\CC_{\vareps_i}$ and $h_i=h_{i+1}\rest\CC_{\vareps_i}$ for $i+1<k$, and
 such that the universes of $(R|\varepsilon_i^{+R})[g_i]$
 and $(\Pg|\varepsilon_i^{+\Pg})[h_i]$
are the same
for each $i<k$.

 Also, ($\dagger$)
$R$ is a non-dropping $n$-maximal $\Sigma_{\Pg}$-iterate of $\Pg$, so letting
\[ i_{\Pg R}:\Pg\to R \]
 be the iteration map, we have $i_{\Pg R}(\delta_i)=\varepsilon_i$ for $i<k$,
 and also $i_{\Pg R}(\pvec_{m}^{\Pg})=\pvec_{m}^R$ for $m\leq n+1$, and therefore in fact
 $\pvec_{m}^R=\pvec_{m}^{\Pg}$
 for $m\leq n+1$.
 
Suppose  we replace $\Pg$ throughout with some
relevant generic $\Sigma_{{\Pg}}$-iterate $N$ of $\Pg$ (with $N$ appearing in some set generic extension of $V$). Then the resulting versions of all  these things still hold, excluding line (\ref{eqn:R_is_eps_k-1-sound})
 and ($\dagger$).
Instead of (the modified version of) line (\ref{eqn:R_is_eps_k-1-sound}), we 
have
\[ R\text{ is }\vareps_{k-1}\text{-sound}\iff N\text{ is }\vareps_{k-1}\text{-sound
 }\iff\Tt_{\Pg N}\text{ is based
 on }\Pg|i_{\Pg N}^{-1}(\vareps_{k-1}) \]
 where $\Tt_{\Pg N}$ is the 
 $\Sigma_{{\Pg}}$-tree leading from $\Pg$ to $N$. And instead of the modified version of paragraph $(\dagger)$,
 we have $(\dagger')$:
 $R$ is a  relevant generic  $\Sigma_{\Pg}$-iterate of $\Pg$, so letting
$i_{\Pg R}:\Pg\to R$
 be the iteration map, we have $i_{\Pg R}(\bar{\delta}_i)=\varepsilon_i$ for $i<k$,
 where $i_{\Pg N}(\bar{\delta}_i)=\delta_i$,
 and also $i_{\Pg R}(\pvec_{n+1}^{\Pg})=\pvec_{n+1}^R$, and therefore in fact
 $\pvec_{n+1}^R=\pvec_{n+1}^{N}$.

 We will be more interested in applying these things when $N$ is a generic iterate of $\Pg$, as opposed to $N=\Pg$,
 but we will often consider the case that $N=\Pg$ and transfer facts about that to generic iterates $N$
 using the  elementarity of $i_{\Pg N}$.
\end{rem}
\begin{lem}
 Let $N\in V[G]$ be a generic non-dropping $\Sigma_{\Pg}$-iterate. 
Let $R=R_{\vec{\delta}}^N$ for some $\vec{\delta}\in[\Delta^N]^{<\om}$.
Let $g\in V[G]$ be $(N,\CC^N)$-generic
and $g'\in V[G]$ be $(R,\CC^R)$-generic
with $\HC^{N[g]}=\HC^{R[g']}$.
Then $\M_\lambda^{N[g]}=\M_\lambda^{R[g']}$.
\end{lem}
\begin{proof}
 We already have that $\M_\lambda^{\Pg[G]}$ and $\M_\lambda^{R[G']}$ have the same universe,
and their $T$-predicates
 agree since $N,R$ are relevant generic $\Sigma_{\Pg}$-iterates.
\end{proof}

\begin{dfn}\label{dfn:mSigma_0_forcing_relation_Mtilde_lambda}
Let $N$ be an $\om$-small premouse
with $\om$ Woodins.
 The \emph{$\mSigma_0^{\Mtilde_{\lambda^N}^N}$ forcing relation $\sforces{\lambda0}{}$} of $N$ is
 the relation
 of pairs $(p,\varphi(\tauvec))$
 with $p\in\CC^N$, $\varphi$
 an $\mSigma_0$ formula of the language for $\M_{\omega_1}$,
 and $\vec{\tau}\in(\Nm_{\lambda^N}^N)^{<\om}$,
 and such that
 \[p\sforces{\lambda0}{}\varphi(\vec{\tau}) \]
 iff, letting $m=\base(\tauvec)$
 and $\vareps=\delta_m^N$ and $\pbar=p\rest\CC_m^N$, then
 \[ N|\vareps^{+N}\sats \pbar\sforces{\CC_m^N}{}\Mtilde_{\vareps}\sats\varphi(\tauvec);
 \]
recall $(\Mtilde_{\vareps})^{N|\vareps^{+N}}$
was introduced in Definition \ref{dfn:local_Mtilde_delta}. 
\end{dfn}

The mtr-suitability
and $\Gammag$-stability of $\Pg$ easily yield the following lemma:

\begin{lem}\label{lem:mSigma_0,lambda_forcing_rel_def}
\label{lem:Sigma_0_lambda_relation_rDelta_2}
Let
$N$ be a relevant generic $\Sigma_{\Pg}$-iterate.
Then:
\begin{enumerate}
\item\label{item:mSigma_0,lambda_forcing_rel_def} $\Big(\sforces{\lambda^N0}{}\Big)^N$ is $\rDelta_2^{N|\lambda^N}$,
uniformly in $N$. \tu{(}In fact,
it is simpler than this,
because we only need to consult $N|\vareps^{+N}$,
where $\vareps=\delta_{\base(\tauvec)}^N$.\tu{)}
 \item The  $\mSigma_0^{\Mtilde_\lambda}$ forcing theorem holds
\tu{(}with respect to $\Big(\sforces{\lambda^N0}{}\Big)^N$\tu{)}.
\item If $N$ is an $\RR$-genericity iterate, as witnessed by $g$, then letting $\lambda=\lambda^N=\om_1^V$, we have
$(\Mtilde_{\lambda}^N)_g=\M_{\omega_1}$.
`\end{enumerate}
\end{lem}

\begin{dfn}\label{dfn:strong_mSigma_1_forcing}
Let $N,\lambda$ be as above. 
 The \emph{strong $\mSigma_1^{\Mtilde_\lambda}$
 forcing relation} $\forces_{\lambda1}^{\mathrm{s}}$ (\emph{of $N$}) is the relation over $N|\lambda$
 where for $p\in\CC$ and $\mSigma_1$ formulas $\varphi$ (in the language for $\M_{\om_1}$)
 and $\vec{\tau}\in(\Nm_\lambda)^{<\om}$,
 letting $\varphi$ be ``$\exists\vec{x}\psi(\vec{\tau},\vec{x})$'' where $\psi$ is $\mSigma_0$, we have
 \[ p\sforces{\lambda1}{\mathrm{s}}\varphi(\vec{\tau})\iff
 \exists\vec{\sigma}\in\Nm_\lambda\ \Big[p\sforces{\lambda0}{}\psi(\vec{\tau},\vec{\sigma})\Big].\qedhere\]
\end{dfn}

\begin{lem}\label{lem:strong_Sigma_1_forcing_definability_and_+1_Woodin_locality}
Let
$N$ be a relevant generic  $\Sigma_{\Pg}$-iterate and $\lambda=\lambda^{N}$.
Then:
\begin{enumerate}
 \item \label{item:strong_mSigma_1^Mtilde_lambda_is_rSigma_2_def}
The strong $\mSigma_1^{\Mtilde_\lambda}$ forcing
relation of $N$ is $\rSigma_2^{N|\lambda}$, uniformly in $N$.
\item\label{item:strong_mSigma_1^Mtilde_lambda_forcing_theorem} The strong $\mSigma_1^{\Mtilde^N_\lambda}$ forcing theorem holds.
\item \label{item:strong_mSigma_1^Mtilde_lambda_witnessed_locally}
If $\tauvec\in(\Nm_\lambda^N)^{<\om}$, $m=\base(\tauvec)$,  
 $\varphi(\yvec)$ is ``$\exists\xvec\psi(\yvec,\xvec)$'' where $\psi$ is $\mSigma_0$,
$p\in\CC^N$ and $\pbar=p\rest\CC_m^N$,
then
\[ p\sforces{\lambda1}{\mathrm{s}}\varphi(\tauvec)\iff\exists\sigmavec\in(N|\delta_{m+1}^{+N})^{\CC_{m+1}^N}
\ \Big[
\pbar\sforces{\lambda0}{}\psi(\tauvec,\sigmavec)\Big].\]
\end{enumerate}
\end{lem}
\begin{proof}
 Parts \ref{item:strong_mSigma_1^Mtilde_lambda_is_rSigma_2_def}
 and \ref{item:strong_mSigma_1^Mtilde_lambda_forcing_theorem}
 are immediate corollaries
 of Lemmas
\ref{lem:mSigma_0,lambda_forcing_rel_def} and \ref{lem:Sigma_0_lambda_relation_rDelta_2}.

Part \ref{item:strong_mSigma_1^Mtilde_lambda_witnessed_locally}:
It is easy enough to see that we may assume $N=\Pg$.
For the non-trivial direction,
suppose $p\sforces{\lambda1}{\mathrm{s}}\varphi(\tauvec)$,
so there is some $k\in(m,\om)$
such that
\begin{equation}\label{eqn:exists_sigmavec_in_Pg|delta_k^+} \exists\sigmavec\in(\Pg|\delta_k^{+\Pg})^{\CC^{\Pg}_k}\ \Big[p\sforces{\lambda0}{}\psi(\tauvec,\sigmavec)\Big]. \end{equation}
It is easy enough, using homogeneity of the forcing,
to see that line (\ref{eqn:exists_sigmavec_in_Pg|delta_k^+}) still holds after replacing $p$ with $\pbar$ (with the same $k$, but maybe changing the witness $\sigmavec$).
So suppose $k>m+1$. Let $\vec{\vareps}=(\delta_0^N,\ldots,\delta_m^{\Pg},\delta_k^{\Pg})$
and $R=R^{\Pg}_{\vec{\vareps}}$
and $j:\Pg\to R$ be the iteration map. Then $j\rest\delta_m^{+\Pg}=\id$,
so $j(\pbar,\tauvec)=(\pbar,\tauvec)$, but $j(\delta_{m+1}^{\Pg})=\delta_{m+1}^R=\delta_k^{\Pg}$,
and 
\[ \ell:\Pg|\delta_{m+1}^{+\Pg}\to R|\delta_{m+1}^{+R} \]
is fully elementary,
where $\ell=j\rest(\Pg|\delta_{m+1}^{+\Pg})$.
But for every $q\in\CC_{m+1}^R$
with $q\leq\pbar$,
there is $(g,h)$ such that $g$ is $(R,\CC_{m+1}^R)$-generic
with $q\in g$,
 $h$ is $(\Pg,\CC_k^{\Pg})$-generic, and such that $R[g]$ and $\Pg[h]$ have the same universe.
Since $R$ is also mtr-suitable,
it follows that 
\[ R\sats\exists\sigmavec\in (R|\delta_{m+1}^{+})^{\CC_{m+1}^R}\ \Big[\pbar\sforces{\lambda 0}{}\psi(\tauvec,\sigmavec)\Big] \]
(using that
line (\ref{eqn:exists_sigmavec_in_Pg|delta_k^+})
holds with $\pbar$ replacing $p$), and since this statement is in fact elementary over $R|\delta_{m+1}^{+R}$, therefore
it pulls back under $j$ to give
\[ \Pg\sats\exists\sigmavec\in(\Pg|\delta_{m+1}^{+})^{\CC_{m+1}^{\Pg}}\ \Big[\pbar\sforces{\lambda 0}{}\psi(\tauvec,\sigmavec)\Big], \]
as desired.
\end{proof}

\begin{dfn}
 Let $N$ be a premouse
 and $\delta_0<\delta_1<\ldots<\delta_{2k-1}<\OR^N$, with each $\delta_i$
a Woodin cardinal in $N$.
 Let $\vec{\delta}=\{\delta_0,\ldots,\delta_{2k-1}\}$.
 Let $\varphi$ be a formula and $x\in N$.
 We write
 \[  \all^{\mathrm{gen}}_{\vec{\delta}}t\ \varphi(t,x) \]
 for the formula  ``$\CC_{\delta_0}$ forces that for every (Turing) degree $s_0$, $\CC_{\delta_1}$
 forces that there is degree $t_0$
 such that $\ldots$ $\CC_{\delta_{2k-2}}$
 forces that for every degree $s_{k-1}$,
 $\CC_{2k-1}$ forces that there is a degree $t_{k-1}$
 such that $\big[(\all i<k\ (s_i\leq_T t_i))$
 and $\varphi(t,x)$, where $t=(t_0,\ldots,t_{k-1})\big]$''.
 
 Similarly write 
 \[ \exists^{\mathrm{gen}}_{\vec{\delta}}t\ \varphi(t,x) \]
 for the natural dual formula; that is,
 ``$\CC_{\delta_0}$ forces the there is a degree $s_0$ such that $\CC_{\delta_1}$ forces that for all degrees $t_0$  \ldots (through $s_{k-1},t_{k-1}$) such that
  $\big[\text{if }(\all i<k\ (s_i\leq_T t_i))$ then $\varphi(t,x)\big]$''.

 We also need the following variant.
 Let $\delta<\OR^N$.
 We write
 \[ \all^{\mathrm{gen}}_{\geq\delta;k} t\ \varphi(t,x) \]
 for the formula ``For all Woodin cardinals
 $\delta_0\geq\delta$,
 $\CC_{\delta_0}$ forces that for every degree $s_0$, there is a Woodin cardinal $\delta_1>\delta_0$ such that $\CC_{\delta_1}$ forces that there is a degree $t_0$
 such that \ldots for all Woodin cardinals
 $\delta_{2k-2}>\delta_{2k-3}$,
 $\CC_{\delta_{2k-2}}$
 forces that for every degree $s_{k-1}$, there is a Woodin cardinal $\delta_{2k-1}$ such that $\CC_{\delta_{2k-1}}$ forces that
 there is a degree $t_{k-1}$ such that
 $\big[(\all i<k\ (s_i\leq_T t_i))$
 and $\varphi(t,x)\big]$''.
 
 And finally define
 \[ \exists^{\mathrm{gen}}_{\geq\delta;k}\ \varphi(t,x) \]
 by analogy with the preceding definitions.
\end{dfn}

The following lemma and its later variants are useful in showing
that our method of locally defining the Martin measure
$\mu$ over segments of  $N[G]$ for  $\RR$-genericity iterates $N$ of $\Pg$ works correctly.
Its proof elaborates on the method used for Lemma \ref{lem:strong_Sigma_1_forcing_definability_and_+1_Woodin_locality}.

\begin{lem}\label{lem:gen-all_independence}
Let $N$ be a relevant generic $
\Sigma_{\Pg}$-iterate  and $\lambda=\lambda^N$.
Let $\varphi$ be $\mSigma_{1}$
and $x\in\Nm_\lambda^N$.
Let $d=\base(x)$.
Then for all $k<\om$, all $\vec{\delta},\vec{\varepsilon}\in[\{\delta_i^{N}\bigm|i\in[d+1,\om)\}]^{2k}$
and all $\theta<\lambda$, $N|\lambda$ satisfies that $\CC_d$ forces that the following three statements are equivalent:
\begin{enumerate}[label=\tu{(}\roman*\tu{)}]
\item\label{item:deltavec} $\all^{\mathrm{gen}}_{\vec{\delta}}s\ \forces_{\CC_{\mathrm{tail}}}\Mtilde_\lambda\sats\varphi(x,s)$,
\item\label{item:epsvec}
$\all^{\mathrm{gen}}_{\vec{\vareps}}s\ \forces_{\CC_{\mathrm{tail}}}\Mtilde_\lambda\sats\varphi(x,s)$,
\item\label{item:above_theta}
$\all^{\mathrm{gen}}_{\geq\theta;k}s\ \forces_{\CC_{\mathrm{tail}}}\Mtilde_\lambda\sats\varphi(x,s)$.\qedhere
\end{enumerate}
\end{lem}
\begin{rem}\label{rem:formalize_mu-forcing_statements}
In the statement of the lemma,  we identify $(N|\lambda)[g_d]$,
for $g_d$ being $(N|\lambda,\CC_d)$-generic,
with the $(N|\delta_d^N,g_d)$-premouse
whose extender sequence is induced by $\es^N\rest(\delta_d^N,\lambda)$.
This determines the meaning
of $\all^{\mathrm{gen}}$  interpreted in $(N|\lambda)[g_d]$.

By Lemma \ref{lem:strong_Sigma_1_forcing_definability_and_+1_Woodin_locality}(\ref{item:strong_mSigma_1^Mtilde_lambda_witnessed_locally}) (or more literally, a relativization thereof), the statement in part \ref{item:deltavec}
of the lemma 
can be expressed
 as saying
that
\[ (N|\delta^{+N})[\widetilde{g}_d]
\sats\all^{\mathrm{gen}}_{\vec{\delta}}s\ \sforces{\CC_{\delta}}{}\ \Mtilde_{\delta}\sats\varphi(x,s),\]
where (a) either $\deltavec=\emptyset$ and $\delta=\delta_{d+1}^N$, or $\deltavec\neq\emptyset$ and $m<\om$ is such that $\delta_m^N=\max(\vec{\delta})$ and $\delta=\delta_{m+1}^N$, (b) $\widetilde{g}_d$ is the standard name for the $\CC_d$-generic filter, (c)
 $(N|\delta^{+N})[\widetilde{g}_d]$
is the natural name for the
 $(N|\delta_d^N,\widetilde{g}_d)$-premouse $Q$ whose extender sequence in the interval $(\delta_d^{N},\delta^{+N})$ is induced by $\es^N\rest(\delta_d^N,\delta^{+N})$,
 and (d) $(\Mtilde_\delta)^{Q}$
 was defined  in  Definition \ref{dfn:local_Mtilde_delta}.  The statement in part \ref{item:above_theta} is formalized similarly, except that
there, the $2k$ quantifiers
corresponding to Woodin cardinals are unbounded over $N|\lambda^N$,
so it is of higher complexity.
(Of course, one would more naively
formalize these statements with the clause ``$\sforces{\CC_{\mathrm{tail}}}{}\Mtilde_\lambda\sats\varphi(x,s)$''
using an unbounded existential quantifier over $N|\lambda$,
but by Lemma \ref{lem:strong_Sigma_1_forcing_definability_and_+1_Woodin_locality},
we can equivalently restrict it in advance in the manner just mentioned.)
\end{rem}
\begin{proof}
Assume $x=\emptyset$ for simplicity; the other case is just an easy relativization thereof.
So $d=-1$ and $\CC_d$ is the trivial forcing.
We  will first prove \ref{item:deltavec} $\Rightarrow$ \ref{item:epsvec}.

Suppose \ref{item:deltavec} holds,
which, as described in Remark \ref{rem:formalize_mu-forcing_statements}, means that
 \begin{equation}\label{eqn:assumption_delta}(N|\delta^{+N})
\sats\all^{\mathrm{gen}}_{\vec{\delta}}s\ \sforces{\CC_{\delta}}{}\Mtilde_{\delta}\sats\varphi(s),\end{equation}
where $\delta$, etc, are as described there. Letting $\varepsilon$ be defined from $\vec{\vareps}$ as $\delta$ is from $\vec{\delta}$, we must
see that
 \begin{equation}\label{eqn:goal_vareps} (N|\vareps^{+N})
\sats\all^{\mathrm{gen}}_{\vec{\vareps}}s\ \sforces{\CC_{\vareps}}{}\Mtilde_{\vareps}\sats\varphi(s).\end{equation}

Since
\[ i_{\Pg N}\rest(\Pg|\lambda^{\Pg}):\Pg|\lambda^{\Pg}\to N|\lambda^N \]
is $\rSigma_2$-elementary (as even if $\OR^{\Pg}=\lambda^{\Pg}$, we have $\rho_1^{\Pg|\lambda^{\Pg}}=\lambda^{\Pg}$),
we may assume $N=\Pg$.

Now suppose line (\ref{eqn:goal_vareps}) fails.
Then
easily $k>0$ and  we may 
 assume  $\delta_i=\delta_{i}^{N}$ for all $i<2k$.
Let $R=R^N_{\vec{\vareps}}$.
We have the iteration map $j:\Pg\to R$,
with
\[j(\delta_m)=j(\delta_m^{\Pg})=\delta_m^{R}=\vareps_m \]
for all $m<2k$, so $j(\deltavec)=\vec{\vareps}$;
similarly $j(\delta)=\vareps$.
Lifting line (\ref{eqn:assumption_delta}) with $j$
therefore gives
\[ (R|\vareps^{+R})\sats\all^{\mathrm{gen}}_{\vec{\vareps}}s\ \sforces{\CC_\vareps}{}\Mtilde_{\vareps}\sats\varphi(s). \]
But because we have  generics $g_i,h_i$ as mentioned in Remark \ref{rem:R_construction_basic_props} (corresponding to $\Pg,R$), it follows that
\[(\Pg|\vareps^{+\Pg})\models\all^{\mathrm{gen}}_{\vec{\vareps}}s\ \forces_{\CC_{\vareps}} \Mtilde_\vareps\models\varphi(s),\]
contradicting the choice of the counterexample.

If \ref{item:deltavec}
fails it is likewise. (In the more general case that $x\neq\emptyset$,
we can take $N$ to be $\delta_{d}^{+N}$-sound, where $d=\base(x)$,
and then we have an iteration map $j:N\to R$ with $\delta_d^{+N}<\crit(j)$, which therefore extends canonically to $j^+:N[g]\to R[g]$,
where $g$ is $(N,\CC_d)$-generic.
So $j^+(x_g)=x_g$, and the foregoing argument easily generalizes.)

We now show \ref{item:deltavec} $\iff$ \ref{item:above_theta}. Suppose for illustration
that $k=2$ and $x=\emptyset$, so again $d=-1$ and $\CC_d$ is trivial. We may therefore again assume that $N=\Pg$;
this is because the statement in part \ref{item:above_theta}
is expressible with integer quantifiers
over $\rSigma_2(\{\xi\})$,
where $\xi$ is any Woodin cardinal of $N$ such that $\xi>\theta$,
and hence sufficiently preserved between $\Pg$ and $N$. We may also assume $\delta_i=\delta_i^N$ for $i<2k$.

Suppose \ref{item:deltavec}
holds under these assumptions. Let $\vareps_0=\delta_{m_0}^N\in\Delta^N$, where $m_0<\om$.
Let $g_0$ be $(N,\CC_{\vareps_0})$-generic.
Let $s_0\in\Dd^{N[g_0]}$.
Let $\vareps_1\in\Delta^N$ with $\vareps_1>\vareps_0$.
Let $R_0=R_{\{\vareps_0,\vareps_1\}}$.
Let $h_0$ be $(R_0,\CC_{\vareps_0}^{R_0})$-generic and such that $\HC^{R_0[h_0]}=\HC^{N[g_0]}$. So $s_0\in R_0[h_0]$.
We have $i_{NR_0}:N\to R_0$ with $i_{NR_0}(\delta_0)=\vareps_0$ and $i_{NR_0}(\delta_1)=\vareps_1$. Let $h_1$ be $(R_0,\CC_{\vareps_1}^{R_0})$-generic, extending $h_0$, and $g_1$ be $(N,\CC_{\vareps_1}^N)$-generic, extending $g_0$, with
$\HC^{R_0[h_1]}=\HC^{N[g_1]}$.
Let $t_0\in\Dd^{R_0[h_1]}$
witness the existential statement at $\vareps_1$ in $R_0[h_1]$, with respect to $s_0$. So also $t_0\in N[g_1]$. Now let $\vareps_2\in\Delta^N$ with $\vareps_2>\vareps_1$. Let $g_2$ be $(N,\CC_{\vareps_2})$-generic, extending $g_1$,
and let $s_1\in\Dd^{N[g_2]}$. Let $\vareps_3\in\Delta^N$ with $\vareps_3>\vareps_2$. We can now continue much as in the first round, further iterating $R_0$ to $R_1=R_{\{\vareps_0,\vareps_1,\vareps_2,\vareps_3\}}$, noting that the tree from $R_0$ to $R_1$ is above $\vareps_1^{+R_0}$,
hence can be extended to $R_0[h_1]$,
and in particular to the parameters produced so far, including $t_0$. The remaining
details are similar to those for 
the equivalence of \ref{item:deltavec}  with \ref{item:epsvec}.

If instead \ref{item:deltavec}
fails; that is, ($N$ satisfies)
\[ \exists^{\mathrm{gen}}_{\vec{\delta}} s\neg\forces_{\CC_{\tail}}\Mtilde_\lambda\sats\varphi(s), \]
then by homogeneity of $\CC_{\tail}$,
\[ \exists^{\mathrm{gen}}_{\vec{\delta}}s\forces_{\CC_{\tail}}\Mtilde_\lambda\sats\neg\varphi(s), \]
and then a very similar calculation
(with quantifiers
inverted)
shows
\[ \exists^{\mathrm{gen}}_{\geq\theta;k}s\ \forces_{\CC_{\tail}}\ \Mtilde_\lambda\sats\neg\varphi(s),\]
which clearly implies
\[ \neg\all^{\mathrm{gen}}_{\geq\theta;k}s\ \forces_{\CC_{\tail}}\ \Mtilde_\lambda\sats\varphi(s).\qedhere\]
\end{proof}

\begin{lem}\label{lem:N|lambda_compute_mu}
 Let $N$ be a relevant generic $\Sigma_{\Pg}$-iterate. Let $G$ be $(N,\CC^N)$-generic. Let $\varphi$ be $\mSigma_1$ and $\tau\in\Nm_\lambda^N$. Let $k<\om$. Let $d=\base(\tau)$ and $\vec{\delta}\in[\Delta^N\cut(d+1)]^{2k}$. Then
\[(\M_\lambda)_G\models\all^*_k s\ [\varphi(\tau_G,s)]\ \ 
\iff\ \  N[G\rest(d+1)]\models\all^{\mathrm{gen}}_{\vec{\delta}}s\ \forces_{\CC_{\mathrm{tail}}}\Mtilde_\lambda\models\varphi(\tau,s).\]
Therefore if $N$ is an $\RR$-genericity
iterate of $\Pg$, as witnessed by $G$, 
and sufficient Turing determinacy holds in $V$,
then
\[ \M_{\om_1}\models\exists^*_ks\ [\varphi(\tau_G,s)]\ \ \iff\ \  N[G\rest(d+1)]\models\all^{\mathrm{gen}}_{\vec{\delta}}s\ \forces_{\CC_{\mathrm{tail}}}\Mtilde_\lambda\models\varphi(\tau,s).\]
\end{lem}
\begin{proof}
For simplicity assume $\tau=\emptyset$, so $d=0$.
Suppose \[N\models\all^{\mathrm{gen}}_{\vec{\delta}}s\ \forces_{\CC_{\mathrm{tail}}}\Mtilde_\lambda\models\varphi(s).\]
By Lemma \ref{lem:gen-all_independence},
 then  
 \[ N\models\all^{\mathrm{gen}}_{\geq 0;k}s\ \forces_{\CC_{\mathrm{tail}}}\Mtilde_\lambda\models\varphi(s).\]
But then since every element of $(\Mtilde_\lambda)_G$ appears in $N[g]$ for some proper segment $g$ of $G$, it easily follows that $(\Mtilde_\lambda)_G\models\all^*_ks\ [\varphi(s)]$.

Conversely, suppose
\[N\models\exists^{\mathrm{gen}}_{\vec{\delta}}s\ \neg\forces_{\CC_{\mathrm{tail}}}\Mtilde_\lambda\models\varphi(s).\]
Then by homogeneity of the forcing $\CC_{\mathrm{tail}}$,
\[N\models\exists^{\mathrm{gen}}_{\vec{\delta}}s\ \forces_{\CC_{\mathrm{tail}}}\Mtilde_\lambda\models\neg\varphi(s),\]
which gives
$(\Mtilde_\lambda)_G\models\exists^*_ks\ [\neg\varphi(s)]$
much as in the previous case,
and hence
$(\Mtilde_\lambda)_G\models\neg\all^*_ks\ \varphi(s)$.
\end{proof}

\begin{dfn}
 For $k\in[1,\om)$,
 the \emph{$\mu_k\Sigma_1^{\Mtilde_\lambda}$ forcing relation $\forces_{\lambda\mu_k 1}$} \tu{(}of a relevant generic $\Sigma_{\Pg}$-iterate $N$\tu{)} is the standard forcing relation for
 $\mu_k\Sigma_1$ formulas, interpreted over $\Mtilde_\lambda$. 
Likewise for \emph{$\muSigma_1$}, and  other such pointclasses. Let $k<\om$. The \emph{strong $\mu_k\Sigma_1^{\Mtilde_\lambda}$ forcing
relation $\forces^{\mathrm{s}}_{\lambda\mu_k1}$ \tu{(}of $N$\tu{)}} is the relation
where given $p\in\CC^N$ and $\vec{\tau}\in(\Nm_\lambda)^{<\om}$
and $d=\base(\vec{\tau})$
and an $\mSigma_1$ formula 
\[\varphi(s,\vec{x})=\exists w\ [\psi(w,s,\vec{x})] \]
where $\psi$ is $\mSigma_0$ (in the language of $\M_{\om_1}$), then letting
$i$ be such that
$\delta^N_{i}=\min(\Delta^N\cut(d+1))$ and $\vec{\delta}=(\delta_{i}^N,\ldots,\delta_{i+2k-1}^N)$
and $\vareps=\delta_{i+2k}^N$, we have
\[ \begin{array}{rl}&p\forces^{\mathrm{s}}_{\lambda\mu_k1}\all^*_k s\ \big[\Mtilde_\lambda\sats\varphi(s,\vec{\tau})\big]\\\iff&  (p\rest\CC_d)\forces_{\CC_d}\all^{\mathrm{gen}}_{\vec{\delta}} s\ \forces_{\CC_{\vareps}}\exists w\in\HC\Big[\forces_{\CC_{\tail}}\Mtilde_\lambda\sats\psi(\check{w},\check{s},\vec{\tau})\Big].\end{array}
\]
Note only $p\rest\CC_d$ is relevant; the rest of $p$ is ignored.
The  \emph{strong $\muSigma_1^{\Mtilde_\lambda}$ forcing
relation $\forces^{\mathrm{s}}_{\lambda\mu1}$ \tu{(}of $N$\tu{)}} is the relation
where given $p\in\CC^N$ and $\vec{\tau}\in(\Nm_\lambda)^{<\om}$
and an $\mSigma_1$ formula $\varphi$,
\[ \begin{array}{rl}&p\forces^{\mathrm{s}}_{\lambda\mu1}\all^*s\ \big[\Mtilde_\lambda\sats\varphi(s,\vec{\tau})\big]\\ \iff& \exists k<\om\ \Big[p\forces^{\mathrm{s}}_{\lambda\mu_k1}\all^*_ks\ \big[\Mtilde_\lambda\sats\varphi(s,\vec{\tau})\big]\Big].\end{array}\]

For pointclasses of the form $\mSigma_{n+2}$,
the \emph{$\mSigma_{n+2}^{\Mtilde_\lambda}$ strong forcing relation $\forces^{\mathrm{s}}_{\lambda,n+2}$
\tu{(}of $N$\tu{)}} is the relation
where for $\mSigma_{n+2}$ formulas
\[ \varphi(\vec{x})\iff\exists r,t\ [T_{n+1}(r,t)\wedge\psi(\vec{x},r,t)] \]
with $\psi$ being $\Sigma_0$, and for $\vec{\tau}\in(\Nm_\lambda)^{<\om}$ and $p\in\CC^N$ we have
\[ p\forces^{\mathrm{s}}_{\lambda,n+2}\varphi(\vec{\tau})\iff\exists \dot{r},\dot{t}\in\Nm_\lambda\ \Big[p\forces T_{n+1}(\dot{r},\dot{t})\wedge\psi(\vec{\tau},\dot{r},\dot{t})\Big].\]
Recall the strong $\mSigma_1$ forcing relation was introduced in Definition \ref{dfn:strong_mSigma_1_forcing}.

Let $k<\om$. The \emph{strong $\mu_k\Sigma_2^{\Mtilde_\lambda}$ forcing
relation $\forces^{\mathrm{s}}_{\lambda\mu_k2}$ \tu{(}of $N$\tu{)}} is the relation
where given $p\in\CC^N$ and $\vec{\tau}\in(\Nm_\lambda)^{<\om}$
and $d=\base(\vec{\tau})$
and an $\mSigma_2$ formula 
\[\varphi(s,\vec{x})=\exists r,t\ [T_1(r,t)\wedge \psi(r,t,s,\vec{x})] \]
where $\psi$ is $\mSigma_0$ (in the language of $\M_{\om_1}$), then letting $i$ be such that
$\delta^N_{i}=\min(\Delta^N\cut(d+1))$ and $\vec{\delta}=(\delta_{i}^N,\ldots,\delta_{i+2k-1}^N)$
and $\vareps=\delta_{i+2k}^N$, we have
\[ \begin{array}{rl}&p\forces^{\mathrm{s}}_{\lambda\mu_k2}\all^*_k s\ \big[\Mtilde_\lambda\sats\varphi(s,\vec{\tau})\big]\\\iff&  (p\rest\CC_d)\forces_{\CC_d}\all^{\mathrm{gen}}_{\vec{\delta}} s\ \forces_{\CC_{\vareps}}\exists r,t\in\HC\Big[\forces_{\CC_{\tail}}\Mtilde_\lambda\sats T_1(\check{r},\check{t})\wedge\psi(\check{r},\check{t},\check{s},\vec{\tau})\Big].\end{array}
\]
The  \emph{strong $\muSigma_2^{\Mtilde_\lambda}$ forcing
relation $\forces^{\mathrm{s}}_{\lambda\mu2}$ \tu{(}of $N$\tu{)}} is derived from the strong $\mu_k\Sigma_2$ forcing relations like for $\muSigma_1^{\Mtilde_\lambda}$.
\end{dfn}

\begin{rem}\label{rem:example_graduated_Qs}
 Suppose $\OR^{\Pg}=\lambda+\lambda$,
 where $\lambda=\lambda^{\Pg}$.
 Let $N$ be an $\RR$-genericity iterate
 of $\Pg$, as witnessed by $G$.
 We will have $(\M^N)_G=\M_{\beta^*}$,
 which therefore has height $\om_1+\om_1$.
 Consider the statement $\varphi(s)$,
 which says ``$\om_1+\alpha$ exists for every ordinal $\alpha<\om_1$ coded 
 by a real in
 some degree $\leq s$''.
 Then $\M_{\beta^*}\sats\all^*_1s\ \varphi(s)$, but $\M_\alpha\sats\neg\all^*_1s\ \varphi(s)$
 for all $\alpha<\beta^*$.
 And $N$ satisfies the statement $\psi$ expressing that this is forced
 of its $\M^N$ (expressed using $\all^{\gen}$), and $N$ has no proper segment of height $\geq\lambda$ satisfying $\psi$.
 This formula $\psi$ asserts that for 
all sufficiently large pairs $
 \{\vareps_0<\vareps_1\}\in[\Delta^N]^2$,
 \begin{equation}\label{eqn:psi_d0d1}\all^{\mathrm{gen}}_{\{\vareps_0,\vareps_1\}}s\ [\M^N\sats\varphi(s)].\end{equation}
Letting $\psi_{\vareps_0,\vareps_1}$
be the statement in (\ref{eqn:psi_d0d1})
(in parameters $\vareps_0,\vareps_1$),
note that
there \emph{is} a proper segment of $N$
 which satisfies $\psi_{\vareps_0,\vareps_1}$,
 namely, $N|(\lambda^N+\vareps_0^{+N})$
 (the generic reals at $\vareps_1$ aren't really relevant here). We will see that this picture
 is a prototype for the general case of least segments of $\M_{\beta^*}$ satisfying
 some formula of form $\all^*s\varrho(s)$.
\end{rem}
\begin{lem}\label{lem:lambda_witness_appears_immediately}\footnote{***This lemma
basically appears as the Claim in proof of 6.44.}
 Let  $\vec{\tau}\in(\Nm_\lambda)^{<\om}$  and $d\in\{0\}\cup\Delta^N$ with $d\geq\base(\vec{\tau})$. Let $k<\om$ and $\vec{\delta}\in[\Delta_{>d}]^{2k}$ and $\eps\in\Delta$ with $\eps>\max(d,\max(\vec{\delta}))$.
 Let $\psi$ be $\mSigma_0$. Then
 \[ \begin{array}{rl} N\sats\ \forces_{\CC_d}\text{``}&\all^{\mathrm{gen}}_{\vec{\delta}}s\ \Big[\forces_{\CC_\tail}\M_\lambda\sats\exists w\ [\psi(w,\vec{\tau},s)]\Big]\\
&\implies
 \all^{\mathrm{gen}}_{\vec{\delta}}s\ \forces_{\CC_{\vareps}}\exists w\in\HC\Big[\forces_{\CC_\tail}\M_\lambda\sats\psi(\check{w},\vec{\tau},s)]\Big]\text{''}.\end{array}\] 
\end{lem}
\begin{proof}
 For simplicity assume $\vec{\tau}=\emptyset$ and $d=0$ (otherwise relativize
 everything above $N|d$).
 Suppose
 \[ N\sats\all^{\mathrm{gen}}_{\vec{\delta}}s\ \Big[\forces_{\CC_\tail}\Mtilde_\lambda\sats\exists w\ [\psi(w,s)]\Big].\]
 Let $g$ be $(N,\CC_{\max(\vec{\delta})})$-generic.
 Let $s\in(\Dd^k)^{N[g]}$
 be such that
 \[ \forces_{\CC_\tail}\Mtilde_\lambda\sats\exists w\ [\psi(w,s)].\]
 We claim that 
 \begin{equation}\label{eqn:N[g]_sats_eps_suffices} N[g]\sats\ \forces_{\CC_\vareps}\exists w\in\HC\ \Big[\forces_{\CC_\tail}\Mtilde_\lambda\sats\psi(\check{w},s)\Big] \end{equation}
 (which clearly suffices; so we are actually  proving a stronger fact than advertised by the lemma).
 
Toward this claim, let $\delta\in\Delta$ with $\delta>\vareps$
 and $p\in\CC_\delta$ be such that
 \[ p\forces_{\CC_\delta}\exists w\in\HC\ \Big[\forces_{\CC_\tail}\Mtilde_\lambda\sats\psi(\check{w},\check{s})\Big].\]
 Then first note that we may assume $p=\emptyset$, by homogeneity of $\CC_\delta$ and $\CC_{\tail}$, and since $s\in N[g]$
 and $\Mtilde_\lambda$ is symmetric.
 
 Now let $\vec{\vareps}=\eps\cap\Delta$
 and let $R=R^N_{\vec{\vareps}\cup\{\delta\}}$.
 So $\crit(i_{NR})>\max(\vec{\delta})$,
 so $i_{NR}$ extends to $i_{NR}^+:N[g]\to R[g]$.
We have $i_{NR}^+(\vareps)=\delta$, and as in the proof of Lemma \ref{lem:gen-all_independence},
 it follows that line (\ref{eqn:N[g]_sats_eps_suffices}) holds, as desired.
\end{proof}

In the following lemma,
note $\rho_1^{N|\lambda}=\lambda$, since $N|\lambda$ has no largest cardinal
and by condensation.
\begin{lem}\label{lem:strong_muSigma_1_lambda}
We have:
\begin{enumerate}
 \item\label{item:strong_mu_k_Sigma_1} The strong $\mu_k\Sigma_1^{\Mtilde_\lambda}$
 forcing relation of $N$ is  $\rSigma_2^{N|\lambda}$, uniformly in $k<\om$.
\item The strong $\muSigma_1^{\Mtilde_\lambda}$ forcing relation of $N$ is  $\rSigma_2^{N|\lambda}$.
\item\label{item:strong_mu_Sigma_1-forcing_theorem} The strong $\mu_k\Sigma_1^{\Mtilde_\lambda}$ and strong $\muSigma_1^{\Mtilde_\lambda}$
forcing theorems hold.
Moreover, letting $G$ be $(N,\CC^N)$-generic,
 $\vec{\tau}\in\Nm_\lambda$
$d=\base(\vec{\tau})$, and $\varphi$ be $\mSigma_1$, if $(\Mtilde_\lambda)_G\sats\all^*_k s\ \varphi(s,\vec{\tau}_G)$
then there is $p\in G\cap\CC_d$
such that \[ p\forces^{\mathrm{s}}_{\lambda\mu_k1}\all^*_ks\ \big[\Mtilde_\lambda\sats\varphi(s,\vec{\tau})\big].\]
\end{enumerate}
\end{lem}

\begin{proof}
Part \ref{item:strong_mu_k_Sigma_1}: Let $p\in\CC^N$ and $\vec{\tau}\in(\Nm_\lambda)^{<\om}$ and $\varphi$ be $\mSigma_1$, and write
\[ \varphi(s,\vec{x})\iff\exists w\ \psi(w,s,\vec{x}) \]
where $\psi$ is $\mSigma_0$
and $\lh(s)=k$.
Then note that by Lemma \ref{lem:lambda_witness_appears_immediately}, \[ N\sats\text{``}p\forces^{\mathrm{s}}_{\lambda\mu_k1}\Mtilde_\lambda\sats\all^*_k s\ \varphi(s,\vec{\tau})\text{''} \]
 iff, letting $d=\base(\vec{\tau})$, there is $\xi<\lambda$ such that letting $t=\Th_1^{N|\lambda}(\xi)$,
 there are  $\vareps_i,\vareps_+<\xi$
 for $i\leq 2k$ 
 such that, according to $t$,
each $\vareps_i$ is  Woodin, and letting $\vareps=\vareps_{2k}$, we have $d<\vareps_0<\ldots<\vareps_{2k-1}<\vareps<\vareps_+=\vareps^{+}$ and
\[p\rest\CC_d\forces_{\CC_d}\ \all^{\mathrm{gen}}_{\vec{\vareps}}s\ \forces_{\CC_\vareps}\exists w\in\HC\ [\psi'(\check{w},\check{s},\vec{\tau})], \]
  where $\vec{\vareps}=\{\eps_0,\ldots,\eps_{2k-1}\}$,
 and $\psi'(x,y,\vec{z})$ asserts
 that $\psi(x,y,\vec{z})$ holds when
 we interpret $\dot{T}$ (in the $\M_{\om_1}$ language) with the set of pairs $(r,u)$
 in the transitive closure of $\{w,s,\vec{\tau}\}$ such that $r\in\RR$ and $\varphi\in u$
 iff thesore is a $\varphi(r)$-witness
 which is a segment of the premouse given by
 translating $N|\vareps^{+N}$ to a premouse over $(N|\vareps,h)$,
 where $h$ is the generic through $\CC_{\vareps}$. Note  also that these are all either $\Sigma_1$ or $\Pi_1$ assertions about ordinals $<\xi$,
 hence determined by $t$.
 
 The remaining parts are straightforward consequences of part \ref{item:strong_mu_k_Sigma_1} and previous lemmas.
\end{proof}

\begin{rem}
 It seems that the (non-strong) $\muSigma_1^{\Mtilde_\lambda}$ forcing
 relation need not be $\rSigma_2^{N|\lambda}$,
 since the forcing might split into
 an infinite maximal antichain $\left<p_k\right>_{k<\om}$ below $p$,
 with each $p_k$ forcing $(\all^*_ks\ \varphi(s))\wedge\all j<k[\neg\all^*_js\ \varphi(s)]$.
\end{rem}

\begin{lem}\label{lem:name_for_1-theory}
 Suppose $\rho_2^{N|\lambda}=\lambda$.
 Let $\tau\in\Nm_\lambda$ and $d=\base(\tau)$.
 Let $t\in\Nm_\lambda$ be defined by
 
 \[ t=\{(p,(\varphi,\tau))\bigm|p\in\CC_d\text{ and } \varphi\text{ is }\muSigma_1\text{ and } p\forces^{\mathrm{s}}_{\lambda\mu1}\varphi(\tau)\}.\]
 Then:
 \begin{enumerate}
  \item\label{item:t_is_CC_d-name}  $t\in\Nm_\lambda$ is a $\CC_d$-name, and
  \item\label{item:theory_name}$N\sats\ \forces_{\CC} t=\Th_{\muSigma_1}^{\Mtilde_\lambda}(\{\tau\})$.
  \item\label{item:t_uniformly_computed_from_Th} $t$ is \tu{(}simply\tu{)} computed
  from $\Th_2^{N|\lambda}(d)$,
  uniformly in $\tau,d$.
  \end{enumerate}
\end{lem}
\begin{proof}
 Part \ref{item:t_uniformly_computed_from_Th} is 
  because the strong $\muSigma_1$ forcing relation
 is $\rSigma_2^{N|\lambda}$
 (by Lemma \ref{lem:strong_muSigma_1_lambda}).  Since $\rho_2^{N|\lambda}=\lambda$, therefore $t\in N|\lambda$,
 so note $t$ is a $\CC_d$-name, so $t\in\Nm_\lambda$,
 giving part \ref{item:t_is_CC_d-name}. Part \ref{item:theory_name} is by the version of the strong $\muSigma_1$ forcing theorem given in Lemma \ref{lem:strong_muSigma_1_lambda} part \ref{item:strong_mu_Sigma_1-forcing_theorem}.\end{proof}

As a corollary we easily get:
\begin{lem}
 Suppose $\rho_2^{N|\lambda}=\lambda$.
 Then the strong $\mSigma_2$ forcing relation $\forces^{\mathrm{s}}_{\lambda 2}$ is $\rSigma_3^{N|\lambda}$, and the strong $\mSigma_2$ forcing theorem holds.
\end{lem}

In the following lemma, recall that
if $\varrho$ is an $\mSigma_0$ formula then
\[\varphi(s,x)\iff\text{``}\exists r,t\ [T_1(r,t)\wedge\varrho(r,t,s,x)]\text{''}\] is $\mSigma_2$.
\begin{lem}\label{lem:localize_forcing_varphi_mSigma_2} Suppose $\rho_2^{N|\lambda}=\lambda$.
Let $\tau\in\Nm_\lambda$
and
$d=\base(\tau)$.
Let $\varrho$ be $\mSigma_0$. Let $0<k<\om$.
Let $\vec{s}$ be a $\CC^N_{2k-2}$-name,
$s_{2k-1}$ be a $\CC^N_{2k-1}$-name, and
$s$ the name for $\vec{s}\conc(s_{2k-1})$.
Let $m\in[2k-1,\om)$,
 $p\in\CC^N_{m}$
and $\sigma,t\in\Nm_\lambda$ with $\base(\sigma),\base(t)\leq m$ and \[  N|\lambda\sats p\forces_{\CC^N_m}\Big[s\in\Dd^k\wedge\forces_{\CC_{\mathrm{tail}}}\Mtilde_\lambda^N\sats T_1(\sigma,t)\wedge\varrho(\sigma,t,s,\tau)\Big].\]
Let $\bar{p}=p\rest\CC^N_{2k-2}$ and
 $\QQ=\Coll(\om,\delta_{2k-1}^N)$. Then
\[ \begin{array}{rcl}N|\lambda&\sats&\bar{p}\forces_{\CC_{2k-2}^N}\exists s',\tau'\ \Big[\emptyset\forces_{\QQ}\\&&(\vec{s},s')\in\Dd^k\wedge\forces_{\CC_{\mathrm{tail}}}\Mtilde^{N}_\lambda\sats T_1(\sigma',t')\wedge\varrho(\sigma',t',(\vec{s},s'),\tau)\Big].\end{array}\]
\end{lem}
\begin{proof} This is
like the proof of  Lemma \ref{lem:lambda_witness_appears_immediately},
but using Lemma \ref{lem:name_for_1-theory}
and that  iteration maps $i_{NR}:N\to R$
fix $\lambda$ and satisfy
$i_{NR}``T_2^{N|\lambda}\sub T_2^{R|\lambda}$.
(if $N=N|\lambda$,
we take the iteration maps
to be formed using degree $2$ ultrapowers),
and also using the slight adaptation of Lemma \ref{lem:name_for_1-theory}
to generic extensions of $N$
(and iterates $R$ thereof)
of form $N[g]$,
where $g$ is $(N,\Coll(\om,\delta))$-generic for some $\delta<\lambda$.
\end{proof}

\begin{cor}\label{cor:lambda_witness_appears_immediately_2}
 \label{lem:lambda_witness_appears_immediately_2}
Suppose $\rho_2^{N|\lambda}=\lambda$.
 Let  $\vec{\tau}\in(\Nm_\lambda)^{<\om}$  and $d\in\{0\}\cup\Delta^N$ with $d\geq\base(\vec{\tau})$. Let $k<\om$ and $\vec{\delta}\in[\Delta_{>d}]^{2k}$.
 Let $\psi$ be $\mSigma_0$. Then
  $N$ satisfies
 \[ \begin{array}{rl}  \forces_{\CC_d}\text{``}&\all^{\mathrm{gen}}_{\vec{\delta}}s\ \Big[\forces_{\CC_\tail}\M_\lambda\sats\exists r,t\ [T_1(r,t)\wedge\psi(r,t,\vec{\tau},s)]\Big]\\
&\implies
 \all^{\mathrm{gen}}_{\vec{\delta}}s\ \exists r,t \in\HC\Big[\forces_{\CC_\tail}\M_\lambda\sats T_1(\check{r},\check{t})\wedge\psi(\check{r},\check{t},\vec{\tau},s)]\Big]\text{''}.\end{array}\] 
\end{cor}

\begin{lem}\label{lem:gen-all_independence_2}
Suppose $\rho_2^{N|\lambda}=\lambda$.
Let $\tau\in\Nm_\lambda^N$
and $d=\base(\tau)$.
Let $\varphi$ be $\mSigma_{2}$.
Then for all $k<\om$ and all $\vec{\delta},\vec{\varepsilon}\in[\Delta_{>d}]^{2k}$,
and all $\theta\in(d,\lambda)$, $N|\lambda$ satisfies that $\CC_d$ forces that the following three statements are equivalent:
\begin{enumerate}[label=\tu{(}\roman*\tu{)}]
\item\label{item:deltavec_2} $\all^{\mathrm{gen}}_{\vec{\delta}}s\ \forces_{\CC_{\mathrm{tail}}}\Mtilde_\lambda\sats\varphi(x,s)$,
\item\label{item:epsvec_2}
$\all^{\mathrm{gen}}_{\vec{\vareps}}s\ \forces_{\CC_{\mathrm{tail}}}\Mtilde_\lambda\sats\varphi(x,s)$,
\item\label{item:above_theta_2}
$\all^{\mathrm{gen}}_{\geq\theta;k}s\ \forces_{\CC_{\mathrm{tail}}}\Mtilde_\lambda\sats\varphi(x,s)$.\qedhere
\end{enumerate}
\end{lem}
\begin{proof}Like   Lemma \ref{lem:gen-all_independence},
making use of
 Lemma \ref{lem:localize_forcing_varphi_mSigma_2}.\end{proof}

\begin{lem}\label{lem:N|lambda_compute_mu_2}
Suppose $\rho_2^{N|\lambda}=\lambda$. Let $G$ be $(N,\CC^N)$-generic. Let $\varphi$ be $\mSigma_2$ and $\tau\in\Nm_\lambda$. Let $k<\om$. Let $d=\base(\tau)$ and $\vec{\delta}\in[\Delta^N_{>d}]^{2k}$. Then
\[(\M_\lambda)_G\models\all^*_k s\ \varphi(\tau_G,s)\ \ 
\iff\ \  N[G\rest (d+1)]\models\all^{\mathrm{gen}}_{\vec{\delta}}s\ \forces_{\CC_{\mathrm{tail}}}\M_\lambda\models\varphi(\tau,s).\]
Therefore if $N$ is an $\RR$-genericity
iterate of $\Pg$, as witnessed by $G$, 
and sufficient Turing determinacy holds in $V$,
then
\[ \M_{\om_1}\models\exists^*_ks\ \varphi(\tau_G,s)\ \ \iff\ \  N[G\rest (d+1)]\models\all^{\mathrm{gen}}_{\vec{\delta}}s\ \forces_{\CC_{\mathrm{tail}}}\M_\lambda\models\varphi(\tau,s).\]
\end{lem}
\begin{proof}By the obvious adaptation of the proof of Lemma \ref{lem:N|lambda_compute_mu}.\end{proof}

\begin{lem}\label{lem:strong_muSigma_1_lambda_2}
Suppose $\rho_2^{N|\lambda}=\lambda$.
Then:
\begin{enumerate}
 \item\label{item:strong_mu_k_Sigma_1_2} The strong $\mu_k\Sigma_2^{\Mtilde_\lambda}$
 forcing relation of $N$ is  $\rSigma_3^{N|\lambda}$, uniformly in $k<\om$.
\item The strong $\muSigma_2^{\Mtilde_\lambda}$ forcing relation of $N$ is  $\rSigma_3^{N|\lambda}$.
\item The strong $\mu_k\Sigma_2^{\Mtilde_\lambda}$ and strong $\muSigma_2^{\Mtilde_\lambda}$
forcing theorems hold.
In particular, letting $G$ be $(N,\CC^N)$-generic,
 $\vec{\tau}\in\Nm_\lambda$
and $d=\base(\vec{\tau})$, and $\varphi$ be $\mSigma_2$, if $(\Mtilde_\lambda)_G\sats\all^*_k s\ \varphi(s,\vec{\tau}_G)$
then there is $p\in G\cap\CC_d$
such that \[ p\forces^{\mathrm{s}}_{\lambda\mu_k2}\all^*_ks\ \big[\Mtilde_\lambda\sats\varphi(s,\vec{\tau})\big].\]
\end{enumerate}
\end{lem}
\begin{proof}Similar to Lemma \ref{lem:strong_muSigma_1_lambda},
using Lemmas \ref{lem:gen-all_independence_2},
\ref{lem:N|lambda_compute_mu_2} 
and \ref{lem:name_for_1-theory} and their proofs.
The witness $w$ from
Lemma \ref{lem:strong_muSigma_1_lambda} is replaced wtih $r,t$, and $\psi'(w,s,\vec{\tau})$
replaced with $T_1'(r,t)\wedge\psi'(r,t,s,\vec{\tau})$, and $T_1'$ asserts that $t$ is computed via the lemmas just mentioned, and their proofs.
\end{proof}

\begin{lem}\label{lem:name_for_2-theory}
 Suppose $\rho_3^{N|\lambda}=\lambda$.
 Let $\tau\in\Nm_\lambda$ and $d=\base(\tau)$.
 Let $t\in\Nm_\lambda$ be defined by
  \[ t=\{(p,(\varphi,\tau))\bigm|p\in\CC_d\text{ and } \varphi\text{ is }\muSigma_2\text{ and } p\forces^{\mathrm{s}}_{\lambda\mu2}\varphi(\tau)\}.\]
 Then
  $t\in\Nm_\lambda$ is a $\CC_d$-name, 
$N\sats\ \forces_{\CC} t=\Th_{\muSigma_2}^{\Mtilde_\lambda}(\{\tau\})$,
and $t$ is \tu{(}simply\tu{)} computed
  from $\Th_3^{N|\lambda}(d)$,
  uniformly in $\tau,d$.
\end{lem}
\begin{proof}Like the proof of Lemma \ref{lem:name_for_1-theory}.
\end{proof}

As a corollary we easily get:
\begin{lem}
 Suppose $\rho_3^{N|\lambda}=\lambda$.
 Then the strong $\mSigma_3$ forcing relation $\forces^{\mathrm{s}}_{\lambda 3}$ is $\rSigma_4^{N|\lambda}$, and the strong $\mSigma_3$ forcing theorem holds.
\end{lem}

\begin{rem}\label{rem:strong_mu_Sigma_n_forcing_at_lambda}
One can now easily generalize the loop of lemmas
through the entire definability hierarchy,
getting that strong $\muSigma_{n}^{\Mtilde_\lambda^N}$ forcing is $\rSigma_{n+1}^{N|\lambda}$, etc, assuming that $\rho_{n}^{N|\lambda}=\lambda$.
\end{rem}

The preceding lemmas suggest that an $\RR$-genericity iterate might compute the $\mu$-hierarchy level-by-level via consulting its extender sequence to define $\mu$
at the right pace (otherwise, if we proceeded naively, the computation of $\mu$ would be slowed down to the same  pace as  the usual $L(\RR)$ hierarchy has). We will execute this,
and generalize (a version of) the lemmas above throughout all proper segments of $N$, and then
up to the degree $n_0$ at $N$ itself.
As foreshadowed by Remark \ref{rem:example_graduated_Qs},
for arbitrary segments $Q$ strictly between $N|\lambda$ and $N$ (and for example when $N=\Pg$, the ordinal height of $Q$ could then be shifted by the relevant iteration maps $i_{NR}$), we will only
 be able to prove slightly weaker versions of some of the Lemmas (for example \ref{lem:gen-all_independence}
and \ref{lem:N|lambda_compute_mu}), which only gives eventual agreement (for large enough tuples of Woodins). (By Remark \ref{rem:example_graduated_Qs},
the stronger version which holds at $\lambda$ cannot hold for all segments in general.)

To assist the analysis, we begin by defining a  system
of names for elements of a symmetric submodel
$\M_G$
of $N[G]$. The intention
is that if $N$ is an $\RR$-genericity
iterate of $\Pg$, as witnessed by $G$,
then $\M_G=\M_{\beta^*}$. 

\begin{dfn}\label{dfn:Nm^N}
We define  an increasing hierarchy
$\left<\Nm_\xi\bigm|\xi\in[\lambda,\OR^N]\right>$;
the elements of $\Nm_\xi$ are the \emph{level $\xi$ construction names}. We also define the \emph{support}
$\supp(\tau)$ of each $\tau\in\Nm_\xi$,
with $\supp(\tau)\in\{-1\}\cup\om$,
and the 
\emph{location} $\loc(\tau)$
of each $\tau\in\Nm_\xi$,
with $\loc(\tau)\in[\xi]^{<\om}$.

For $\xi=\lambda$, it only remains to define
$\loc(\tau)$; we set $\loc(\tau)=\emptyset$.

Suppose $\xi\in[\lambda,\OR^N)$
is a limit ordinal and we have defined
$\Nm_\xi$ and
 $\supp\rest\Nm_\xi$ and $\loc\rest\Nm_\xi$.
The names in $\Nm_{\xi+\om}$  will represent
objects output by  $\mu^{N,G}$-rud functions
applied to $(\Mtilde_\xi)_G\cup\{(\Mtilde_\xi)_G\}$,
where $\mu^{N,G}$ is  Martin measure relativized to the Turing degrees in $\widetilde{\HC}^G$
(that is, with all quantifiers ranging
over these degrees), and $\Mtilde_\xi$
is a certain $\CC^N$-name (proper class in $N$ if $\xi=\OR^N$); we have already defined $\Mtilde_\lambda$. 

It might not be immediately clear that
all of the notions introduced below are  well-defined (in particular, $\tau_G$).
This will be clarified by the end of \S\ref{sec:through_gap}. Also see Remark \ref{rem:will_formalize_name_definitions}.
But it should be clear that the formal classes $\Nm_\xi$
and the functions $\supp$ and $\loc$ are well-defined.

Fix a recursive enumeration
$\left<f_i\right>_{i<\om}$ for (schemes for)
$\mu$-rud functions of arity $\geq 2$ (here ``$\mu$'' is just a symbol).
Let $a_i+2$ be the arity of $f_i$,
so $a_i\geq 0$.
In the construction name $(\eta,i,\pi)$ below, the $i$ indicates the function (scheme) $f_i$ to be applied, and $\eta,\pi$ determine
the inputs to $f_i$.
For limits $\eta\in[\lambda,\OR^N)$ define
\[ \Nm_{\eta+\om}=\Nm_\lambda\cup\Big\{(\eta,i,\pi)\Bigm| (i<\om)\wedge(\pi:a_i\to\Nm_\eta)\Big\}.\]
Define $\supp\rest\Nm_{\eta+\om}$ extending
$\supp\rest\Nm_{\eta}$ by setting 
\[ \supp(\eta,i,\pi)=\max(\rg(\supp\circ\pi)).\] 
Let $\loc\rest\Nm_{\eta+\om}$
extend $\loc\rest\Nm_\eta$,
where
for $\tau=(\eta,i,\pi)\in\Nm_{\eta+\om}\cut\Nm_\eta$,
\[ \loc(\eta,i,\pi)=\{\eta\}\cup\bigcup_{j<a_i}\loc(\pi(j)).\]

For limits $\xi\in(\lambda,\OR^N]$, we set
 $\Nm_\xi=\bigcup_{\gamma<\xi}\Nm_\gamma$
(which recursively determines $\supp\rest\Nm_\xi$ and $\loc\rest\Nm_\xi$).

 We next define
the interpretation $\tau_G$
of $\tau\in\Nm_\xi$ for $(N,\CC^N)$-generics $G$, recursively in $\xi$.
(Recall we have already defined $\tau_G$ for $\tau\in\Nm_\lambda$,
as the the conventional interpretation.)
Let $\xi\in[\lambda,\OR^N)$.
Define
\[\xi_G=\{\sigma_G\bigm|\sigma\in\Nm_{\xi}\} \]
(so $\lambda_G=\widetilde{\HC}_G$) and
\[ \lambda'_G=\Big\{\sigma_G\Bigm|\sigma\in\Nm_\lambda\wedge(\Mtilde_\lambda)_G\sats T(\sigma_G)\Big\}.\]
Let $\tau=(\xi,i,\pi)\in\Nm_{\xi+\om}$.
Then $f_i$ is a scheme for a $\mu$-rud function;
let $f_i^{\mu^{N,G}}$ be the resulting $\mu^{N,G}$-rud function, and
(temporarily for intuition,
to be formalized in Definition \ref{dfn:formal_Nm^N})
define
\[ \tau_G= f_i^{\mu^{N,G}}(\xi_G,\lambda'_G,\pi(0)_G,\ldots,\pi(a_i-1)_G).\]

For $\xi\in(\lambda,\OR^N]$, 
define $\widetilde{\M}_\xi$ to be the
natural $\CC^N$-name for the transitive structure
\[ \big(\{\sigma_G\bigm|\sigma\in\Nm_\xi\},(\M_\lambda)_G\big)\]
(a structure in the $\M(\RR)$-language).
(This is a conventional name, not a
construction name. If $\xi=\OR^N$ then this name is a proper class of $N$.)

Given a strong cutpoint $\gamma<\lambda$
of $N$, given
 $g$ which is $(N,\Coll(\om,\gamma))$-generic, and given $\xi\in[\lambda,\OR^N]$,
 define $\Nm_\xi^{N[g]}$ with respect to $N[g]$
just as $\Nm_\xi$ is defined over $N$.
\end{dfn}

\begin{rem}\label{rem:will_formalize_name_definitions}
Note that $\Nm_\xi$, $\xi\mapsto\Nm_\xi$,
and the functions $\supp$ and $\loc$
are well-defined, and independent of
the interpretation $\tau_G$ and the names $\Mtilde_\xi$ etc introduced above.
But we are yet to see that we actually
have $\tau_G\in N[G]$, and therefore yet to see
that $\Mtilde_\xi$ and other notions introduced above are really well-defined.
The reader will easily observe that
we do not make any formal use of $\tau_G$, $\Mtilde_\xi$, etc, until after we have formalized these notions; they just provide intuitive motivation for the formal notions to be introduced.

We will in fact show that if $\tau\in\Nm_\xi$
then $\tau_G\in(N|\xi)[G]$.
Also, for limits $\xi\in(\lambda,\OR^N]$, let $\forces_{\xi0}$ be the 
 $\Sigma_0^{\Mtilde_\xi}$ forcing relation
 (over names in $\Nm_\xi$, for truth
 over $\Mtilde_\xi$). 
 We will  show that $\forces_{\xi0}$
 is $\Delta_1^{N|\xi}(\{\lambda\})$,
 uniformly in such $\xi$, and
 in fact, for each limit $\gamma\in[\lambda,\xi)$ and $n<\om$,
the $\mSigma_n^{\Mtilde_\gamma}$-forcing relation $\forces_{\gamma n}$ (with respect to the relevant Turing degrees) is definable over $N|\gamma$,
uniformly in $(\gamma,n)$,
 and so by the forcing theorem for such formulas, we get $\tau_G\in(N|\xi)[G]$
 for $\tau\in\Nm_\xi$,
 and the evaluation map $\tau\mapsto\tau_G$ (with domain $\Nm_\xi$)
 is $\Delta_1^{(N|\xi)[G]}(\{N|\lambda,\widetilde{\HC}_G\})$.
We prove the definability of
$\forces_{\gamma n}$ and $\forces_{\xi0}$
inductively in $\xi$ (where again $\gamma<\xi$ and $n<\om$). For $\xi$ a limit of limits, it follows immediately by induction,
so suppose it holds at a limit $\xi\in[\lambda,\OR^N)$; we want to establish the
definability of $\forces_{\xi n}$ for each $n<\om$, and through the next lemma, hence the $\Delta_1^{N|(\xi+\om)}(\{\lambda\})$-definability of $\forces_{\xi+\om,0}$.
By inductive hypothesis, $\Mtilde_\xi$ is well-defined, and $\tau_G\in(N|\xi)[G]$ for all $\tau\in\Nm_\xi$.

In the end, we will be able to replace
the talk of $\mu^{N,G}$ in the definition of $\tau_G$ above with
a formal definition
which we will end up showing computes
$\mu^{N,G}$ correctly
 over the relevant segments $(N|\xi)[G]$ of $N[G]$.
This will yield a well-defined (and level-by-level definable) $\Mtilde_\xi$ etc,
and we will then see (at least in the relevant circumtsances) that it yields the objects defined above.
\end{rem}

\begin{dfn}\label{dfn:names_for_ordinals_and_M}
Given $\alpha<\xi$,
let 
 $\sigma_\alpha\in\Nm_{\xi}^N$
 be the canonical name (in $\Nm_{\xi}^N$)
for $\alpha$. That is, if $\alpha<\lambda^N$
then $\sigma_\alpha=\check{\alpha}$ as usual.
If $\alpha\in[\lambda^N,\xi)\cap\Lim$
then $\sigma_\alpha=(\alpha,i_{\mathrm{Ord}},\emptyset)$
where $i_{\mathrm{Ord}}$ is the index for the natural $\mu$-recursive function scheme $f$ such that that $f(A,B)=A\cap\OR$
whenever $A$ is a rudimentarily closed transitive set.
If $\alpha=\beta+n+1$ where $\beta\in[\lambda^N,\xi)\cap\Lim$ and $n<\om$ then $\sigma_\alpha=(\beta,i_{\mathrm{Ord}+n+1},\emptyset)$ where $i_{\mathrm{Ord}+n+1}$ is chosen similarly.
 (We may just write ``$\alpha$''
 in forcing statements where formally it should be ``$\sigma_\alpha$''.)
 
 Similarly
 if $\alpha\in[\lambda^N,\xi)\cap\Lim$
let $m_\alpha\in\Nm_{\xi}$
be the canonical name for $(\Mtilde_\alpha)_G$.
(This definition is made formally, independently
of our earlier introduction of $\Mtilde_\alpha$.
That is, $m_\alpha=(\alpha,i_{\M},\emptyset)$
where $i_{\M}$ indexes the natural $\mu$-recursive function scheme $f$ such that $f(A,B)=(A,B)$ if $B\notin A$, and $f(A,B)=A$ otherwise.)
\end{dfn}

The following lemma is by standard fine structure:
\begin{lem}\label{lem:Sigma_0_to_Sigma_omega_algo}
There is  a recursive function
 $(\varphi,\vec{i})\mapsto\psi_{\varphi,\vec{i}}$ 
 sending pairs $(\varphi,\vec{i})$
 consisting of:
\begin{enumerate}[label=--]
  \item  an $\mSigma_0$ formula
 $\varphi=\varphi(\vec{x})$ in the $\M(\RR)$ language
with free variables $\vec{x}$,
\item a tuple $\vec{i}=(i_0,\ldots,i_{k-1})\in\om^k$, where $k=\lh(\vec{x})$
\tu{(}representing the tuple $\vec{f}=(f_{i_0},\ldots,f_{i_{k-1}})$ of $\mu$-recursive function schemata, and recall that $f_i$ has arity $a_i+2$\tu{)},
\end{enumerate}
to formulas $\psi_{\varphi,\vec{i}}$ in the $\M(\RR)$ language, such that for all limits $\gamma\geq\om_1$
such that Turing determinacy
holds in $\M_{\gamma+\om}$,
and all
\[ \vec{y}=\vec{y}_0\conc\ldots\conc\vec{y}_{k-1}\in(\M_\gamma)^{<\om}\]
with $\lh(\vec{y}_{j})=a_{i_j}$ for each $j<k$, writing
$(m,t)=(\M_\gamma,T^{\M_\lambda})$, we have
\[ \M_{\gamma+\om}\sats\varphi\big(f_{i_0}(m,t,\vec{y}_0),\ldots,
f_{i_{k-1}}(m,t,\vec{y}_{k-1})\big)\iff\M_\gamma\sats\psi_{\varphi,\vec{i}}(\vec{y}).\]

This recursive function also analogously reduces 
$\Sigma_0^{(\M_{\xi+\om})_G}$
to $\mSigma_\om^{(\M_\xi)_G}$, assuming
\begin{enumerate}[label=\tu{(}\roman*\tu{)}]
 \item\label{item:no_new_reals}
$\HC^{(\M_{\xi+\om})_G}=\widetilde{\HC}_G$, and
\item\label{item:Turing_det_holds} Turing determinacy holds
in $(\M_{\xi+\om})_G$.
\end{enumerate}
\end{lem}

\begin{dfn}\label{dfn:Sigma_0_to_Sigma_om_algo}
Let  $(\varphi,\vec{i})\mapsto\psi_{\varphi,\vec{i}}$ be the natural algorithm
witnessing Lemma \ref{lem:Sigma_0_to_Sigma_omega_algo}.

Let $I_{33}$ be the least integer
which indexes the $\mu$-recursive
function scheme $f$ such that $f(x,y,z)=z$.

For $i<\om$ let $\pad(i)$
be the natural $i'<\om$
such that $a_{i'}=a_i+1$
and
\[ f_{i'}(x,y,z_0,\ldots,z_{a_i})=f_i(z_{a_i},y,z_0,\ldots,z_{a_{i-1}}).\qedhere\]
\end{dfn}

Thus, it suffices
to show that the $\mSigma_n^{\Mtilde_\xi}$ forcing relation $\forces_{\xi n}$
is definable over $N|\xi$,
uniformly in $n$,
and that \ref{item:no_new_reals} and \ref{item:Turing_det_holds} hold.

The definability of $\forces_{\xi n}$
is verified by induction on complexity of $\mSigma_\om$ formulas. Here we do not (yet)
proceed precisely level-by-level comparing the $\mSigma_n$ and $\rSigma_n$
hierarchies (of $(\Mtilde_\xi)_G$
and $N|\xi$ respectively), because
we are not yet prepared to
show how the fine structure matches up 
between the two sides (this will come later). But we do proceed by induction on the complexity
of formulas of the $\M(\RR)$ language. We will in fact show that for various recursive classes $\Gamma$ of formulas,  of bounded complexity
(in the $\M(\RR)$ language), the forcing relation $\forces^{N[g]}_{\xi\Gamma}$, asserting
\[ p\forces_{\CC_{\mathrm{tail}}}\Mtilde_\xi\sats\varphi(\vec{\tau}),\]
for forcing over $N[g]$, for $p\in\CC_{\mathrm{tail}}$,  formulas $\varphi$ in $\Gamma$ and $\vec{\tau}\in(\Nm_{\xi}^{N[g]})^{<\om}$, is definable over $N[g]$, uniformly in $\gamma,g$, for $\gamma<\lambda$ and
 $g$ being $(N,\Coll(\om,\gamma))$-generic.  
All the induction steps excluding
the $\mu$-quantifier are standard, so we ignore these.
So suppose we have
appropriately
defined $\forces_{\xi \Gamma}$ for some class of formulas $\Gamma$;
so 
$\forces_{\xi\Gamma}$
is first order over $N|\xi$,
and  $\forces^{N[g]}_{\xi\Gamma}$ defined
over $N[g]$ in the same manner.
We will explain how to define $\forces_{\xi,\all^*_\mu\Gamma}$
in terms of this
(that is, for formulas $\psi(\vec{x})$ of
form $\all^*s\varphi(s,\vec{x})$,
where $\varphi(s,\vec{x})$
is in $\Gamma$).
Let $\varphi$ be in $\Gamma$. So we have already defined the relation
\[ p\forces_{\CC^N}\big[\Mtilde_\xi\sats\varphi(\vec{\tau},s)\big], \]
(where $p$ varies over $\CC^N$ and $s,\vec{\tau}$ over $(\Nm_\xi)^{<\om}$).
We show that the relation
\[ p\forces_{\CC^N}\big[\Mtilde_\xi\sats\all^* s\ \varphi(\vec{\tau},s)\big] \]
is also appropriately definable; likewise with respect to $N[g]$. The process will be uniform in $\varphi$, leading to the desired definition of $\forces_{\xi\all^*_\mu\Gamma}$.

\begin{rem}\label{rem:prelude_to_psi_0}
The formal definition of the 
$\mSigma_0$ forcing relation
is only given later in Definition \ref{dfn:psi_0}, through the formula $\psi_0$, and the proof of its correctness in Lemma \ref{lem:Sigma_P-iterate_good}.
Formally, one can skip to
Definition \ref{dfn:language_real+generic-pm}
at this point.
But the
following ``calculations''
are provide a sketch of key ideas which will motivate further the formulation of $\psi_0$, before spelling it out.
We give these ``calculations'' assuming that we have a definition that works up to a given point.
So where we write, for example, ``$\forces_{\CC}\Mtilde_\xi\sats\varphi(\vec{\tau})$'' below, where $\xi\in[\lambda^N,\OR^N]$ and  $\varphi$ is $\mSigma_0$, we have not yet really specified what this means, but we 
will later fill it in using the formula (in parameter $\lambda^N$) $\psi_0(\lambda^N,\cdot,\cdot,\cdot)$, introduced in Definition \ref{dfn:psi_0};
the formula $\psi_0$ is $\rSigma_1$,
and (in the right context, using parameter $\lambda^N$)
it will define an $\rDelta_1(\{\lambda^N\})$ relation,
and in fact $\rDelta_1^{N|\xi}(\{\lambda^N\})$ for each limit $\xi\in(\lambda^N,\OR^N]$.\end{rem}

The following lemma is the analogue of Lemma \ref{lem:gen-all_independence} and its variants,
but a key difference is that we now only get agreement above some lower bound $m$:

\begin{lem}\label{lem:xi_gen-all_independence}
\tu{(}We have $\xi\in[\lambda^N,\OR^N)\cap\Lim$ and the formula class $\Gamma$.\tu{)} Let
$\vec{\tau}\in(\Nm_\xi)^{<\om}$.
Then there is $m<\om$ such that $\base(\vec{\tau})\leq\delta_m^N$
and for all
$\varphi\in~\Gamma$,
all $k<\om$, all $\vec{\delta},\vec{\varepsilon}\in[\Delta^N_{\geq m}]^{2k}$ and
all $\theta\in\Delta_{\geq m}^N$, $N$ satisfies that $\CC_d$ forces the following three statements are equivalent:
\begin{enumerate}[label=\tu{(}\roman*\tu{)}]
\item\label{item:xi_deltavec} $\all^{\mathrm{gen}}_{\vec{\delta}}s\ \forces_{\CC_{\mathrm{tail}}}\Mtilde_\xi\sats\varphi(\vec{\tau},s)$,
\item\label{item:xi_epsvec}
$\all^{\mathrm{gen}}_{\vec{\vareps}}s\ \forces_{\CC_{\mathrm{tail}}}\Mtilde_\xi\sats\varphi(\vec{\tau},s)$,
\item\label{item:xi_above_theta}
$\all^{\mathrm{gen}}_{\geq\theta;k}s\ \forces_{\CC_{\mathrm{tail}}}\Mtilde_\xi\sats\varphi(\vec{\tau},s)$.
\end{enumerate}
\end{lem}
\begin{proof}
 Since $i_{\Pg N}:\Pg\to N$ is $\rSigma_1$-elementary,
 and considering Remark \ref{rem:prelude_to_psi_0}, we may assume that $N=\Pg$, which is fully sound.
 
 We may assume by induction that the lemma holds when we replace $\xi$ with a limit $\xi'\in[\lambda,\xi)$ (note that Lemma \ref{lem:gen-all_independence} and its variants already established it when $\xi$ is replaced with $\lambda$ itself).
 Suppose the lemma fails at $\xi$;
 recall $\xi<\OR^{\Pg}$.

 \ref{item:xi_deltavec} $\iff$ \ref{item:xi_epsvec} at $\xi$ (above some $m$):
Suppose otherwise. For simplicity assume that $\lh(\vec{\tau})=1$, and just write $\tau$ instead of $\vec{\tau}$. 
Let $\delta^\Pg_d=\base(\tau)$.
We will define a correct above-$\delta_d^{\Pg}$ normal tree $\Tt$
on $\Pg$, of form $\Tt=\Tt_0\conc\Tt_1\conc\ldots$, with each $\Tt_n$ based on an interval of finitely many Woodins, with a unique cofinal branch $b$,
and such that $M^\Tt_b$ is illfounded, a contradiction.

Since $m=d+1$ does not witness the equivalence,
we can find a counterexample consisting
a
formula $\varphi$,
some  $k<\om$,
and tuples $\vec{\delta},\vec{\varepsilon}\in[\Delta_{\geq m}^\Pg]^{2k}$, and note we may assume $\vec{\delta}=\{\delta_{m}^\Pg,\delta_{m+1}^\Pg,\ldots,\delta_{m+2k-1}^\Pg\}$, 
so $\delta_i\leq\vareps_i$ for each $i<2k$.
Let $\vec{\theta}=\{\delta_0^\Pg,\ldots,\delta_d^\Pg\}$.
 Let $R=R^\Pg_{\vec{\theta}\cup\vec{\vareps}}$ and $i_{\Pg R}:\Pg\to R$ the iteration map. Note  $\delta_d^{+\Pg}<\crit(i_{\Pg R})$ and
 $i_{\Pg R}(\delta_i)=\vareps_i$
 for all $i<2k$.
 
Since $\delta_d^{+\Pg}<\crit(i_{\Pg R})$, we have $i_{\Pg R}(\desc(\tau)_1)=\desc(\tau)_1$. 

\begin{clm}\label{clm:i(xi)>xi_or_etc}
Either
 $i_{\Pg R}(\xi)>\xi$
 or $i_{\Pg R}(\desc(\tau)_0)>_{\mathrm{lex}}\desc(\tau)_0$.
\end{clm}
\begin{proof}Otherwise,
note that $i_{\Pg R}(\xi)=\xi$
 and $i_{\Pg R}(\tau)=\tau\in(\Nm_\xi)^R$. Consider corresponding generic extensions $R[H]=\Pg[G]$ of
$R$ and $\Pg$ respectively,
constructed as  in the proof of Lemma \ref{lem:gen-all_independence}
and taken with
 $H\rest\delta_d^R=G\rest\delta_d^\Pg$.
This gives $(\Mtilde_\xi^R)_G=
(\Mtilde_\xi^\Pg)_H$ and $\tau_G=\tau_H$ (this uses that $\base(\tau)=\delta_d^\Pg$ and $R|\delta_d^{+R}=\Pg|\delta_d^{+\Pg}$ and $H\rest\delta_d^R=G\rest\delta_d^\Pg$). But
this contradicts the disagreement between $R$ and $\Pg$ at $i_{\Pg R}(\vec{\delta})=\vec{\varepsilon}$ (which holds by the choice of $\varphi,k,\vec{\delta},\vec{\varepsilon}$ and the elementarity of $i_{\Pg R}$).
\end{proof}

Now let $\Tt_0$ be the tree from $R_0=\Pg$ to $R_1=R$. By the contradictory hypothesis, we  can repeat the process but starting with $m$ such that $\delta_m^\Pg>\max(\vec{\varepsilon})$, producing
 a tree $\Tt_1'$ on $\Pg$,
 with last model $R'$. 
 Let $\Tt_1$ be the equivalent tree on $R_1=R$ (which exists because $\Pg,R_1$ are equivalent above $\max(\vec{\varepsilon})$).
 Proceed in this manner,
 defining $\Tt_n$ for all $n<\om$. Each $\Tt_n$ is based on an interval of only finitely many Woodins,
 and does not drop on its main branch, 
 and the interval for $\Tt_n$ below that for $\Tt_{n+1}$. The concatenation $\Tt=\Tt_0\conc\Tt_1\conc\ldots$ is a correct normal tree, and there is a unique $\Tt$-cofinal branch $b$.
 But  by Claim \ref{clm:i(xi)>xi_or_etc},
 $M^\Tt_b$ is illfounded,
 a contradiction.
 
 \ref{item:xi_deltavec} $\Leftrightarrow$ \ref{item:xi_above_theta} at $\xi$ (above some $m$): This is established by combining
 the method of the previous part 
 with the proof of Lemma \ref{lem:gen-all_independence}.
 We leave the straightforward execution to the reader.
\end{proof}

\begin{lem}
\tu{(}We have $\xi\in[\lambda,\OR^N)\cap\Lim$ and the formula class $\Gamma$.\tu{)}
~Let $\vec{\tau}\in(\Nm_\xi)^{<\om}$.
Let $m$ witness Lemma \ref{lem:xi_gen-all_independence} for $\vec{\tau}$
\tu{(}in particular with $\base(\vec{\tau})\leq\delta_m^N$\tu{)}.
Let $\varphi\in\Gamma$.
Let $\theta\in\Delta_{\geq m}^N$.
Let $g_d$ be $(N,\CC_d^N)$-generic.
Then:
\begin{enumerate}
\item\label{item:individual_k_equivalent_statements} For each $k<\om$,
$(N|\xi)[g_d]$ satisfies that the following statements are equivalent:
\begin{enumerate}[label=\tu{(}\alph*\tu{)}] \item\label{item:forced_all^*_ks} $\forces_{\CC_{\mathrm{tail}}}\all_k^*s\ \Big[\Mtilde_\xi\sats\varphi(\vec{\tau},s)\Big]$,
\item\label{item:game_xi_above_theta}
$\all^{\mathrm{gen}}_{\geq\theta;k}s\ \forces_{\CC_{\mathrm{tail}}}\Mtilde_\xi\sats\varphi(\vec{\tau},s)$.
\item\label{item:exists_betavec_>=_m} $\exists\vec{\beta}\in[\Delta_{\geq m}]^{2k}\all^{\mathrm{gen}}_{\vec{\beta}}s\ \forces_{\CC_{\mathrm{tail}}}\Mtilde_\xi\sats\varphi(\vec{\tau},s)$
\item\label{item:all_betavec_>=_m}  $\all\vec{\beta}\in[\Delta_{\geq m}]^{2k}\all^{\mathrm{gen}}_{\vec{\beta}}s\ \forces_{\CC_{\mathrm{tail}}}\Mtilde_\xi\sats\varphi(\vec{\tau},s)$
\item\label{item:all_ell_exists_betavec_>=_ell} $\all\ell<\om\exists\vec{\beta}\in[\Delta_{\geq \ell}]^{2k}\all^{\mathrm{gen}}_{\vec{\beta}}s\ \forces_{\CC_{\mathrm{tail}}}\Mtilde_\xi\sats\varphi(\vec{\tau},s)$
\end{enumerate}
\item\label{item:forcing_mu_measure_one_many_first_order} Either:
\begin{enumerate}[label=\tu{(}\alph*\tu{)}]
\item We have:
\begin{enumerate}
\item For some $k<\om$,
$(N|\xi)[g_d]\sats\ \forces_{\CC_{\mathrm{tail}}}\all_k^*s\ \Big[\Mtilde_\xi\sats\varphi(\vec{\tau},s)\Big]$,
\item\label{item:N|xi_sats_psi} $(N|\xi)[g_d]\sats\psi(\vec{\tau})$
\end{enumerate}
or
\item We have:
\begin{enumerate}\item For all $k<\om$,
$(N|\xi)[g_d]\sats\ \forces_{\CC_{\mathrm{tail}}}\exists^*_ks\ \Big[\Mtilde_\xi\sats\neg\varphi(\vec{\tau},s)\Big]$,
\item\label{item:N|xi_sats_neg_psi} $(N|\xi)[g_d]\sats\neg\psi(\vec{\tau})$,
\end{enumerate}
\end{enumerate}
where
\[\psi(\vec{\tau})\ \iff\ \exists k<\om\all\ell<\om\exists\vec{\beta}\in[\Delta_{\geq\ell}]^{2k}\all^{\mathrm{gen}}_{\vec{\beta}}s\ \forces_{\CC_{\mathrm{tail}}}\Mtilde_\xi\sats\varphi(\vec{\tau},s).\]
Moreover, ``$\psi(\vec{\tau})$'' is a first order assertion over $(N|\xi)[g_d]$ \tu{(}as is implicit in the notation in \ref{item:N|xi_sats_psi} and
\ref{item:N|xi_sats_neg_psi}\tu{)}.
\end{enumerate}
\end{lem}
\begin{proof}
Part \ref{item:individual_k_equivalent_statements}:  Note
here that because $k$ is fixed, the statements are directly first order over $(N|\xi)[g_d]$; note that the statement in \ref{item:forced_all^*_ks},
written in expanded form, is just
\[\forces_{\CC_{\mathrm{tail}}}\ \all^{\Dd}t_0\exists^{\Dd}s_0\ldots\all^{\Dd}t_{k-1}\exists^{\Dd}s_{k-1}\ \Big[\big(\all i<k\ (t_i\leq_T s_i)\big)\wedge\Mtilde_\xi\sats\varphi(\vec{\tau},s)\Big]. \]
The equivalence of \ref{item:forced_all^*_ks}--\ref{item:all_ell_exists_betavec_>=_ell} is an straightforward consequence
of Lemma \ref{lem:xi_gen-all_independence}.

Part \ref{item:forcing_mu_measure_one_many_first_order}:
By part \ref{item:individual_k_equivalent_statements},
it is clear that either:
\begin{enumerate}[label=\tu{(}\alph*\tu{)}]
\item\label{item:some_k_option} We have:
\begin{enumerate}[label=\roman*.]
\item For some $k<\om$,
$(N|\xi)[g_d]\sats\ \forces_{\CC_{\mathrm{tail}}}\all_k^*s\ \Big[\Mtilde_\xi\sats\varphi(\vec{\tau},s)\Big]$,
\item\label{item:N|xi_sats_psi'} For some $k<\om$,
$(N|\xi)[g_d]$ satisfies the statement in \ref{item:all_ell_exists_betavec_>=_ell},\end{enumerate}
or
\item We have:
\begin{enumerate}[label=\roman*.]\item For all $k<\om$,
$N[g_d]\sats\ \forces_{\CC_{\mathrm{tail}}}\exists^*_ks\ \Big[\Mtilde_\xi\sats\neg\varphi(\vec{\tau},s)\Big]$,
\item For all $k<\om$,
$(N|\xi)[g_d]$
satisfies the negation of the statement in \ref{item:all_ell_exists_betavec_>=_ell}.
\end{enumerate}
\end{enumerate}
So the only issue
is the claim that $\psi(\vec{\tau})$
is a first order assertion
over $(N|\xi)[g_d]$. (This is not superficially immediate, because of the nesting of the ``$\all^{\mathrm{gen}}_{\vec{\beta}}s$'' quantifier (of length $2k$) within ``$\exists k<\om$''.) To see it is indeed first order, fix $n\in(0,\om)$
such that $\forces_{\xi}^{N[g]}\varphi(\cdot)$
is uniformly $\rSigma_n^{N[g]}$ definable (uniformly in $g$, that is, for $g$ being $(N,\Coll(\om,\gamma))$-generic for some $\gamma<\lambda^N$), and note that
\ref{item:some_k_option}\ref{item:N|xi_sats_psi'} holds iff
\[ (N|\xi)[g_d]\sats\exists k<\om\all\ell<\om\exists\vec{\beta}\in[\Delta_{\geq\ell}]^{2k}\exists q,t\ [T_n(q,t)\wedge\varrho(q,t,\vec{\beta},\vec{\tau})],\]
where $\varrho$ is the $\rSigma_1$ formula asserting that there is $\alpha\in\OR$ such that $q=(\alpha+1,(\vec{\beta},\vec{\tau}))$
and  writing $\vec{\beta}=(\beta_0,\ldots,\beta_{2k-1})$, $P=N[g]$
and $\Dd_i$ for the set of $\CC_{\beta_i}$-names in $P|\beta_i^{+P}$ for Turing degrees, we have
\[ \text{``}\alpha=\beta_{2k-1}^+\text{''}\in t \]
and
\[\begin{array}{lcl} \all ^{\CC_{\beta_0}}p_0\ \all^{\Dd_0} t_0\ \exists^{\CC_{\beta_1}}p_1\leq p_0\ \exists^{\Dd_1} s_0 \\
   \ldots\\
   \all^{\CC_{\beta_{2k-2}}}p_{2k-2}\leq p_{2k-1}\ \all^{\Dd_{2k-2}} t_{k-1}\ \exists^{\CC_{\beta_{2k-1}}} p_{2k-1}\leq p_{2k-2}\ \exists^{\Dd_{2k-1}} s_{k-1}\\\ \  \text{``}p_{2k-1}\forces_{\CC_{\beta_{2k-1}}}
   \Big[\big(\all i<k(t_i\leq_T s_i)\big)\wedge\forces_{\CC_{\mathrm{tail}}}\Mtilde_\xi\sats\varphi(\vec{\tau},s)\Big]\text{''}\in t,\end{array} \]
where $s$ denotes $(s_0,\ldots,s_{k-1})$. (Here $\varrho$ is indeed $\rSigma_1$, because the quantifiers are all bounded by
$(N|\alpha)[g_d]$.)
\end{proof}

We now proceed toward the formal definition of the forcing relation.

\begin{dfn}\label{dfn:language_real+generic-pm}
 The language of \emph{real+generic-premice} has predicate symbols $\widetilde{\es},\es$, and constant symbols $\dot{z},\dot{g}$.
 Given an $\om$-small  real premouse $N=L[\es,z]$
 with $\omega$ Woodin cardinals,
 and given $g$ which is $(N,\CC^N_n)$-generic for some $n<\om$,
 we define an associated real+generic-premouse as the structure
 $\widehat{N[g]}=(N[g],\widehat{\es}_0,\es_0,z_0,g_0)=(N[g],\es^N,\es^{N[g]},z,g)$,
 where $\es^{N[g]}$ is the canonical
 extension of $\es^N\rest(\delta_n^N,\OR^N)$ to  $N[g]$ (and its initial segments). Thus, $\widehat{N[g]}$
 can automatically determine $\es^N,z,g$ (and hence $N$) as well as its own extender sequence $\es^{N[g]}$.
 We allow $n=-1$ and $g=\emptyset$, in which case
 $N$ and $\widehat{N[g]}$
 are trivially equivalent.\end{dfn}
 
 \begin{dfn}\label{dfn:rename}
 Let $N$ be an $\om$-small premouse
  with $\om$ Woodin cardinals and $\lambda=\lambda^N$.
 Let
 $n<\om$
  and $g$ be $(N,\CC_n^N)$-generic.

  For $\sigma\in\Nm^N$, $\rename_{\rightarrow}(\sigma,g)$ denotes
 the natural translation
 of $\sigma$ to a name $\sigma'\in\Nm^{N[g]}$.
 That is, define $\sigma'$ recursively
 in the rank of $\sigma$, as follows:
 \begin{enumerate}[label=--]
  \item  if $\sigma\in\Nm_{\lambda}^N$,
  then $\sigma'=\sigma_g$
  where 
$\sigma_g$ denotes the usual ``partial evaluation'' of $\sigma$ via $g$
(in particular, $\sigma_g\in\Nm_\lambda^{N[g]}$
and  $(\sigma_g)_{G'}=\sigma_G$
whenever $G'$ is $(N[g],\CC^{N[g]})$-generic and $G=g\conc G'$), and
  \item  if $\sigma\notin\Nm_{\lambda}^N$,
  so $\sigma$ is of form $(\eta,i,\pi)$,
  then $\sigma'=(\eta,i,\pi')$
 where $\dom(\pi')=\dom(\pi)$
 and $\pi'(k)=\pi(k)'$ for $k\in\dom(\pi)$.
\end{enumerate}
It will follow that $(\sigma')_{G'}=\sigma_G$ in when $G,G'$ are as above.

We also make a similar definition in reverse.
Given $\sigma\in\Nm^{N[g]}$, $ \rename_{\leftarrow}(\sigma,g)$ denotes the natural translation of $\sigma$ to a name $\sigma''\in\Nm^N$:
proceed recursively as above,
but for $\sigma\in\Nm_\lambda^{N[g]}$,
$\sigma''$ is the $N$-least $\sigma''\in\Nm_{\lambda}^N$
such that $(\sigma'')_g=\sigma$ (this again denotes ``partial evaluation'').
(Note that
for every $\sigma\in\Nm_\lambda^{N[g]}$
there is $\tau\in\Nm_\lambda^N$
such that $\tau_g=\sigma$.)

If $\sigmavec=(\sigma_0,\ldots,\sigma_{k-1})\in(\Nm^N)^{<\om}$
then $\rename_\rightarrow(\sigmavec,g)$
denotes \[(\rename_\rightarrow(\sigma_0,g),\ldots,\rename_\rightarrow(\sigma_{k-1},g)).\]
And if $\pi:X\to\Nm^N$
where $X$ is finite
then $\rename_\rightarrow(\pi,g)$
denotes the function $\pi':X\to\Nm^{N[g]}$
where $\pi'(x)=\rename_\rightarrow(\pi(x),g)$. Likewise for $\rename_\leftarrow$.

Let $\rnm_\leftarrow(\lambda,y,x)$
denote the natural formula,
in free variables $\lambda,y,x$, and which is $\rSigma_1$
in the language of passive real+generic-premice,
such that whenever $N,n,g$ are as above
with $\lambda^N<\OR^N$,
$\pi'\in N$ and $\pi\in N[g]$,
then
\[ \pi'=\rename_\leftarrow(\pi,g)\iff\widehat{N[g]}\sats \rnm_\leftarrow(\lambda^N,\pi',\pi).\]

Let $\rnm^2_\leftarrow(y,x)$
denote the natural formula,
in free variables $y,x$,
and which is $\rSigma_2$ in the language of passive real+generic premice,
such that whenever $N,n,g$ are as above
with $\lambda^N=\OR^N$,
$\pi'\in N$ and $\pi\in N[g]$,
then
\[ \pi'=\rename_\leftarrow(\pi,g)\iff
\widehat{N[g]}\sats\rnm^2_{\leftarrow}(\pi',\pi).\qedhere\]
 \end{dfn}

\begin{dfn}\label{dfn:forcing_relation}
The following function $C$ (for \emph{complexity}) gives a (not very impressive) upper bound on the complexity of the forcing relation for a given formula $\varphi$ of the $\M(\RR)$-language.
Let $C:\om\to\om$
be defined recursively as follows:
\begin{enumerate}[label=--]
 \item if $\varphi$
is an $\mSigma_0$ formula
then $C(\varphi)=1$,
\item $C(\neg\varphi)=C(\varphi)+1$,
\item $C(\varphi\wedge\psi)=\max(C(\varphi),C(\psi))$,
\item $C(\exists x\varphi)=C(\varphi)+3$,
\item $C(\all^*_\mu s\varphi)=C(\varphi)+3$.
\end{enumerate}

We next define a recursive function
$F:\om\cross\om\to\om$ ($F$ for \emph{forcing})
such that for formulas $\sigma(\vec{u})$ of the $\M(\RR)$ language (in free variables $\vec{u}$)
and $\rSigma_1$ formulas $\psi(\lambda,p,\varphi,\pi)$
of the passive real-premouse
language (in free variables $\lambda,p,\varphi,\pi)$, 
$F(\psi,\sigma)$
is a formula $\varrho_{\psi\sigma}(\lambda,p,\pi)$ of the real+generic-premouse language (in free variables $\lambda,p,\pi$).
In the case of interest,
for appropriate $\om$-small premice $N$ with $\om$ Woodin cardinals
and $P$ such that $N|\lambda^N\pins P\ins N$,
$\psi^P(\lambda^N,\cdot,\cdot,\cdot)$ will define the $\mSigma_0$ forcing relation; so
for $p\in\CC^N$, $\mSigma_0$ formulas $\varphi=\varphi(\vec{v})$ and $\pi:X\to\Nm^P$
where $X$ is a finite set of variables with $\vec{v}\sub X$, this will mean that
\[ \Big(p\forces\Mtilde^P\sats\varphi(\pi)\Big)\iff P\sats\psi(\lambda^N,p,\varphi,\pi)\]
(the notation $\varphi(\pi)$ just
means that each free variable $u\in\vec{v}$ is interpreted by $\pi(u)$). We then want $\varrho_{\psi\sigma}$
to be a formula such that if in fact $N|\lambda^N\pins P\pins N$, then
for all $n<\om$
and all $g$ which are $(N,\CC_n^N)$-generic,
all   $p\in\CC^{P[g]}$
and $\pi:X\to\Nm^{P[g]}$
where $X$ is any finite set of variables with $\vec{u}\sub X$,
  we will have
\[ \Big(p\forces_{\CC^{P[g]}}\Mtilde^{P[g]}\sats\sigma(\pi)\Big)\iff \widehat{P[g]}\sats\varrho_{\psi\sigma}(\lambda^N,p,\pi).\]
Note however that in general the formula $\psi$
can be any $\rSigma_1$ formula
of the passive real-premouse language
in the specified free variables.
After having defined $F$,
we will use it to help us write
 a specific formula $\psi$
 having the desired properties.

 We will also
 need to define a variant
 to deal with the case that $P=N|\lambda^N$ (and hence $\lambda^N\notin P$);
 we will do this later.
 
Now for $(\psi,\sigma)\in\om\cross\om$,
we define $\varrho_{\psi\sigma}=F(\psi,\sigma)$ recursively in the length of $\sigma$ as follows.
 \begin{enumerate}
 \item\label{item:forcing_mSigma_0_varphi} if $\sigma$ is $\mSigma_0$ then $\varrho_{\psi\sigma}(\lambda,p,\pi)=$
\[\text{`}\exists q\in\dot{g}\ \exists\pi'\ \Big[ \rnm_{\leftarrow}(\lambda,\pi',\pi)\wedge L[\widehat{\es},\dot{z}]\sats\psi(\lambda,q\cup p,\sigma,\pi')\Big]\text{'},\] using
the symbols
$\widehat{\es},\dot{z},\dot{g}$
of the language of real+generic-premice),
 \item\label{item:forcing_neg_varphi} if $\neg\sigma$ is not $\mSigma_0$ then $\varrho_{\psi,\neg\sigma}(\lambda,p,\pi)=
\text{`}\all q\leq_{\lambda} p\ [\neg\varrho_{\psi\sigma}(\lambda,q,\pi)]\text{'}$,
\item\label{item:forcing_varphi_0_wedge_varphi_1} if $\varphi_0\wedge\varphi_1$ is not $\mSigma_0$ then $\varrho_{\psi,\sigma_0\wedge\sigma_1}(\lambda,p,\pi)=\text{`}\varrho_{\psi\sigma_0}(\lambda,p,\pi)\wedge\varrho_{\psi\sigma_1}(\lambda,p,\pi)\text{'}$,
 \item\label{item:forcing_exists_x_varphi} $\varrho_{\psi,\exists u\sigma}(\lambda,p,\pi)=\text{`}\all q\leq_\lambda p\ \exists r\leq_\lambda q\ \exists\vartheta\in\Nm_{\OR}\ \big[\varrho_{\psi\sigma}(\lambda,r,\pi_{u\mapsto\vartheta})\big]\text{'}$,
 where $\pi_{u\mapsto\vartheta}$
 denotes the map $\pi'$ with domain $\dom(\pi)\cup\{u\}$
 such that $\pi'\rest (\dom(\pi)\cut\{u\})\sub\pi$
 and $\pi'(u)=\vartheta$, and
\item\label{item:all^*_mu_s_forcing_relation} $\varrho_{\psi,\all^*_\mu 
u\sigma}(\lambda,p,\pi)=$
\[\begin{split} \text{`}&\text{Let }d=\base(\pi)\text{.
 Then }\\
& p\rest\CC_d\forces_{\CC_d}\exists^\om k\all^\om m\exists\vec{\beta}\in[\Delta_{\geq m}]^{2k}\exists q,t\Big[T_{C(\varphi)}(q,t)\wedge\theta(\lambda,\pi,q,t,k,\vec{\beta})\Big]\text{',}\end{split}\]
 where $\theta$ is the $\rSigma_1$ formula asserting
 \begin{enumerate}[label=]
 \item `$k>0$ and there is $\alpha\in\OR$ such that $q=(\alpha+1,(\lambda,\vec{\beta},\pi))$
and letting $\vec{\beta}=(\beta_0,\ldots,\beta_{2k-1})$
and letting $\Dd_i$ be the set of $\CC_{\beta_i}$-names in $L[\es,\dot{z},\dot{g}]|\beta_i^{+L[\es,\dot{z},\dot{g}]}$ for Turing degrees, we have
\[ \text{the formula ``}\alpha=\beta_{2k-1}^+\text{''}\in t \]
and
\[\begin{array}{lcl} \all ^{\CC_{\beta_0}}p_0\ \all^{\Dd_0} t_0\ \exists^{\CC_{\beta_1}}p_1\leq_\lambda p_0\ \exists^{\Dd_1} s_0 \\
   \ldots\\
   \all^{\CC_{\beta_{2k-2}}}p_{2k-2}\leq_\lambda p_{2k-3}\ \all^{\Dd_{2k-2}} t_{k-1}\ \exists^{\CC_{\beta_{2k-1}}} p_{2k-1}\leq_\lambda p_{2k-2}\ \exists^{\Dd_{2k-1}} s_{k-1}\\\ \  \text{the formula ``}p_{2k-1}\forces_{\CC_{\beta_{2k-1}}}
   \Big[\big(\all i<k(t_i\leq_T s_i)\big)\wedge\varrho_{\psi\sigma}(\lambda,\emptyset,\pi_{u\mapsto s})\Big]\text{''}\in t\text{'},\end{array} \]
where $s$ denotes $(s_0,\ldots,s_{k-1})$. (Here $\varrho$ is indeed $\rSigma_1$, because the quantifiers are all bounded by
$(L[\es,\dot{z},\dot{g}]|\alpha)$ (using the language of real+generic-premice).)\qedhere
\end{enumerate}
\end{enumerate}
\end{dfn}
Notice that the function $C$
was used in clause \ref{item:all^*_mu_s_forcing_relation}
of Definition \ref{dfn:forcing_relation} 
(in ``$T_{C(\varphi)}(q,t)$'').

\begin{dfn}
 We also define a recursive function $F^*:\om\to\om$, and write $\varrho^*_\varphi=F(\varphi)$,
 with the intention that if $P=N|\lambda^N\pins N$ then
 \[ \Big(p\forces_{\CC^{P[g]}}\Mtilde^{P[g]}\sats\varphi(\pi)\Big)\iff \widehat{P[g]}\sats\varrho^*_\varphi(p,\pi).\]
 Note that since we are now considering defining forcing over $N|\lambda^N$, it is a proper class forcing over this model.

 Clauses \ref{item:forcing_neg_varphi}--\ref{item:forcing_exists_x_varphi} of the definition of $F$ are replicated for $F^*$, except that we drop the parameter $\lambda$ (and of course $\leq_\lambda$ is replaced by the definition of the ordering)
 .
 Clause \ref{item:all^*_mu_s_forcing_relation}
 is basically as before,
 though we can also replace ``$\all^\om m\exists\vec{\beta}\in[\Delta_{\geq m}]^{2k}$'' with just ``$\exists\vec{\beta}\in[\Delta]^{2k}$''. We leave it to the reader to modify the function $C$ if needed. The main difference
 is for $\mSigma_0$ formulas $\varphi$:
 instead of using clause \ref{item:forcing_mSigma_0_varphi}, we define $\varrho^*_\varphi(p,\pi)$ using
 Definition \ref{dfn:mSigma_0_forcing_relation_Mtilde_lambda} and
 the proof of Lemma \ref{lem:mSigma_0,lambda_forcing_rel_def}
 (and its easy adaptation to intermediate generic extensions $N[g]$).
\end{dfn}

Using the functions $F,F^*$
for the main conversion of
formulas into forcing statements,
we now specify the formula $\psi$
we will actually use;
this will define the $\mSigma_0$ forcing relation (in the appropriate context).

In the following definition
we specify an $\rSigma_1$ formula $\psi_0(\lambda,p,\varphi,\pi)$ of the real-premouse language. The intention
is that if $N$ is an appropriate  $\om$-small premouse
with $\om$ Woodins and
 $\lambda^N<\OR^N$,
$p\in\CC^N$, $\varphi(\vec{u})$
is an $\mSigma_0$ formula,
$\pi:X\to\Nm^N$ where $X$ is finite and $\vec{u}\sub X$, then
\[ \Big(p\forces_{\CC^N}\Mtilde^N\sats\varphi(\pi)\Big)
 \iff
 N\sats\psi_0(\lambda^N,p,\varphi,\pi).
\]
As is standard for a level-by-level definition of a forcing relation, we look for a sequence of relations $\forces_\eta$ which handle, in this case, forcing truth over $\Mtilde_\eta$.
We have already defined the $\mSigma_0$ part of $\forces_{\lambda}$ (and sketched the definition of the full $\forces_{\lambda}$).
As is also standard, we use the algorithm of Definition \ref{dfn:Sigma_0_to_Sigma_om_algo}
to reduce the $\mSigma_0$ part of $\forces_{\eta+\om}$ to $\forces_\eta$. The key step, which
makes use of our special circumstances,
is the extension of the $\mSigma_0$ part of $\forces_\eta$ to the full $\forces_\eta$;
this employs the formulas $\varrho_{\psi\varphi}$ and $\varrho^*_\varphi$. (Things are also a little specialized here because we are restricting to names in $\Nm^N$ and the model $\Mtilde^N$.) A small organizational subtlety also arises in that our definition of $(\psi,\varphi)\mapsto\varrho_{\psi\varphi}$ refers to $\psi$,
but of course we don't know what the $\psi$ of interest is until we have written it; in the end this does not matter.

\begin{dfn}\label{dfn:psi_0}
Let $\psi_0(\lambda,p_0,\varphi_0,\sigma_0)$ 
denote the natural $\Sigma_1$ formula of the passive real-premouse
language asserting
``there are $\CC,\eta^*,I,I'$ and a sequence  $\left<\forces^\varphi_\eta\right>_{(\eta,\varphi)\in I'}$ 
such that:
\begin{enumerate}
\item $\lambda,\eta^*\in\Lim$, $\lambda\leq\eta^*$
and $L[\es,\dot{z}]|\eta^*\sats$``$\lambda$ is a limit of Woodin cardinals'',\footnote{The requirement that $L[\es,\dot{z}]|\eta^*\sats$``$\lambda$ is a limit of Woodin cardinals''
is included more for the reader's orientation,
and to let us make clear sense
of things like $\CC^{L[\es,\dot{z}]|\lambda}$ and $\Nm_{\eta^*}^{L[\es,\dot{z}]}$;
in the end we will only interpret
$\psi_0$ over an $\om$-small premouse $N$
with $\om$ Woodins, and with $\lambda=\lambda^N$,
so we will in fact have $N\sats$``$\lambda$ is a limit of Woodins'', not only $N|\eta^*$. But here
we do not want $\psi_0$ to assert ``$L[\es,\dot{z}]\sats$``$\lambda$ is a limit of Woodins'',
because we want $\psi_0$ to be $\Sigma_1$.}
\item $\CC=\CC^{L[\es,\dot{z}]|\lambda}$,
\item $I=[\lambda,\eta^*]\cap\Lim$
and $I'\sub I\cross\om$,
\item if $\eta\in I$
and $\varphi$ is an $\mSigma_0$ formula
then $(\eta,\varphi)\in I'$,
\item if $\eta\in[\lambda,\eta^*)\cap\Lim$ then $\{\eta\}\cross\om\sub I'$,
\item there are only finitely many non-$\mSigma_0$ formulas $\varphi$ such that $(\eta^*,\varphi)\in I'$,\footnote{We make this arestriction
so that $\psi_0$ can actually be written in $\rSigma_1$ form.}
\item $(\eta^*,\varphi')\in I'$
for each $\varphi$ with $(\eta^*,\varphi)\in I'$
and each subformula $\varphi'$ of $\varphi$,
\item $\forces^\varphi_\eta\ \sub \CC\cross(\Nm^{L[\es,\dot{z}]|\eta})^{<\om}$,
and we write $p\forces_\eta\varphi(\pi)$
for $(p,\pi)\in\ \forces^\varphi_\eta$,
\item\label{item:forcing_relation_at_lambda_def}
if $(\lambda,\varphi)\in I'$
and $\pi\in(\Nm^{L[\es,\dot{z}]|\lambda})^{<\om}$
then
\[ \big(p\forces_\lambda\varphi(\pi)\big)
\iff L[\es,\dot{z}]|\lambda\sats\varrho^*_\varphi(p,\pi),\]
 \item\label{psi_0_next_mSigma_0} if $\eta+\om\in I$
 and $\varphi$ is $\mSigma_0$
 and $n<\om$ and $\sigma:n\to\Nm^{L[\es,\dot{z}]|(\eta+\om)}$ then:
 \begin{enumerate}
  \item\label{item:if_names_always_at_eta}  if there is $\left<i_k,\pi_k\right>_{k<n}$ such that for each $k<n$, we have
 \[ \sigma(k)=(\eta,i_k,\pi_k), \]
then writing $\vec{i}=(i_0,\ldots,i_{n-1})$,
we have
 \[ \Big(p\forces_{\eta+\om}\varphi(\sigma)\Big) \iff
 \Big(p\forces_{\eta}\psi_{\varphi,\vec{i}}(\pi_0\conc\ldots\conc\pi_{n-1})\Big) \]
 (recall $\psi_{\varphi,\vec{i}}$
 was specified in Definition \ref{dfn:Sigma_0_to_Sigma_om_algo}), 
\item\label{item:if_names_not_always_at_eta} letting $\sigma':n\to\Nm^{L[\es,\dot{z}]|(\eta+\om)}$
be such that for all $k<n$,
\begin{enumerate}
\item if $\sigma(k)$ has form $(\eta,i,\pi)$
for some $i,\pi$, then $\sigma'(k)=\sigma(k)$,
\item if $\sigma(k)$ has form $(\gamma,i,\pi)$
for some $\gamma,i,\pi$ with $\gamma<\eta$,
then\[ \sigma'(k)=(\eta,\pad(i),\pi\conc\left<m_{\gamma}\right>) \]
(cf.~Definitions \ref{dfn:Sigma_0_to_Sigma_om_algo} and \ref{dfn:names_for_ordinals_and_M}), and
\item if $\sigma(k)\in\Nm^{L[\es,\dot{z}]|\lambda}$
then $\sigma'(k)=(\eta,I_{33},\left<\sigma(k)\right>)$ (cf.~Definition \ref{dfn:Sigma_0_to_Sigma_om_algo}),
\end{enumerate}
we have\footnote{Note that if  the hypothesis of clause \ref{item:if_names_always_at_eta}
holds then $\sigma'=\sigma$,
so that clause \ref{item:if_names_not_always_at_eta} holds trivially.}
 \[ \Big(p\forces_{\eta+\om}\varphi(\sigma)\Big) \iff \Big(p\forces_{\eta+\om}\varphi(\sigma')\Big).\]
 \end{enumerate}
 \item if $\eta\in(\lambda,\eta^*]$
 is a limit of limits then for each $\mSigma_0$ formula $\varphi$,
 we have \[{\forces^\varphi_\eta}=
 \Big(\bigcup_{\gamma<\eta}{\forces^\varphi_\gamma}\Big),\]
 \item if $\eta\in(\lambda,\eta^*]$ and $\left<\forces^\varphi_\eta\right>_{\varphi\in \mSigma_0}$
 is $\Sigma_1$-definable over 
 $L[\es,\dot{z}]|\eta$ in the parameter $\lambda$,
 and $\psi$ is the least $\Sigma_1$ formula such that
 \[ \left<\forces^\varphi_\eta\right>_{\varphi\in \mSigma_0} =\Big\{(p,\varphi,\pi)\in (L[\es,\dot{z}]|\eta)\Bigm|(L[\es,\dot{z}]|\eta)\sats\psi(\lambda,p,\varphi,\pi)\Big\},\]
 and if $(\eta,\varphi)\in I'$
 and $\sigma:n\to\Nm^{L[\es,\dot{z}]|\eta}$ where $n<\om$ and all free variables of  $\varphi$ are ${<n}$,
 then
 \[ \Big(p\forces_\eta\varphi(\sigma)\Big)\iff L[\es,\dot{z}]|\eta\sats\varrho_{\psi\varphi}(\lambda,p,\sigma),\]
 \item\label{item:psi_0_conclusion} $\varphi_0$ is an $\mSigma_0$ formula
and there is $n<\om$ such that all free variables of $\varphi_0$ are ${<n}$ and $\sigma_0:n\to\Nm^{L[\es,\dot{z}]|(\eta^*+\om)}$,\footnote{It might be that $\eta^*+\om=\OR^N$,
but note that this statement
is still naturally $\rSigma_1$,
as $\pi_0$ is required to have finite domain
and $\Nm^N\sub N$.}
and letting $\sigma_0'$
 be defined from $\sigma_0$ like in clause \ref{item:if_names_not_always_at_eta} (but with $\eta^*$ replacing  $\eta$ there)
 and $\vec{i}'_0$ be defined
 from $\sigma_0'$ like  $\vec{i}$
  in clause \ref{item:if_names_always_at_eta},
  and also $\left<\pi'_j\right>_{j<n}$ from $\sigma_0'$ 
  like $\left<\pi_j\right>_{j<n}$
  in clause \ref{item:if_names_always_at_eta},
 then $(\eta^*,\psi_{\varphi_0,\vec{i}'_0})\in I'$
 and
 $p_0\forces_{\eta^*}\psi_{\varphi_0,\vec{i}'_0}(\pi_0'\conc\ldots\conc\pi'_{n-1})$\text{''}.\qedhere
\end{enumerate}
\end{dfn}

\begin{dfn}\label{dfn:psi_0^neg}
 Let $\psi_0^{\neg}$ be the $\Sigma_1$ formula of the passive real-premouse
language obtained from $\psi_0$
by replacing, in condition \ref{item:psi_0_conclusion},
the very last clause
\[\text{``}p_0\forces_{\eta^*}\psi_{\varphi_0,\vec{i}'_0}(\pi_0'\conc\ldots\conc\pi'_{n-1})\text{''}\]
with its negation, i.e.
\[\text{``}p_0\not\forces_{\eta^*}\psi_{\varphi_0,\vec{i}'_0}(\pi_0'\conc\ldots\conc\pi'_{n-1})\text{''}.\qedhere\]
\end{dfn}

\begin{lem}
 Let $N$ be an $\om$-small premouse with $\om$ Woodins.
 \begin{enumerate}
  \item\label{item:agreement_of_witnesses}
Suppose \tu{(}i\tu{)} $N\sats\psi_0(\lambda^N,p_0,\varphi_0,\sigma_0)$ or 
\tu{(}ii\tu{)} $N\sats\psi_0^{\neg}(\lambda^N,p_0,\varphi_0,\sigma_0)$, as witnessed by
 \[ \Big(\CC_0,\eta^*_0,I_0,I'_0,\left<\forces^{\varphi}_{0\eta}\right>_{(\eta,\varphi)\in I'_0}\Big),\]
 and also \tu{(}iii\tu{)} $N\sats\psi_0(\lambda^N,p_0,\varphi_0,\sigma_0)$ or \tu{(}iv\tu{)} $N\sats\psi_0^{\neg}(\lambda^N,p_0,\varphi_0,\sigma_0)$, as witnessed by
 \[ \Big(\CC_1,\eta^*_1,I_1,I'_1,\left<\forces^{\varphi}_{1\eta}\right>_{(\eta,\varphi)\in I'_1}\Big).\]
 \tu{(}So either \tu{(}i\tu{)} holds,
 with two witnessing tuples,
 or \tu{(}ii\tu{)} holds, with two witnessing tuples,
 or both \tu{(}i\tu{)} and \tu{(}ii\tu{)}
 hold, and we have witnesses for each.\tu{)}
Let $\eta^*=\min(\eta^*_0,\eta^*_1)$.
 Then:
 \begin{enumerate}[label=\tu{(}\alph*\tu{)}]
  \item the two witnessing tuples have no disagreements within their common domains; that is, 
$\CC_0=\CC_1$ and for all $(\eta,\varphi)\in I'_0\cap I'_1$,
 all $p\in\CC_0$
 and all $\pi\in(\Nm^{N|\eta^*})^{<\om}$,
 \[ \big(p\forces_{0\eta}\varphi(\pi)\big)
\iff \big(p\forces_{1\eta}\varphi(\pi)\big) \]
\tu{(}using the notation of Definition \ref{dfn:psi_0}\tu{)},
 \item for each $\eta\in(\lambda,\eta^*_1]\cap\Lim$, $\left<\forces^\varphi_{1\eta}\right>_{
 \varphi\in\mSigma_0}$ is $\Sigma_1$-definable
 over $N|\eta$ in the parameter $\lambda^N$,
 as witnessed by $\psi_0$; that is, 
 \[ \left<\forces^\varphi_{1\eta}\right>_{\varphi\in\mSigma_0}=\Big\{(p,\varphi,\pi)\in(N|\eta)\ \Big|\ (N|\eta)\sats\psi_0(\lambda^N,p,\varphi,\pi)\Big\}.
 \]
\end{enumerate} 
\item\label{item:mSigma_0_forcing_Delta_1}For all $p_0\in\CC$, $\mSigma_0$ formulas $\varphi_0$, and $\sigma_0\in(\Nm^N)^{<\om}$,
we have 
\[N\sats\Big(\psi_0(\lambda^N,p_0,\varphi_0,\sigma_0)
 \Leftrightarrow\neg\psi^{\neg}_0(\lambda^N,p_0,\varphi_0,\sigma_0)\Big).
\]
\end{enumerate}
\end{lem}
\begin{proof}
Part \ref{item:agreement_of_witnesses} is a straightforward induction on $\eta$,
and part \ref{item:mSigma_0_forcing_Delta_1} likewise,
on the rank of $\sigma_0$. (Note that $\psi^{\neg}_0$ is only intuitively asserting that
$p_0$ fails to force $\varphi_0$, not that it forces $\neg\varphi_0$.)
\end{proof}

\begin{rem}
 Note that $\left<\forces^\varphi_{1\lambda^N}\right>_{\varphi\in\mSigma_0}$ is also $\rSigma_2$-definable over $N|\lambda^N$ (without parameters);
 this follows from clause \ref{item:forcing_relation_at_lambda_def}
 of $\psi_0$.
\end{rem}

\begin{dfn}\label{dfn:internal_mSigma_0_forcing_relation}
 Let $N$ be an $\om$-small premouse
 with $\om$ Woodins.
 Then for $p\in\CC^N$,
 $\varphi\in\mSigma_0$
 and $\pi\in(\Nm^N)^{<\om}$ we write
 \[ \Big(p\sforces{\CC}{N,\mathrm{int}}\Mtilde\sats\varphi(\pi)\Big)\iff N\sats\psi_0(\lambda^N,p,\varphi,\pi);\]
 the \emph{int} stands for \emph{internal}.
 We may drop the ``$N$'' from ``$\sforces{\CC}{N,\mathrm{int}}$''
 if $N$ is clear from context.
\end{dfn}

\begin{rem}
Recall that by \cite{jensen_fs}, for any $A\sub V$,
$A$-rud functions can be expressed as a composition of rud functions and the function $x\mapsto x\cap A$.
We consider $\mu$-rud function schemes using finitely many variables taken from an infinite sequence $\left<x_n\right>_{n<\om}$. Let $F'_1(x_0,x_1),\ldots,F'_{15}(x_0,x_1)$ be the list of $\mu$-rud function schemes
given in \cite[between 1.7 and 1.8]{schindler2010fine} (but with input variables $x_0,x_1$; recall that $F'_i$ has arity 2 for each $i$,
and
$F'_{15}(x_0,x_1)$ is (symbolically) $x_0\mapsto x_0\cap\mu$). We add another scheme $F'_0$ to this list, defined $F'_{0}(x_0,x_1)=x_0$.
Let $F_i(x_0,x_1,x_2,x_3)=F'_i(x_0,x_1)$,
so $F_i$ formally has more input variables, but they are ignored.
Say a  scheme $f$ is \emph{$0$-good}
if
\[ f\in\{F_0(x_0,x_1,x_2,x_3),\ldots,F_{15}(x_0,x_1,x_2,x_3)\}.\]
For $n<\om$, say scheme
is \emph{$(n+1)$-good}
if it has form
\[ f(x_0,\ldots,x_{2^{n+3}-1})=F_i(g(x_0,\ldots,x_{2^{n+2}-1}),h(x_{2^{n+2}},\ldots,x_{2^{n+3}-1})) \]
for some $i<16$ and $n$-good schemes $g,h$ (the notation means we substitue $x_{2^{n+1}+i}$ for $x_i$ in the original $h$). For $n<\om$, say a  scheme $f$ is \emph{$n$-nice}
iff $f=g\com h$, where $g$ is $n$-good
and
\[ h(x_0,\ldots,x_{k-1})=(x_{\pi(0)},\ldots,x_{\pi(2^{n+2}-1)}) \]
where $k\in[1,2^{n+2}]$ and $\pi:2^{n+2}\to k$.\footnote{When $n=0$,
we have allowed $k\in[1,4]$ instead
of just $k\in[1,2]$, as
letting $\mathscr{F}$ be the
set of $0$-nice schemes
with $4$ input variables,
and
given a limit $\eta$,
letting $\mathscr{T}$ be the 
set of elements of $\Nm^N$
of form $\tau=(\eta,i,\pi)$
with $f_i\in\mathscr{F}$,
and letting $u$ be the universe of $\Mtilde_\eta^{N,G}$
and $t=T^{\Mtilde_\eta^{N,G}}$,
we will have \[ s^{\mu^{N,G}}(u\cup\{u,t\}))=\Big\{\tau_G\bigm|\tau\in\mathscr{T}\}=\Big\{f(u,t,x,y)\bigm|f\in\mathscr{F}\wedge x,y\in u\Big\},\]
whereas the will not hold with ``$4$ input variables''
replaced by ``$2$ input variables'',
since in the definition of $\tau_G$,
we  put (for convenience) $u$ and $t$ as the inputs
to the first two variables.}

Say $f$ is \emph{good} (nice) if it is $n$-good ($n$-nice) for some $n<\om$.
Let $\mathscr{N}_n$ ($\mathscr{G}_n$)
be the set of $n$-nice ($n$-good) schemes.
Note that $\mathscr{N}_n$ and $\mathscr{G}_n$
are finite, and all schemes
in $(k+1)$ variables $(x_0,\ldots,x_k)$, where $k<\om$, are equivalent to some nice one.

Given $A\sub V$ and a  scheme $f$, let $f^A$ be the resulting $A$-rud function. Let $s^A(u)=\bigcup_{i<16}(F^A_i)``u^4$. We have $u\sub s^A(u)$ since
$F_0(x_0,\ldots,x_3)=x_0$. If $u$ is transitive then so is $s^A(u)$ (cf.~\cite{schindler2010fine}).
And essentially by \cite{jensen_fs} (see also \cite{schindler2010fine}),
the $A$-rud closure of any set $u$
is just $\bigcup_{n<\om}(s^A)^n(u)$.
Note that for $n<\om$, 
\[ (s^A)^{n+1}(u)=\bigcup_{f\in\mathscr{G}_n^A}f``u^{2^{n+2}}=\bigcup_{f\in\mathscr{N}_n^A}f``u^{a_f}=\bigcup_{f\in\mathscr{N}_n^A\wedge a_f\geq 2}f``u^{a_f}\]
where $\mathscr{G}^A_n=\{f^A\bigm|f\in\mathscr{G}_n\}$ and $\mathscr{N}^A_n=\{f^A\bigm|f\in\mathscr{N}_n\}$ and $a_f$ denotes the arity of $f\in\mathscr{N}_n^A$. We have   $(s^A)^n(u)\sub (s^A)^{n+1}(u)$,
and if $u$ is transitive then so is $(s^A)^n(u)$.

Now we may from now on assume that $\left<f_i\right>_{i<\om}$ enumerates (recursively) just the set of nice schemes $f$
with $a_f\geq 2$. There are only finitely many such schemes which are $n$-nice for a given $n$.\end{rem} 

Using these notions we refine the hierarchy of $\Nm^N$:

\begin{dfn}
For $\xi\in[\lambda,\OR^N)\cap\Lim$
and $n<\om$ 
let $\Nm^N_{\xi+n+1}$ be the set $\tau\in\Nm^N_{\xi+\om}$ such that $\tau\in\Nm_\xi^N$ or $\tau=(\xi,i,\pi)$ for some $i$ such that $f_i$ is $k$-nice for some $k\leq n$.
So $\Nm^N_{\xi+\om}=\bigcup_{n<\om}\Nm^N_{\xi+n}$.\end{dfn}

We can now formalize the intuitively
introduced notions
of Definition \ref{dfn:Nm^N}:

\begin{dfn}\label{dfn:formal_Nm^N}
Let $N$ be an $\om$-small premouse with $\om$ Woodins.
Let $G$ be $(N,\CC^N)$-generic.
Working in $N[G]$, we define
an inner model $\Mtilde^{N,G}$,
as follows. 
For $\sigma,\tau\in\Nm^N$
define
\[ \sigma\approx\tau\iff\exists p\in G\ \Big[p\sforces{\CC}{\mathrm{int}}\Mtilde\sats\sigma=\tau\Big], \]
and note that $\approx$ is a $\Delta_1^N(\{\lambda^N\})$-definable proper class equivalence relation of $N[G]$, uniformly in $N,G$.
However, each equivalence class
is itself a proper class of $N[G]$.
Also for $\sigma,\tau\in\Nm^N$ define
\[ \sigma\ \dot{\in}\ \tau\iff\exists p\in G\ \Big[p
\sforces{\CC}{\mathrm{int}}\Mtilde\sats\sigma\in\tau\Big].\]
Note that  $\dot{\in}$ is a $\Delta_1^{N[G]}(\{\lambda^N\})$-definable relation
on $\Nm^N$, which respects $\approx$.
Let $t^N\in\Nm_{\lambda^N+1}^N$
be $t^N=(\lambda^N,i,\pi)$,
where $i<\om$ is such that $f_i(x_0,x_1,x_2,x_3)=F_0(x_1,x_0)=x_1$
(note this is $0$-nice) and $\pi(0)=\pi(1)=\check{\emptyset}$.

For $\eta\in[\lambda^N,\OR^N]$ let $\approx_\eta$, $\dot{\in}_\eta$ be the restrictions of $\approx$, $\dot{\in}$ to $\Nm^N_\eta$ respectively.
Let $[t]_\eta$ be the $\approx_\eta$-equivalence class
of $t\in\Nm_\eta^N$ and
 $\mathscr{U}'_\eta=\{[t]_\eta\bigm|t\in\Nm_\eta^N\}$.
 For $\pi\in(\Nm_\eta^N)^{<\om}$,
 we will also later use the notation
$[\pi]_\eta=([\pi(0)],\ldots,[\pi(k-1)]_\eta)$ where $k=\lh(\pi)$.
Let $\mathscr{E}'_\eta$
be the relation induced on $\mathscr{U}'_\eta$ by $\dot{\in}_\eta$. Let $e_\eta'$ be the equality relation on $\mathscr{U}'_\eta$. If $\eta>\lambda^N$
let $\mathscr{M}'_\eta$ denote the structure \[ (\mathscr{U}'_\eta,[t^N]_\eta,\mathscr{E}_\eta',e_\eta') \]
with signature that of $\M_\gamma$ (with universe $\mathscr{U}'_\eta$ and binary relations $\mathscr{E}'_\eta$ and $e_\eta'$ and constant $[t^N]'$).
If $\eta=\lambda^N$ let $T=\widetilde{T}^{N,G}$
and noting that $\mathscr{U}'_{\lambda^N}=\widetilde{HC}^{N,G}$, let $\mathscr{M}'_{\lambda^N}$ denote the structure
\[ (\mathscr{U}'_{\lambda^N},\widetilde{T}^{N,G},\mathscr{E}'_{\lambda^N},e_{\lambda^N}'),\]
with signature that of $\M_{\omega_1}$.

If $\mathscr{M}'_\eta$ is extensional
and wellfounded,
then let
\[ \mathscr{M}_\eta=(\mathscr{U}_\eta,t_\eta,\mathscr{E}_\eta,e_\eta) 
\]
 denote
its transitive collapse and $\pi_\eta:\mathscr{M}_\eta\to\mathscr{M}'_\eta$
the uncollapse map
(so if $\eta>\lambda^N$ then $\pi_\eta(t_\eta)=[t^N]_\eta)$;
otherwise easily $\M_\eta=\M'_\eta$
and $\pi_\eta=\id$, so in this case $t_\eta=\widetilde{T}^{N,G}$).

We write $(\mathscr{M}')^{N,G}_\eta=\mathscr{M}'_\eta$, etc.
In the case that $\eta=\OR^N$ we may drop the
subscript ``$\eta$'' from this notation,
writing $(\mathscr{M}')^{N,G}=(\mathscr{M}')^{N,G}_{\OR^N}$, etc.
\end{dfn}

  \begin{dfn}
  Let $N$ be an $\om$-small premouse with $\om$ Woodins.
We say that $N$ is \emph{$\mu$-homogeneous}
iff for all $\xi\in\Lim\cap[\lambda^N,\OR^N)$ and all $\pi\in(\Nm^N_\xi)^{<\om}$
and all $\mSigma_\om$ formulas $\varphi$,
there are $d<m<\om$ such that for all $k<\om$ and all $\vec{\delta},\vec{\vareps}\in[\Delta^N_{\geq m}]^{2k}$, $N|\xi$ satisfies that $\CC_d^N$ forces
\[
\all^*_{\vec{\delta}}s\ \sforces{\CC_{\mathrm{tail}}}{\mathrm{int}}\varphi(s,\pi)\iff
\all^*_{\vec{\vareps}}s\ \sforces{\CC_{\mathrm{tail}}}{\mathrm{int}}\varphi(s,\pi)\iff
\all^*_{>\delta_m^N}s\ \sforces{\CC_{\mathrm{tail}}}{\mathrm{int}}\varphi(s,\pi),\]
and \emph{$\mu$-determined}
iff for all $(\xi,\pi,\varphi)$ as above,
there are $d<m<\om$ such that for all $(k,\vec{\delta})$ as above, 
$N|\xi$ satisfies that $\CC_d^N$ forces
\begin{equation}\label{eqn:mu-det_equiv}\all^*_{\vec{\delta}}s\ \sforces{\CC_{\mathrm{tail}}}{\mathrm{int}}\varphi(s,\pi)\iff \exists^*_{\vec{\delta}}s\ \sforces{\CC_{\mathrm{tail}}}{\mathrm{int}}\varphi(s,\pi), \end{equation}
and \emph{$\mu$-nice} if $\mu$-homogeneous
and $\mu$-determined.

For $n<\om$, say that $N$ is \emph{$(\mu,n)$-nice}
iff $N$ is $\mu$-nice, 
$N$ is $5(n+1)$-sound, $\lambda^N\leq\rho_{5(n+1)}^N$,
and $\mu$-homogeneity
and $\mu$-determinacy
hold with respect to all $\mSigma_n$ formulas $\varphi$ and all $\pi\in(\Nm^N)^{<\om}$.
  \end{dfn}

Note that if $N$ is $\mu$-nice,
then the stronger version of \emph{$\mu$-determined} holds, which results when we replace line (\ref{eqn:mu-det_equiv})
with the equivalence 
\[\begin{split}&\all^*_{\vec{\delta}}s\ \sforces{\CC_{\mathrm{tail}}}{\mathrm{int}}\varphi(s,\pi)\ \ \ \ \iff \exists^*_{\vec{\delta}}s\ \sforces{\CC_{\mathrm{tail}}}{\mathrm{int}}\varphi(s,\pi)\\
\iff &\all^*_{>\delta_m^N}s\ \sforces{\CC_{\mathrm{tail}}}{\mathrm{int}}\varphi(s,\pi)\iff\exists^*_{>\delta_m^N}s\
\sforces{\CC_{\mathrm{tail}}}{\mathrm{int}}\varphi(s,\pi)\end{split}\]
Likewise for $(\mu,n)$-niceness.
Note that $(\mu,n)$-niceness is expressed by a first-order sentence.

  \begin{dfn}\label{dfn:Lhat(R)^M}
Given $\gamma\in[\alphagap,\beta^*]\cap\Lim$
and $\M=\M_\gamma$,
$\widehat{L}(\RR)^{\M}$ denotes the corresponding level $\J_{f(\gamma)}$ of $L(\RR)$
 (cf.~Definition \ref{dfn:between_hierarchies}).
 This model can also be defined
 in the codes inside $\M$. If $\M\sats$``$\Theta$ exists'' then $\widehat{L}(\RR)^{\M}=L(\RR)^\M$. Otherwise
  $\widehat{L}(\RR)^{\M}$ is the union of transitive models which
 satisfy ``There is no largest ordinal and $V=\J(\RR)$''\footnote{And recall
 that in our indexing of the $\J$-hierarchy,
 $\SS_\delta$ only has limit ordinal height when $\delta$ is a limit. Because
 $\M\sats$``$\Theta$ does not exist'',
 $f(\gamma)$ is a limit of limits.}, are coded by sets of reals in $\M$,  contain an isomorphic copy of $\RR\cap\M$
 as their own set of reals, have
 a least initial segment $\bar{\M}$ which satisfies
 $T^\M$, and whose ordinal height $\xi$ is $<\OR^{\bar{\M}}+\om\cdot\OR^{\M}$ (and hence $\xi$ is wellfounded).
 
 Given a model $\M$ with  the same signature and similar first-order properties to those of $\M_\gamma$
(including Turing determinacy, $T^{\M}\sub\HC^{\M}$,  $\M$ is built by constructing relative to the iterated Martin measure of ${\M}$ over $(\HC^{\M},T^{\M})$, and $T^{\M}$ is a consistent $\Sigma_1$ theory in the $L(\RR)$ language),
 we define $\widehat{L}(\RR)^{\M}$ analogously
in the codes over ${\M}$, if possible. Note then that if there is
 $\alpha\in\OR$ such that $T_{\M}$ encodes exactly $\Th_{\Sigma_1}^{\Ss_{\bar{\alpha}}(\RR^{\M})}(\RR^\M)$, then $\xi=\OR(\widehat{L}(\RR)^{\M})$ is wellfounded, and in fact $\xi\leq\alpha+\om\cdot\beta$ where $\OR^{\M}=\omega_1^{\M}+\beta$.
\end{dfn}

Note that if $\M$ is as in Definition \ref{dfn:Lhat(R)^M}
 and $\widehat{L}(\RR)^{\M}$ is well-defined
 then $\Th_{\Sigma_1}^{\widehat{L}(\RR)^{\M}}(\RR^{\M})$ is $\mSigma_1^{\M}$,
 and in case $\M=\M_\gamma$,
 we have $\Th_{\Sigma_1}^{\widehat{L}(\RR)^{\M}}(\RR)=\Th_{\Sigma_1}^{\J_{\alphag}}(\RR)=T^{\M_\gamma}$,
 since $f(\gamma)\leq f(\beta^*)\leq\betagap$.

\begin{lem}\label{lem:Sigma_P-iterate_good}
 Let $N$ be a generic $\Sigma_{\Pg}$-iterate. Let $G$, possibly
 appearing in some generic extension of $V$, be $(N,\CC^N)$-generic.
 Write $\mathscr{M}'_\eta=(\mathscr{M}')^{N,G}_\eta$, etc. Let $\xi_0,\xi_1,\eta\in[\lambda^N,\OR^N]$ with $\xi_0\leq\xi_1\leq\eta\in\Lim$.
Then:
\begin{enumerate}
 \item\label{item:ext_and_wf} $\mathscr{M}'_{\xi_0}$ is extensional
 and wellfounded,
 \item $\mathscr{U}_{\xi_0}\sub\mathscr{U}_{\xi_1}$
 and $t_{\xi_0}=t_{\xi_1}$, so
so $\mathscr{M}_{\xi_1}$
 is an end-extension of $\mathscr{M}_{\xi_0}$,
 \item $\mathscr{U}_\eta$
 is rudimentarily closed and amenable,
 \item\label{item:HC_stable} $\HC^{\M_\eta}=\mathscr{U}_{\lambda^N}$,
 \item $\M_{\xi_0}\sats$``Turing determinacy holds'',
 \item $\OR(\mathscr{M}_\eta)=\eta$,
 \item if $\xi=\xi_0<\eta$
 then $\OR(\mathscr{M}_{\xi+n+1})\leq\xi+5n+5$
 for each $n<\om$,
 \item if $\eta>\lambda^N$
 then letting
  $s=\left<\Nm_\gamma,\approx_\gamma,\dot{\in}_\gamma,\mathscr{M}'_\gamma,\mathscr{M}_\gamma,\pi_\gamma\right>_{\gamma<\eta}$, we have:
 \begin{enumerate}[label=\tu{(}\alph*\tu{)}]
 \item $s\sub N[G]|\eta$
 and $s$ is
 $\Sigma_1^{N[G]|\eta}(\{N|\lambda^N,G\})$-definable, uniformly in $\eta$,
 \item $\M_\eta\sub(N|\eta)[G]$ and $\M_\eta$ is $\Sigma_1^{N[G]|\eta}(\{N|\lambda^N,G\})$-definable, uniformly in $\eta$,
 \end{enumerate}
   \item\label{item:mu-nice_at_eta} $N|\eta$ is $\mu$-nice,
\item \tu{(}$\mSigma_0$ forcing theorem\tu{)} for all $k<\om$, all $\mSigma_0$ formulas $\varphi(x_0,\ldots,x_{k-1})$ and all $\pi\in(\Nm_\eta)^{k}$, recalling $[\pi]_\eta$ denotes $([\pi(0)]_\eta,\ldots,[\pi(k-1)]_\eta)$,  \[\Big(\M_\eta\sats\varphi([\pi]_\eta)\Big)\iff\exists p\in G\ \Big[N|\eta\sats p\sforces{\CC}{\mathrm{int}}\Mtilde\sats\varphi(\pi)\Big].\]

 \item\label{item:forcing_theorem_mSigma_om_over_M_xi} \tu{(}$\mSigma_\om$ forcing theorem\tu{)}  Suppose $\xi_0<\eta$.
 Let $\varphi$ be an $\mSigma_\om$ formula
 of arity
 $k<\om$ and
  $\pi\in(\Nm_{\xi_0})^k$.
  Then
  \[ \Big(\M_{\xi_0}\sats\varphi([\pi]_{\xi_0})\Big)\iff\exists p\in G\ \Big[N|\xi_0\sats p\sforces{\CC}{\mathrm{int}}\Mtilde\sats\varphi(\pi)\Big].\]

 \item\label{item:tau_interpreted_correctly} Suppose $\xi=\xi_0<\xi_1$ and $\xi\in\Lim$. Let $\tau\in\Nm^N$
 be of form $\tau=(\xi,i,\pi)$
 and $k<\om$ be such that
 $\pi$ is $(2+k)$-ary.
 Then
 \[ [\tau]_{\xi_1}=f_i^{\mu^{N,G}}(\M_{\xi},T^{\M_\xi},[\pi(0)]_{\xi},\ldots,[\pi(k-1)]_\xi);\]
moreover, 
there is $X\in\M_{\xi+\om}\sub N[G]|(\xi+\om)$ such that all computations of $\mu^{N,G}$-measure relevant to computing $f_i^{\mu^{N,G}}(\M_\xi,T^\M_\xi,X_0,\ldots,X_{k-1})$
for any $X_0,\ldots,X_{k-1}\in\M_\xi$
are witnessed by measure one trees in $X$,
\item\label{item:no_new_Sigma_1_truths} $\Th_{\Sigma_1}^{\widehat{L}(\RR)^{\M_\eta}}(\RR)=T^{\M_\eta}=T^{\M_{\lambda^N}}$.
 \end{enumerate}
 \end{lem}
 
  \begin{proof}
Let $\eta_\infty$ be largest limit ordinal $\leq\OR^N$
such that $N|\eta_\infty$ is $\mu$-nice.
We first show that the lemma holds for all
 $\eta\in[\lambda^N,\eta_\infty]\cap\Lim$, by induction on $\eta$. If $\eta=\lambda^N$ it is clear,
  so suppose $\eta>\lambda^N$. If $\eta$ is a limit of limits then everything follows
  easily by induction, so suppose $\eta=\eta'+\om$ where $\lambda^N\leq\eta'\in\Lim$,
  and parts \ref{item:ext_and_wf}--\ref{item:tau_interpreted_correctly} hold at $\eta'$;
  we verify them at $\eta$.
  Let $\xi=\eta'$.

  Part \ref{item:forcing_theorem_mSigma_om_over_M_xi}: Note that we interpret truth of $\mSigma_\om$
  formulas over $\M_\xi$ via $\widetilde{\HC}^{N,G}=\HC\cap\M_\xi$. Recall also that ``$\all^*_\mu s$''
  is defined ``$\exists k<\om\all^*_k s$'',
  and ``$\all^*_k s$'' is a first order quantifier (interpreted with $\HC\cap\M_\xi$), and ``$\exists^*_\mu s$''
   is analogous, making the interpretation of $\mSigma_\om$ formulas unambiguous (that is, we don't require that we have $\mu^{N,G}$-measure one sets in any particular model to witness the truth of the quantifier). 
  But now by the $\mSigma_0$ forcing theorem
  at $\xi$, a straightforward
  induction on the complexity of $\varphi$,
  using the $\mu$-niceness of $N|\eta$
  together with calculations like those earlier in this section,
and the definability of $\sforces{\CC}{\mathrm{int}}$ (which ensures that $N$-genericity is enough),
  establishes the $\mSigma_\om$ forcing
  theorem at $\xi$, i.e.~part  \ref{item:forcing_theorem_mSigma_om_over_M_xi}.
    
We now consider the reaining parts. Let $\varphi$ be $\mSigma_\om$
and $\pi\in(\Nm_\xi)^{<\om}$
and let
\[ x=\{n<\om\bigm|\M_\xi\sats\varphi(n,[\pi]_\xi)\}.\]
Then by the homogeneity of $\CC$,
note that $x\in\widetilde{HC}^{N,G}$
(in fact $x\in N[G\rest d]$
where $d=\supp(\pi)$).\footnote{***Note that $\sforces{\CC}{\mathrm{int}}$
also works in $N[G\rest d]$.}
Also, $\mu^{N,G}$-Turing determinacy
holds with respect to all sets of tuples of degrees which are $\mSigma_\om$-definable over $\M_\xi$, by the niceness of $N$.

Now for all $\sigma,\tau\in\Nm_{\xi+\om}$ of form
$\tau=(\xi,i,\pi)$ and $\sigma=(\xi,i',\pi')$,
the following are equivalent:
\begin{enumerate}[label=\tu{(}\roman*\tu{)}]
 \item \label{item:tau_approx_sigma}
 $\tau\approx\sigma$
\item\label{item:tau=sigma_forced} $\exists p\in G\ \Big[p\sforces{
\CC}{\mathrm{int}}\Mtilde\sats\tau=\sigma\Big]$ \item\label{item:translated_tau=sigma_forced}
$\exists p\in G\ \Big[p\sforces{\CC}{\mathrm{int}}\psi_{u=v,(i,i')}(\pi,\pi')\Big]$
\item \label{item:translated_tau=sigma_holds} $\M_\xi\sats\psi_{u=v,(i,i')}([\pi]_\xi,[\pi']_\xi)$
\item\label{item:equality_holds} $f_i^{\mu^{N,G}}(\M_\xi,T^{\M_\xi},[\pi]_\xi)=f_{i'}^{\mu^{N,G}}(\M_\xi,T^{\M_\xi},[\pi']_{\xi})$.
\end{enumerate}
Here \ref{item:tau_approx_sigma} $\Leftrightarrow$ \ref{item:tau=sigma_forced}
by definition of $\approx$ (\ref{dfn:formal_Nm^N}),
\ref{item:tau=sigma_forced}
$\Leftrightarrow$ \ref{item:translated_tau=sigma_forced}
by definition of the $\mSigma_0$ forcing relation $\sforces{\CC}{\mathrm{int}}$
\ref{dfn:psi_0} clause
\ref{psi_0_next_mSigma_0},
\ref{item:translated_tau=sigma_forced}
$\Leftrightarrow$ \ref{item:translated_tau=sigma_holds}
by induction with part \ref{item:forcing_theorem_mSigma_om_over_M_xi},and \ref{item:translated_tau=sigma_holds}
$\Leftrightarrow$ \ref{item:equality_holds}
because the algorithm $(\varphi,\vec{i})\mapsto\psi_{\varphi,\vec{i}}$
correctly translates $\mSigma_0^{\M_\xi}$
to $\mSigma_\om^{\J(\M_\xi)}$ (with $\mu$ interpreted as $\mu^{N,G}$ in both cases).

The analogous equivalence holds for $\dot{\in}$. Moreover, if $\sigma,\tau\in\Nm_{\xi+\om}$
are arbitrary, with $\sigma=(\gamma,i,\pi)$
and $\tau=(\gamma',i',\pi')$,
then we similarly have \[\sigma\approx\tau\iff  f_i^{\mu^{N,G}}(\M_{\gamma},T^{\M_\gamma},[\pi]_\gamma)=f_{i'}^{\mu^{N,G}}(\M_{\gamma'},T^{\M_{\gamma'}},[\pi']_{\gamma'}),\]
and the analogous equivalence for $\dot{\in}$. 

Parts \ref{item:ext_and_wf}--\ref{item:tau_interpreted_correctly} of the lemma at $\eta=\xi+\om$ now  easily follow
(making use of the comments above regarding $\widetilde{HC}^{N,G}$ and $\mu^{N,G}$-Turing determinacy). We leave the remaining details of these parts to the reader.

Part \ref{item:no_new_Sigma_1_truths}:
Since the previous parts hold at $\eta$,
the statement under question is well-defined.
Suppose it fails at $\eta$,
and for simplicity
it fails with respect to some
 $x=\emptyset$
(as opposed to the more general $x\in\RR^{\M_\eta}$, which
involves a straightforward relativization,
and which we leave to the reader) and some
$\Sigma_1$ formula $\varphi_1$ of the $L(\RR)$-language. In particular,
$\J_{\alphagap}\not\sats\varphi_1$.
We will use the failure to show
that in fact $\SS_{\alphagap}\sats\varphi_1$, by
showing we have something like a $\varphi_1$-witness with a strategy in $\SS_{\alphagap}$, and using this to verify that $\SS_{\alphagap}\sats\varphi_1$
in a similar manner as from a $\varphi_1$-witness.
The failure is an $\mSigma_1^{\M_\eta}$ fact (a witness to the $\mSigma_1^{\M_\eta}$ assertion is just a set of reals which encodes a sequence of models
of the right form).
Let $\varphi=\text{``}\exists z \varphi_0(z)\text{''}$,
with $\varphi_0(z)$ being $\mSigma_0$,
asserting the failure (in the mentioned  manner, and as witnessed by $\varphi_1$). Fix $\tau\in\Nm_\eta^N$,
with $\tau=(\xi,i,\sigma)$,
such that $\M_\eta\sats\varphi_0(\tau_G)$.
Then $\M_\xi\sats\psi_{\varphi_0,i}(\sigma)$,
so we can fix $p\in G$
such that 
\begin{equation}\label{eqn:p_forces_psi_varphi_0} N|\xi\sats p\sforces{\CC}{\mathrm{int}}\Mtilde\sats\psi_{\varphi_0,i}(\sigma).\end{equation}
Let $m<n<\om$ be such that
$\psi_{\varphi_0,i}$ is $\mSigma_m$,
and 
$p'\sforces{\CC}{\mathrm{int}}\Mtilde\sats\varphi'(\sigma')$
is an $\rSigma_n$ relation of $(p',\varphi',\sigma')$ when restricted to $\mSigma_0$ formulas $\varphi'$ or subformulas of $\psi_{\varphi_0,i}$ (with arbitrary $p'\in\CC$
and $\sigma'\in\Nm$;
an inspection of the definition
of $\sforces{\CC}{\mathrm{int}}$
easily reveals that there is such an $n$).

Taking $\ell<\om$ sufficiently large (say $\ell=5(m+1)+n$)
and $H=\cHull_{\ell+1}^{N|\xi}(\omega_1^N)$
and  $\pi:H\to N|\xi$ the uncollapse,
then $H$ is sound with $\rho_{\ell+1}=\omega_1^N=\omega_1^{\Pg}=\omega_1^H<\lambda^H\leq\rho_{\ell}^H$
(noting $\lambda^N\leq\rho_\om^{N|\xi}$), and we may assume $p,\sigma\in\rg(\pi)$.
It follows that $H\pins N$, so $H\pins \Lp_{\Gammag}(\Pg|\omega_1^{\Pg})$ where $\Pg$ is an $x$-mouse. Therefore $\Sigma_H\in\J_{\alphagap}$. By line (\ref{eqn:p_forces_psi_varphi_0}),
\begin{equation}\label{eqn:pbar_forces_psi_varphi_0} H\sats \bar{p}\sforces{\CC}{\mathrm{int}}\Mtilde\sats\psi_{\varphi_0,i}(\bar{\sigma}),\end{equation}
where $\pi(\pbar,\sigmabar)=(p,\sigma)$
(also $\pbar\in\CC^{H}$
and $\sigmabar\in\Nm^H$). Since $N|\xi$
satisfies the full lemma, note that $H$ does also (everything is of bounded complexity).
Moreover, $H$ is $(\mu,m)$-nice,
as $N|\xi$ is $(\mu,m)$-nice,
as $N|\eta$ is nice.

Let $\mathscr{T}$
be the set of all non-trivial $\ell$-maximal
trees $\Uu$  on $H$, via $\Sigma_H$, which are based on $H|\delta_n^H$ for some $n<\om$, have successor length
and are such that $b^\Uu$ does not drop.
Given $\Uu,\Vv\in\mathscr{T}$,
write $\Uu\leq\Vv$
iff $\Vv=\Uu\conc\Vv'$
for some $\Vv'$ on $M^\Uu_\infty$
which is above $\delta_{n}^{M^\Uu_\infty}$
where $n$ is least such that $\Uu$ is based on $H|\delta_n^H$.
Let $\Uu\in\mathscr{T}$
and $H'=M^\Uu_\infty$. Let $\mathscr{G}_\Uu$ be the set of
all $g$ such that for some $k<\om$,
$g$ is  $(H',\Coll(\om,\delta_k^{H'}))$-generic. Let $g\in\mathscr{G}_{\Uu}$.
Then $T_{\Uu,g}$ denotes the set of all pairs $(\psi,x)$ such that $\psi$ is $\Sigma_1$
in the $L(\RR)$-language and $x\in\RR\cap H'[g]$ and there is a pre-$\psi(x)$-witness
$R$ such that $R\pins H'[g]$,
where $H'[g]$ is considered as an
$(H'|\delta_k^{H'},g)$-premouse.
Let $T_{\Uu}=\bigcup_{g\in\mathscr{G}_\Uu}T_{\Uu,g}$. Let $T_{\geq\Uu}=\bigcup_{\Vv\geq\Uu}
T_{\Vv}$ (note that $\Vv\geq\Uu$ implies $\Vv\in\mathscr{T}$). Note that
for each $\Uu\in\mathscr{T}$,
there is $\alpha<\alphagap$
such that $\J_\alpha\models T_{\geq\Uu}$,
and letting $\alpha_{\Uu}$ be the least such, for all $\Uu,\Vv\in\mathscr{T}$, we have
\[ \Uu\leq\Vv\implies T_{\geq\Uu}\supseteq T_{\geq\Vv}\implies \alpha_{\Uu}\geq\alpha_{\Vv}. \]
So (with an application of DC)
we can fix $\Uu_0\in\mathscr{T}$
such that for all $\Vv\in\mathscr{T}$
with $\Uu_0\leq\Vv$, we have  $\alpha_{\Vv}=\alpha_{\Uu_0}$. Let $H_0=M^{\Uu_0}_\infty$ and let $j_0$ be such that $\Uu_0$ is based on $H|\delta_{j_0}^H$.
Let $\alpha_0=\alpha_{\Uu_0}$.
Note that $\alpha_0$ starts an S-gap.
Let $\beta_0<\alphagap$ be the end of that S-gap.

We claim that $T_{\geq\Uu_0}=\Th_{\Sigma_1}^{\J_{\alpha_0}}(\RR)$. For certainly
$T_{\geq\Uu_0}\sub\Th_{\Sigma_1}^{\J_{\alpha_0}}$. But by the minimality of $\alpha_0$, $T_{\geq\Uu_0}$ is also cofinal in $\Th_{\Sigma_1}^{\J_{\alpha_0}}(\RR)$ with respect to  the standard prewellorder of $\Sigma_1$ truth, but then the usual
``comparison of ranks'' argument
shows that equality holds.
(That is, by properties
of $\Pg$ and elementarity,
$T_{\geq\Uu_0}$
is ``simply closed''.
We leave the precise formulation of this to the reader, but it should mean
essentially that it is closed under
straightforward logical deduction.
Suppose for example
that $\alpha_0$ is a limit of limits, and
 $\J_{\alpha_0}\sats\psi(y)$ for some $\Sigma_1$ formula $\psi$
and $y\in\RR$. Let $\gamma$
be least such that $\J_\gamma\sats\psi(y)$,
so $\gamma<\alpha_0$ (as $\alpha_0$ is a limit of limits).
Let $\psi'(y')\in T_{\geq\Uu_0}$ be such
that the least $\gamma'$ such that $\J_{\gamma'}\sats\psi'(y')$, has $\gamma'>\gamma$.
Then because $T_{\geq\Uu_0}$
is simply closed,
either
\begin{enumerate}[label=(\roman*)]
 \item\label{item:option_i}
 $\text{``}\exists\beta\in\OR \Big[\SS_\beta\sats\psi'(y')\wedge\psi(y)\Big]\text{''}\in T_{\geq\Uu_0}$
\item\label{item:option_ii}  $\text{``}\exists\beta\in\OR\ \Big[\SS_\beta\sats\psi'(y')\wedge\neg\psi(y)\Big]\text{''}\in T_{\geq\Uu_0}$.
\end{enumerate}
But
since $\gamma<\gamma'$
and $\J_{\alpha_0}\sats T_{\geq\Uu_0}$,
\ref{item:option_ii} does not hold, so \ref{item:option_i} holds.
But then again by simple closure,
$\psi(y)\in T_{\geq\Uu_0}$. The case that
$\alpha_0$ is a successor-limit
is likewise but with the $\Ss$-hierarchy.)

Working in a generic extension of $V$, let $H'$ be an $\RR$-genericity iterate
of $H_0$, formed with a tree above $\delta_{j_0}^{H_0}$.
Then $H'$ satisfies the full lemma,
by the elementarity of
the iteration map
$j:H\to H'$.
Let $G'$ be a generic
witnessing that $H'$ is an $\RR$-genericity iterate, and one which meets
all dense 
subsets $D\sub\CC^{H'}$ which are $\bfrSigma_{\ell}^{H'}$-definable (not just those in $H'$), with $j(\pbar)\in G'$;
a construction like that for Remark \ref{rem:get_RR-genericity_generic}
works for this, since these
$D$ are amenable to $H'|\lambda^{H'}$.
So $\mathscr{N}=\Mtilde^{H',G'}$ is well-defined,  $\lambda^{H'}=\omega_1$,
and note that
$\Mtilde_{\omega_1}^{H',G'}=(\HC,T)$
where $T=\Th_{\Sigma_1}^{\J_{\alpha_0}}$.
Let $\beta_0$ ens the S-gap starting at $\alpha_0$, so $\beta_0<\alphagap$.
The fact that $H'$ is $\mu$-nice
and satisfies part \ref{item:no_new_Sigma_1_truths}
easily gives that $\mathscr{N}=\M^{[\alpha_0,\beta_0]}_\gamma$
where $\gamma=\OR^{\mathscr{N}}$,
and $\widehat{L}(\RR)^{\mathscr{N}}=\SS_\beta$
for some $\beta\in[\alpha_0,\beta_0]$,
and since $\OR^{\mathscr{N}}\leq\beta$,
therefore $\mathscr{N}\in\J_{\alphagap}$.
Now
\[ H'\sats j(\pbar)\sforces{\CC}{\mathrm{int}}\Mtilde\sats\psi_{\varphi_0,i}(j(\sigmabar)_{G'}),\]
by line (\ref{eqn:pbar_forces_psi_varphi_0}),
and $H'$ is $(\mu,m)$-nice,
and an inspection of $\sforces{\CC}{\mathrm{int}}$ (and our choice of $G'$) therefore gives that
\[ \Mtilde^{H',G'}\sats\psi_{\varphi_0,i}(j(\sigmabar))\]
and (again using that $H'$ is $(\mu,m)$-nice, and that $\J^\mu(H')\in\J_{\alphagap}\models$``Turing determinacy'')
\[ \J^\mu(H')\sats\varphi_0(f_i(j(\sigmabar)_{G'}))\]
(it's not relevant here whether $\OR^{\mathscr{N}}<\beta_0^*$ or not,
where $\beta_0^*$ is the ``end of S-gap''
in the $\M^{[\alpha_0,\beta_0]}$-hierarchy). So
\[ \J^\mu(H')\sats\exists z\varphi_0(z),\]
and recall that this says there is a sequence of sets of reals coding models witnessing that $L(\RR)\sats\varphi_1$.
But then since $\J^\mu(H')\in\J_{\alphagap}$,
we get $\J_{\alphagap}\models\varphi_1$,
a contradiction, completing the proof of part \ref{item:no_new_Sigma_1_truths}.

This completes the induction
up to $\eta_\infty$.
Now suppose $\eta_\infty<\OR^N$, i.e.~$N$ is not $\mu$-nice. The $\mu$-homogeneity of $N$ follows directly from properties of $\Pg$,
so it is $\mu$-determinacy which fails for $N|(\eta_\infty+\om)$. By homogeneity of $\CC$ it follows that we can fix
$(\varphi,\sigma)$ and $d<\om$ such that ($*$) $d\geq\supp(\pi)$
and for all $m\in[d,\om)$
there are $k<\om$ and $\vec{\delta}\in[\Delta_{\geq m}^N]^{2k}$ 
such that $N|\eta_0$ satisfies that $\CC_d^N$ forces
\[\Big(\all^*_{\vec{\delta}}s\ \sforces{\CC_{\mathrm{tail}}}{\mathrm{int}}\Mtilde\sats\varphi(s,\sigma)\Big)\wedge \Big(\all^*_{\vec{\delta}}s\ \sforces{\CC_{\mathrm{tail}}}{\mathrm{int}}\Mtilde\sats\neg\varphi(s,\sigma)\Big). \]
But now taking $\ell<\om$ large enough
and $H=\cHull_{\ell+1}^{N|\eta_\infty}(\emptyset)$, we can argue much like in the proof of part \ref{item:no_new_Sigma_1_truths} to obtain a failure of Turing determinacy inside $\SS_{\alphagap}$, a contradiction.
\end{proof}

\begin{dfn}
Let $N$ be an $\om$-small premouse
with $\om$ Woodins and $\lambda^N<\OR^N$.
We say that $N$ is \emph{$\Mtilde$-good}
iff
$N$ is $\mu$-nice and the conclusion of Lemma \ref{lem:Sigma_P-iterate_good} holds for $N$.
\end{dfn}

Note that there is an $\rPi_1$ formula $\psi$
such that if $\lambda^N<\OR^N$
then $N$ is $\Mtilde$-good iff $N\sats\psi(\lambda^N)$.

\begin{lem}\label{lem:OR^N_leq_beta^*}
Work in a generic extension of $V$.
 Let $N$ be an $\RR$-genericity $\Sigma_{\Pg}$-iterate, as witnessed by $G$.
 Then $\OR^N\leq\beta^*$ and $\Mtilde^{N,G}=\M_{\OR^N}$. 
\end{lem}
\begin{proof}
 The previous lemma applies,
 so $\Mtilde^{N,G}$ is well-defined etc;
 That $\Mtilde_{\lambda^N}^{N,G}=\M_{\om_1}$ is clear by nature of $\Pg,G$. 
 So suppose $\beta^*<\OR^N$.
 By Remark  \ref{rem:between_hierarchies},
 we have $f(\beta^*)\leq\betag< f(\beta^*)+\om^2$, 
where $f$ is the function as there. 
It follows that there is $n<\om$
such that for a cone  of reals $x$,
there is a real $y$
which is $\mSigma_{n+1}^{\M_{\beta^*}}(\{x\})$ but $y\notin\OD_{\alphagap}(x)$.
Let $x_0$ be at the base of such a cone $C$,
and let $g=G\rest\CC_k^N$
with $k$ large enough that $x_0\in N[g]$.
Let $g'=G\rest\CC_{k+1}^N$.
Let $x$ be a real equivalent to $(g',N|\delta_{k+1}^N)$. Then $x_0\leq_T x$.
So let $y$ witness that $x\in C$.
Then $y\notin N[g']$.
But by homogeneity of $\CC_{\mathrm{tail}}$ and the $\mu$-definability
of $y$ over $\M_{\beta^*}$ from $x$,
we get that $y$ is definable
from $x$ over $(N|\beta^*)[g']$,
so $x\in N[g']$, a contradiciton.
\end{proof}

\subsection{The generic premouse}
In this section we assume that $\SS_{\alphagap}$ is admissible.

Having shown that we can realize
some initial segment of $\M_{\beta^*}$
as a derived model of $\RR$-genericity iterates $N$ of $\Pg$,
we want to 
arrange (by choosing $N$ appropriately) that that initial segment
is in fact the full $\M_{\beta^*}$,
and to arrange  fine structural
correspondence between $N$ and $\M_{\beta^*}$. To achieve this,
we will arrange
that $N$ is generic over $\M_{\beta^*}$,
and more generally, that $N|\gamma$
is generic over $\M_\gamma$.

We will force over $\M_{\beta^*}$ with Turing Prikry forcing $\PP$, for
forcing an $\om$-small  premouse containing $\om$ Woodins, which is $\Gammagap$-exact below the supremum $\lambda$ of its Woodins (cf.~Definition \ref{dfn:mtr}),
where the forcing conditions are pairs $(p,\Xvec)$ in which
$\Xvec$ is a countable Boolean
combination of uniformly $\bfmSigma_{n^*}^{\M_{\beta^*}}$
sets (have to define $n^*$; see Definition \ref{dfn:PP} for details). We now proceed toward defining the forcing.
Definitions \ref{dfn:suitable}
and \ref{dfn:short_tree} are standard (see, for instance, \cite{HOD_as_core_model}).

\begin{dfn}\label{dfn:suitable}
Let $X\in\HC$ be
transitive and let $P$ be an $\om$-small $X$-premouse.

For $k<\om$, we say 
that $P$ is \dfnemph{$k$-suitable-like} iff
$P$ has exactly $k$ Woodins $\delta_0<\ldots<\delta_{k-1}$ strictly above
$\rank(X)$, and letting $\delta_{-1}=\rank(X)$, every set in $P$
has cardinality $\leq\delta_{k-1}$, and $P\sats\ZF^-$. We say that $P$ is 
\dfnemph{$\om$-suitable-like}
iff $P$ has $\om$ Woodins and $\OR^P=\lambda^P$.

Recall Definition \ref{dfn:mtr}.
We say that $P$ is \dfnemph{bounded}
iff $P$ is $\delta$-bounded for all $\delta<\OR^P$,
and \dfnemph{full} iff $P$ is $\delta$-full
for all strong cutpoints $\delta<\OR^P$.

Let $k\leq\om$. We say that $P$ is \dfnemph{$k$-suitable} iff $P$ is $k$-suitable-like, bounded and full. We say that $P$ is 
\dfnemph{suitable}
iff $P$ is $k$-suitable for some $k\leq\om$.

Note that the
definitions above all use the fixed pointclass
$\Gammagap$ implicitly.
If we want to make the same definitions
with some other pointclass $\Gamma$,
then we add ``$\Gamma$-'' as a prefix,
as in ``$\Gamma$-full'', etc.
\end{dfn}

\begin{rem}
If $P$ is almost mtr-suitable
then $P|\lambda^P$ is $\om$-suitable.

If $P$ is $k$-suitable then no $R\pins P$ is
$k$-suitable. For suppose not, and $k<\om$.
Easily $k>0$.
Note that $\delta_{k-1}^R\notin\{\delta_0^P,\ldots,\delta_{k-1}^P\}$,
hence is not Woodin in $P$.
Let $S\pins P$ be least such that 
$R\ins S$ and $\rho_\om^S\leq\delta_{k-1}^R$, and note that $\delta_{k-1}^R$ is a strong cutpoint of $S$,
which contradicts
the $\delta_{k-1}^R$-boundedness of $P$. If $k=\om$ it is similar.
\end{rem}

\begin{dfn}\label{dfn:short_tree}
Let $P$ be a full $k$-suitable-like premouse,
where $k\leq\om$.

Let $\Tt$ be a $0$-maximal (equivalently, $\om$-maximal) iteration tree on $P$.
We say $\Tt$ is \dfnemph{short} iff for every limit $\lambda\leq\lh(\Tt)$ there
is\[Q\pins\Lp_{\Gammagap}(M(\Tt\rest\lambda))\] such that
$Q$ is a Q-structure for $M(\Tt\rest\lambda)$,
and if $\lambda<\lh(\Tt)$ then $Q\ins M^\Tt_\lambda$.
We say $\Tt$ is  \dfnemph{maximal} iff $\Tt$ has limit length and  every proper segment of $\Tt$ is short, but $\Tt$ is not short.

The \dfnemph{short-tree strategy}  $\Psi_{\Gamma P}$ for $P$ is the putative partial $0$-maximal
iteration strategy $\Psi$ such that, given a countable limit length, short tree $\Tt$ on 
$P$, $\Psi(\Tt)$ is the $\Tt$-cofinal branch $b$ such that $Q(\Tt,b)$ exists and
$Q(\Tt,b)\pins\Lp_\Gamma(M(\Tt))$, if such $b$ exists, and $\Psi(\Tt)$ is
undefined otherwise. (Standard arguments show that there is at most one such
$b$.)

We say that $P$ is \dfnemph{short-tree-iterable} iff whenever $\Tt$ is a
countable limit length, short tree on $P$ via $\Psi_{\Gamma P}$, then 
$\Psi_{\Gamma P}(\Tt)$ is
defined, and every putative tree via $\Psi_{\Gamma P}$ is an iteration tree.

Suppose $P$ is short-tree iterable and let $\Psi=\Psi_{\Gamma P}$. We say that
$\Psi$ is \dfnemph{fullness preserving} iff for every successor length tree $\Tt$ via 
$\Psi$, if $b^\Tt$ does not drop then $M^\Tt_\infty$ is full.
We say that $\Psi$, and also $P$, are \dfnemph{stable}\footnote{This
is clearly related to the notion
of $\Gamma$-stability
from Definition \ref{dfn:descent}
(which continues after Remark \ref{rem:soundness_for_later_ones}).
The present definitions can also
be relativized to a pointclass
$\Gamma$ with the prefix ``$\Gamma$-'',
and thus, we are presently introducing
a new definition of \emph{$\Gamma$-stable}.
There is no formal ambiguity
between the current notion
and that of \ref{dfn:descent},
because in \ref{dfn:descent},
triples $(N,n,\eta)$
are $\Gamma$-stable, not premice $N$.
However, we also used the
terminology \emph{$\Gamma$-stable}
informally for premice $N$,
omitting explicit specification of 
 $(n,\eta)$. But there is still
 no ambiguity, because in \ref{dfn:descent},
$N$ was a projecting structure with $\rho_{n+1}^N\leq\eta<\rho_n^N$,
whereas here $P$ is not projecting.} iff 
$P$ is $k$-suitable (hence bounded) and for every successor length tree $\Tt$ via $\Psi$,
\begin{enumerate}[label=(\roman*)]
\item\label{item:stable_no_drop} if $b^\Tt$ does not drop then $M^\Tt_\infty$ is $k$-suitable,
and 
\item\label{item:stable_drop} if $b^\Tt$ drops, $R\pins S\ins M^\Tt_\infty$,
 $\nu(\Tt)\leq\OR^R$,
 $R$ is a strong cutpoint of $S$, 
 and there is $n<\om$ such that
 $\rho_{n+1}^S\leq\OR^R<\rho_n^S$,
 then $S\pins\Lp_{\Gammagap}(R)$.
\end{enumerate}

The preceding definitions also relativize in the obvious way to trees and
iteration strategies
above some $\delta<\OR^P$, or acting on some interval $[\delta,\gamma)$ where $\delta<\gamma\leq\OR^P$.
In this way we define
\dfnemph{stable above $\delta$}, etc.
\end{dfn}

\begin{rem}
Suppose $P$ is a stable short-tree-iterable  $k$-suitable premouse.
Then a stronger variant of
 clause \ref{item:stable_drop} of the definition
 of stability  holds:
 suppose $\Tt,R,S,n$ have the same properties
 as there, except that $R$ is only assumed to be a cutpoint of $S$, not a strong cutpoint. Then we claim there is an above-$\OR^R$, $(n,\om_1+1)$-strategy
 for $S$ in $\SS_{\alphagap}$.
For suppose otherwise,
and let $S$ be the least such segment of $M^\Tt_\infty$, again as witnessed by $n$.
Then by \ref{item:stable_drop},
$R$ is not a strong cutpoint of $S$,
so there is $E\in\es_+^S$
such that $\kappa=\crit(E)=\OR^R$.
We have that $S|\kappa^{+S}$ is passive,
since $\kappa$ is a cutpoint of $S$.
By minimality of $S$
every proper segment of 
$S|\kappa^{+S}$ is above-$\kappa$
iterable in $\SS_{\alphagap}$.
It follows (using admissibility
to collect together iteration strategies)
that $\kappa^{+S}$ is not a cutpoint
of $S$ (otherwise $\kappa^{+S}$
is a strong cutpoint,
so by \ref{item:stable_drop},
$S$ is above-$\kappa^{+S}$-iterable
in $\SS_{\alphagap}$,
but then using admissibility
to collect together strategies,
$S$ is above-$\kappa$-iterable
in $\SS_{\alphagap}$, contradiction).
Let $E\in\es_+^S$ be least such that
$\crit(E)=\kappa<\kappa^{+S}<\lh(E)$.
Let $\Tt'$ be the ($0$-maximal)
tree determined by $\Tt'\conc\left<E\right>$.
Let $S'=M^{\Tt'}_\infty=\Ult_n(S,E)$
(recall $b^\Tt$ drops,
so $\deg^{\Tt'}_\infty=n$ even if $M^\Tt_\infty=S$)
and $R'=S|\kappa^{+S}$.
Then $b^{\Tt'}$ drops
and the hypotheses of \ref{item:stable_drop}
apply to $(\Tt',S',R',n)$,
and $\nu(\Tt')=\kappa^{+S}$,
so $S'\pins\Lp_{\Gammagap}(R')$,
so $S'$ is above-$\kappa^{+S}$,
$(n,\om_1+1)$-iterable in $\SS_{\alphagap}$,
but then this induces an above-$\kappa$,
$(n,\om_1+1)$-strategy for $S$ by copying,
which is therefore in $\SS_{\alphagap}$,
a contradiction.
\end{rem}

\begin{rem}
 One might expect that \ref{item:stable_no_drop} in the definition of stability would imply \ref{item:stable_drop}, considering common arguments involving an analysis of fatal drops in an iteration tree.  But let us
 point out a situation where this does not seem to work. Suppose there are $\rho<\delta<\zeta<\OR^P$
 such that $\delta$ is a limit cardinal of $P|\zeta$ and a strong cutpoint of $P|\zeta$, but $\rho=\rho_\om^{P|\zeta}$. Suppose
 there are cofinally many $\alpha<\delta$
 such that $P|\alpha$ is active with a $P$-total extender. Suppose $\Tt$ is a successor length tree on $P$ which is based on the interval $[\rho,\delta]$. So $\Tt$ drops in model, and can be considered as a tree on $P|\zeta$. By $k$-suitability, we know that $P|\zeta\pins\Lp_\Gamma(P|\delta)$. But suppose
 that $b^\Tt$ does not drop further in model below the image of $P|\zeta$, so $M^\Tt_\infty$ is an ultrapower of $P|\zeta$, and is $\delta'$-sound
 where $\delta'=i^\Tt_{0\infty}(\delta)$.  Then it does not seem obvious that $M^\Tt_\infty\ins\Lp_\Gamma(M^\Tt_\infty|\delta')$.
\end{rem}

\begin{dfn}
Let $P$ be a full  $k$-suitable-like premouse, where $k\leq\om$.
Let $N$ be a premouse. We say that $N$ is a \dfnemph{pseudo-iterate} of $P$ iff
there is an iteration tree $\Tt$ via $\Psi_{\Gammagap P}$ such that either:
\begin{enumerate}[label=--]
 \item  $\Tt$ has successor length and $N=M^\Tt_\infty$, or
 \item 
 $\Tt$ has limit length, is maximal
 and $N=\Lp_\Gamma(M(\Tt))$.
 \end{enumerate}
 Note that $\Tt$ is uniquely determined by $P,N$.

 A pseudo-iterate
$N$ of $P$ is
\dfnemph{pseudo-non-dropping} iff, letting $\Tt$ be as above, if $\Tt$ has successor length then $b^\Tt$ does 
not drop. 
\end{dfn}

\begin{lem}\label{lem:pseudo_comp_gen_it}
Let $0<k<\om$ and let
$\mathscr{P}=\left\{P_n\right\}_{n<\om}$ be a  countable set of $k$-suitable-premice.
Suppose that
\[ P_n|\delta_{k-2}^{P_n}=P_m|\delta_{k-2}^{P_m} \]
for each $m,n$; let $\delta_{k-2}$ be the common $\delta_{k-2}^{P_m}$.
Suppose that for each $n$, $P_n$ is above-$\delta_{k-2}$ short-tree-iterable and
stable above $\delta_{k-2}$.
Let $\Psi_{n}$ be the above-$\delta_{k-2}$ short-tree-strategy for $P_n$. Let $x\in\RR$. Then there is a $k$-suitable premouse $P$ which is an
above-$\delta_{k-2}$
pseudo-non-dropping pseudo-iterate of  $P_n$, and $x$ is $P$-generic for
$\BB_{\delta_{k-1}^P}^P$.
\end{lem}
\begin{proof}
For simplicity, assume $k=1$.
We compare the $P_n$'s, folding in genericity iteration. Such constructions
have frequently been used in this context without much explanation (cf.~for example \cite{HOD_as_core_model}), but they need a little more care than might initially be apparent, 
because extenders $E$ used for comparison can cause
drops (maybe temporarily), and create new inaccessibles below $\lh(E)$, clouding what should be meant by ``genericity iteration''. A related kind of construction
was given in detail in \cite{backcomp},
but that construction also
involves other details
which are irrelevant here.
Constructions closer to our present situation 
were also given in detail in \cite{odle_v2}. But for the reader's convenience
we repeat the relevant ideas here.

We
use
the
following algorithm for extender selection.\footnote{If the $P_n$ are iterable
for stacks of normal pseudo-trees, one could instead compare first, and then do
a genericity iteration.}

We define padded normal trees $\Tt_n$ on $P_n$, with $\Tt_n$ via
$\Psi_n$. We only allow $\eta=\pred^{\Tt_n}(\gamma+1)$  if
$E^{\Tt_n}_\eta\neq\emptyset$. Suppose for some $\eta$, we have defined
$\Tt_n\rest\eta+1$ for all $n$. Let $R_n=M^{\Tt_n}_\eta$. 

If (a) there is $n_0<\om$
such that $R_{n_0}\ins R_n$ for all $n<\om$,
then let $\xi_\eta=\OR^{R_{n_0}}$.

If (b) otherwise,
let $\xi_\eta$ be the least $\xi$ such that for
some $m,n$, we have $R_m|\xi\neq R_n|\xi$.

Now we proceed with both cases. Suppose first that (c) $x\sats\varphi$ for
all extender algebra axioms $\varphi$ induced by extenders $E$ such that for
some $\gamma\leq\xi_\eta$,
\begin{enumerate}[label=-]
\item if $k>1$ then $\delta_{k-2}^{P_m}<\lh(E)$ for all $m$,
 \item $E=\es_\gamma(R_m)$ for all $m$, and
 \item $\nu(E)$ is inaccessible in $R_m||\xi_\eta$.
\end{enumerate}
Then if (a) holds we stop the process,
and if (b) holds we set $E^{\Tt_m}_\eta=\es_{\xi_\eta}^{R_m}$
for each $m$.

Now suppose that (c) fails. Let $E$
witness this with $\lh(E)$ minimal. Then we set $E^{\Tt_m}_\eta=E$ for each $m$.

If we reach a limit stage $\eta<\om_1$ such that (d) for some $n$,
$E^{\Tt_n}_\gamma\neq\emptyset$ for cofinally many $\gamma<\eta$, and $\Tt_n$
is maximal, then we terminate the process. At other limit stages $\eta<\om_1$ we continue
using our strategies.

It is easy to see that for each $m,n$ and $\eta,\gamma$, if $\eta<\gamma$ and
$E^{\Tt_m}_\eta\neq\emptyset\neq E^{\Tt_n}_\gamma$, then
$\lh(E^{\Tt_m}_\eta)\leq\nu(E^{\Tt_n}_\gamma)$, and in particular, each tree is
normal.

\begin{clm}\label{clm:termination_by_a_and_c}
If the process terminates due to the conjunction of \tu{(}a\tu{)} and \tu{(}c\tu{)} above, then for all $m,n$,
$b^{\Tt_m}$ does not drop, and $M^{\Tt_m}_\infty=M^{\Tt_n}_\infty$.\end{clm}

\begin{proof} We first show that
$b^{\Tt_{n_0}}$ does not drop,
where $n_0$ is as in (a). Suppose otherwise.
Then $R_{n_0}$ is not sound,
so $R=R_{n_0}=R_n$ for all $n<\om$,
so  for each $n$,
$b^{\Tt_n}$ drops and $\core_\om(R)\pins
M^{\Tt_n}_{\beta}$ for some $\beta$. Choosing $\beta_n$ to be the largest
such $\beta$ (recall that $\Tt_n$ is padded), $E^{\Tt_n}_{\beta_n}\neq\emptyset$
and $E^{\Tt_n}_{\beta_n}\in\es_+(\core_\om(R))$. Since $R_n=R$ for all $n$, it
follows that $E^{\Tt_n}_{\beta_n}$ and $\beta_n$ are independent of $n$. Write
$\beta=\beta_n$ and $E=E^{\Tt_n}_{\beta_n}$. Therefore $E$ was chosen at stage
$\beta$ due to a bad extender algebra axiom; i.e., at stage $\beta$, (c) above
failed. So $\nu=\nu(E)$ is
inaccessible in $M^{\Tt_n}_\beta||\xi_\beta$.

Now there is $\delta$ such that for each $n$, we have $\beta=\pred^{\Tt_n}(\delta+1)$ and
\[ i^{*\Tt_n}_{\delta+1,\eta}=\pi:\core_\om(R)\to R\]
is the core embedding, with $\delta$ and $F=E^{\Tt_n}_\delta$ both independent
of $n$, and $\crit(\pi)=\crit(F)<\nu$, and in each tree, $F$
triggers the drop in model to $\core_\om(R)=M^{*\Tt_n}_{\delta+1}$. Since
$\nu$ is a
cardinal of $M^{\Tt_n}_\beta||\xi_\beta$, but $F$ triggers the drop in model,
therefore $\nu$ is not a cardinal of $M^{\Tt_n}_\beta$, so
$\xi_\beta<\OR(M^{\Tt_n}_\beta)$, and so case (b) held at stage $\beta$. But also
since $\nu$ is a cardinal of $M^{\Tt_n}_\beta|\xi_\beta$, we have
$M^{\Tt_n}_\beta|\xi_\beta\ins\core_\om(R)$. Therefore $M^{\Tt_n}_\beta|\xi_\beta$ is
independent of $n$. This contradicts the choice of $\xi_\beta$.

So $b^{\Tt_{n_0}}$ does not drop.
So by stability of $\Psi_{n_0}$, $M^{\Tt_{n_0}}_\infty$ is $k$-suitable.
So if $b^{\Tt_n}$ does not drop,
then by stability of $\Psi_n$,
$M^{\Tt_n}_\infty=M^{\Tt_{n_0}}_\infty$.
Suppose that $b^{\Tt_n}$ drops,
so $M^{\Tt_{n_0}}_\infty\pins M^{\Tt_n}_\infty$. Let $\delta=\delta_{k-1}^{M^{\Tt_{n_0}}_\infty}$,
the largest (Woodin) cardinal of $M^{\Tt_{n_0}}_\infty$. We have
$M^{\Tt_{n_0}}_\infty=\Lp_{\Gammagap}(M^{\Tt_{n_0}}_\infty|\delta)$.
Note that $\nu(\Tt_n)\leq\delta$.
Since $b^{\Tt_n}$ drops, there is  $(S,d)$ such that
$M^{\Tt_{n_0}}_\infty\ins S\ins M^{\Tt_n}_\infty$
and $\rho_{d+1}^S\leq\delta<\rho_d^S$,
but then taking $(S,d)$ least such,
$\delta$
is Woodin in $S$, and hence a strong cutpoint of $S$. So by condition \ref{item:stable_drop}
of stability (for $\Psi_n$), $S\pins\Lp_{\Gammagap}(S|\delta)$,
contradicting that $M^{\Tt_{n_0}}_\infty=\Lp_{\Gammagap}(S|\delta)$.
\end{proof}

So if the process terminates via the conjunction of (a) and (c), then setting $P=M^{\Tt_{n_0}}_\infty$, we are done.

Suppose the process terminates via (d),
as witnessed by $n_0$ (so $\Tt_{n_0}$ is non-padded cofinally often below the limit stage $\eta$, and $\Tt_{n_0}$ is maximal).
If $\Tt_n$ is also non-padded cofinally below $\eta$, then we are done, so suppose otherwise.
We claim that $M^{\Tt_n}_\infty=\Lp_{\Gammagap}(M(\Tt_{n_0}))$.
For let us first observe that
$\Lp_{\Gammagap}(M(\Tt_{n_0}))\ins M^{\Tt_n}_\infty$.
By stability, $M^{\Tt_n}_\infty\npins\Lp_{\Gammagap}(M(\Tt_{n_0}))$. But then
if $\Lp_{\Gammagap}(M(\Tt_{n_0}))\nins M^{\Tt_n}_\infty$ then the comparison
of $\Lp_{\Gammagap}(M(\Tt_{n_0}))$ with $M^{\Tt_n}_\infty$ is non-trivial,
and noting that the comparison is above $\delta(\Tt_{n_0})$
and short tree strategies
are enough to perform it,
it succeeds, which is impossible.
So $\Lp_{\Gammagap}(M(\Tt_{n_0}))\ins M^{\Tt_n}_\infty$. Now suppose
that $\Lp_{\Gammagap}(M(\Tt_{n_0}))\pins M^{\Tt_n}_\infty$. Then by stability for $\Psi_n$, $b^{\Tt_n}$ drops.
But now one reaches a  contradiction like at the end of the proof of Claim \ref{clm:termination_by_a_and_c}.
So $M^{\Tt_n}_\infty=\Lp_{\Gammagap}(M(\Tt_{n_0}))$, which is sound,
so $b^{\Tt_n}$ does not drop, which suffices.

These are the only two possible
kinds of termination.
But the usual arguments show that the process must terminate in countably many steps, so we are done.
\end{proof}

We will often want to apply the preceding lemma assuming that some $P_n$ is the 
$k$-suitable segment of a non-dropping $\Sigma_\Pg$-iterate of $\Pg$. In this case, let 
$\Uu$
be the resulting successor length normal tree on $\Pg$. By $\Gammagap$-stability of $\Pg$ (in the sense of Definition \ref{dfn:descent}), $b^\Uu$ does not drop, and $P=M^\Uu_\infty|i^{\Uu}((\delta_{k-1}^\Pg)^{+\Pg})$.

The following definition is the obvious adaptation of the analogue in \cite{HOD_as_core_model}.

\begin{dfn}[$Q_s$]\label{dfn:Q_s}
Given $n<\om$ and $s\in\Dd^n$ of sufficiently rapid growth, we will define an
$n$-suitable premouse $Q_s$ (over $x^{\Pg}$). 
Moreover, for
each $m<n$, $Q_{s\rest m}$ will also
be defined and will be the $m$-suitable segment of $Q_{s}$.

Define
$Q_\emptyset=\Lp_{\Gammagap}(x^{\Pg})$.
Now fix $s\in\Dd^n$ where $n\geq 1$. Let $m+1<\lh(s)$ be such that $Q_{s\rest m}$ is defined.
Then $Q_{s\rest m+1}$ is
defined iff there is an $(m+1)$-suitable- premouse $N$ such that $N$ has $m$-suitable
segment $Q_{s\rest
m}$, $N$ is coded by some real $z\in s(m)$,  is short-tree-iterable
above $\delta_{m-1}^N$ and stable above $\delta_{m-1}^N$. In this case, we let $Q_{s\rest m+1}$ be the
premouse $P$ ouput by the pseudo-comparison/genericity iteration
process of the proof of \ref{lem:pseudo_comp_gen_it}, for making some
(equivalently, all) reals $x\in s(m)$ generic.\footnote{Given Turing equivalent
reals $x_1,x_2$, and an active premouse $R$, it is easy to show that $F^R$
induces an extender algebra axiom false for $x_1$ iff $F^R$ induces one false
for $x_2$. So the iteration trees produced by the proof of
\ref{lem:pseudo_comp_gen_it} are independent of $x\in s(m)$.}
\end{dfn}

\begin{lem}\label{lem:Q_s_definability}
The function $s\mapsto Q_{s}$, and its domain $D$, are both $\Delta_1^{\M_{\om_1}}(\{x^{\Pg}\})$.\end{lem}
\begin{proof}
The assertion ``$N$ is a $k$-suitable-like premouse''
is $\Pi^1_1$ (of variable $N$). The assertion ``$N$ is full
and short-tree iterable'' is $\Pi_1^{\J_{\alphagap}}$. The assertion
``$N$ is bounded''
is $\Sigma_1^{\J_{\alphagap}}$.
The assertion ``the short-tree strategy is stable'' is $\Pi_1^{\J_{\alphagap}}\wedge\Sigma_1^{\J_{\alphagap}}$. (With
$\Pi_1^{\J_{\alphagap}}$ we can assert that the short-tree strategy is fullness preserving; with $\Sigma_1^{\J_{\alphagap}}$ we can  handle the boundedness aspect of property
\ref{item:stable_no_drop},  and handle property \ref{item:stable_drop}.) 
So $\{s\in D\bigm|\lh(s)=1\}$ is $\Delta_1^{\M_{\omega_1}}(\{x^{\Pg}\})$. For the domain where $\lh(s)>1$, one must proceed recursively, computing first $Q_{s\rest(\lh(s)-1)}$. The assertion ``$P$ is the result of the pseudo-comparison/genericity iteration of all (the relevant) $N$ (as described in \ref{dfn:Q_s})'' is uniformly arithmetic over Boolean-combination-$\Sigma_1^{\J_{\alphagap}}(\{y\})$ for any real $y$ coding $P$ and all reals $\leq_T s$. (With $\Sigma_1^{\J_{\alphagap}}$ we can express that the trees use the correct strategy, and with $\Sigma_1^{\J_{\alphagap}}\wedge\Pi_1^{\J_{\alphagap}}$ we can express that $P$ is full, and bounded at its top Woodin.)  Since $\M_{\om_1}$ has the predicate $T$ at its disposal, it easily follows that $D$ and the graph of  $s\mapsto Q_s$
are $\Delta_1^{\M_{\om_1}}(\{x^{\Pg}\})$.
\end{proof}

\begin{dfn}
Let $s\in\Dd^{<\om}$ and $Y\sub\Dd^{<\om}$. We say $Y$ is a
\emph{tree} iff $Y$ is closed under initial segment.

If $Y$ is a tree, we say that $Y$ is \dfnemph{$\mu$-cone-splitting with stem $s$} iff $s\in Y$ and
 for each $t\in Y$,
 either
(i) $t\pins s$, or
 (ii) $s\ins t$ and for a $\mu$-cone of $x\in\Dd$  we have $t\conc\left<x\right>\in Y$.

If $Y$ is a tree, we say that $Y$ is \dfnemph{$\mu$-cofinally-splitting
with stem $s$} iff likewise,
except that we replace ``a $\mu$-cone of''
with ``Turing cofinally many''.
\end{dfn}
\begin{dfn}[Prikry forcing, $\PP$, $\PP^\gamma$, $\PP^-$]\label{dfn:PP} Let
$\PP^-_n$  be
the set of all
tuples $s\in\Dd^n$ such that $Q_s$ is defined.
Let $\PP^-=\bigcup_{n<\om}\PP^-_n$. For $s\in\PP^-$, let $\PP^-_{s,n}$ be the
set of
all $t\in\Dd^n$ such that $s\conc t\in\PP^-$. Let
$\PP^-_s=\bigcup_{n<\om}\PP^-_{s,n}$.

We will define the partial order $\PP=(P,<)$, where $P$ is the set of
\dfnemph{(Prikry) conditions}. For each $\gamma$ and
$n<\om$ such that
$\gamma+n\leq\beta^*+n^*$ we also define a
sub-order $\PP^{\gamma+n}\sub\PP$, and set
$\PP=\PP^{\beta^*+n^*}$.

A \dfnemph{potential condition} is 
either
\begin{enumerate}[label=\tu{(}\roman*\tu{)}]
 \item a pair of
form $(s,\alpha)$, where $s\in\PP^-$ and $\alpha<\om_1$, or
\item\label{item:mu-cone-splitting_tree} a $\mu$-cone-splitting tree $Y\sub(\PP^-)^{<\om}$.
\end{enumerate}

We may identify a potential condition $Y$ as in \ref{item:mu-cone-splitting_tree} above with the pair $(s,\vec{X})$ where 
$\vec{X}=\left<X_n\right>_{n<\om}$ and
\begin{enumerate}[label=--]
\item $Y$ has stem $s$,
\item for each $n<\om$,
and $Y\cap\Dd^{\lh(s)+n}=s\conc X_n=\{s\conc t\bigm|t\in X_n\}$.
\end{enumerate}

Let $P_0$ be the set of potential
conditions.

The ordering $<_0$ restricted to potential conditions of form $(s,\vec{X})\in P_0$ is as usual (cf.~\cite{HOD_as_core_model});
that is, $(s,\vec{X})\leq(t,\vec{Y})$
iff  $s\conc\vec{X}\sub t\conc\vec{Y}$
(note this implies $t=s\rest\lh(t)$).
For those of form $(s,\alpha)$,
declare $(s,\alpha)$ to be $<_0$-equivalent to the
potential condition
$(s,\Xvec)$, where $\Xvec=\left<X_{s,\alpha,n}\right>_{n<\om}$, where $X_{s,\alpha,n}$ is the set of tuples $t\in\PP^-_{s,n}$ such that if $n>0$
then the least measurable cardinal of $Q_{s\conc t}$ is ${>\alpha}$.

We will have $P\sub P_0$ and define the ordering $<$ to be ${<_0}\rest P$.

Fix a limit $\gamma$. We set $\PP^\gamma$ to be the sub-ordering whose set of
conditions is $\M_\gamma\cap P_0$. In particular, the conditions in $\PP^{\om_1}$
are just those of form $(s,\alpha)\in P_0$ with $\alpha<\om_1$.

Let $1\leq n<\om$ and $q=(s,\vec{X})\in P_0$. We put $q\in\PP^{\gamma+n}$ iff
there are sequences $\left<q_i\right>_{i<\om}$
and $\left<q'_i\right>_{i<\om}$
of potential  conditions,
each with stem $s$,
and such that each $q_i$ is $\bfmuSigma_n^{\M_\gamma}$-definable,
each $q'_i$ is $\bfmuPi_n^{\M_\gamma}$-definable,
and $q=\bigcap_{i<\om}q_i\cap\bigcap_{i<\om}q'_i$.

We finally define $\PP=\PP^{\beta^*+n^*}$.
\end{dfn}

By the next lemma, which follows easily from Turing completeness, we can obtain
potential conditions
easily from sequences of measure one sets:
\begin{lem}
Let $(s,\vec{Y})$ be such that $s\in\PP^-$ and $\vec{Y}=\left<Y_n\right>_{n<\om}$, where  $Y_n\in\mu_n$ for each $n$.
Then there is $(s,\vec{X})\in\PP$
such that $X_n\sub Y_n$ for each $n$,
and $\vec{X}$ is simply definable from $(s,\Yvec)$. In fact, let $X_n$ be the set of all $t\in\PP_{s,n}^-$ such that for each $i\leq\lh(t)$ and each $m<\om$,
there are $\mu_m$-measure one many $u\in\PP_{s\conc (t\rest i),m}^-$ such that $s\conc (t\rest i)\conc u\in s\conc Y_{i+m}$\tu{]}.
Then $(s,\vec{X})$ is as advertised.
\end{lem}

\begin{lem}
Let $s\in\PP^-$ and $\gamma\in\Lim$ and $\gamma+n+1<\beta^*+n^*$
and $\left<q_n\right>_{n<\om}\sub\PP^{\gamma+n+1}$, be such that
for each $n$, $q_n\sub\Dd^{<\om}$ \tu{(}as opposed to $q_n=(s,\alpha)$ with $\alpha\in\OR$\tu{)}
and  $s=\stem(q_n)$.
Then $q=(\bigcap_{n<\om}q_n)\in\PP^{\gamma+n+1}$, and $\stem(q)=s$.
\end{lem}

\begin{dfn}
Say that a set $D\sub\PP$
is \dfnemph{cone-strongly predense} iff
\[ \all s\in\PP^-\ \ex k<\om\ \exists^*_k t\ \ex q\in D\ [\stem(q)=s\conc t],\]
and \dfnemph{cofinal-strongly predense} iff
likewise, but with $\exists^*_k$ weakened
to $\all^*_k$.

Working in a generic extension of $V$, a filter $G\sub\PP$ is
\dfnemph{sufficiently generic} iff it meets
all cone-strongly predense sets
$D\in\Ss_{\om^2}(\M_{\beta^*})$.
\end{dfn}

The proofs to come
don't actully need this much genericity;
one could restrict to sets $D$
definable within a certain
$\mu$-definability class over $\M_{\beta^*}$. But it will suffice
for our purposes to consider sufficient genericity as defined.

\begin{lem}\label{lem:suitable_implies_strongly_predense}
We have:
\begin{enumerate}
\item Every cone-strongly predense set if cofinal-strongly predense.
\item Every cofinal-strongly predense set is predense.
\end{enumerate}
\end{lem}

\begin{dfn}\label{dfn:N} Work in $\M_{\beta^*}$. Let $\widetilde{N}$
denote the natural class $\PP$-name for the
the potential premouse of height $\beta^*$ such that for sufficiently generic $G\sub\PP$,
$\widetilde{N}_G=\J_{\zeta}[Q]$
where $Q$ is the stack of all $Q_{\stem(q)}$ for $q\in G$,
and $\beta^*=\om_1+\zeta$.
For $\gamma\in\Lim\inter[\om_1,\beta^*]$ we
write
$\Ntilde^{\M_\gamma}$ or $\Ntilde|\gamma$ for the natural name for $\Ntilde_G|\gamma$.\end{dfn}

Sufficient genericity implies that $\om_1=\OR^Q$ above.  We will show  that $\Ntilde_G$ is a
premouse and $\Ntilde_G\sats$``$\lambda$ is a cardinal'' where $\lambda=\om_1$. We
won't actually deal with the generic extension $\M_{\beta^*}[G]$ beyond
$\Ntilde_G$.

\begin{dfn}For $s\in\PP^-$ with $s\neq\emptyset$ let $\delta^{Q_s}$ denote the
largest (Woodin) cardinal of $Q_s$;
and let $\delta^{Q_\emptyset}=\om$.
Let $\PP^-_{\Pg}$ denote the set of all $s\in\PP^-$ such that $\lh(s)>0$
and $s_{\lh(s)-1}\geq_T\Pg$.
For $s\in\PP^-_{\Pg}$, we define
a premouse $R_s$  extending $Q_s$,
as follows. Form the
$Q_s$-genericity iteration  $\Tt$ of $\Pg$
at $\delta_0^{\Pg}$,
after first iterating the least measurable of $\Pg$ out to $\delta^{Q_s}$.
Because $\Pg$ is extender algebra
generic over
 $Q_s$ and since $Q_s=\Lp_{\Gammag}(Q_s|\delta^{Q_s})$ and $\Pg$ is stable, we get $i^\Tt_{0\infty}(\delta_0^{\Pg})=\OR^{Q_s}$.
 Let $N=M^\Tt_\infty$.
Then $N[Q_s]$ translates to a premouse $R$
 over $(N|\delta_0^N,Q_s)$.
We define $R_s$ as the premouse extending $Q_s$, which is given
by the P-construction of $R$ above $Q_s$ (see below).\end{dfn}

\begin{rem} In this situation, we have
\begin{enumerate}[label=(\roman*)]
 \item\label{item:R_s_is_premouse_extending_Q_s}  $R_s$ is indeed a premouse extending $Q_s$,
 \item\label{item:Q_s_card_seg_of_R_s} $Q_s$ is
a cardinal segment of $R_s$, and
\item\label{item:R_s_it_above_Q_s} $R_s$ is iterable above $Q_s$, via a tail
$\Sigma_{R_s}$ of $\Sigma_{\Pg}$.
\end{enumerate}
For \ref{item:R_s_is_premouse_extending_Q_s} and \ref{item:Q_s_card_seg_of_R_s}: 
We have $N|\delta_0^{+N}=\Lp_\Gammagap(N|\delta_0^N)$ and $\delta_0^N=\OR^{Q_s}$, so \[
    N[Q_s]|(\OR^{Q_s})^{+N[Q_s]}
=\Lp_\Gammagap(N|\delta_0^N,Q_s)\]
and
\[R_s|(\OR^{Q_s})^{+R_s}=\Lp_\Gammagap(Q_s).\] So if \ref{item:R_s_is_premouse_extending_Q_s} fails or  \ref{item:Q_s_card_seg_of_R_s}fails, it is straightforward to see that \footnote{Regarding the soundness of proper segments of $R_s$, there is a slight subtlety
for segments of $N,R_s$ of height $\geq\OR^{Q_s}$ which project $\leq\OR^{Q_s}$. We already know that every proper segment of $R$ is sound (as a premouse over $(N|\delta_0^N,Q_s)$). But the language of this structure has a symbol for this coarse object
(and also symbols for each of its elements).
Given that $\pow(\delta^{Q_s})\cap R_s\sub Q_s$, which gives that $\rho_\om^{R_s|\alpha}\geq\OR^{Q_s}$ for each $\alpha\in(\OR^{Q_s},\OR^{R_s})$,
one also needs to see that the element $\OR^{Q_s}$ gets in to the relevant fine structural hulls. These hulls will be formed
using all elements of $\OR^{Q_s}$,
and since $\delta^{Q_s}$ is the largest cardinal of $Q_s$, there would therefore only be a problem 
for a hull of the form $\Hull_1^{R_s|\alpha}(\OR^{Q_s})$, where $R_s|\alpha$ is passive.
But this does not arise: we are interested in the case that
 also $\rho_1^{R_s|\alpha}=\OR^{Q_s}$. But then $\rho_1^{N|\alpha}=\delta_0^N=\OR^{Q_s}$,
 and $N|\alpha$ is passive (as $R_s|\alpha$ is passive). But then $p_1^N\not\sub\delta_0^N$, because $N|\delta_0^N\preccurlyeq_1 N$, by condensation. Hence,
 the relevant hull here is $\Hull_1^{R_s}(\OR^{Q_s}\cup p_1^N)$, and note that
 $\OR^{Q_s}$ is in this hull.}
\[ \pow(\delta^{Q_s})\cap\Lp_\Gammagap(Q_s)\not\sub Q_s.\]
 But
\[ \pow(\delta^{Q_s})\cap Q_s=\pow(\delta^{Q_s})\cap\Lp_\Gammagap(Q_s|\delta^{Q_s})=\pow(\delta^{Q_s})\cap L(T,Q_s|\delta^{Q_s}), \]
where $T$ is a $\Gammagap$ tree of a very good scale on a $\Gammagap$-universal set.
And $Q_s$ is definable over its universe from $Q_s|\delta^{Q_s}$ (this can be seen by using the Jensen stack above $\delta^{Q_s}$, since this is a regular uncountable cardinal in $Q_s$ (see \cite{V=HODX_pub} for more details)),
so $Q_s\in L(T,Q_s|\delta^{Q_s})$, so
$L(T,Q_s|\delta^{Q_s})=L(T,Q_s)$, and so
\[ \pow(\delta^{Q_s})\cap\Lp_\Gammagap(Q_s)=\pow(\delta^{Q_s})\cap L(T,Q_s)=\pow(\delta^{Q_s})\cap Q_s,
\]
contradiction.

Part \ref{item:R_s_it_above_Q_s} is routine.

We will only iterate $R_s$ above $Q_s$. 
Let $n=\deg(\Pg)$.
Then since $N$ is $\delta_0^N$-sound
and basically by the usual fine structural correspondence
of P-construction etc,
$\rho_{n+1}^{R_s}\leq\OR^{Q_s}<\lambda^{R_s}\leq\rho_n^{R_s}$,
 $R_s$ is $\OR^{Q_s}$-sound,
 and $p_{n+1}^{R_s}\cut\OR^{Q_s}=p_{n+1}^N$.
Note also that $\Th_{n+1}^{\Pg}(p_{n+1}^{\Pg})$ can be recovered
from $R'[\Pg|\delta_0^{\Pg}]$ for any above-$\OR^{Q_s}$,
non-dropping, $n$-maximal iterate $R'$ of $R_s$
 (first compute
$N|\delta_0^N$, and then translate
$\es^{R'}\rest[\delta_0^N,\OR^{R'})$
above $\delta_0^N$ to yield a premouse
$N'$ extending $N$,
and note that $N'$ is an iterate of $\Pg$).
And
whenever $g$
is $(R',\Coll(\om,\OR^{Q_s}))$-generic, there is 
a real $z\notin\Lp_\Gamma((Q_s,g))$
which is $\rSigma_{n+1}^{R'[g]}$-definable
in the parameter $(p',Q_s,g)$,
where $p'=i_{R_sR'}(p_{n+1}^{N})$
(consider the definable surjection
from $\OR^{Q_s}$ to $(\OR^{Q_s})^{+Q_s}$
which comes from the $\OR^{Q_s}$-soundness
of $R_s$).
\end{rem}

\begin{dfn} For a filter $G\sub\PP$, let $e^G$ denote
$\bigcup_{p\in G}\stem(p)$.

Working in a generic extension of $V$. A \dfnemph{sufficiently Prikry generic
iterate of $\Pg$}
is a generic $\Sigma_{\Pg}$-iterate $R$ (hence, an iterate of $\Pg$), such that $R|\lambda^R=\Ntilde_G|\om_1$ for some
sufficiently generic filter $G\sub\PP$.
Let $s\in\PP^-_{\Pg}$. 
A \dfnemph{sufficiently Prikry generic
iterate of $R_s$}
is likewise, except that it is a generic $\Sigma_{R_s}$-iterate $R$ (hence, an above-$Q_s$ iterate of $R_s$).
A \dfnemph{sufficiently Prikry generic iterate} is some such (of either  $\Pg$ or  some $R_s$).\end{dfn}

The following lemma is an immediate consequence
of Remark  \ref{rem:get_RR-genericity_generic}:
\begin{lem}\label{lem:suff_gen_implies_RR-gen}
 In any generic extension of $V$,
 every sufficiently Prikry generic iterate   is an $\RR$-genericity iterate.
\end{lem}

\begin{lem}\label{lem:suff_gen_it}Let $s\in\PP^-_{\Pg}$ and let
$Y$ be a $\mu$-cofinally-splitting tree with stem $s$. Then for
sufficiently large $\lambda$, in $V^{\Coll(\om,V_\lambda)}$, there is a
sufficiently Prikry generic iterate $N$ of $R_s$,  as witnessed by a sufficiently
generic filter $G\sub\PP$
such that $e^G\sub Y$. If $Y$ is in fact a $\mu$-cone-splitting tree, then we can take $G$ to meet all cofinal-strongly predense sets. Likewise with $s=\emptyset$ and $\Pg$ replacing $R_s$.\end{lem}

\begin{proof}Work in $V[G]$ where $G$ is $(V,\Coll(\om,V_\lambda))$-generic, for sufficiently large $\lambda$.
Enumerate the cone-strongly
predense sets of $\Ss_{\om^2}(\M_{\beta^*})$ (or any countable-in-$V[G]$ collection of such) as
$\left<D_n\right>_{n<\om}$.
Recursively define $s_n,Y_n,M_n,\Tt_n\in V$ such
that:
\begin{enumerate}
 \item $s_0=s$, $Y_0=Y$, $M_0=R_s$,
 \item $s_n\in\PP^-$ and $Y_n$ is a $\mu$-cofinal tree with $\stem(Y_n)=s_n$,
 \item $M_n$ is a countable non-dropping $\Sigma_{R_s}$-iterate  with
$Q_{s_{n}}=M_{n}|((\delta_{\lh(s_n)-1}^{M_n})^+)^{M_n}$,
 \item there is $q\in D_n$ such that $s_n\pins\stem(q)=s_{n+1}$ and $Y_{n+1}\sub
Y_n\inter q$,
 \item $\Tt_n$ is a countable non-dropping
tree on $M_n$ with last model $M_{n+1}$,
and $\Tt_0\conc\ldots\conc\Tt_n$ is via $\Sigma_{R_s}$,
\item $\Tt_n$ is based on the
interval $(\delta_{\lh(s_n)-1}^{M_n},\delta_{\lh(s_{n+1})-1}^{M_n})$ (here if
$n=0$ and $\lh(s)=0$ the interval should be $(0,\delta_{\lh(s_1)-1}^{M_0})$).
\end{enumerate}

Suppose we have $s_i$,
$Y_i,M_i,\Tt_i$ for $i\leq n$. Fix $k<\om$ and $A\in\mu_k$ witnessing that $D_n$ is
cone-strongly predense, with respect to $s_n$.
We may assume $k>0$; for
 illustration assume $k=2$. Let $x_0\in\RR$ code
$M_n|((\delta_{\lh(s_n)}^{M_n})^+)^{M_n}$. Let $c_0\in\Dd$ be the base of a cone of degrees $d$ such that $\{e\in\Dd\bigm|(d,e)\in A\}\in\mu$.
Let $d_0\in\Dd$ be such that $d_0\geq_Tx_0\oplus c_0$
and $s_n\conc\left<d_0\right>\in Y_n$.
Let $M'$ be the $\Sigma_{\Pg}$-iterate
given by iterating $\Pg$ out
to $Q_{s_n\conc\left<d_0\right>}$,
and $\Tt'$ such that $\Tt_0\conc\ldots\conc\Tt_{n-1}\conc\Tt'$
is the corresponding tree.
Let $x_1\in\RR$
code $M^{\Tt'}_\infty|(\delta_{\lh(s_n)+1})^{+M^{\Tt'}_\infty}$.
Let $c_1\in\Dd$ be a base of a cone
of $e\in\Dd$ such that $(d_0,e)\in A$.
Let $d_1\geq c_1\oplus x_1$ be such that $s_n\conc\left<d_0,d_1\right>\in Y_n$. Let $M_{n+1}$ be the $\Sigma_{\Pg}$-iterate given by iterating $\Pg$ out to $Q_{s_n\conc\left<d_0,d_1\right>}$,
and $\Tt_{n+1}$  such that $\Tt_0\conc\ldots\conc\Tt_{n+1}$ is the corresponding tree. Since $(d_0,d_1)\in A$ we can fix $q\in D_n$ with $\stem(q)=s_n\conc\left<d_0,d_1\right>$.
Let $Y_{n+1}=Y_n\cap q$ and $s_{n+1}=s_n\conc\left<d_0,d_1\right>$. This determines $s_{n+1},Y_{n+1},M_{n+1},\Tt_{n+1}$ as required.
The case that $k\neq 2$ is similar.

The variant meeting all cofinal-strongly predense sets, assuming $Y$ is a measure one tree,
is similar and left to the reader.\end{proof}

We now describe canonical names for elements of $\Ntilde_G$:

\begin{dfn}[$\Ntilde$-names, $\Nhat|\gamma$, $f^Q$]
An \dfnemph{$\Ntilde$-name} (or just \dfnemph{name}) is an element of
$\om^{<\om}\cross(\Lim\inter\beta^*)^{<\om}$. For $\gamma\in\Lim\cap[\om_1,\beta^*]$, we write  $\Nhat|\gamma$ or
$\Nhat^{\M_\gamma}$ to denote the class of $\Ntilde$-names in
$\om^{<\om}\cross\gamma^{<\om}$,
and $\Nhat=\Nhat|\beta^*$.

For potential premice $Q$ of height $\geq\gamma$, let $f^Q$ be the standard function
interpreting, in $Q$, names in $\Nhat|\gamma$. (The first component of a name
determines some sequence of $\es^Q$-rud functions,
and the second component
determines segments of $Q$ at which to
interpret them.) We may write ``$\xvec\in\Nhat|\gamma$'' for
``$\xvec\in(\Nhat|\gamma)^{<\om}$''. Given $\xvec\in\Nhat|\gamma$ of length $n$,
let $f^Q(\xvec)=(f^Q(x_0),\ldots,f^Q(x_{n-1}))$. For $G$
sufficiently generic and $\xvec\in\Nhat$, let $\xvec_G=
f^{\Ntilde_G}(\xvec)$ (assuming this is well-defined (i.e. assuming a large enough initial segment of $\Ntilde_G$ is a premouse; we will show that it is)).
\end{dfn}
\begin{dfn}\label{dfn:ordinal_names_PP}
 For $\alpha<\beta^*$, let $o_\alpha\in\Nhat|\max(\omega_1,\alpha+\om)$
 be the natural name for $\alpha$.
 Similarly, for $\vec{\alpha}\in\OR^{<\om}$
 or $\vec{\alpha}\in[\OR]^{<\om}$, let
 let $o_{\vec{\alpha}}\in\Nhat$ be the natural
 name for $\vec{\alpha}$.
 We might also abuse notation and just write
  ``$\alpha$'' or ``$\vec{\alpha}$''
 in forcing statements instead of ``$o_\alpha$'' or ``$o_{\vec{\alpha}}$''.
\end{dfn}

\begin{dfn}\label{dfn:external_PP_forcing_relation}
Let $\gamma\in\Lim\cap[\omega_1,\beta^*]$.
 The \emph{external $\PP$ forcing relation
 $\sforces{\PP}{\mathrm{ext},\Ntilde|\gamma}$
 at $\gamma$} is the relation  of tuples $(p,\varphi,\vec{x})$ such that $p\in\PP$,
 $\varphi$ is an $\rSigma_\om$ formula
 in the language of passive premice,
 $\vec{x}\in(\Nhat|\gamma)^{<\om}$, and where
 \[ q\sforces{\PP}{\mathrm{ext},\Ntilde|\gamma}\varphi(\vec{x}) \] iff $\Coll(\om,\M_{\beta^*})$ forces
 over $V$
 that
 for every sufficiently generic filter $G\sub\PP$ with $q\in G$,
 letting $N=\Ntilde_G$, we have
 $N|\gamma\sats\varphi(f^N(\vec{x}))$.
 And the \emph{external $\PP$ forcing relation $\sforces{\PP}{\mathrm{ext},\Ntilde}$}
 is just $\sforces{\PP}{\mathrm{ext},\Ntilde|\beta^*}$.
\end{dfn}

\begin{rem}
 In the case of $\gamma=\beta^*$, we will only
 be interested in $\varphi$ of certain limited complexity. We will not be interested in truth in the
wider universe $\M_\gamma[G]$, and so we may
just write ``$\gamma$'' instead of ``$\Ntilde|\gamma$'' with the same meaning.
\end{rem}

\begin{dfn}
We say that $(\gamma,n)$ is  \emph{low} iff
 $\gamma\in\Lim\cap[\omega_1,\beta^*]$,
 $n<\om$, and either:
\begin{enumerate}[label=--]
\item $\om_1\leq\gamma<\beta^*$, or
\item $\om_1=\gamma=\beta^*$ and $n\leq 1$, or
\item $\om_1<\gamma=\beta^*$ and $n=0$.
\end{enumerate}
 If $\varphi$ is a formula in the language
 of passive premice, we say that $(\gamma,\varphi)$
 is \emph{low}
 iff $\varphi$ is $\Sigma_n$,
 where $(\gamma,n)$ is low.\footnote{Note
 that this is $\Sigma_n$, not $\rSigma_n$.
 This suffices for the present section,
but in  \S\ref{subsec:fine}
 we will need a variant using $\rSigma_n$.}
\end{dfn}

\begin{dfn}\label{dfn:prior_condition_deciding}\footnote{The notions defined here will be slightly refined in Definition \ref{dfn:stem_forcing_relations}.
Note that here we use the $\Sigma_n$ hierarchy,
whereas in \ref{dfn:stem_forcing_relations} we use the $\rSigma_n$ hierarchy.}
Given $(\gamma,\varphi)$ low,
 $\xvec\in(\Nhat|\gamma)^{<\om}$ and $s\in\PP^-$,
we will define below the
\dfnemph{condition $q=q^{\gamma}_{s,\varphi(\xvec)}$ deciding $\varphi(\xvec)$
for $\Ntilde|\gamma$ at $s$}.
 We will obtain that $q\in\PP^{\gamma+\om}$, $\stem(q)=s$, and  
either $q\sforces{\PP}{\mathrm{ext},\gamma}\varphi(\xvec)$ or
$q\sforces{\PP}{\mathrm{ext},\gamma}\neg\varphi(\xvec)$.

We
will also $\mu$-define over $\M_\gamma$, for low $(\gamma,\varphi)$,
the \dfnemph{$\varphi$ stem forcing relation $\sforces{-}{\gamma}\varphi$ at $\gamma$}, in such a manner that
\begin{equation}\label{eqn:PP_Sigma_0_forcing_thm} \begin{split}\Big(s\sforces{-}{\gamma}\varphi(\vec{x})\Big)&\iff \Big(q^\gamma_{s,\varphi(\vec{x})}\sforces{\PP}{\mathrm{ext},\gamma}\varphi(\vec{x})\Big)\\&\iff\Big(\exists p\in\PP\ \Big[\stem(p)=s\wedge p\sforces{\PP}{\mathrm{ext},\gamma}\varphi(\vec{x})\Big]\Big),\end{split}\end{equation}
for $s\in\PP^-$ and $\vec{x}\in(\Nhat|\gamma)^{<\om}$. The \dfnemph{$\Sigma_0^{\Ntilde|\gamma}$ stem forcing relation $\sforces{-}{\gamma,0}$ at $\gamma$} is the resulting relation of three variables $(\varphi,s,\vec{x})$, where $\varphi$ is any $\Sigma_0$ formula;
likewise the \dfnemph{$\Sigma_1^{\Ntilde|\omega_1}$ stem forcing relation $\sforces{-}{\omega_1,1}$ at $\omega_1$}.
These things will be defined by simultaneous recursion
on $\gamma$, with a sub-recursion on $\varphi$.\footnote{Actually
the same definitions make sense
more generally,
without the full restriction that $(\gamma,\varphi)$ be of low complexity, but we will define the relevant notions for this generalization later.}

We will also have that:
\begin{enumerate}[label=\tu{(}$\dagger$\arabic*\tu{)}]
\item \label{item:dagger1} the map $(s,\varphi,\vec{x})\mapsto q^{\omega_1}_{s,\varphi(\vec{x})}$, where $\varphi$ is restricted to $\rSigma_1$,
is $\mDelta_1^{\M_{\omega_1}}(\{\xg\})$,
\item \label{item:dagger2} $\sforces{-}{\omega_1,1}$ is $\mDelta_1^{\M_{\omega_1}}(\{\xg\})$.
 \item\label{item:dagger3}   for $\delta\in\Lim\cap(\om_1,\beta^*]$,
the map
\[ (\gamma,s,\varphi,\vec{x})\mapsto q^\gamma_{s,\varphi(\vec{x})} \]
is $\mDelta_1^{\M_{\delta}}(\{\xg\})$,
where the first coordinate $\gamma$, is restricted to $\Lim\cap[\om_1,\delta)$; this definability is moreover uniform in
$\delta$,
\item \label{item:dagger4} $\sforces{-}{\delta,0}$ is $\mDelta_1^{\M_\delta}(\{\xg\})$, uniformly in  $\delta\in\Lim\cap(\omega_1,\beta^*]$.
\end{enumerate}

 First consider the case that $\gamma=\om_1$
 and $\varphi$ is $\rSigma_1$.
 Let $\vec{x}\in\Nhat|\om_1$ and $s\in\PP^-$.
 Set $q^{\om_1}_{s,\varphi(\vec{x})}=(s,\alpha)$ where $\alpha$ is least such that
 $\vec{x}\sub\max(\OR^{Q_s},\alpha)$.
 Note that  $q=q^{\om_1}_{s,\varphi(\vec{x})}\in\M_{\om_1}$, and
 $(s,\varphi,\vec{x})\mapsto q^{\om_1}_{s,\varphi(\vec{x})}$
 is $\mDelta_1^{\M_{\om_1}}(\{\xg\})$.
Set
 \[ s\sforces{-}{\omega_1}\varphi(\vec{x}) \]
to hold iff there is  $d\in\PP^-_{s,1}$ such that:
 \begin{enumerate}[label=--]
 \item $\alpha<\kappa$,
 where $\kappa$ is the least measurable of
 $Q=Q_{s\conc\left<d\right>}$ with $\kappa\in(\delta_{\lh(s)-1}^{Q},\delta_{\lh(s)}^Q)$, and
 \item
 $Q\sats\varphi(f^{Q}(\vec{x}))$\end{enumerate}
(equivalently, for all $d\in\PP^-_{s,1}$
such that $\alpha<\kappa$ where $\kappa$ is the least measurable of $Q=Q_{s\conc\left<d\right>}$,
the same conclusion holds).  Note that
line (\ref{eqn:PP_Sigma_0_forcing_thm}) holds
for $(\omega_1,\varphi)$.

Given $q^\gamma_{s,\varphi(\vec{x})}$,
define $q^\gamma_{s,\neg\varphi(\vec{x})}=q^\gamma_{s,\varphi(\vec{x})}$ and
\[ s\sforces{-}{\gamma}\neg\varphi(\vec{x})\iff \neg\Big(s\sforces{-}{\gamma}\varphi(\vec{x})\Big).\]

Given $q^\gamma_{s,\varphi(\vec{x})}$,
and $q^\gamma_{s,\psi(\vec{y})}$,
define
$q^\gamma_{s,\varphi(\vec{x})\wedge\psi(\vec{y})}=q^\gamma_{s,\varphi(\vec{x})}\cap q^\gamma_{s,\psi(\vec{y})}$
and
\[s\sforces{-}{\gamma}\big(\varphi(\vec{x})\wedge\psi(\vec{y})\big)\iff \Big(s\sforces{-}{\gamma}\varphi(\vec{x})\Big)\wedge\Big( s\sforces{-}{\gamma}\psi(\vec{y})\Big).\]

Now suppose we have defined $q^\gamma_{s,\varphi(y,\vec{x})}$
for all $y,\vec{x}$; we define $q^{\gamma}_{s,\psi(\vec{x})}$ where $\psi(\vec{x})=\text{``}\exists v\varphi(v,\vec{x})\text{''}$. We set
\[ s\sforces{-}{\gamma}\exists v\varphi(v,\vec{x})\]
true iff
\begin{equation}\label{eqn:exists_k_and_measure_one_many}\exists k<\om\ \all^*_k u\ \exists y\in(\Nhat|\gamma)\ \Big[s\conc u\sforces{-}{\gamma}\varphi(y,\vec{x})\Big]. \end{equation}
If line (\ref{eqn:exists_k_and_measure_one_many}) holds then we put $s\conc t\in q=q^\gamma_{s,\psi(\vec{x})}$ iff 
\[  \all i\leq\lh(t)\ \exists k<\om\ \all^*_ku\ 
 \exists y\in(\Nhat|\gamma)\ \Big[s\conc(t\rest i)\conc u\sforces{-}{\gamma}\varphi(y,\vec{x})\Big].\]
Note that $q\in\PP$ and $\stem(q)=s$.
Note also that
$s\conc t\in q$ iff
\[  \exists k\in[\lh(t),\om)\ \all^*_ku\ 
\all i\leq\lh(t)\ \exists y\in(\Nhat|\gamma)\ \Big[s\conc(t\rest i)\conc (u\rest(k-i))\sforces{-}{\gamma}\varphi(y,\vec{x})\Big].\]

On the other hand,  if line (\ref{eqn:exists_k_and_measure_one_many})
fails, then we put $s\conc t\in q=q^\gamma_{s,\psi(\vec{x})}$ iff
\[ \all i\leq\lh(t)\ \all k<\om\ \exists^*_ku\ \all y\in(\widehat{N}|\gamma)\ \Big[s\conc(t\rest i)\conc u\sforces{-}{\gamma}\neg\varphi(y,\vec{x})\Big].\]
Again $q\in\PP$ and $\stem(q)=s$.

Now suppose $\omega_1\leq\gamma<\beta^*$ and we have defined $\gamma^\gamma_{s,\varphi(\vec{x})}$ for all $s,\varphi,\vec{x}$,
and also $\sforces{-}{\gamma}$.
We must define $q^{\gamma+\om}_{s,\varphi(\vec{x})}$ for $\Sigma_0$ formulas $\varphi$ and $\vec{x}\in(\widehat{N}|(\gamma+\om))^{<\om}$, and also the relation $\sforces{-}{\gamma+\om,0}$.
We do this by translating
$\varphi(\vec{x})$
down to some $\Sigma_\om$ statement over $\widetilde{N}|\gamma$
about names in $\widehat{N}|\gamma$. That is,
fix the natural algorithm \[ (\varphi,\vec{i})\mapsto(\psi'_{\varphi,\vec{i}},\vec{j}_{\varphi,\vec{i}}),\]
much like
the algorithm  of Definition \ref{dfn:Sigma_0_to_Sigma_om_algo},
such that for all $\Sigma_0$ formulas $\varphi$
and
\[ \vec{x}=((\vec{i}_0,\vec{\xi}_0),\ldots,(\vec{i}_{k-1},\vec{\xi}_{k-1}))\in(\widehat{N}|(\gamma+\om))^{<\om} \]
(so $\vec{i}_\ell\in\om^{<\om}$
and $\vec{\xi}_\ell\in(\gamma+\om)^{<\om}$ for each $\ell<k$),
then letting $\vec{i}=(\vec{i}_0,\ldots,\vec{i}_{k-1})$,
then we have
\[ \vec{y}=((\vec{j}_0,\vec{\zeta}_0),\ldots,(\vec{j}_{k-1},\vec{\zeta}_{k-1}))\in(\widehat{N}|\gamma)^{<\om}, \]
where
$(\vec{j}_0,\ldots,\vec{j}_{k-1})=\vec{j}_{\varphi,\vec{i}}$,
and
$\vec{\zeta}_{\ell}=\vec{\xi}_\ell\cut\{\gamma\}$
for $\ell<k$,
and the truth of $\varphi(\vec{x})$
will be uniformly equivalent to that of $\psi'_{\varphi,\vec{i}}(\vec{y})$.
We then define $q^{\gamma+\om}_{s,\varphi(\vec{x})}=q^{\gamma}_{s,\psi'(\vec{y})}$
and set
\[ s\sforces{-}{\gamma+\om,0}\varphi(\vec{x})\iff s\sforces{-}{\gamma}\psi'(\vec{y}),\]
where $\psi'=\psi'_{\varphi,\vec{i}}$.

This completes the recursive definitions.
It is now straightforward to verify that
$q^\gamma_{s,\varphi(\vec{x})}\in\M_{\gamma+\om}$, and that the (uniform) definability claimed in \ref{item:dagger1}--\ref{item:dagger4}  above holds; note though that this makes crucial use
of the two special features of the $\M$-hierarchy (that it starts with $\M_{\om_1}=(\HC,T^{\M_{\om_1}})$ and constructs using $\mu$).

Let $(\gamma,n)$ again be low. Let $\alpha<\om_1$ and $\xvec\in\Nhat|\gamma$. We  define
the \textbf{condition
$q^{\gamma,n}_{s,\alpha,\xvec}$ deciding
$\Th_{n}^{\Ntilde|\gamma}(\alpha\un\{\xvec\})$ at $s$} as the meet
of $(s,\alpha)$\footnote{Recall
this notation from Definition \ref{dfn:PP}; $(s,\alpha)$ is a condition in $\PP^{\om_1}$.}with all conditions $q^{\gamma}_{s,\varphi(\betavec,\xvec)}$, for $\varphi$ being
$\Sigma_{n}$ and $\betavec\in\alpha^{<\om}$. Note that
$q^{\gamma,n}_{s,\alpha,\xvec}\in\PP^{\gamma+\om}$
and the map $(s,\gamma,n,\alpha,\vec{x})\mapsto q^{\gamma,n}_{s,\alpha,\vec{x}}$ is $\mDelta_1^{\M_{\gamma+\om}}$,
uniformly in $\gamma$. The \dfnemph{measure one
$\Sigma_n$-type $t=t^{\Ntilde|\gamma}_{n,s}(\alpha,\vec{x})$
of $\alpha\cup\{\vec{x}\}$ at $s$} is just the set of all
$\Sigma_n$ formulas $\varphi(\vec{\beta},\vec{v})$
in parameters $\vec{\beta}$ and variables $\vec{v}$
such that  $\vec{\beta}\in\alpha^{<\om}$ and $s\sforces{\PP^-}{\gamma}\varphi(\vec{\beta},\vec{x})$.
\end{dfn}

\begin{lem}\label{lem:rSigma_0_forcing_theorem_for_PP}
 Work in a generic extension of $V$.
 Let $G\sub\PP$ be a sufficiently generic filter
 and $N=\Ntilde_G$. Then:
 \begin{enumerate}
  \item\label{item:N_is_a_premouse} $N$ is an $\om$-small premouse  with $\om$ Woodins
  and $\lambda^N=\omega_1$ and $\OR^N=\beta^*$,
  \item\label{item:x_G=f^N(x)} $x_G=f^N(x)$ for each $x\in\widehat{N}$, 
  \item\label{item:G_contains_eventual_q_ss}
  for each low $(\gamma,\varphi)$ and
   each $\vec{x}\in(\widehat{N}|\gamma)^{<\om}$, we have:
  \begin{enumerate}[label=\tu{(}\alph*\tu{)}]
\item\label{item:each_varphi(x)_decided_by_some_p_in_G}  There is $s\in\PP^-$ such that
  $q^{\gamma}_{s,\varphi(\vec{x})}\in G$.
  \item\label{item:s_ins_t_forcing_relationship}  Suppose $s\in\PP^-$ and $q^{\gamma}_{s,\varphi(\vec{x})}\in G$.
 Let $p\in G$ with $s\ins t=\stem(p)$.
 Then $q^{\gamma}_{t,\varphi(\vec{x})}\in G$,
 and moreover,  $s\sforces{-}{\gamma}\varphi(\vec{x})$
 iff $t\sforces{-}{\gamma}\varphi(\vec{x})$.
  \item\label{item:PP_Sigma_0_forcing_theorem}
 The following are equivalent:
  \begin{enumerate}
\item $N|\gamma\sats\varphi(\vec{x}_G)$,
\item There is $p\in G$ such that
$\stem(p)\sforces{-}{\gamma}\varphi(\vec{x})$,
\item For each $s\in\PP^-$,
if $q^\gamma_{s,\varphi,\vec{x}}\in G$
then $s\sforces{-}{\gamma}\varphi(\vec{x})$.
\end{enumerate}  
\item\label{item:int_ext_forcing_relations_agree} In $V$, for all $s\in\PP^-$, the following are equivalent:
\begin{enumerate}
\item $\M_{\gamma}\sats s\sforces{-}{\gamma}\varphi(\vec{x})$
\item $q^{\gamma}_{s,\varphi(\vec{x})}\sforces{\PP}{\mathrm{ext},\gamma}\varphi(\vec{x})$
\item $\Coll(\om,\M_{\beta^*})$ forces that there is a sufficiently generic filter $H\sub\PP$
such that $q^\gamma_{s,\varphi(\vec{x})}\in H$
and $\Ntilde_H\sats\varphi(\vec{x}_H)$.
\end{enumerate}
 \end{enumerate}
 \item\label{item:condition_deciding_theory} for each low $(\gamma,n)$, each $\vec{x}\in(\Nhat|\gamma)^{<\om}$, we have:
 \begin{enumerate}[label=\tu{(}\alph*\tu{)}]
   \item\label{item:each_ctbl_theory_decided_by_some_p} There are $s\in\PP^-$ and $\alpha<\omega_1$ such that
  $q^{\gamma,n}_{s,\alpha,\vec{x}}\in G$.
\item\label{item:each_ctbl_theory_at_s_implies} Suppose $s\in\PP^-$, $\alpha<\omega_1$ and $q^{\gamma,n}_{s,\alpha,\vec{x}}\in G$. Then:
\begin{enumerate}[label=\tu{(}\roman*\tu{)}]
\item\label{item:at_extension_t} Let $p\in G$ with $s\ins t=\stem(p)$.
Then $q^{\gamma,n}_{t,\alpha,\vec{x}}\in G$.
\item\label{item:condition_deciding_theory_sets_it_to_t-theory_is_mu-OD} $\Type_{\Sigma_n}^{N|\gamma}(\alpha\cup\{\vec{x}_G\})=t^{\gamma,n}_{s,\alpha,\vec{x}}\in  N|\lambda^N$.
\end{enumerate}
 \end{enumerate}
 \end{enumerate}
\end{lem}
\begin{proof}
Sufficient genericity immediately
gives that $N|\omega_1$
is an $\om$-small premouse with $\om$ Woodins
cofinal in its ordinals. Therefore
$N$ is a potential premouse.
Part \ref{item:x_G=f^N(x)} is directly by definition. Part \ref{item:G_contains_eventual_q_ss}\ref{item:each_varphi(x)_decided_by_some_p_in_G}
holds because $G$ meets the cone-strongly pre-dense set
 \[ \{q^\gamma_{s,\varphi(\vec{x})}\bigm|s\in\PP^-\},\]
 and part 
Part \ref{item:G_contains_eventual_q_ss}\ref{item:s_ins_t_forcing_relationship}
is straightforward.

We now prove part \ref{item:G_contains_eventual_q_ss}\ref{item:PP_Sigma_0_forcing_theorem},
by induction on $\gamma\in\Lim\cap[\omega_1,\beta^*]$, with a subinduction on $\varphi$.
Suppose $\gamma=\om_1$ and $\varphi$ is $\rSigma_1$. Because there is $s\in\PP^-$ such that $q^{\omega_1}_{s,\varphi(\vec{x})}\in G$, and
because $N|\xi\preccurlyeq_1 N|\omega_1$
for every $N|\omega_1$-cardinal $\xi$,
part \ref{item:PP_Sigma_0_forcing_theorem} is easily seen to hold with respect to $\om_1,\varphi,s$. Now suppose either that $\gamma\in[\omega_1,\beta^*)$, and we have already dealt with $\rSigma_0$ formulas $\varphi$ (and $\rSigma_1$ if $\gamma=\omega_1$). The propagation of the induction through $\wedge$ and $\neg$ is clear. Now suppose $\gamma\in[\omega_1,\beta^*)$ and the inductive hypothesis holds for $(\gamma,\psi)$
with respect to all $(\vec{x},y)\in(\Nhat|\gamma)^{<\om}$,
where $\psi$ has free variables within $(\vec{u},v)$,
and $\varphi(\vec{u})$ is the formula $\exists v \psi(\vec{u},v)$. Let $s\in\PP^-$ and suppose $q^\gamma_{s,\varphi(\vec{x})}\in G$
and $s\sforces{-}{\gamma}\varphi(\vec{x})$.
Then by definition, we can fix $k<\om$
such that for all $\ell\in[k,\om)$ and all $t\in\Dd^\ell$ with $s\conc t\in q^\gamma_{s,\varphi(\vec{x})}$,
there is $y\in\Nhat|\gamma$ such that
$s\conc t\sforces{-}{\gamma}\psi(\vec{x},y)$.
Now let
$C$ be the set of all conditions of the form
\begin{enumerate}[label=--]
 \item 
$q^\gamma_{r,\psi(\vec{x},y)}$, where $r\in\PP^-$ and $y\in\Nhat|\gamma$ and $r\sforces{-}{\gamma}\psi(\vec{x},y)$,
or
\item $q^\gamma_{r,\neg\exists v\psi(\vec{x},v)}$, where $r\in\PP^-$ and $r\sforces{-}{\gamma}\neg\exists v\psi(\vec{x},v)$.
\end{enumerate}
Note that $C\in\M_{\beta^*}$
and $C$ is cone-strongly-predense.
Therefore there is $p\in C\cap G$.
If $p=q^\gamma_{r,\psi(\vec{x},y)}$
where $r\sforces{-}{\gamma}\psi(\vec{x},y)$,
then by induction, we have $N\sats\psi(\vec{x}_G,y_G)$, so $N\sats\varphi(\vec{x}_G)$, as desired. So suppose $p=q^\gamma_{r,\neg\exists v\psi(\vec{x},v)}$ where $r\sforces{-}{\gamma}\neg\exists v\psi(\vec{x},v)$. By part \ref{item:s_ins_t_forcing_relationship}, we may assume $r=s\conc t$ where $\lh(t)\geq k$,
and therefore (since $q^\gamma_{s,\varphi(\vec{x})}\in G$) $r\in q^\gamma_{s,\varphi(\vec{x})}$. But then by the remarks above we can fix $y\in\Nhat|\gamma$ such that $r\sforces{-}{\gamma}\psi(\vec{x},y)$, which easily contradicts
the fact that $r\sforces{-}{\gamma}\neg\exists v\psi(\vec{x},v)$. This completes the induction through formulas at level $\gamma$.
Given this, if $\gamma<\beta^*$,
then part \ref{item:PP_Sigma_0_forcing_theorem} holds for $\rSigma_0$ formulas at level $\gamma+\om$ by induction
and the correctness of the algorithm
$(\varphi,\vec{i})\mapsto(\psi'_{\varphi,\vec{i}},\vec{j}_{\varphi,\vec{i}})$
used in Definition \ref{dfn:prior_condition_deciding}. And finally, if $\gamma$ is a limit of limits,
then part \ref{item:PP_Sigma_0_forcing_theorem}
follows immediately at $\gamma$ for $\rSigma_0$ formulas $\varphi$ by induction.

Part \ref{item:G_contains_eventual_q_ss}\ref{item:int_ext_forcing_relations_agree}
is now an
immediate corollary of what we have established.

Part \ref{item:condition_deciding_theory}\ref{item:each_ctbl_theory_decided_by_some_p}
is like \ref{item:G_contains_eventual_q_ss}\ref{item:each_varphi(x)_decided_by_some_p_in_G}, and \ref{item:condition_deciding_theory}
\ref{item:each_ctbl_theory_at_s_implies}\ref{item:at_extension_t}
is  straightforward.

Part \ref{item:condition_deciding_theory}\ref{item:each_ctbl_theory_at_s_implies}\ref{item:condition_deciding_theory_sets_it_to_t-theory_is_mu-OD}: The fact
that $\Type^{N|\gamma}_{\Sigma_n}(\alpha\cup\{\vec{x}_G\})=t^{\gamma,n}_{s,\alpha,\vec{x}}$
follows immediately from part \ref{item:G_contains_eventual_q_ss},
so we just need to see that this type is in
$N|\lambda^N$. If $n=0$ it is automatically true. And if $(\gamma,n)=(\omega_1,1)$,
it is just because $\rho_1^{N|\lambda^N}=\lambda^N$, since $\lambda^N$ is a limit of $N$-cardinals. Now suppose otherwise. Then $\gamma<\beta^*$,
and 
note that there is $n'<\om$
such that $t^{\gamma,n}_{s,\alpha,\vec{x}}$
is an $\OD^{\gamma,n'}_{\mu}(Q_s)$ subset
of $\alpha$,
and so by the minimality
of $\beta^*$, and fullness
of the $Q_r$'s, we get $t^{\gamma,n}_{s,\alpha,\vec{x}}\in Q_r$ if $\alpha<\OR^{Q_r}$,
so $t^{\gamma,n}_{s,\alpha,\vec{x}}\in N|\lambda^N$.

 Part \ref{item:N_is_a_premouse}: If $N$ fails to be a premouse, or $\omega_1$
fails to be an $N$-cardinal, 
 then there is $\gamma\in[\om_1,\beta^*)$
 and $m<\om$ such that $N|\gamma$ is an $m$-sound premouse and $\rho=\rho_{m+1}^N<\om_1\leq\rho_m^N$. But this contradicts
 part \ref{item:condition_deciding_theory}.
\end{proof}

Putting everything together, we can now establish
that $\M_{\beta^*}$ is ``the'' derived model
of sufficiently Prikry generic iterates,
at least in a naive\footnote{The more refined analysis in the \S\ref{subsec:fine} will lead to a less naive version.} sense:

\begin{lem}
Work in a generic extension of $V$.
Let either:
\begin{enumerate}[label=--]
 \item $s\in\PP^-_{\Pg}$
 and $R$ be a sufficiently
Prikry generic iterate of $R_s$, or
\item $R$ be a sufficiently Prikry generic iterate of $\Pg$,
\end{enumerate}
as witnessed by $G\sub\PP$. 
Let $H$ witness that $R$ is an $\RR$-genericity iterate \tu{(}see Lemma \ref{lem:suff_gen_implies_RR-gen}\tu{)}.
Then
 $(\Ntilde^{\M_{\beta^*}})_G=R$ and $\M_{\beta^*}=(\Mtilde^R)_H$.
\end{lem}
\begin{proof}
Let $N=\Ntilde_G$.
We have $\OR^N=\beta^*$.
By construction, $N|\lambda^N=R|\lambda^R$.
By Lemma \ref{lem:OR^N_leq_beta^*} (in the case that $R$ is an iterate of $R_s$,
 apply the lemma to the $\Sigma_{\Pg}$-iterate $P'$ (of $\Pg$) which corresponds to $R$; in particular, $R$ and $P'$ are equivalent modulo a small generic), it therefore suffices to see that $\OR^R=\beta^*$.
Similarly by Lemma \ref{lem:OR^N_leq_beta^*}, $\OR^R\leq\beta^*$.
 But if $\OR^R<\beta^*$, then $R\pins N$,
but since $R$ projects to $Q_s$
and is not sound,
$N$ is not a premouse, contradicting Lemma \ref{lem:rSigma_0_forcing_theorem_for_PP}.
\end{proof}

\subsection{Fine correspondence}\label{subsec:fine}

Having established the model correspondence
between $\M_{\beta^*}$ and 
sufficiently Prikry generic iterates $N$ of $\Pg$ etc,
we now want to refine our understanding
of this correspondence,  analyzing the definability hierarchies over such models.
We will demonstrate quantifer-by-quantifier (in the appropriate sense) correspondence, up to the level at which $\Pg$ projects and $\M_{\beta^*}$
computes a new $\OD^\mu(x)$-real for some $x\in\RR$, deducing that these correspond.

Roughly, we want to give
a (reasonably) optimal definition over $\M_{\beta^*}$
of the $\rSigma_n^{\Ntilde}$ forcing relation
(for the relevant values of $n$),
and a likewise definition over premice $N$
of the $\mSigma_{n'}^{\Mtilde^N}$ forcing
relation (for the relevant values of $n'$,
and assuming $N$ satisfies the appropriate first order properties),
and verify the corresponding forcing theorems.
We will first formally define the relevant
putative forcing relations (over the relevant models),
and then later observe that these definitions
yield the actual (external) forcing relation,
i.e.~that the corresponding forcing theorem holds.
The putative forcing relations will
be named for their intended intuitive/external meaning,
but defined purely in terms of features of the model over which they are being defined.

\begin{dfn}\label{dfn:tau(beta),q(beta)_if_omega_1<=beta}
Let $\gamma\in\Lim\cap(\omega_1,\beta^*]$.
 Let $n>0$ and suppose $\omega_1<\rho_n^{\M_\gamma}$.
Let $\beta\in(\omega_1,\rho_n^{\M_\gamma})$.
Let \[ C=\cHull_{\mSigma_n}^{\M_\gamma}(\beta\cup\HC\cup\pvec_n^{\M_\gamma}) \]
and $\pi:C\to\M_\gamma$ the uncollapse and
 $\vec{p}=\pi^{-1}(\pvec_n^{\M_\gamma})$.
Then $\vec{p}^{\Ntilde|\gamma}_n(\beta)$
denotes $o_{\vec{p}}$ (Definition \ref{dfn:ordinal_names_PP}), and $\tau^{\Ntilde|\gamma}_n(\beta)$ denotes
the natural name in $\Nhat|(\OR^C+\om)$ for
$\Th_{\rSigma_n}^{\Ntilde|\OR^C}(o_\beta\cup\{o_{\vec{p}}\})$.
\end{dfn}

Note that $\tau_n^{\Ntilde|\gamma}(\beta)$
is the ``natural name for'' $\Th_{\rSigma_n}^{\Ntilde|\OR^C}(o_\beta\cup\{o_{\vec{p}}\})$,
as opposed to $\Th_{\rSigma_n}^{\Ntilde|\gamma}(o_\beta\cup\{o_{\pvec_n^{\M_\gamma}}\})$,
and $\tau_n^{\Ntilde|\gamma}(\beta)\in\Nhat|(\OR^C+\om)$,
as opposed to just $\tau_n^{\Ntilde|\gamma}(\beta)\in\Nhat|(\gamma+\om)$.
(Note that $\beta,\vec{p}\sub\OR^C$,
and so there is indeed such a ``natural name'' in $\Nhat|(\OR^C+\om)$.)

\begin{dfn}\label{dfn:stem_forcing_relations}
 Let $\gamma\in\Lim\cap[\om_1,\beta^*]$.
 Recall that the $\Sigma_0^{\Ntilde|\gamma}$ stem forcing relation $\sforces{\PP^-}{\gamma,0}$ at $\gamma$
  was specified in Definition \ref{dfn:prior_condition_deciding},
  as was $\sforces{\PP^-}{\omega_1,1}$
  (for $\Sigma_1^{\Ntilde|\omega_1}$ stem forcing).

 Now let $(\gamma,n)$ be such that
$(\om_1,1)\leq(\gamma,n)\leq(\beta^*,n^*)$.
We define the \dfnemph{witnessed $\rSigma_{n+1}^{\Ntilde|\gamma}$ stem forcing relation} $\sforces{\PP^-}{\gamma\mathrm{w},n+1}$ of $\M_\gamma$, recursively in $n$.\footnote{Note that we now deal with the $\rSigma$ hierarchy, not $\Sigma$. However, we supress this from the forcing notation; we will not need the forcing notation from Definition \ref{dfn:prior_condition_deciding} other than that for $\Sigma_0$ and $\Sigma_1$, but anyway, $\rSigma_0=\Sigma_0$ and $\rSigma_1=\Sigma_1$,
even syntactically.}

If $\gamma>\omega_1$, then the \dfnemph{witnessed $\rSigma_{1}^{\Ntilde|\gamma}$ stem forcing relation}
$\sforces{\PP^-}{\gamma\mathrm{w}1}$
of $\M_\gamma$
is the relation of tuples $(s,\varphi,\vec{x})$
such that $s\in\PP^-$, $\varphi(\vec{v})$
is an $\rSigma_{1}$ formula of the passive premouse language, of form
\[ \varphi(\vec{v})\iff\exists y\ \psi(y,\vec{v})\] where $\psi$ is $\Sigma_0$, and $\vec{x}\in(\Nhat|\gamma)^{<\om}$, and where we define
\[ s\sforces{\PP^-}{\gamma\mathrm{w}1}\varphi(\vec{x})\ \iff\ \exists \sigma\in(\Nhat|\gamma)\ \Big[s\sforces{\PP^-}{\gamma0}\psi(\sigma,\vec{x})\Big].\]

And the \dfnemph{witnessed $\rSigma_2^{\Ntilde|\omega_1}$ stem forcing relation} $\sforces{\PP^-}{\omega_1\mathrm{w},2}$ is
the relation of  $(s,\varphi,\vec{x})$
such that $s\in\PP^-$, $\varphi(\vec{v})$ is $\rSigma_2$ in the passive premouse language (with free variables all among $\vec{v}$),
$\vec{x}\in(\Nhat|\omega_1)^{<\om}$,
and there is $z\in\PP_{s,1}^-$
such that letting $Q=Q_{s\conc\left<z\right>}$,
and $\kappa$ be the least measurable cardinal of $Q$ with $\kappa>\OR^{Q_s}$,
then there is an $\mSigma_1$ min-term
$u$ and $\theta<\kappa$ and $\vec{\alpha}\in\theta^{<\om}$
such that
$f^Q(\vec{x})=u^Q(\vec{\alpha})$
and $\Th_{\Sigma_1}^{Q}(\theta)$
codes a putative witness to $(\varphi(\vec{v})
,(u,\vec{\alpha}))$.

Suppose we have defined $\sforces{\PP^-}{\gamma\mathrm{w},n+1}$,
where $n\geq 0$,
and $n\geq 1$ if $\gamma=\omega_1$.
Then the \dfnemph{$\mu$-witnessed $\rSigma_{n+1}^{\Ntilde|\gamma}$ stem forcing relation} $\sforces{\PP^-}{\gamma\mu,n+1}$ of $\M_\gamma$ is
the relation of $(s,\varphi,\vec{x})$ as before, but now
\[ s\sforces{\PP^-}{\gamma\mu,n+1}\varphi(\vec{x})\iff\exists k<\om\all^*_k t\ \Big[s\conc t\sforces{\PP^-}{\gamma\mathrm{w},n+1}\varphi(\vec{x})\Big]\]
(in case it escapes the reader's visual attention, the distinction
 between the two forcing notions is 
 denoted by the differing superscripts ``$\mathrm{w}$''
 and ``$\mu$'').
 
Suppose now that $(\gamma,n)<(\beta^*,n^*)$. Then given $s\in\PP^-$ and $\vec{x}\in(\Nhat|\gamma)^{<\om}$, we use
$\sforces{\PP^-}{\gamma\mathrm{w},n+1}$ just like in  \ref{dfn:prior_condition_deciding} to define
the \dfnemph{condition $r=r^{\gamma}_{s,\varphi(\xvec)}$\footnote{We use different
notation here to help distinguish from
the similar notion in Definition \ref{dfn:prior_condition_deciding}.} deciding $\varphi(\xvec)$
for $\Ntilde|\gamma$ at $s$}.
We will observe later that $r\in\PP^{\gamma+n+1}\sub\PP$, $\stem(r)=s$ and
either $r\sforces{\PP}{\gamma}\varphi(\xvec)$ or
$r\sforces{\PP}{\gamma}\neg\varphi(\xvec)$).
Also as before, given also
  $\alpha<\om_1$, this determines
the \textbf{condition
$r^{\gamma,n+1}_{s,\alpha,\xvec}$ deciding
$\Th_{n+1}^{\Ntilde|\gamma}(\alpha\un\{\xvec\})$ at $s$} as the intersection
of all conditions $r^{\gamma}_{s,\varphi(\betavec,\xvec)}$, for $\varphi$ being
$\rSigma_{n+1}$ and $\betavec\in\alpha^{<\om}$. This gives
$r^{\gamma,n+1}_{s,\alpha,\xvec}\in\PP^{\gamma+n+1}$,
and also determines the \dfnemph{measure one
$\rSigma_{n+1}$-type $t=t^{\Ntilde|\gamma}_{s,n+1}(\alpha,\vec{x})$
of $\alpha\cup\{\vec{x}\}$ at $s$},
much as before. 
If $t\in Q=Q_{s\conc\left<z\right>}$
for all $z\in\PP^-_{s,1}$ such that $\alpha\leq\kappa$, where $\kappa$ is the least measurable of $Q$, let
$\tau=\tau^{\Ntilde|\gamma}_{s,n+1}(\alpha,\vec{x})$ denote the natural name in $\widehat{N}|\omega_1$ for
$t$; that is, just let $\tau$ be lexicographically least in $\Nhat|\omega_1$ such that $f^Q(\tau)=t$.
If $t\notin Q$ for such $Q$ then let $\tau\in\Nhat|\omega_1$ be the natural name for $\emptyset$. (We will show later that in fact,
$t\in Q$, by arguing like in the proof of  Lemma \ref{lem:rSigma_0_forcing_theorem_for_PP}\ref{item:each_ctbl_theory_decided_by_some_p}. We will  
 observe later that
\[ r^{\gamma,n+1}_{s,\alpha,\vec{x}}\sforces{\PP}{\gamma}\Type_{\rSigma_{n+1}}(o_\alpha\cup\{\vec{x}\})=\tau.\]
Let $t^{\Ntilde|\gamma}_{n+1,s}(\alpha)=t^{\Ntilde|\gamma}_{n+1,s}(\alpha,\pvec^{\M_\gamma}_{n+1})$,
unless $n=0$ (so $\gamma>\omega_1$) and $p_1^{\M_\gamma}=\emptyset$, in which case
let $t^{\Ntilde|\gamma}_{1,s}(\alpha)=t^{\Ntilde|\gamma}_{1,s}(\alpha,\{\omega_1\})$.
Define $\tau^{\Ntilde|\gamma}_{n+1,s}(\alpha)$ analogously
(so $\tau^{\Ntilde|\gamma}_{n+1,s}(\alpha)\in(\Nhat|\om_1)$). Let $C$ be the premouse
determined by $t^{\Ntilde|\gamma}_{n+1,s}(\alpha)$. Then $\vec{p}^{\Ntilde|\gamma}_{n+1,s}(\alpha)$ denotes $o_{\vec{p}}$,
where $\vec{p}$ is
 the transitive collapse
of $\pvec_{n+1}^{\M_\gamma}$ in $C$, or the transitive collapse of $\{\omega_1\}$ if $n=0$ and $p_1^{\M_\gamma}=\emptyset$.

The \dfnemph{witnessed $\rSigma_{n+2}^{\Ntilde|\gamma}$ stem forcing relation} $\sforces{\PP^-}{\gamma\mathrm{w},n+2}$
is the relation of $(s,\varphi,\vec{x})$
such that $s\in\PP^-$, $\varphi(\vec{v})$
is an $\rSigma_{n+2}$ formula of the passive premouse language, and where we define
\[ s\sforces{\PP^-}{\gamma\mathrm{w},n+2}\varphi(\vec{x}) \] iff                             
there are  $\beta<\rho_{n+1}^{\M_\gamma}$  and $\vec{\alpha}\in\beta^{<\om}$
and an $\rSigma_{n+1}$ min-term $u$
and $\vec{p},\vec{q},\tau$
such that: 
\begin{enumerate}[label=(\roman*)]
\item if $\omega_1<\rho_{n+1}^{\M_\gamma}$ then $\omega_1<\beta$,
 \item 
 if  $\rho_1^{\M_\gamma}=\omega_1$ and $p_1^{\M_\gamma}=\emptyset$ then $\vec{p}=\{\om_1\}$,
\item if $\rho_1^{\M_\gamma}>\omega_1$ or $p_1^{\M_\gamma}\neq\emptyset$ then $\pvec=\pvec_{n+1}^{\M_\gamma}$,
 \item \label{item:s_stem-forces_x=u(alpha,p)}
$s\sforces{\PP^-}{\gamma\mathrm{w},n+1}\text{``}\vec{x}=u(o_{\vec{\alpha}},o_{\pvec})\text{''}$, and
\item\label{item:s_stem-forces_tau_is_putative_witness} $s\sforces{\PP^-}{\gamma0}$``$\tau\text{  codes a putative witness to }(\varphi(\vec{v}),(u,(o_{\vec{\alpha}},o_{\vec{q}})))$'', where
\item $\tau=\tau^{\Ntilde|\gamma}_{n+1}(\beta)$
 and $\vec{q}=\vec{p}^{\Ntilde|\gamma}_{n+1}(\beta)$ 
 (recall that if $\om_1<\beta<\rho_{n+1}^{\M_\gamma}$,
 these names were defined via Definition \ref{dfn:tau(beta),q(beta)_if_omega_1<=beta},
 and if $\beta<\om_1=\rho_{n+1}^{\M_\gamma}$, they were defined in the previous paragraph).\qedhere
 \end{enumerate}
\end{dfn} 
\begin{dfn}\label{}
Let $G\sub\PP$ be a sufficiently generic filter (which might appear in some generic extension of $V$).
Given $(\gamma,n)\in\Lim\cross\om$ such that either:
\begin{enumerate}[label=--]
 \item  $\gamma=\omega_1$ and $n=1$, or
\item $\omega_1<\gamma\leq\beta^*$ and $n=0$,
\end{enumerate}
 we say that the \dfnemph{$\sforces{\PP^-}{\gamma n}$-stem forcing theorem holds for $G$}
 iff for all $\rSigma_n$ formulas $\varphi$
 in the language of passive premice
and all $\vec{x}\in(\Nhat|\gamma)^{<\om}$,
we have
\[ \Big((\Ntilde_G|\gamma)\Big)\sats\varphi(\vec{x}_G)\iff\exists s\in\PP^-\ \Big[\Big(s\sforces{\PP^-}{\gamma n}\varphi(\vec{x})\Big)\wedge r^\gamma_{s,\varphi(\vec{x})}\in G\Big].\]

Given $\gamma\in\Lim\cap(\omega_1,\beta^*]$,
we say that the \dfnemph{$\sforces{\PP^-}{\gamma\mathrm{w}1}$-stem forcing theorem holds for $G$}
iff 
for all $\rSigma_{1}$ formulas $\varphi$
in the language of passive premice
of form $\exists y\ \psi(y,\vec{v})$,
where $\psi$ is $\Sigma_0$, and
all $\vec{x}\in(\Nhat|\gamma)^{<\om}$,
we have  $(\Ntilde_G|\gamma)\sats\varphi(\vec{x}_G)$
iff
there is $s\in\PP^-$ such that
$s\sforces{\PP^-}{\gamma\mathrm{w}1}\varphi(\vec{x})$,
as witnessed by $\sigma\in\Nhat|\gamma$,
with $r^\gamma_{s,\psi(\sigma,\vec{x})}\in G$.

We say the \dfnemph{$\sforces{\PP^-}{\om_1\mathrm{w}2}$-stem forcing theorem holds for $G$} iff for all $\rSigma_2$ formulas $\varphi$
in the language of passive premice
and all $\vec{x}\in(\Nhat|\om_1)^{<\om}$,
we have $(\Ntilde_G|\gamma)\sats\varphi(\vec{x}_G)$
iff
there is $s\in\PP^-$
such that $s\sforces{\PP^-}{\om_1\mathrm{w}2}\varphi(\vec{x})$, as witnessed by $\theta$, where $(s,\theta)\in G$.

Given $(\gamma,n)$
 such that $(\omega_1,1)\leq(\gamma,n)<(\beta^*,n^*)$, we say that the \dfnemph{$\sforces{\PP^-}{\gamma\mathrm{w},n+2}$-stem forcing theorem holds for $G$}
iff 
for all $\rSigma_{n+2}$ formulas $\varphi$
in the language of passive premice
and
all $\vec{x}\in(\Nhat|\gamma)^{<\om}$,
we have  $(\Ntilde_G|\gamma)\sats\varphi(\vec{x}_G)$
iff there is $s\in\PP^-$
such that $s\sforces{\PP^-}{\gamma\mathrm{w},n+2}\varphi(\vec{x})$, as witnessed by $\beta,\vec{\alpha},u,\vec{p},\vec{q},\tau$,
with
\begin{enumerate}[label=--]
 \item 
$r^\gamma_{s,\psi_1(\vec{x},u,o_{\vec{\alpha}},o_{\vec{p}})}\in G$
and $r^\gamma_{s,\psi_2(\tau,\varphi,u,o_{\vec{\alpha}},o_{\vec{q}})}\in G$,
where  $\psi_1(\ldots)$ and $\psi_2(\ldots)$
are the statements respectively forced in \ref{item:s_stem-forces_x=u(alpha,p)} and \ref{item:s_stem-forces_tau_is_putative_witness} of
the definition of $\sforces{\PP^-}{\gamma\mathrm{w},n+2}$ (in \ref{dfn:stem_forcing_relations}),
and
\item if $\rho_{n+1}^{\M_\gamma}=\om_1$ then
$r^{\gamma,n+1}_{s,\beta,\pvec}\in G$.
\end{enumerate}

For these stem forcing relations $\sforces{\cdot}{\cdot}$, we say that the \dfnemph{$\sforces{\cdot}{\cdot}$-stem forcing theorem holds}
iff for all sufficiently large $\alpha$,
$\Coll(\om,V_\alpha)$ forces that for all sufficiently generic filters $G\sub\PP$,
the $\sforces{\cdot}{\cdot}$-stem forcing theorem holds for $G$.
\end{dfn}

\begin{dfn}
 $o_\alpha$
denotes the natural name in $\Nhat$ for $\alpha$,
and $O_\alpha$ denotes the natural name in $\Mhat$
for $\alpha$.\footnote{***This definition should maybe be moved.}

Notation:\footnote{***This notation might not have been fully integrated.} $\Ntilde|\gamma$ is the name for the model of height $\gamma$, and $\Nhat|\gamma$ is the collection of canonical names for its elements (indexed below $\gamma$). Similarly,
$\Mtilde_\gamma$ is the name for the model of height $\gamma$, and $\Mhat_\gamma$ the collection
of canonical names for its elements (again indexed below $\gamma$).
\end{dfn}

\begin{rem}
 Note that although we have only claimed to define
 conditions in $\PP=\PP^{\beta^*+n^*}$,  we have
 also defined both the witnessed and $\mu$-witnessed $\rSigma_{n^*+1}$  stem forcing relations.
\end{rem}

\begin{dfn}\label{dfn:tau^Mtilde(beta),q(beta)_if_omega_1<=beta}
***Given the appropriate assumptions on $N$: Let  $n>0$ with $\lambda^N<\rho_n^{N}$.
Let $\beta\in(\lambda^N,\rho_n^N)$.
Let \[ C=\cHull_{\rSigma_n}^{N}(\beta\cup\pvec_n^{N}) \]
and $\pi:C\to N$ the uncollapse and
 $\vec{p}=\pi^{-1}(\pvec_n^{N})$.
Then $\vec{p}^{\Mtilde^N}_n(\beta)$
denotes $O_{\vec{p}}$, and $\tau^{\Mtilde^N}_n(\beta)$ denotes
the natural name in $\Mhat_{\OR^C+\om}$ for
\[ \Th_{\muSigma_n}^{\Mtilde_{\OR^C}}(O_\beta\cup\widetilde{\HC}\cup\{O_{\vec{p}}\}) \]
(using the natural $\mu$-definition to write the formula of the $\M(\RR)$ language used in specifying the name).
\end{dfn}

\begin{dfn}\label{dfn:CC_forcing_relations}
  The \dfnemph{$\mSigma_0^{\Mtilde_{\lambda^N}}$
 forcing relation $\sforces{\CC}{\lambda^N0}$ of $N|\lambda^N$} is the relation $\sforces{\lambda0}{}$  (in the sense of $N$) specified in Definition \ref{dfn:mSigma_0_forcing_relation_Mtilde_lambda}.
  
 Let $\gamma\in\Lim\cap(\lambda^N,\OR^N]$.
 The \dfnemph{$\mSigma_0^{\Mtilde_\gamma}$
 forcing relation $\sforces{\CC}{\gamma0}$ of $N|\gamma$}
  is just
 $\sforces{\CC}{N|\gamma,\mathrm{int}}$ (see Definition \ref{dfn:internal_mSigma_0_forcing_relation}); so
for $p\in\CC^N$, $\mSigma_0$ formulas $\varphi$ and $\vec{x}\in(\Mhat_\gamma)^{<\om}$, we have
 \[ p\sforces{\CC}{\gamma0}\varphi(\vec{x})\iff\psi_0(\lambda^N,p,\varphi,\vec{x}).\]

Now let $(\gamma,n)$ be such that
$(\lambda^N,0)\leq(\gamma,n)\leq(\OR^N,n_0)$.
We define the \dfnemph{witnessed $\mSigma_{n+1}^{\Mtilde_\gamma}$ forcing relation
$\sforces{\CC}{\gamma\mathrm{w},n+1}$} of $N|\gamma$.

 The \dfnemph{witnessed $\mSigma_1^{\Mtilde_\gamma}$ forcing relation} $\sforces{\CC}{\gamma\mathrm{w}1}$
 is the relation of tuples $(p,\varphi,\vec{x})$ such that $p\in\CC$, $\varphi$ is $\mSigma_1$ of form
 \[ \varphi(\vec{u})\ \iff\ \exists y\ \psi(y,\vec{u}) \]
 where $\psi$ is $\mSigma_0$, and $\vec{x}\in(\Mhat_\gamma)^{<\om}$, and where we define
 \[p\sforces{\CC}{\gamma\mathrm{w}1}\varphi(\vec{x})\iff \exists\sigma\in\Mhat_\gamma\ \Big[p\sforces{\CC}{\gamma0}\psi(\sigma,\vec{x})\Big].\]
(This relation is $\rSigma_1^{N|\gamma}(\{\lambda^N\})$
if $\gamma>\lambda^N$,
and is $\rSigma_2^{N|\lambda^N}$
if $\gamma=\lambda^N$. Note that even if $\OR^N=\lambda^N$,
we have $\rho_1^{N|\lambda^N}=\lambda^N$.)

For $m<\om$,
the \dfnemph{$m$-good $\muSigma_{1}^{\Mtilde_\gamma}$ forcing relation} $\sforces{\CC,\geq m}{\gamma\mu,1}$ is the relation
of $(p,\varphi,\vec{x})$
such that $p\in\CC$, $\varphi$
is $\muSigma_{1}$ of form
\[ \varphi(v)\iff\exists k<\om\all^*_ks\ \psi(s,v) \]
where $\psi$ is $\mSigma_{1}$,
and $\vec{x}\in(\Mhat_\gamma)^{<\om}$,
and where we define
\[ p\sforces{\CC,\geq m}{\gamma\mu,1}\varphi(\vec{x}) \]
iff letting $d=\max(\supp(p),\supp(\vec{x}))$ and $m'=\max(m,d+1)$, either
\begin{enumerate}[label=(\roman*)]
 \item\label{item:gamma>lambda^N_m-good_mu_Sigma_1_forcing} $\gamma>\lambda^N$
 and there are $k<\om$, $\vec{\delta}\in[\Delta^N_{\geq m'}]^{2k}$, $\beta\in[\lambda^N,\gamma)\cap\Lim$ and $i<\om$ with
$\max(\loc(\vec{x}))\leq\beta$
and such that
\[ p\sforces{\CC_d}{}\text{``}\all^{\mathrm{gen}}_{\vec{\delta}}s\  \sforces{\CC_{\mathrm{tail}}}{\gamma0}\Mtilde_{\beta+i}\sats\psi(s,\vec{x})\text{''}, \]
 or
\item\label{item:gamma=lambda^N_m-good_mu_Sigma_1_forcing} $\gamma=\lambda^N$ and
\[ \psi(s,\vec{x})\iff \exists y\ \psi'(s,\vec{x},y) \]
where $\psi'$ is $\mSigma_0$ and there are $k<\om$, $\vec{\delta}'\in[\Delta^N_{\geq m'}]^{2k+1}$
such that, letting $\vec{\delta}=\vec{\delta}'\rest 2k$
and $\delta_{2k}=\delta_i^N$,
\[ p\sforces{\CC_d}{}\text{``}\all^{\mathrm{gen}}_{\vec{\delta}}s\ \sforces{\CC_{i}}{}\exists y\in\HC\ \sforces{\CC_{\mathrm{tail}}}{\lambda^N0}\Mtilde_{\lambda^N}\sats\psi'(s,\vec{x},y)\text{''}. \]
\end{enumerate}
(Condition \ref{item:gamma>lambda^N_m-good_mu_Sigma_1_forcing} is $\rSigma_{1}^{N|\gamma}(\{\lambda^N\})$,
and condition \ref{item:gamma=lambda^N_m-good_mu_Sigma_1_forcing}
is $\rSigma_2^{N|\lambda^N}$.)

Given the $m$-good
$\muSigma_{n+1}^{\Mtilde_\gamma}$ forcing relations $\sforces{\CC,\geq m}{\gamma\mu,n+1}$,
we then define the \dfnemph{(stable) $\muSigma_{n+1}^{\Mtilde_\gamma}$ forcing relation} $\sforces{\CC}{\gamma\mu,n+1}$ (thus, without the parameter $m$) as \[ p\sforces{\CC}{\gamma\mu,n+1}\varphi(\vec{x})\iff\all m<\om\  \Big[p\sforces{\CC,\geq m}{\gamma\mu,n+1}\varphi(\vec{x})\Big].\]

Suppose $n>0$.
If $\gamma>\lambda^N$,
let $n'=n$, and if $\gamma=\lambda^N$
let $n'=n+1$. Suppose $\rho_{n'}^{N|\gamma}=\lambda^N$.
Let $\varrho\in\Mhat_{\lambda^N}$
and $d=\supp(\varrho)$.
Then $\tau^{\Mtilde_\gamma}_n(\varrho)$ denotes the $\CC_d$-name
$\tau\in N|\lambda^N$ such that
\[ \sforces{\CC_d}{}\tau=\Big\{\varphi\Bigm| \ (\varphi\text{ is }\muSigma_n)\wedge\Big(\emptyset\sforces{\CC_{\mathrm{tail}}}{\gamma\mu,n}\ \varphi(\varrho,\pvec)\Big)\Big\}\]
where $\pvec=\pvec_n^{N|\gamma}\cut\{\lambda^N\}$ (note this uses the (stable) $\muSigma_{n}^{\Mtilde_\gamma}$
forcing relation, as defined in $\N[g]$, using $\CC_{\mathrm{tail}}$ there). We will see later
that this is indeed a name
in $N|\lambda^N$, and in fact can be computed from
\[ \Th_{\rSigma_{n'}}^{N|\gamma}(\{\varrho,\lambda^N,\pvec_n^{N|\gamma}\}\cup(\delta_d^N)^{+N}),\]
which is in $N$ since $\rho_{n'}^{N|\gamma}=\lambda^N$.

If $n>0$, the \dfnemph{witnessed $\mSigma_{n+1}^{\Mtilde_\gamma}$ forcing relation}
$\sforces{\CC}{\gamma\mathrm{w},n+1}$
is the relation of tuples $(p,\varphi,\vec{x})$
such that $p\in\CC^N$, $\varphi(\vec{v})$ is $\mSigma_{n+1}$, $\vec{x}\in(\Mhat_\gamma)^{<\om}$, and where we define
\[ p\sforces{\CC}{\gamma\mathrm{w},n+1}\varphi(\vec{x})\]
iff  either
\begin{enumerate}[label=\tu{(}\roman*\tu{)}]
 \item\label{item:rho_n^N|gamma>lambda^N_etc}   $\rho_n^{N|\gamma}>\lambda^{N}$\footnote{This is not a typo.
We will show later that $\rho_n^{N|\gamma}=\rho_n^{\M_\gamma}$, but since we are presently looking at things from the persepctive of $N$, we write it as ``$\rho_n^{N|\gamma}>\lambda^N$''.}and there are $\beta\in(\lambda^N,\rho_n^{N|\gamma})$ 
and $\vec{\alpha}\in[\beta]^{<\om}$
and $\varrho\in\Mhat_{\lambda^N}$
and an $\mSigma_n$ min-term $u$ such that
\begin{enumerate}[label=(\alph*)]
\item\label{item:clause_b}
$p\sforces{\CC}{\gamma\mathrm{w}n}\text{``}\vec{x}=u(\varrho,O_{\vec{\alpha}},O_{\pvec})\text{''}$, and
 \item\label{item:clause_a} $p\sforces{\CC}{\gamma0}\text{``}\tau$ codes a putative witness to $(\varphi(\vec{v}),(u,(\varrho,O_{\vec{\alpha}},O_{\vec{q}})))$'',
\end{enumerate}
where $\pvec=\pvec_n^{N|\gamma}$
and $\tau=\tau^{\Mtilde_\gamma}_n(\beta)$
and $\vec{q}=\vec{p}^{\Mtilde_\gamma}(\beta)$,
or
\item\label{item:gamma>rho_n_witnessed_mSigma_n+1_forcing_rel} $\gamma>\rho_n^{N|\gamma}=\lambda^N$
and there are $\varrho\in\Mhat_{\lambda^N}$
and an $\mSigma_n$ min-term $u$ such that
\begin{enumerate}[label=(\alph*)]
\item\label{item:clause_b'}
$p\sforces{\CC}{\gamma\mathrm{w}n}\text{``}\vec{x}=u(\varrho,O_{\pvec})\text{''}$, and
 \item\label{item:clause_a'} $p\sforces{\CC}{\gamma0}\text{``}\tau\text{ codes a putative witness to } (\varphi(\vec{v}),(u,(\varrho,O_{\pvec})))$'',
\end{enumerate}
where $\pvec=\pvec_n^{N|\gamma}\cut\{\lambda^N\}$ and $\tau=\tau^{\Mtilde_\gamma}_n(\varrho)$, or
\item $\gamma=\rho_{n+1}^{N|\gamma}=\lambda^N$
and there are $\varrho\in\Mhat_{\lambda^N}$
and an $\mSigma_n$ min-term $u$ such that
\begin{enumerate}[label=(\alph*)]
\item\label{item:clause_b''} 
$p\sforces{\CC}{\gamma\mathrm{w}n}\text{``}\vec{x}=u(\varrho)\text{''}$, and
 \item\label{item:clause_a''} $p\sforces{\CC}{\gamma0}\text{``}\tau\text{ codes a putative witness to } (\varphi(\vec{v}),(u,\varrho))$'',
\end{enumerate}
where $\tau=\tau^{\Mtilde_\gamma}_n(\varrho)$ (in this last case
we could have taken $\varrho$
to be a name for a real coding $\vec{x}$ directly in some simple way, instead of using a min-term $u$).
\end{enumerate}
 
If $n>0$ and $m<\om$, the \dfnemph{$m$-good $\muSigma_{n+1}^{\Mtilde_\gamma}$ forcing relation} $\sforces{\CC,\geq m}{\gamma\mu,n+1}$ is the relation
of $(p,\varphi,\vec{x})$
such that $p\in\CC$, $\varphi$
is $\muSigma_{n+1}$ of form
\[ \varphi(\vec{v})\iff\exists k<\om\all^*_kv_0\ \psi(v_0,\vec{v}) \]
where $\psi$ is $\mSigma_{n+1}$,
and $\vec{x}\in(\Mhat_\gamma)^{<\om}$,
and where we define
\[ p\sforces{\CC,\geq m}{\gamma\mu,n+1}\varphi(\vec{x}) \]
iff letting $d=\max(\supp(p),\supp(\vec{x}))$
and $m'=\max(m,d+1)$, either
\begin{enumerate}[label=(\roman*)]
 \item\label{item:rho_n>lambda_m-good_muSigma_n+1_forcing_rel} $\rho_n^{N|\gamma}>\lambda$
 and there are $k<\om$, $\vec{\delta}\in[\Delta^N_{\geq m'}]^{2k}$, $\beta\in(\lambda^N,\rho_n^{N|\gamma})$, $\vec{\alpha}\in\beta^{<\om}$, $\varrho\in\Mhat_{\lambda^N}$,
and an $\mSigma_n$ min-term $u$ such that
\begin{enumerate}[label=(\alph*)]
 \item\label{item:rho_n>lambda_m-good_muSigma_n+1_forcing_rel_clause_a}  $p\sforces{\CC}{\gamma\mathrm{w}n}$``$\vec{x}=u(\varrho,O_{\vec{\alpha}},O_{\vec{p}})$'', and
 \item\label{item:rho_n>lambda_m-good_muSigma_n+1_forcing_rel_clause_b}
$p\sforces{\CC_d}{}$``$\all^{\mathrm{gen}}_{\vec{\delta}}s\ \sforces{\CC_{\mathrm{tail}}}{\gamma0}\Big[\tau\text{ codes a putative witness to }(\psi(v_0,\vec{v}),(u',(s,(\varrho,O_{\vec{\alpha}},O_{\vec{q}}))))\Big]$'',
\end{enumerate}
where $\pvec=\pvec_n^{N|\gamma}$
and $u'$ is the $\mSigma_n$ min-term given by $u'(a,b)=(a,u(b))$ 
and $\tau=\tau_n^{\Mtilde_\gamma}(\beta)$ and $\vec{q}=\vec{p}^{\Mtilde_\gamma}(\beta)$, or
\item $\gamma>\rho_n^{N|\gamma}=\lambda^N$
and there are $k<\om$,
$\vec{\delta}'\in[\Delta^N_{\geq m'}]^{2k+1}$, $\varrho\in\Mhat_{\lambda^N}$
and an $\mSigma_n$ min-term $u$ such that,
letting $\vec{\delta}=\vec{\delta}'\rest 2k$ and $\delta_{2k}=\delta_i^N$,
\begin{enumerate}[label=(\alph*)]
\item\label{item:clause_b'''}
$p\sforces{\CC}{\gamma\mathrm{w}n}\text{``}\vec{x}=u(\varrho,O_{\pvec})\text{''}$, and
 \item\label{item:clause_a'''}
 \[ \begin{split}p\sforces{\CC_d}{}\text{``}&\all^{\mathrm{gen}}_{\vec{\delta}}s\
 \sforces{\CC_i}{}\exists y\in\HC\ \sforces{\CC_{\mathrm{tail}}}{\gamma0}\\&\Big[\tau(\varrho,s,y)\text{ codes a putative witness to } (\psi(v_0,\vec{v}),(u',((\varrho,s,y),O_{\pvec})))\Big]\text{''},\end{split}\]
\end{enumerate}
where $\pvec=\pvec_n^{N|\gamma}\cut\{\lambda^N\}$ and $\tau(\varrho,s,y)=(\tau^{\Mtilde_\gamma}_n(\varrho,s,y))^{\CC_{\max(\vec{\delta})}}$ (that is,  the latter name
gets computed in $(N|\gamma)[g]$,
where $g$ is $(N,\CC\rest(i+1))$-generic),
and $u'$ is the min-term given by
\[u'((a,b,c),d)=(u(a,d),b) \] (so
 $u'((\varrho,s,y),\pvec)=(u(\varrho,\pvec),s)$, yielding $(\vec{x},s)$ where desired), or
\item $\gamma=\rho_{n+1}^{N|\gamma}=\lambda^N$
and there are $k<\om$,
$\vec{\delta}'\in[\Delta^N_{\geq m'}]^{2k+1}$, $\varrho\in\Mhat_{\lambda^N}$
and an $\mSigma_n$ min-term $u$ such that,
letting $\vec{\delta}=\vec{\delta}'\rest 2k$ and $\delta_{2k}=\delta_i^N$,
\begin{enumerate}[label=(\alph*)]
\item\label{item:clause_b''''}
$p\sforces{\CC}{\gamma\mathrm{w}n}\text{``}\vec{x}=u(\varrho)\text{''}$, and
 \item\label{item:clause_a''''} \[\begin{split}p\sforces{\CC_d}{}\text{``}&\all^{\mathrm{gen}}_{\vec{\delta}}s\
 \sforces{\CC_i}{}\exists y\in\HC\ \sforces{\CC_{\mathrm{tail}}}{\gamma0}\\&\Big[\tau(\varrho,s,y)\text{ codes a putative witness to } (\psi(v_0,\vec{v}),(u',((\varrho,s,y))))\Big]\text{''}\end{split}\],
\end{enumerate}
where $\tau(\varrho,s,y)=(\tau^{\Mtilde_\gamma}_n(\varrho,s,y))^{\CC_{\max(\vec{\delta})}}$ 
and $u'$ is the min-term given by
\[u'((a,b,c))=(u(a),b).\qedhere\]
\end{enumerate}
\end{dfn}
We also consider the following variant
of the $\muSigma_{n+1}$ forcing relation:
\begin{dfn}\label{dfn:proj_forcing_rel}
Let $n<\om$
be such that
\begin{enumerate}[label=(\roman*)]
\item\label{item:gamma>lambda} $\lambda^N<\gamma$ and $\rho_{n+1}^{N|\gamma}\leq\lambda^N$ (so if $\rho_{n+1}^{N|\gamma}<\lambda^N$
then $\gamma=\OR^N$ and $n=n_0$), or
\item\label{item:gamma=lambda} $\lambda^N=\gamma$ and $\rho_{n+2}^{N|\gamma}\leq\lambda^N=\rho_{n+1}^{N|\gamma}$ (so if $\rho_{n+2}^{N|\gamma}<\lambda^N$ then $n_0=n+1$ and $\lambda^N=\OR^N$)
\end{enumerate}
Let $m<\om$. Then the \dfnemph{$m$-good $\muSigma_{n+1}^{\Mtilde_\gamma}$ projecting forcing relation} $\sforces{\CC,\geq m}{\gamma\mu,n+1,\mathrm{proj}}$ is the relation
of $(p,\varphi,\vec{x})$
such that 
$p\in\CC$, $\varphi$
is $\muSigma_{n+1}$ and $\vec{x}\in(\Mhat_\gamma)^{<\om}$,
and where we define
\[ p\sforces{\CC,\geq m}{\gamma\mu,n+1,\mathrm{proj}}\varphi(\vec{x}) \]
iff there are $m'<\om$ and  $\sigma\in\Mhat_{\lambda^N}$ and an $\mSigma_{n+1}$ min-term $r$ such that
$\max(\supp(p),\supp(\vec{x}),\supp(\sigma),m)<m'$, and letting
\begin{enumerate}[label=(\roman*)]
 \item 
$\vec{p}=\pvec_{n+1}^{N|\gamma}\cut\{\lambda^N\}$,\footnote{Clearly it wouldn't matter if we left $\lambda^N\in\vec{p}$ if it happened to be in $\vec{p}^{N|\gamma}_{n+1}$,
and it also wouldn't affect the definability,
but it is more natural to remove it, since we will later verify that in the context of interest,
$\vec{p}=\vec{p}^{\M_\gamma}_{n+1}$.}
if $\lambda^N<\gamma$, and
\item $\vec{p}=\pvec_{n+2}^{N|\gamma}$, if $\lambda^N=\gamma$,
\end{enumerate}
then letting $\vec{\xi}=\loc(\vec{x})$,
\begin{enumerate}[label=--]
\item $p\sforces{\CC}{\gamma\mathrm{w},n+1}\text{``}O_{\vec{\xi}}=r(\sigma,\vec{p})\text{''}$ and
\item 
 $p\sforces{\CC,\geq m'}{\gamma\mu,n+1}\varphi(\vec{x})$.\qedhere
 \end{enumerate}
\end{dfn}
\begin{dfn}\label{dfn:CC_forcing_theorems}
For the $\CC$-forcing relations  $\sforces{
\cdot}{\cdot}$ of $N|\gamma$
 introduced
in Definition \ref{dfn:CC_forcing_relations} we say that the \dfnemph{forcing theorem holds for $\sforces{\cdot}{\cdot}$}
iff for all sufficiently large ordinals $\alpha$, in $V^{\Coll(\om,\alpha)}$,
for all formulas $\varphi$
of the relevant complexity,
all $\vec{x}\in(\Mhat_\gamma)^{<\om}$
and all sufficiently generic filters $G\sub\CC^N$, we have
\[ \Big((\Mtilde_G)_\gamma\sats\varphi(\vec{x}_G)\Big)\ \iff\ \exists p\in G\ \Big[p\sforces{\cdot}{\cdot}\varphi(\vec{x})\Big].\qedhere\]
\end{dfn}

We can now proceed to the central facts regarding the fine structural correspondence between $\M_{\beta^*}$
and sufficiently Prikry generic iterates of $\Pg$.
We break the presentation of the facts into two cases --
the first case for $\M_\gamma$ and $N|\gamma$ where $\gamma>\omega_1$, and the second case when $\gamma=\omega_1$. The two cases are very similar to one another, with the main difference being just
that  when $\gamma>\omega_1$,
$\mSigma_{n+1}^{\M_\gamma}$ corresponds to $\rSigma_{n+1}^{N|\gamma}$,
but when $\gamma=\omega_1$,
$\mSigma_{n+1}^{\M_{\omega_1}}$ corresponds to $\rSigma_{n+2}^{N|\gamma}$.
Presenting both cases
simultaneously would make everything less readable,
so we split into the two cases, even though this leads to some repetition.

\begin{lem}\label{lem:fs_correspondence}
Let $N$ be a sufficiently Prikry generic iterate of $\Pg$.

Suppose $\lambda^{\Pg}<\OR^{P_0}$,
so $\om_1<\beta^*$.  Let $(\om_1+
\om,0)\leq(\gamma,n)\leq(\beta^*,\min(n^*,n_0))$.\footnote{We will show in Lemma \ref{lem:n^*=n_0} that
(as $\omega_1<\beta^*$),
 actually $n^*=n_0$.} Then
we have:
\begin{enumerate}
  \item \label{item:rho_n=rho_n}
  $\rho_{n}^{\M_\gamma}=\rho_{n}^{N|\gamma}\geq\lambda^N=\om_1$.
  \item\label{item:p_n=p_n} Either:
  \begin{enumerate}[label=--]
  \item $p_{n}^{\M_\gamma}=p_{n}^{N|\gamma}$, or
  \item $n=1$, $\rho_1^{\M_\gamma}=\om_1=\lambda^N=\rho_1^{N|\gamma}$, $p_1^{\M_\gamma}=\emptyset$,
  $p_1^{N|\gamma}=\{\lambda^N\}$ and $\lambda^N$ is the largest cardinal of $N|\gamma$,
  \end{enumerate}
  \item\label{item:1-CC-forcing} Regarding forcing $\Mtilde_\gamma$
  over $N|\gamma$ with $\CC^N$:
  \begin{enumerate}[label=\tu{(}\alph*\tu{)}]
     \item\label{item:1-Delta_0^M} the $\mSigma_0^{\Mtilde_\gamma}$ forcing
    relation $\sforces{\CC}{\gamma0}$ is $\rDelta_1^{N|\gamma}(\{\lambda^N\})$,
   \item\label{item:1-witnessed_1^N} the witnessed $\mSigma_{n+1}^{\Mtilde_\gamma}$ forcing relation $\sforces{\CC}{\gamma\mathrm{w},n+1}$ is:
   \begin{enumerate}[label=--]
    \item $\rSigma_1^{N|\gamma}(\{\lambda^N\})$, if $n=0$,
    \item $\rSigma_{n+1}^{N|\gamma}(\pvec_n^{N|\gamma})$, if $n>0$,
   \end{enumerate}
   \item\label{item:1-mu_1^N}
  the $\muSigma_{n+1}^{\Mtilde_\gamma}$ forcing
   relation $\sforces{\CC}{\gamma\mu,n+1}$  is:
   \begin{enumerate}[label=--]
     \item $\all^\om \rSigma_{1}^{N|\gamma}(\{\lambda^N\})$,
   if $n=0$,
   \item $\all^\om \rSigma_{n+1}^{N|\gamma}(\pvec_n^{N|\gamma})$, if $n>0$,
   \end{enumerate}
  \item\label{item:mu_Sigma_n+1_projecting_forcing_rel_def} if $\rho_{n+1}^{N|\gamma}\leq\lambda^N$, then for each $m<\om$, the $m$-good $\muSigma_{n+1}^{\Mtilde_\gamma}$ projecting forcing
   relation $\sforces{\CC,\geq m}{\gamma\mu,n+1,\mathrm{proj}}$ is
   $\rSigma_{n+1}^{N|\gamma}(\{\pvec_{n+1}^{N|\gamma},m\})$,\footnote{This relation was introduced in \ref{dfn:proj_forcing_rel}. Note that for this relation, the parameter is $(\pvec_{n+1}^{N|\gamma},m)$,
 as opposed to just $\pvec_n^{N|\gamma}$ or $\{\lambda^N\}$.
Of course it follows
that the relation is also $\rSigma_{n+1}^{N|\gamma}(\{\pvec_{n+1}^{N|\gamma}\})$,
i.e.~we can dispense with the trivial parameter $m$. However, including the $m$ allows us to assert the uniformity of definitions in part \ref{item:forcing_rels_unif_def}.}
  \end{enumerate}
  \item\label{item:1-PP-forcing} Regarding forcing $\Ntilde|\gamma$ over $\M_\gamma$ with $\PP$:
  \begin{enumerate}[label=\tu{(}\alph*\tu{)}]
     \item\label{item:1-Delta_0^N} the  $\rSigma_0^{\Ntilde|\gamma}$ stem forcing
    relation $\sforces{\PP^-}{\gamma0}$ is $\mDelta_1^{\M_\gamma}(\{\xg\})$,
 
   \item\label{item:witnessed_rSigma_n+1_forcing_rel_def} 
     the witnessed $\rSigma_{n+1}^{\Ntilde|\gamma}$ stem forcing
   relation
   $\sforces{\PP^-}{\gamma\mathrm{w},n+1}$
   is $\mSigma_{n+1}^{\M_\gamma}(\{\xg,\pvec_n^{\M_\gamma}\})$,
   \item\label{item:mu-witnessed_forcing_rel_def} 
   the $\mu$-witnessed  $\rSigma_{n+1}^{\Ntilde|\gamma}$ stem forcing
   relation
   $\sforces{\PP^-}{\gamma\mu,n+1}$
   is $\muSigma^{\M_\gamma}_{n+1}(\{\xg,\pvec_n^{\M_\gamma}\})$.
  \end{enumerate}
  \item\label{item:forcing_rels_unif_def}  The forcing/stem-forcing relations are moreover uniformly definable as follows:
  \begin{enumerate}[label=\tu{(}\alph*\tu{)}]
  \item\label{item:mSigma_0_forcing_rel_unif_def} $\sforces{\CC}{\gamma0}$
  is 
 $\rDelta_1^{N|\gamma}(\{\lambda^N\})$ uniformly in limits $\gamma\in(\lambda^N,\beta^*]$;
  that is,
  there are $\rSigma_1$ formulas
  $\psi,\psi'$ such that
  for each such $\gamma$ and all $p,\varphi,\vec{x}\in N|\gamma$,
  \[ \Big(p\sforces{\CC}{\gamma0}\varphi(\vec{x})\Big)\iff \Big(N|\gamma\sats\psi(p,\varphi,\vec{x})\Big)\iff \Big(N|\gamma\sats\neg\psi'(p,\varphi,\vec{x})\Big), \]
  \item\label{item:rSigma_0_stem_forcing_rel_unif_def}  $\sforces{\PP^-}{\gamma0}$
  is likewise $\mDelta_1^{\M_\gamma}(\{x_0\})$
  uniformly in limits $\gamma\in(\omega_1,\beta^*]$,
  \item\label{item:mSigma_n+1,muSigma_n+1_forcing_rel_unif_def_lambda<rho_n} the 
   forcing relations mentioned
   in  \ref{item:1-CC-forcing}\ref{item:1-witnessed_1^N},\ref{item:1-mu_1^N},\ref{item:mu_Sigma_n+1_projecting_forcing_rel_def}
 are defined in the stated manner
uniformly in pairs $(\gamma,n)\in\Lim\cross\omega$ satisfying
\begin{equation}\label{eqn:range_for_gamma,n} (\lambda^N+\om,0)\leq(\gamma,n)\leq(\beta^*,\min(n^*,n_0))
\text{ and }\lambda^{N}<\rho_n^{N|\gamma},\end{equation}
   meaning that there is a recursive
   function $n\mapsto\psi_n$
   such that $\psi_n$ is a formula
   of the stated complexity
   and for each $(\gamma,n)$ as in line 
   \tu{(}\ref{eqn:range_for_gamma,n}\tu{)},
 $\psi_n(\vec{p},\cdot,\cdot,\cdot)$ defines the stated
   forcing relation over $N|\gamma$
   when $\pvec$ is the stated parameter
   \tu{(}note that the use of the otherwise trivial parameter $m$
   in \ref{item:1-CC-forcing}
\ref{item:mu_Sigma_n+1_projecting_forcing_rel_def} is needed here\tu{)},
   \item\label{item:mSigma_n+1,muSigma_n+1_forcing_rel_unif_def_lambda=rho_n} part \ref{item:mSigma_n+1,muSigma_n+1_forcing_rel_unif_def_lambda<rho_n}
   still holds after  replacing ``$\lambda^N<\rho_n^{N|\gamma}$''
   with ``$\lambda^N=\rho_n^{N|\gamma}$'' \tu{(}but now the witnessing formulas $\psi_n$  are different\tu{)},
   \item parts \ref{item:mSigma_n+1,muSigma_n+1_forcing_rel_unif_def_lambda<rho_n}
   and \ref{item:mSigma_n+1,muSigma_n+1_forcing_rel_unif_def_lambda=rho_n}
   still hold
   after replacing `` \ref{item:1-CC-forcing}\ref{item:1-witnessed_1^N},\ref{item:1-mu_1^N},\ref{item:mu_Sigma_n+1_projecting_forcing_rel_def}'' with
   ``\ref{item:1-PP-forcing}\ref{item:witnessed_rSigma_n+1_forcing_rel_def},\ref{item:mu-witnessed_forcing_rel_def}''
and ``$\lambda^N$'' with ``$\omega_1$''
and ``$N|\gamma$'' with ``$\M_\gamma$''.
   \end{enumerate}
      \item\label{item:CC_PP_forcing_theorems} We have:
      \begin{enumerate}[label=\tu{(}\alph*\tu{)}]
      \item the  forcing theorem holds for each of the  forcing relations mentioned in parts \ref{item:1-CC-forcing}\ref{item:1-Delta_0^M}--\ref{item:1-mu_1^N},
      \item\label{item:projecting_forcing_relation_forcing_theorem} if $\rho_{n+1}^{N|\gamma}\leq\lambda^N$, then
there is $m_0<\om$, which depends on $\gamma$ but not on $n$, such that for all $m\in[m_0,\om)$,
the forcing theorem holds for the $m$-good $\muSigma_{n+1}^{\Mtilde_\gamma}$ projecting forcing
   relation $\sforces{\CC,\geq m}{\gamma\mu,n+1,\mathrm{proj}}$ \tu{(}mentioned in part \ref{item:1-CC-forcing}\ref{item:mu_Sigma_n+1_projecting_forcing_rel_def}\tu{)}; moreover, if $\gamma=\beta^*$ then $m_0=0$ suffices,
   \item the stem forcing theorem holds for the forcing relations mentioned in parts \ref{item:1-PP-forcing}\ref{item:1-Delta_0^N},\ref{item:witnessed_rSigma_n+1_forcing_rel_def}.
   \end{enumerate}
  \item\label{item:Hull_n+1_ordinals_match} Let $X\sub\gamma$. Then:
  \begin{enumerate}[label=\tu{(}\alph*\tu{)}]
   \item \label{item:Hull_rho_p_match_n=0}
If $n=0$ then
  \begin{enumerate}[label=\tu{(}\roman*\tu{)}]
\item\label{item:Hull_match_n=0} $\gamma\cap\Hull_{\mSigma_{1}}^{\M_\gamma}(\HC\cup X)=
   \gamma\cap\Hull_{\rSigma_{1}}^{N|\gamma}(\lambda^N\cup X\cup\{\lambda^N\})$,
\item\label{item:rho_match_n=0}   $\rho_{1}^{\M_\gamma}=\max(\rho_{1}^{N},\lambda^N)$, and
\item\label{item:p_match_n=0} either:
  \begin{enumerate}[label=--]
   \item $p_1^{\M_\gamma}=p_1^{N|\gamma}$, or
   \item $\rho_1^{\M_\gamma}=\om_1=\lambda^N\geq\rho_1^{N|\gamma}$, $p_1^{\M_\gamma}=\emptyset$, $p_1^{N|\gamma}=\{\lambda^N\}$ and $\lambda^N$ is the largest cardinal of $N|\gamma$.
  \end{enumerate}
  \end{enumerate}
\item  If $n>0$ then
  \begin{enumerate}[label=\tu{(}\roman*\tu{)}]
\item $\gamma\cap\Hull_{\mSigma_{n+1}}^{\M_\gamma}(\HC\cup X\cup \pvec_n^{\M_\gamma})=
   \gamma\cap\Hull_{\rSigma_{n+1}}^{N|\gamma}(\lambda^N\cup X\cup \pvec_n^{N|\gamma})$,
\item  $\rho_{n+1}^{\M_\gamma}=\max(\rho_{n+1}^{N|\gamma},\lambda^N)$, and
\item $p_{n+1}^{\M_\gamma}=p_{n+1}^{N|\gamma}$.
\end{enumerate}
\end{enumerate}
  \end{enumerate}
  
Now suppose instead that $\gamma=\omega_1=\lambda^N$, and
let $n<\om$ be such that $(\omega_1,n)\leq(\beta^*,\min(\{n^*,n_0\})$.\footnote{We will show in Lemma \ref{lem:n^*=n_0}
that if $\omega_1<\beta^*$ then $n^*=n_0$,
and if $\omega_1=\beta^*$ then $n^*+1=n_0$.} Then we have:
\begin{enumerate}[resume*]
  \item \label{item:rho_n=rho_n_om_1}
  $\gamma=\omega_1=\rho_{n}^{\M_{\gamma}}=\rho_{n+1}^{N|\gamma}=\lambda^N$ \tu{(}therefore, if $\omega_1=\beta^*$
  then $n<n_0$\tu{)}.
  \item\label{item:p_n=p_n_om_1}  $p_{n}^{\M_\gamma}=p_{n+1}^{N|\gamma}=\emptyset$,
  \item\label{item:1-CC-forcing_om_1} Regarding forcing $\Mtilde_{\gamma}$
  over $N|\gamma$ with $\CC^N$:
  \begin{enumerate}[label=\tu{(}\alph*\tu{)}]
     \item\label{item:1-Delta_0^M_om_1} the $\mSigma_1^{\Mtilde_\gamma}$ forcing
    relation $\sforces{\CC}{\gamma1}$ is $\rDelta_2^{N|\gamma}$,
   \item\label{item:1-witnessed_1^N_om_1} the witnessed $\mSigma_{n+1}^{\Mtilde_\gamma}$ forcing relation $\sforces{\CC}{\gamma\mathrm{w},n+1}$ is
$\rSigma_{n+2}^{N|\gamma}$,
   \item\label{item:1-mu_1^N_om_1}
  the $\muSigma_{n+1}^{\Mtilde_\gamma}$ forcing
   relation $\sforces{\CC}{\gamma\mu,n+1}$  is
$\all^\om \rSigma_{n+2}^{N|\gamma}$,
  \item\label{item:mu_Sigma_n+1_projecting_forcing_rel_def_om_1}  for each $m<\om$, the $m$-good $\muSigma_{n+1}^{\Mtilde_\gamma}$ projecting forcing
   relation $\sforces{\CC,\geq m}{\gamma\mu,n+1,\mathrm{proj}}$ is
   $\rSigma_{n+2}^{N|\gamma}(\{m\})$,
  \end{enumerate}
  \item\label{item:1-PP-forcing_om_1} Regarding forcing $\Ntilde|\gamma$ over $\M_\gamma$ with $\PP$:
  \begin{enumerate}[label=\tu{(}\alph*\tu{)}]
     \item\label{item:1-Delta_0^N_om_1} the  $\rSigma_1^{\Ntilde|\gamma}$ stem forcing
    relation $\sforces{\PP^-}{\gamma1}$ is $\mDelta_1^{\M_\gamma}(\{\xg\})$,
 
   \item\label{item:witnessed_rSigma_n+1_forcing_rel_def_om_1} 
     the witnessed $\rSigma_{n+2}^{\Ntilde|\gamma}$ stem forcing
   relation
   $\sforces{\PP^-}{\gamma\mathrm{w},n+2}$
   is $\mSigma_{n+1}^{\M_\gamma}(\{\xg\})$,
   \item\label{item:mu-witnessed_forcing_rel_def_om_1} 
   the $\mu$-witnessed  $\rSigma_{n+2}^{\Ntilde|\gamma}$ stem forcing
   relation
   $\sforces{\PP^-}{\gamma\mu,n+2}$
   is $\muSigma^{\M_\gamma}_{n+1}(\{\xg\})$.
  \end{enumerate}
  \item\label{item:forcing_rels_unif_def_om_1}  The forcing/stem-forcing relations are moreover uniformly definable as follows:
  \begin{enumerate}[label=\tu{(}\alph*\tu{)}]
  \item\label{item:mSigma_n+1,muSigma_n+1_forcing_rel_unif_def_lambda<rho_n_om_1} the 
   forcing relations mentioned
   in  \ref{item:1-CC-forcing_om_1}\ref{item:1-witnessed_1^N_om_1},\ref{item:1-mu_1^N_om_1},\ref{item:mu_Sigma_n+1_projecting_forcing_rel_def_om_1}
 are defined in the stated manner
uniformly in $n<\om$ satisfying
\begin{equation}\label{eqn:range_for_gamma,n_om_1} (\omega_1,0)\leq(\omega_1,n)\leq(\beta^*,\min(n^*,n_0)),
\end{equation}
   meaning that there is a recursive
   function $n\mapsto\psi_n$
   such that $\psi_n$ is a formula
   of the stated complexity
   and for each $n$ as in line 
   \tu{(}\ref{eqn:range_for_gamma,n_om_1}\tu{)},
 $\psi_n(x,\cdot,\cdot,\cdot)$ defines the stated
   forcing relation over $N|\gamma$
   when $x$ is the stated parameter,
   \item part \ref{item:mSigma_n+1,muSigma_n+1_forcing_rel_unif_def_lambda<rho_n_om_1}
   still holds
   after replacing `` \ref{item:1-CC-forcing_om_1}\ref{item:1-witnessed_1^N_om_1},\ref{item:1-mu_1^N_om_1},\ref{item:mu_Sigma_n+1_projecting_forcing_rel_def_om_1}'' with
   ``\ref{item:1-PP-forcing_om_1}\ref{item:witnessed_rSigma_n+1_forcing_rel_def_om_1},\ref{item:mu-witnessed_forcing_rel_def_om_1}''
and ``$N|\gamma$'' with ``$\M_\gamma$''.
   \end{enumerate}
      \item\label{item:CC_PP_forcing_theorems_om_1} The  forcing theorem holds for each of the  forcing relations mentioned in parts \ref{item:1-CC-forcing_om_1}\ref{item:1-Delta_0^M_om_1}--\ref{item:mu_Sigma_n+1_projecting_forcing_rel_def_om_1},\footnote{We have no need for an analogue of the $m_0$
      from part \ref{item:1-CC-forcing}\ref{item:mu_Sigma_n+1_projecting_forcing_rel_def} here;
 in other words, $m_0=0$ works. This is because $\lambda^N$ is fixed by the relevant iteration maps.}
   and the stem forcing theorem holds for the forcing relations mentioned in parts \ref{item:1-PP-forcing_om_1}\ref{item:1-Delta_0^N_om_1},\ref{item:witnessed_rSigma_n+1_forcing_rel_def_om_1}.
  \end{enumerate}
\end{lem}
\begin{proof}
We give the direct proof of parts \ref{item:rho_n=rho_n}--\ref{item:Hull_n+1_ordinals_match}.
Given what was already discussed in \S\ref{subsec:the_generic_M(R*)},
the remaining details for parts \ref{item:rho_n=rho_n_om_1}--\ref{item:CC_PP_forcing_theorems_om_1}
are similar but simpler,
so we will omit further discussion of these.

The $\mSigma_0^{\Mtilde_\gamma}$
forcing relation $\sforces{\CC}{\gamma0}$ of $N|\gamma$
is by definition uniformly definable,
via $\psi_0(\lambda^N,\cdot,\cdot,\cdot)$
and $\psi_0^{\neg}(\lambda^N,\cdot,\cdot,\cdot)$
(Definitions \ref{dfn:psi_0} and \ref{dfn:psi_0^neg} respectively),
and the corresponding forcing theorem
was established in Lemma \ref{lem:Sigma_P-iterate_good}.
This gives parts 
 \ref{item:1-CC-forcing}\ref{item:1-Delta_0^M} and
\ref{item:forcing_rels_unif_def}\ref{item:mSigma_0_forcing_rel_unif_def},
and the corresponding piece of part \ref{item:CC_PP_forcing_theorems}.
The (uniform) definability
of the $\rSigma_0^{\Ntilde|\gamma}$ stem forcing relation $\sforces{\PP^-}{\gamma0}$
(parts \ref{item:1-PP-forcing}\ref{item:1-Delta_0^N} and \ref{item:forcing_rels_unif_def}\ref{item:rSigma_0_stem_forcing_rel_unif_def})
is by properties \ref{item:dagger3} and \ref{item:dagger4} of Definition
\ref{dfn:prior_condition_deciding},
and the corresponding stem forcing theorem is by Lemma \ref{lem:rSigma_0_forcing_theorem_for_PP},
particularly its part \ref{item:G_contains_eventual_q_ss}.

Regarding the (uniform) definability
when $n>0$,
the appropriate (uniform) definability of the remaining forcing relations 
follows straightforwardly from the definitions (see in particular \ref{dfn:stem_forcing_relations} and \ref{dfn:CC_forcing_relations}), by proceeding via induction on $n$,
 and we leave this verification to the reader.\footnote{Note that
 we are referring here to the
 formal definitions we gave;
 we have not yet verified that
 those definitions yield something
 useful, i.e.~that the relevant forcing theorems hold.} (One point maybe worth highlighting here is that in Definition \ref{dfn:CC_forcing_relations},
 in order to define the witnessed $\mSigma_{n+1}$ forcing relation $\sforces{\CC}{\gamma\mathrm{w},n+1}$
 when $n>0$ and $\gamma>\lambda^N$
 and $\rho_n^{N|\gamma}=\lambda^N$,
 we defined and referred to the names $\tau_n^{\Mtilde_\gamma}(\varrho)$,
 for $\varrho\in\Mhat_{\lambda^N}$.
 In order to see that $\sforces{\CC}{\gamma\mathrm{w},n+1}$ is appropriately definable, one wants to see that $\tau_n^{\Mtilde_\gamma}(\varrho)$
 is easily computed from
 \begin{equation}\label{eqn:rSigma_n_theory_in_N|gamma_which_computes_name}\Th_{\rSigma_n}^{N|\gamma}(\{\varrho,\lambda^N,\pvec_n^{N|\gamma}\}\cup(\delta_d^N)^{+N}),\end{equation} where $d=\supp(\varrho)$, and uniformly so. But by induction, 
 the stable $\muSigma_n$ forcing relation $\sforces{\CC}{\gamma\mu n}$ 
is $\all^\om\rSigma_n(\pvec)$
where $\pvec=\pvec_{n-1}^{N|\gamma}$,
if $n>1$, or $\all^\om\rSigma_1^{N|\gamma}(\{\lambda^N\})$,
if $n=1$. So in order to check the truth of a forcing statement of form $p\sforces{\CC}{\gamma\mu n}\varphi(\vec{x})$,
one just has to check that the corresponding $\om$-sequence
of statements all belong
to the appropriate $\rSigma_n$ theory. Combining this with some basic forcing calculations,
it follows that the name $\tau_n^{\Mtilde_\gamma}(\varrho)$
is indeed easily and uniformly enough computable from the theory indicated
in line (\ref{eqn:rSigma_n_theory_in_N|gamma_which_computes_name}).

 This completes the proof of parts \ref{item:1-CC-forcing},
 \ref{item:1-PP-forcing}
 and \ref{item:forcing_rels_unif_def}.

We now consider the rest of parts \ref{item:rho_n=rho_n}--\ref{item:Hull_n+1_ordinals_match},
proceeding by induction on $n$, assuming that if $\gamma=\beta^*$
then $n\leq\min(n^*,n_0)$.\footnote{We show in Lemma \ref{lem:n^*=n_0} that (as $\omega_1<\beta^*$) actually $n^*=n_0$.}We break into two stages: $n=0$ and $n>0$.

\begin{stage*} $n=0$.
 
Since $n=0$, parts \ref{item:rho_n=rho_n}
and \ref{item:p_n=p_n} are trivial.
The witnessed $\mSigma_1^{\Mtilde_\gamma}$ and witnessed $\rSigma_1^{\Ntilde|\gamma}$
forcing theorems (of \ref{item:1-CC-forcing}\ref{item:1-witnessed_1^N} and \ref{item:1-PP-forcing}\ref{item:witnessed_rSigma_n+1_forcing_rel_def}) are immediate consequences of the 
$\mSigma_0^{\Mtilde_\gamma}$ and $\rSigma_0^{\Ntilde|\gamma}$ forcing theorems.

\begin{clm}
Part \ref{item:Hull_n+1_ordinals_match} holds.
\end{clm}
\begin{proof}
 We have
\[\gamma\cap\Hull_{\mSigma_{1}}^{\M_\gamma}(\HC\cup X)\sub
   \gamma\cap\Hull_{\rSigma_{1}}^{N|\gamma}(\lambda^N\cup X\cup\{\lambda^N\})
  \] 
  because if $\xi<\gamma$,
  $t$ is an $\mSigma_1$ min-term,
  $\vec{\alpha}\in X^{<\om}$,
   $z\in\HC$ and
  \[ \M_\gamma\sats\xi=t(z,\vec{\alpha}),\]
then by the witnessed $\mSigma_1^{\Mtilde_\gamma}$ forcing theorem for $N|\gamma$
there is some $p\in\CC^N$
and $\dot{z}\in\Mhat_{\lambda^N}$
such that
\[ p\sforces{\CC^N}{\gamma\mathrm{w}1}O_\xi=t(\dot{z},O_{\vec{\alpha}}),\]
and since $\sforces{\CC^N}{\gamma\mathrm{w}1}$ is $\rSigma_1^{N|\gamma}(\{\lambda^N\})$
it follows that
\[ \xi\in\Hull_{\rSigma_1}^{N|\gamma}((\lambda^N+1)\cup X).\]

The converse is similar, except
that it is important that we use the stem forcing relation, instead of a standard forcing relation,
so that we can use some $s\in\PP^-\sub\HC$
instead of $p\in\PP$ (along with parameters from $X$) to define a given ordinal $\xi$;
we also get $\lambda^N=\omega_1$
in the hull on the left automatically,
considering the $\M(\RR)$ language.

This establishes \ref{item:Hull_n+1_ordinals_match}\ref{item:Hull_rho_p_match_n=0}\ref{item:Hull_match_n=0}.
The rest of part \ref{item:Hull_n+1_ordinals_match} is an easy consequence. (If $p_1^{N|\gamma}=\{\lambda^N\}$ then $p_1^{\M_\gamma}=\emptyset$ because $\om_1=\lambda^N\in\Hull_1^{\M_\gamma}(\emptyset)$, because the $\M(\RR)$ language has a symbol for $\omega_1$. If $p_1^{\M_\gamma}=\emptyset$ but $\rho_1^{\M_\gamma}=\omega_1=\rho_1^{N|\gamma}$
 then $p_1^{N|\gamma}=\{\lambda^N\}$
 because $\Hull_1^{N|\gamma}(\lambda^N)=N|\lambda^N$.)
\end{proof}

We next consider the forcing theorem for  the $\muSigma_1^{\Mtilde_\gamma}$ forcing relation $\sforces{\CC}{\gamma\mu,1}$,
which except at the very last stage of induction follows easily from the following claim.

\begin{clm}\label{clm:muSigma_1_forcing_theorem_<_OR^N}
Suppose that either $\gamma<\OR^N$
or $0<n_0$.
Let $\vec{x}\in(\Nhat|\gamma)^{<\om}$, $d=\supp(\vec{x})$.
Then there is $m>d$ such that
for all $\mSigma_1$ formulas $\varphi$,
all $k<\om$,
all $\vec{\delta},\vec{\vareps}\in[\Delta_{\geq m}^N]^{2k}$,
and all $\theta\in\Delta_{\geq m}^N$,
 $N$ satisfies
that $\CC_d$ forces 
the following
statements
are equivalent:
\begin{enumerate}[label=\tu{(}\roman*\tu{)}]
\item\label{item:deltavec_unbounded_witness}$\all^{\mathrm{gen}}_{\vec{\delta}}s\ \sforces{\CC_{\mathrm{tail}}}{\gamma\mathrm{w}1}\varphi(\vec{x},s)$
\item\label{item:deltavec_bound_witness} $\exists\gamma'\in\OR\ \all^{\mathrm{gen}}_{\vec{\delta}}s\ \sforces{\CC_{\tail}}{\gamma0}\Mtilde_{\gamma'}\sats\varphi(\vec{x},s)$
\item\label{item:epsvec_unbounded_witness} $\all^{\mathrm{gen}}_{\vec{\vareps}}s\ \sforces{\CC_{\mathrm{tail}}}{\gamma\mathrm{w}1}\varphi(\vec{x},s)$
\item\label{item:>=all-gen>=theta} $\all^{\mathrm{gen}}_{\geq\theta;k}s\ \sforces{\CC_{\mathrm{tail}}}{\gamma\mathrm{w}1}\varphi(\vec{x},s)$.
\item\label{item:each_Woodin_bounded_witness} $\all \ell<\om\ \exists\vec{\beta}\in[\Delta_{\geq \ell}]^{2k}\ \exists\gamma'\in\OR\ \all^{\mathrm{gen}}_{\vec{\beta}}s\ \sforces{\CC_{\tail}}{\gamma0}\ \Mtilde_{\gamma'}\sats\varphi(\vec{x},s)$.
\end{enumerate}
\end{clm}

Note here that in clauses \ref{item:deltavec_bound_witness} and \ref{item:each_Woodin_bounded_witness}
above,
although $\gamma\in\Lim$, $\gamma'$ might be a successor. And $\Mtilde_{\gamma'}$
is just the natural name in $\Mhat_\gamma$ for the $\gamma'$th
level of the hierarchy of the generic $\Mtilde_G$.

\begin{proof} We consider three cases.

 \begin{case}\label{case:gamma<OR^N_muSigma_1_forcing_thm}
$\gamma<\OR^N$.

Let $\left<\Tt_n\right>_{n<\om}$
be the standard decomposition of $\Tt$,
the tree leading from $\Pg$ to $N$
(see Definition \ref{dfn:standard_decomposition}).
Let $N_0=\Pg$ and
 $N_{n+1}=M^{\Tt_n}_\infty$
and $j_{n,\infty}:N_{n}\to N$
be the iteration map. Fix $n$ such that $\gamma,\vec{x}\in\rg(j_{n\infty})$.
Note then that it suffices
to prove the corresponding equivalence for $N_{n}$, $\gamma'=j_{n\infty}^{-1}(\gamma)$ and $\vec{x}'=j_{n\infty}^{-1}(\vec{x})$,
 since the equivalence is preserved by $j_{n\infty}$. (The statement
that the equivalence holds is not quite first order over $N_{n}|\gamma'$,
(because
 of the unboundedness of the quantifiers in \ref{item:>=all-gen>=theta}),
but if $m$ witnesses it for $N_{n}$,
then for each $k$, it just says
that a certain statement $\psi_k$ holds
(of the relevant parameters),
each of which are preserved by $j_{n\infty}$.) Now for notational simplicity,
let us just assume that $n=0$,
 so $N_n=\Pg$;
the other case is just a relativization
above $N_{n+1}|\delta_n^{N_{n+1}}$,
using the $\delta_n^{N_{n+1}}$-soundness of $N_{n+1}$ (take $m\geq n+1$ in this case).

Using the fact that a given ordinal is eventually fixed under the relevant iteration maps,\footnote{***Could add a general lemma on this.} fix $m> d$ such that for all $k<\om$ and
all $\vec{\delta}\in[\Delta_{>m}^{\Pg}]^{2k}$, 
letting  $j:\Pg\to R^{\Pg}_{\vec{\theta}\cup\vec{\delta}}$
be the iteration map, where $\vec{\theta}=(\delta_0^{\Pg},\ldots,\delta_m^{\Pg})$, then $j(\gamma)=\gamma$
and $j(\loc(\vec{x}))=\loc(\vec{x})$,
and therefore $j(\vec{x})=\vec{x}$.
With this $m$, the equivalence of 
\ref{item:deltavec_unbounded_witness},
\ref{item:epsvec_unbounded_witness}
and \ref{item:>=all-gen>=theta}
is as in the proof of Lemma \ref{lem:xi_gen-all_independence}.
 With this, it is easy to see that
 once we have
 shown that \ref{item:deltavec_unbounded_witness}
 implies \ref{item:deltavec_bound_witness},
 the rest easily follows.
 
 So let us show \ref{item:deltavec_unbounded_witness}
 $\Rightarrow$ \ref{item:deltavec_bound_witness}.
  For purposes of illustration suppose $k=2$;
  the general case just involves more notation.
  We may then assume that $\vec{\delta}=(\delta_{m+1}^N,\delta_{m+2}^N,\delta_{m+3}^N,\delta_{m+4}^N)$,
  since if \ref{item:deltavec_bound_witness} holds
  for this $\vec{\delta}$,
  then we can use an iteration map like $j$
  above to deduce it for the remaining tuples in $[\Delta_{\geq m}^N]^{2k}$.
  Fix $g$ which is $(\Pg,\CC_d^{\Pg})$-generic,
  and suppose 
  that in $\Pg[g]$,
  \ref{item:deltavec_unbounded_witness}
 holds but
 \ref{item:deltavec_bound_witness}
 fails. 
 Then 
 \begin{equation}\label{eqn:no_bounded_witness} \Pg[g]\sats\all\gamma'\in\OR\ \exists^{\mathrm{gen}}_{\vec{\delta}}s\ \sforces{\CC}{\gamma0}\Mtilde_{\gamma'}\sats\neg\varphi(\vec{x},s),\end{equation}
where we have used the homogeneity of $\CC$
and that $\supp(\vec{x})=d$.
Let $\vec{\vareps}=(\delta_{m+1}^N,\delta_{m+3}^N,\delta_{m+4}^N,\delta_{m+6}^N)$
and $\vec{\nu}=(\delta_{m+2}^N,\delta_{m+3}^N,\delta_{m+5}^N,\delta_{m+6}^N)$.
Let $R=R^{\Pg}_{\vec{\theta}\cup\vec{\vareps}}$
and $R'=R^{\Pg}_{\vec{\theta}\cup\vec{\nu}}$.
Let $h_i$, for $i\in\{1,2,3,4,5,6\}$,
be $(N,\CC_{m+i}^N)$-generic,
with $g\sub h_1\sub\ldots\sub h_6$.
Let $g_i$, for $i\in\{1,2,3,4\}$,
be $(R,\CC_{m+i}^R)$-generic,
with $g\sub g_1\sub\ldots\sub g_4$,
and such that  $\HC^{R[g_1]}=\HC^{N[h_1]}$,
$\HC^{R[g_2]}=\HC^{N[h_3]}$,
$\HC^{R[g_3]}=\HC^{N[h_4]}$,
and $\HC^{R[g_4]}=\HC^{N[h_6]}$.
Let $g'_i$, for $i\in\{1,2,3,4\}$,
be $(R',\CC_{m+i}^{R'})$-generic,
with $g\sub g'_1\sub\ldots\sub g'_4$,
and such that $\HC^{R'[g'_1]}=\HC^{N[h_2]}$, $\HC^{R'[g'_2]}=\HC^{N[h_3]}$,
$\HC^{R'[g'_3]}=\HC^{N[h_5]}$,
and $\HC^{R'[g'_4]}=\HC^{N[h_6]}$.

Let $j:\Pg\to R$ and $j':\Pg\to R'$ be the iteration maps,
which extend to the generic extensions given by $g$,
and $j(\gamma)=\gamma=j'(\gamma)$
and $j(\vec{x})=\vec{x}=j'(\vec{x})$.
So (by \ref{item:deltavec_unbounded_witness}
in $\Pg[g]$)
\begin{equation}\label{eqn:R'[g]_sats_all^gen_s} R'[g]\sats\all^{\mathrm{gen}}_{\vec{\nu}}s\ \sforces{\CC}{\gamma\mathrm{w}1}\varphi(\vec{x},s),\end{equation}
but (by line (\ref{eqn:no_bounded_witness}) for each $\gamma'<\gamma$,
\begin{equation}\label{eqn:each_gamma'_R[g]_sats_exists^gen_s} R[g]\sats\exists^{\mathrm{gen}}_{\vec{\vareps}}s\ \sforces{\CC}{\gamma0}\Mtilde_{\gamma'}\sats\neg\varphi(\vec{x},s).\end{equation}
So for each $\gamma'<\gamma$,
fix a Turing degree $x_{\gamma'}\in\HC^{R[g_1]}$ witnessing
the first (existential) degree quantifier
in the ``$\exists^{\mathrm{gen}}_{\vec{\vareps}}s$'' quantifier in line (\ref{eqn:each_gamma'_R[g]_sats_exists^gen_s}) in $R[g_1]$. Note that $\HC^{R[g_1]}=\HC^{N[h_1]}$ is countable in $N[h_2]$,
and since $\HC^{N[h_2]}=\HC^{R'[g'_1]}$,
we can therefore fix a Turing degree $x_\infty\in\HC^{R'[g'_1]}$ such that
$x_\infty\geq_T x_{\gamma'}$ for all $\gamma'<\gamma$. Applying the first (universal) degree quantifier in the ``$\all^{\mathrm{gen}}_{\vec{\nu}}s$''
quantifier in line (\ref{eqn:R'[g]_sats_all^gen_s})
in $R'[g'_1]$ to $x_\infty$,
let $s_0$ be a Turing degree
in $\HC^{R'[g'_2]}$
such that $x_\infty\leq_T s_0$
and $s_0$ satisfies the first existential degree quantifier of ``$\all^{\mathrm{gen}}_{\vec{\nu}}s$''
in $R'[g'_2]$.
Since $\HC^{R'[g'_2]}=\HC^{R[g_2]}$
and $x_{\gamma'}\leq_T x_\infty\leq_T s_0$ for each $\gamma'<\gamma$,
the first universal degree quantifier of ``$\exists^{\mathrm{gen}}_{\vec{\vareps}}s$'' applies in $R[g_2]$ to $s_0$,
with respect to $\gamma'$.
Therefore, for each $\gamma'<\gamma$
we can fix a Turing degree $y_{\gamma'}\in\HC^{R[g_3]}$ witnessing the second existential degree quantifier
in ``$\exists^{\mathrm{gen}}_{\vec{\vareps}}s$''
in $R[g_3]$
(with respect to $\gamma'$),
then pick
$y_\infty\in\HC^{R'[g'_3]}$
such that $y_{\gamma'}\leq_T y_\infty$
for each $\gamma'$,
apply the next universal degree quantifier to $y_\infty$
in $R'[g'_3]$,
and hence find a Turing degree
$s_1\in R'[g'_4]$
with $y_\infty\leq_T s_1$
and such that
\[ R'[g'_4]\sats\ \sforces{\CC}{\gamma\mathrm{w}1}\varphi(\vec{x},s)\]
where $s=(s_0,s_1)$.
By homogeneity of $\CC$, and since all the names here have small enough support, it follows that there is $\gamma'<\gamma$
such that
\[ R'[g'_4]\sats\ \sforces{\CC}{\gamma 0}\Mtilde_{\gamma'}\sats\varphi(\vec{x},s).\]
On the other hand, our choice of $(s_0,s_1)$ with respect to the ``$\exists^{\mathrm{gen}}_{\vec{\vareps}}$'' quantifier with respect to $\gamma'$ guarantees that 
\[ R[g_4]\sats\ \sforces{\CC}{\gamma0}\Mtilde_{\gamma'}\sats\neg\varphi(\vec{x},s).\]
But since $R,R'$ are both given by corresponding P-constructions above $\delta_{m+6}^N$,
note that $g_4$ and $g'_4$ can be extended
to $G$ and $G'$ respectively, such that $(\Mtilde^R)_G=(\Mtilde^{R'})_{G'}$,
and this gives a contradiction.
Other values of $k$ are similar. 
This proves the claim in case $\gamma<\OR^N$.\end{case}

\begin{case} $\gamma=\OR^N$ but $\lambda^N<\rho_1^N$.

We may substitute $N_n$ for $\Pg$
where $\vec{x}\in\rg(j_{n\infty})$
(notation as at the start of the proof in Case \ref{case:gamma<OR^N_muSigma_1_forcing_thm}).
For  $j_{n\infty}:N_n\to N$
is a near $1$-embedding, and:
\begin{enumerate}[label=--]
\item \ref{item:deltavec_bound_witness}
is $\rSigma_1$,
\item \ref{item:each_Woodin_bounded_witness}
is $\all^\om\rSigma_1$, and
 \item 
 \ref{item:deltavec_unbounded_witness},
 \ref{item:epsvec_unbounded_witness}
 and \ref{item:>=all-gen>=theta}
are $\rSigma_2$ (in the relevant parameters); for example \ref{item:deltavec_unbounded_witness}
is expressed over $N_n$ by a simple assertion about $\Th_{\rSigma_1}^{N_n}(\max(\vec{\delta})^{+N_n}\cup\{p\})$
for the appropriate $p$,
and \ref{item:>=all-gen>=theta}
is a simple assertion about the theory
$\Th_{\rSigma_1}^{N_n}(\lambda^{N_n}\cup\{p\})$ for the appropriate $p$,
which by the case hypothesis
is in $N_n$, and mapped correctly
by $j_{n\infty}$,
\end{enumerate}
and using these facts it is easy to
see that $j_{n\infty}$ preserves the truth of the claim.

The proof that the claim holds for $N_n$ is just like that in the previous case,
taking $m\geq n$, but also using the considerations just mentioned to see that
the iteration maps $i:N_n\to R^{N_n}_{\vec{\theta}}$ (for $\vec{\delta}\in[\Delta^{N_n}]^{<\om}$) preserve the truth of the relevant forcing statements. 
\end{case}

\begin{case}\label{case:gamma=OR^N_and_lambda^N=rho_1^N}$\gamma=\OR^N$ and $\lambda^N=\rho_1^N$.
 
The equivalence
 of \ref{item:deltavec_unbounded_witness}, \ref{item:deltavec_bound_witness}, \ref{item:epsvec_unbounded_witness},  and \ref{item:each_Woodin_bounded_witness} is just as in the previous case.
However, it seems the argument
 used there does not suffice
 to show the equivalence of these with \ref{item:>=all-gen>=theta}:  since $\lambda^N=\rho_1^N$,
 this clause is no longer $\rSigma_2$, and it seems to be too complex to be obviously preserved by the relevant iteration maps. But a different method is available under the case hypothesis.
 Note that by part \ref{item:Hull_n+1_ordinals_match}
 with $\gamma=\OR^N$ and $n=0$
 (for $\Sigma_1$ hulls),
 parts \ref{item:rho_n=rho_n} and \ref{item:p_n=p_n} hold for $n=1$.
 Let $t$ be an $\mSigma_1$ min-term
 and let $z\in\HC$ with $t^{\M_{\beta^*}}(p_1^{\M_{\beta^*}},z)=\vec{\xi}$
 where $\vec{\xi}=\loc(\vec{x})$.
 Fix $G\sub\CC$ witnessing that $N$ is an
 $\RR$-genericity iterate of $\Pg$.
 So $\M_{\beta^*}=(\Mtilde^N)_G$.
 Let $m\in[d,\om)$ with $z\in N[g]$
 where $g=G\rest m$.
 Then by homogeneity of $\CC$,
 \[ N[g]\sats\ \sforces{\CC}{\beta^*\mathrm{w}1}O_{\xi}=t^{\Mtilde}(O_p,z) \]
 where $p=p_1^{\M_{\beta^*}}$.
 
 \begin{sclm}\label{sclm:N_to_R_suff_Prik_it_fixes_xvec}
 Let $R$ be a $\Sigma_{\Pg N}$-iterate of $N$ which is itself a sufficiently Prikry generic iterate of $\Pg$,
 and such that $\delta_m^N<\crit(i_{NR})$. Then $i_{NR}(\vec{x})=\vec{x}$.\end{sclm}
 \begin{proof}All that we have established so far
 for $N$ also applies also to $R$, and in particular parts \ref{item:rho_n=rho_n} and \ref{item:p_n=p_n} hold for $R$, with $n=1$. Therefore $p_1^R=i_{NR}(p_1^N)=p_1^N$ (for certainly also $\lambda^R=i_{NR}(\lambda^N)=\lambda^N$),
 and so $i_{NR}(p)=p$.
 But letting $i_{NR}^+:N[g]\to R[g]$ be the canonical extension, $i_{NR}^+(z)=z$, and so 
 \[ R[g]\sats\ \sforces{\CC_{\tail}}{\beta^*\mathrm{w}1}O_{i_{NR}(\xi)}=t^{\Mtilde}(O_p,z), \]
 and since $\M_{\beta^*}=(\Mtilde^N)_G$
 and $\M_{\beta^*}=(\Mtilde^R)_{G'}$
 with some $G'$ with $g=G'\rest m$,
 it follows that $i_{NR}(\vec{\xi})=\vec{\xi}$, and since $m\geq d$,
 therefore $i_{NR}(\vec{x})=\vec{x}$.\end{proof}
 
 Now let us show that 
 \ref{item:deltavec_unbounded_witness} implies \ref{item:>=all-gen>=theta}.
 Suppose $N[g]$ satisfies
 \ref{item:deltavec_unbounded_witness}.
For illustration suppose $k=2$
and $\vec{\delta}=(\delta_{m}^N,\delta_{m+1}^N,\delta_{m+2}^N,\delta_{m+3}^N)$.
Let $\theta\in\Delta_{\geq m}^N$.
We must show that
\begin{equation}\label{eqn:to_show} N[g]\sats\all^{\mathrm{gen}}_{\geq\theta;k}s\ \sforces{\CC}{\beta^*\mathrm{w}1}\varphi(\vec{x},s).\end{equation}
So let $n_0\geq m$ with $\vareps_0=\delta_{n_0}^N\geq\theta$;
we let $\vareps_0$ be the ``first Woodin played'' by the $\all$-player in the
game corresponding
to the statement in (\ref{eqn:to_show}).
Note then we may assume $n_0>m$.
We may assume $g_0=G\rest(n_0+1)$ is the generic played. Let $z_0\in N[g_0]$
be the Turing degree played.

We now find the first response
for the $\exists$-player.
Let $R_0=R^N_{\vec{\theta}\cup\{\vareps_0\}}$, where $\vec{\theta}=(\delta_0^N,\ldots,\delta_{m-1}^N)$.
So $i_{NR_0}\rest\delta_{m-1}^N=\id$
and $i_{NR_0}(\delta_{m}^N)=\delta_{m}^{R_0}=\vareps_0$.
Let $g'_0$ be $(R_0,\CC_{m}^{R_0})$-generic with $g\sub g'_0$
and $\HC^{R[g'_0]}=\HC^{N[g_0]}$.
Let $\Uu_0$ be the successor length tree on $N$ given by iterating $N|\delta_{m}^N$ out to $R_0|\delta_{m}^{R_0}$,
with $\delta(\Uu_0)=\delta_{m}^{R_0}$
(so $M^{\Uu_0}_\infty|\delta_{m}^{M^{\Uu_0}_\infty}=R_0|\delta_{m}^{R_0}$, but it need not be true that $M^{\Uu_0}_\infty=R_0$).
Let $n_1<\om$ be such that $\vareps_0<\delta_{n_1}^N$ and
\begin{equation}\label{eqn:M^U_0_infty|delta_m+1_in_HC} M^{\Uu_0}_\infty|\delta_{m+1}^{M^{\Uu_0}_\infty}\in\HC^{N[G\rest n_1]}.\end{equation}
Set $\vareps_1=\delta^N_{n_1+1}$. Let $R_1=R^N_{\vec{\theta}\cup\{\vareps_0,\vareps_1\}}$.
Then $\delta_{m+1}^{R_1}=\vareps_1$
and $R_1|\vareps_1$ is a correct iterate of $M^{\Uu_0}_\infty|\delta_{m+1}^{M^{\Uu_0}_\infty}$
(the proof of this fact uses line (\ref{eqn:M^U_0_infty|delta_m+1_in_HC})). Let $\Uu_1$ be the correct successor length tree on $M^{\Uu_0}_\infty$
iterating $M^{\Uu_0}_\infty|\delta_{m+1}^{M^{\Uu_0}_\infty}$ out to $R_1|\vareps_1$, with $\delta(\Uu_1)=\vareps_1$. Let $g_1=G\rest\vareps_1$
and $g'_1$ be $(R_1,\CC_{m+1}^{R_1})$-generic with $g'_0\sub g'_1$
and such that $\HC^{N[g_1]}=\HC^{R_1[g'_1]}$.
Let $s_0$ be a Turing degree
in $\HC^{R_1[g'_1]}$ with $s_0\geq_T z_0$
witnessing the first existential degree
quantifier of ``$\all^{\mathrm{gen}}_{i^{\Uu_0,\Uu_1}(\vec{\delta})}s$''
 in $M^{\Uu_0,\Uu_1}_\infty[g'_1]$ (this holds there as $i^{\Uu_0,\Uu_1}$ is $\rSigma_2$-elementary). Set $(\vareps_1,g_1,s_0)$
 to be the first move by the $\exists$-player.

 Let $\vareps_2\in\Delta^N$ with $\vareps_2>\vareps_1$; the $\all$-player will play $\vareps_2$. Note that
 we may assume $\delta_{n_2}^N<\vareps_2$
 where $n_2>n_1+1$ is such that $M^{\Uu_0,\Uu_1}_\infty|\delta_{m+2}^{M^{\Uu_0,\Uu_1}_\infty}\in\HC^{N[G\rest n_2]}$.
Let $g_2=G\rest\vareps_2$;
we may assume the $\all$-player plays $g_2$.
 Let the $\all$-player play Turing degree $z_1\in N[g_2]$.
 
 Let $R_2=R^N_{\vec{\theta}\cup\{\vareps_0,\vareps_1,\vareps_2\}}$.
  Let $g'_2$
be $(R_2,\CC_{m+2}^{R_2})$-generic
with $g'_1\sub g'_2$
and $\HC^{R_2[g'_2]}=\HC^{N[g_2]}$.
Now $R_2|\vareps_2$ is a correct
iterate of $M^{\Uu_0,\Uu_1}_\infty|\delta_{m+2}^{M^{\Uu_0,\Uu_1}_\infty}$;
let $\Uu_2$ be the correct successor length tree
on $M^{\Uu_0,\Uu_1}_\infty$
iterating $M^{\Uu_0,\Uu_1}_\infty|\delta_{m+2}^{M^{\Uu_0,\Uu_1}_\infty}$
out to $R_2|\vareps_2$.
Let $n_3<\om$ be such that
$\vareps_2<\delta_{n_3}^N$
and $M^{\Uu_0,\Uu_1,\Uu_2}_\infty|\delta_{m+3}^{M^{\Uu_0,\Uu_1,\Uu_2}_\infty}\in\HC^{N[G\rest n_3]}$.
Let $\vareps_3=\delta_{n_3+1}^N$.
Let $R_3=R^N_{\vec{\theta}\cup\{\vareps_0,\ldots,\vareps_3\}}$.
As before, let $\Uu_3$
be the correct successor length tree on 
 $M^{\Uu_0,\Uu_1,\Uu_2}_\infty$
iterating $M^{\Uu_0,\Uu_1,\Uu_2}_\infty|\delta_{m+3}^{M^{\Uu_0,\Uu_1,\Uu_2}_\infty}$ out to $R_3|\delta_{m+3}^{R_3}$,
with $\delta(\Uu_3)=\delta_{m+3}^{R_3}$.
Let $g_3=G\rest\vareps_3$,
and let $g_3'$ be $(R_3,\CC_{m+3}^{R_3})$-generic,
with $\HC^{R_3[g'_3]}=\HC^{N[g_3]}$,
and $s_1\in\HC^{R_3[g'_3]}$
witnessing the second existential degree quantifier of ``$\all^{\mathrm{gen}}_{\vec{\vareps}}s$''
in $M^{(\Uu_0,\ldots,\Uu_3)}_\infty[g'_3]$ with respect to $s_0,z_1$.
So
\begin{equation}\label{eqn:M^U<4_sats_forced_exists} M^{(\Uu_0,\ldots,\Uu_3)}_\infty[g'_3]\sats\ \sforces{\CC}{\OR\mathrm{w}1}\varphi(\vec{x}',s) \end{equation}
where $s=(s_0,s_1)$ and $\vec{x}'=i^{(\Uu_0,\ldots,\Uu_3)}_{0\infty}(\vec{x})$.

Now continuing further, but without having to consider further Turing degrees, find trees $\Uu_4,\Uu_5,\ldots$, with $\Uu_{k+1}$ on $M^{(\Uu_0,\ldots,\Uu_k)}_\infty$,
and integers $n_4,n_5,\ldots$,
with $n_3<n_4<n_5<\ldots$,
and such that for $k\geq 3$, we have $M^{(\Uu_0,\ldots,\Uu_k)}_\infty|\delta_{m+k+1}^{M^{(\Uu_0,\ldots,\Uu_k)}_\infty}\in\HC^{N[G\rest n_{k+1}]}$, and setting $\vareps_{k+1}=\delta_{n_{k+1}+1}^N$
and $R_{k+1}=R^N_{\vec{\theta}\cup\{\vareps_0,\ldots,\vareps_{k+1}\}}$,
then $M^{(\Uu_0,\ldots,\Uu_{k+1})}_\infty|\delta_{m+k+1}^{M^{(\Uu_0,\ldots,\Uu_{k+1})}_\infty}=R_{k+1}|\delta_{m+k+1}^{R_{k+1}}$.

Let $\Uu=(\Uu_0,\Uu_1,\ldots)$.
Now we did not (at least not explicitly) arrange that $M^\Uu_\infty$ is sufficiently Prikry generic, but in any case we can iterate it further to some $R$ which is sufficiently Prikry generic,
and then
by Subclaim \ref{sclm:N_to_R_suff_Prik_it_fixes_xvec}, $\OR^R=\beta^*$ and $i_{NR}(\vec{\xi})=\vec{\xi}$, and therefore $\OR^{M^\Uu_\infty}=\beta^*$ and $i^\Uu_{0\infty}(\vec{\xi})=\vec{\xi}$,
and therefore $i^\Uu_{0\infty}(\vec{x})=\vec{x}$.
So by line (\ref{eqn:M^U<4_sats_forced_exists}), we have
\[M^\Uu_\infty[g'_3]\sats\ \sforces{\CC}{\OR\mathrm{w}1}\varphi(\vec{x},s).\]

Let $G'\sub\CC$ be $(M^\Uu_\infty,\CC)$-generic with $g'_3\sub G'$
and such that for each $k\geq 3$,
$\HC^{M^\Uu_\infty[G'\rest(k+1)]}=\HC^{N[G\rest\vareps_{k+1}]}$.
Then $G'$ witnesses that $M^\Uu_\infty$
is an $\RR$-genericity iterate.
So by Lemma \ref{lem:Sigma_P-iterate_good}
(and since $\OR^{M^\Uu_\infty}=\beta^*$),
 $(\Mtilde^{M^\Uu_\infty})_{G'}=\M_{\beta^*}=(\Mtilde^N)_G$.
But therefore $\M_{\beta^*}\sats\varphi(\vec{x}_{G'})$
and $\vec{x}_{G'}=\vec{x}_G$,
so by homogeneity of $\CC$,
\[ N[g_3]\sats\ \sforces{\CC}{\OR\mathrm{w}1}\varphi(\vec{x},s),\] as desired, completing
the proof that \ref{item:deltavec_unbounded_witness} implies \ref{item:>=all-gen>=theta}.

For the converse, i.e.~that
\ref{item:>=all-gen>=theta} implies \ref{item:deltavec_unbounded_witness}, use a very similar argument, flipping the roles of the $\all$- and $\exists$-players (but again only applying
iteration maps to clause \ref{item:deltavec_unbounded_witness}).\qedhere
\end{case}
\end{proof}

We now complete the proof of the forcing theorem for  the $\muSigma_1^{\Mtilde_\gamma}$ forcing relation $\sforces{\CC}{\gamma\mu,1}$,
by adapting the previous claim
to the very last stage of induction 
(in the case that $\rho_1^{\Pg}=\om$):
 
\begin{clm}\label{clm:muSigma_1_forcing_theorem_=_OR^N}
Suppose  $\gamma=\OR^N$.
Suppose $0=n_0$; that is, $\rho_1^{\Pg}=\om$.
Let $\vec{x}\in(\Nhat|\gamma)^{<\om}$
and $d=\supp(\vec{x})$.
Then there is $m>d$ such that
for all $\mSigma_1$ formulas $\varphi$,
all $k<\om$,
all $\vec{\delta},\vec{\vareps}\in[\Delta_{\geq m}^N]^{2k}$,
and all $\theta\in\Delta_{\geq m}^N$,
 $N$ satisfies
that $\CC_d$ forces 
the following
statements
are equivalent:
\begin{enumerate}[label=\tu{(}\roman*\tu{)}]
\item\label{item:deltavec_bound_witness_end} $\exists\gamma'\in\OR\ \all^{\mathrm{gen}}_{\vec{\delta}}s\ \sforces{\CC_{\tail}}{\gamma0}\Mtilde_{\gamma'}\sats\varphi(\vec{x},s)$
\item\label{item:epsvec_unbounded_witness_end}$\exists\gamma'\in\OR\ \all^{\mathrm{gen}}_{\vec{\vareps}}s\ \sforces{\CC_{\tail}}{\gamma0}\Mtilde_{\gamma'}\sats\varphi(\vec{x},s)$
\item\label{item:>=all-gen>=theta_end} $\all^{\mathrm{gen}}_{\geq\theta;k}s\ \sforces{\CC_{\mathrm{tail}}}{\gamma\mathrm{w}1}\varphi(\vec{x},s)$
\item\label{item:each_Woodin_bounded_witness_end} $\all \ell<\om\ \exists\vec{\beta}\in[\Delta_{\geq \ell}]^{2k}\ \exists\gamma'\in\OR\ \all^{\mathrm{gen}}_{\vec{\beta}}s\ \sforces{\CC_{\tail}}{\gamma0}\ \Mtilde_{\gamma'}\sats\varphi(\vec{x},s)$.
\end{enumerate}
\end{clm}
\begin{proof}
 The equivalence of 
 \ref{item:deltavec_bound_witness_end}, 
\ref{item:epsvec_unbounded_witness_end} and \ref{item:each_Woodin_bounded_witness_end}
follows as before, using the $\delta_i^{N_i}$-soundness of $N_i$
(where $N_i$ is as before).
The proof that \ref{item:deltavec_bound_witness_end}
$\Rightarrow$
\ref{item:>=all-gen>=theta_end}
is just like the proof that 
Claim \ref{clm:muSigma_1_forcing_theorem_<_OR^N}
\ref{item:deltavec_unbounded_witness}
$\Rightarrow$
Claim  \ref{clm:muSigma_1_forcing_theorem_<_OR^N}.\ref{item:>=all-gen>=theta}
in Case \ref{case:gamma=OR^N_and_lambda^N=rho_1^N}
of Claim  \ref{clm:muSigma_1_forcing_theorem_<_OR^N}'s proof.
So we just need to see that (above some $m$, $\CC_d$ forces that) \ref{item:>=all-gen>=theta_end} $\Rightarrow$ \ref{item:deltavec_bound_witness_end}.
But this can be shown by using
the kind of argument given in the proof
that Claim \ref{clm:muSigma_1_forcing_theorem_<_OR^N}\ref{item:deltavec_unbounded_witness} $\Rightarrow$ Claim \ref{clm:muSigma_1_forcing_theorem_<_OR^N}\ref{item:deltavec_bound_witness}
in Case \ref{case:gamma<OR^N_muSigma_1_forcing_thm} of Claim  \ref{clm:muSigma_1_forcing_theorem_<_OR^N}'s proof,
but executed in the manner of
the proof that Claim \ref{clm:muSigma_1_forcing_theorem_<_OR^N}\ref{item:>=all-gen>=theta} $\Rightarrow$ Claim \ref{clm:muSigma_1_forcing_theorem_<_OR^N}\ref{item:deltavec_unbounded_witness}
in Case \ref{case:gamma=OR^N_and_lambda^N=rho_1^N}
of Claim  \ref{clm:muSigma_1_forcing_theorem_<_OR^N}'s proof.
\end{proof}

The next claim is the last piece in stage $n=0$:
\begin{clm}
Part \ref{item:CC_PP_forcing_theorems}\ref{item:projecting_forcing_relation_forcing_theorem} holds.\end{clm}
\begin{proof}Assume
$\rho_1^{N|\gamma}\leq\lambda^N$.

Suppose first that $\gamma<\OR^N$. Then it suffices
to prove that if $\gamma'\in\Lim\cap(\lambda^{\Pg},\OR^{\Pg})$, then there is $m_0<\om$
which works for $\Pg|\gamma'$,
i.e.~that whenever $p\in\CC^{\Pg}$
and $\varphi$ is $\mSigma_{n+1}$
and $\vec{x}\in(\Nhat^{\Pg}|\gamma')^{<\om}$
and
\begin{equation}\label{eqn:m_0_proj_forces} \Pg|\gamma'\sats p\sforces{\CC,\geq m_0}{\gamma'\mu,n+1,\mathrm{proj}}\varphi(\vec{x}),\end{equation}
then for all $m\in[m_0,\om)$,
\begin{equation}\label{eqn:m_proj_forces} \Pg|\gamma'\sats p\sforces{\CC,\geq m}{\gamma'\mu,n+1,\mathrm{proj}}\varphi(\vec{x}).\end{equation}
But for this, it suffices to take $m_0$ large enough
that $\gamma$ is stabilized, in the sense that for all $k<\om$
and all $\vec{\delta}\in[\Delta_{\geq m_0}^{\Pg}]^{2k}$,
we have $i_{\Pg R}(\gamma)=\gamma$
where $R=R^{\Pg}_{\vec{\theta}\cup\vec{\delta}}$
where $\vec{\theta}=(\delta_0^{\Pg},\ldots,\delta_{m_0-1}^{\Pg})$.
For with such $m_0$,
we can use the usual arguments
involving such models (and using soundness)
to propagate
the truth of (\ref{eqn:m_0_proj_forces})
to that of (\ref{eqn:m_proj_forces})
for each $m\geq m_0$,
using an argument like in the proof
of Subclaim \ref{sclm:N_to_R_suff_Prik_it_fixes_xvec} to see that $j(\vec{x})=\vec{x}$ for the iteration maps that arise in this propagation.

Now suppose that $\gamma=\OR^N=\beta^*$.
We work directly with $N$,
not $\Pg$.
In this case we must see that $m_0=0$ suffices. Propagate the
truth 
much as in the foregoing arguments
which directly involved $N$.
We
automatically get that $\gamma=\beta^*$ is ``stable'',
i.e. $\OR^R=\gamma$ for the relevant iterates $R$ of $N$,
and considering the definition
of the projecting forcing relation,
the proof of Subclaim \ref{sclm:N_to_R_suff_Prik_it_fixes_xvec} again gives that $j(\vec{x})=\vec{x}$ for the corresponding iteration maps $j:N\to R$. So arguments like before work;
we leave the details to the reader.
\end{proof}

This completes the inductive stage for $n=0$.\end{stage*}

\begin{stage*} $n>0$.
 
 We must verify
 parts \ref{item:rho_n=rho_n},
 \ref{item:p_n=p_n},
 \ref{item:CC_PP_forcing_theorems}
 and
 \ref{item:Hull_n+1_ordinals_match}.
  
Parts \ref{item:rho_n=rho_n} and \ref{item:p_n=p_n}: Since $n\leq n_0$ (by assumption),
 we have $\rho_n^{N|\gamma}\geq\lambda^N$. So these
 two parts are by part \ref{item:Hull_n+1_ordinals_match}
 for $n-1$. 

 Write
$\rho=\rho_n^{N|\gamma}=\rho_n^{\M_\gamma}$. 

\begin{clm}
The forcing theorem
for the witnessed $\mSigma_{n+1}$
forcing relation \tu{(}of part  \ref{item:1-CC-forcing}\ref{item:1-witnessed_1^N}\tu{)}, and the stem-forcing theorem for the witnessed $\rSigma_{n+1}$
stem forcing relation \tu{(}of part
 and \ref{item:1-PP-forcing}\ref{item:witnessed_rSigma_n+1_forcing_rel_def}\tu{)} both hold.
 \end{clm}
 \begin{proof}
We consider two cases, corresponding to the value of $\rho$.

\begin{case} $\rho>\om_1$.

By parts \ref{item:rho_n=rho_n} and \ref{item:p_n=p_n} with $(\gamma,n)$,
the case hypothesis (which implies $\omega_1<\rho_1^{\M_\gamma}=\rho_1^{N|\gamma}$) implies $\pvec_n^{\M_\gamma}=\pvec_n^{N|\gamma}$.
Now given $\beta\in(\lambda^N,\rho)$, let
\begin{enumerate}
\item $H^{\M_\gamma}_\beta=\Hull_{\mSigma_n}^{\M_\gamma}(\beta\cup\HC\cup\{\pvec_n^{\M_\gamma}\})$,
\item $C^{\M_\gamma}_\beta$ be the transitive collapse of $H^{\M_\gamma}_\beta$
\item  $\pi^{\M_\gamma}_\beta:C^{\M_\gamma}_\beta\to H^{\M_\gamma}_\beta$ be the uncollapse map,
\item 
$t^{\M_\gamma}_\beta=\Th_{\mSigma_{n}}(\beta\cup\HC\cup\{\pvec_n^{\M_\gamma}\})$, 
and\item\label{item:t^M_gamma,mu_beta} $t^{\M_\gamma,\mu}_\beta=\Th_{\muSigma_{n}}(\beta\cup\HC\cup\{\pvec_n^{\M_\gamma}\})$.
\end{enumerate}
Also let
\begin{enumerate}[resume*]
\item $H^{N|\gamma}_\beta=\Hull_{\rSigma_n}^{N|\gamma}(\beta\cup\{\pvec_n^{N|\gamma}\})$,
\item $C^{N|\gamma}_\beta$ be the transitive collapse of $H^{N|\gamma}_\beta$,
\item $\pi^{N|\gamma}_\beta:C^{N|\gamma}_\beta\to H^{N|\gamma}_\beta$ be the uncollapse map, and
\item $t^{N|\gamma}_\beta=\Th_{\rSigma_n}^{N|\gamma}(\beta\cup\{\pvec_n^{N|\gamma}\})$.
\end{enumerate}

By part \ref{item:Hull_n+1_ordinals_match} with $(\gamma,n-1)$ (hence, the latter regards $\mSigma_n$ and $\rSigma_n$ hulls), we have
\[ \OR\cap H^{\M_\gamma}_\beta=\OR\cap H^{N|\gamma}_\beta \]
(note that since $\lambda^N<\rho_1^N$, we have $p_1^{\M_\gamma}= p_1^{N|\gamma}$),
so we have (and define $\xi$ as)
\[ \xi=\OR\cap C^{\M_\gamma}_\beta=\OR\cap C^{N|\gamma}_\beta,\]
and
$\pi^{\M_\gamma}_\beta\rest\xi=\pi^{N|\gamma}_\beta\rest\xi$,
and hence (and define $\pvec$ as)
\[ \pvec=(\pi^{N|\gamma}_\beta)^{-1}(\pvec_n^{N|\gamma})=(\pi^{\M_\gamma}_\beta)^{-1}(\pvec_n^{\M_\gamma}).\]
By condensation, it also follows that
\[ C^{\M_\gamma}_\beta=\M_\xi \text{ and }C^{N|\gamma}_\beta=N|\xi.\]

Therefore, recalling the names $\pvec^{\Mtilde_\gamma}_n(\beta)$ and $\tau^{\Mtilde_\gamma}_n(\beta)$
from  \ref{dfn:tau^Mtilde(beta),q(beta)_if_omega_1<=beta}, we get
\begin{enumerate}[label=--]
\item $\pvec^{\Mtilde_\gamma}_n(\beta)=O_{\pvec}$,
and
\item $\tau^{\Mtilde_\gamma}_n(\beta)\in \Mhat_{\xi+\om}$ and
$\tau^{\Mtilde_\gamma}_n(\beta)$
is the natural name for 
\[ \Th_{\muSigma_n}^{\Mtilde_\xi}(O_\beta\cup\widetilde{\HC}\cup\{O_{\pvec}\}),\]
\end{enumerate}
and note here that $\Th_{\muSigma_n}^{\M_\xi}(\beta\cup\HC\cup\{\pvec\})$
is just $t^{\M_\gamma,\mu}_\beta(\pvec_n^{\M_\gamma}/\pvec)$
(that is, the theory obtained from $t^{\M_\gamma,\mu}_\beta$ (defined in clause \ref{item:t^M_gamma,mu_beta} above),  by substituting  $\pvec$ for $\pvec_n^{\M_\gamma}$).
Similarly, recalling the names $\pvec^{\Ntilde|\gamma}_n(\beta)$
and $\tau_n^{\Ntilde|\gamma}(\beta)$ from 
\ref{dfn:tau(beta),q(beta)_if_omega_1<=beta}, we get
\begin{enumerate}[label=--]
\item $\pvec^{\Ntilde|\gamma}_n(\beta)=o_{\vec{p}}$,
and
\item $\tau_n^{\Ntilde|\gamma}(\beta)\in\Nhat|(\xi+\om)$ and $\tau_n^{\Ntilde|\gamma}(\beta)$ is the natural name for
$\Th_{\rSigma_n}^{\Ntilde|\xi}(o_\beta\cup\{o_{\pvec}\})$,
\end{enumerate}
and $\Th_{\rSigma_n}^{N|\xi}(\beta\cup\{\pvec\})$
is just $t^{N|\gamma}_\beta(\pvec_n^{N|\gamma}/\pvec)$.

With the above considerations in mind,
it is straightforward to see that
the forcing theorem and stem forcing theorem corresponding to the relations $\sforces{\cdot}{\cdot}$ in  \ref{item:1-CC-forcing}\ref{item:1-witnessed_1^N} and \ref{item:1-PP-forcing}\ref{item:witnessed_rSigma_n+1_forcing_rel_def}
hold, and we leave the details to the reader. (One key point is that
the statements required to be forced/stem-forced 
in these definitions are only $\mSigma_n$/$\rSigma_n$ respectively,
and hence we already know that the corresponding forcing/stem-forcing theorem holds. For example
with $\sforces{\CC}{\gamma\mathrm{w},n+1}$, these statements are written
in clauses \ref{item:clause_b} and \ref{item:clause_a} of part \ref{item:rho_n^N|gamma>lambda^N_etc}
of the definition of $\sforces{\CC}{\gamma\mathrm{w},n+1}$ in  \ref{dfn:CC_forcing_relations}. The
main complexity in  this relation arises through the reference to $\tau_n^{\Mtilde_\gamma}(\beta)$ and $\pvec_n^{\Mtilde_\gamma}(\beta)$.)
\end{case}

\begin{case}$\rho=\om_1$.
 
In this case the forcing theorem for $\sforces{\CC}{\gamma\mathrm{w},n+1}$
is proven overall similarly to in the previous case, but now instead of the variable $\beta$ ranging over ordinals ${<\rho_n^{N|\gamma}}$ and corresponding name $\tau_n^{\Mtilde_\gamma}(\beta)$ for theories,
we have the variable $\varrho$
ranging over elements of $\Mhat_{\lambda^N}$ and the name $\tau_n^{\Mtilde_\gamma}(\varrho)$; see clause \ref{item:gamma>rho_n_witnessed_mSigma_n+1_forcing_rel} in the definition of
$\sforces{\CC}{\gamma\mathrm{w},n+1}$ in  \ref{dfn:CC_forcing_relations}.
(More precisely,
the name $\varrho$
in case $\rho=\om_1$
is analogous to the pair $(\beta,\varrho)$ in case $\rho>\om_1$,
in that the roles of $\beta$ and $\varrho$ in case $\rho>\omega_1$
is covered by just $\varrho$ in case $\rho=\om_1$.)
By the forcing theorem for the $\muSigma_n$ forcing relation $\sforces{\CC}{\gamma\mu  n}$
and  the homogeneity of $\CC$,
we get
\[ (\tau^{\Mtilde_\gamma}_n(\varrho))_G=\Th_{\muSigma_n}^{(\Mtilde_\gamma)_G}(\varrho_G\cup\{\pvec\}), \]
where $\pvec=\pvec^{N|\gamma}\cut\{\lambda^N\}$.
Using this equality, it is not difficult to verify the
forcing theorem for $\sforces{\CC}{\gamma\mathrm{w},n+1}$.

The stem-forcing theorem
for $\sforces{\PP^-}{\gamma\mathrm{w},n+1}$ is straightforward; the key fact is that
for each $\vec{x}\in(\Nhat|\gamma)^{<\om}$
and each $\alpha<\omega_1$,
there is some $s\in\PP^-$
such that the condition $r^{\gamma,n}_{s,\alpha,\vec{x}}\in G$, which decides 
the theory $\Th_{\rSigma_n}^{\Ntilde|\gamma}(o_\alpha\cup\{\vec{x}\})$.
(Letting $Q$ be as in the definition of $\tau^{\Ntilde|\gamma}_{s,n}(\alpha,\vec{x})$ (within \ref{dfn:stem_forcing_relations}),
we have that the theory $t$ decided
by $r^{\gamma,n}_{s,\beta,\vec{x}}$
is in $Q$, by the minimality of $(\beta^*,n^*)$ and $\muSigma_n^{N|\gamma}(\{\xg,\vec{x},\beta\})$-definability
of $t$.)
\end{case}

This completes the proof of the claim.
\end{proof}

Part \ref{item:Hull_n+1_ordinals_match} for $n>0$ is an easy consequence of these two facts, analogously to when $n=0$.

It only remains to verify the following claim:

\begin{clm}
We have:
\begin{enumerate}
 \item 
The forcing
theorem for $\sforces{\CC}{\gamma\mu,n+1}$ \tu{(}mentioned in part \ref{item:1-CC-forcing}\ref{item:1-mu_1^N}\tu{)} holds,
\item Part \ref{item:CC_PP_forcing_theorems}\ref{item:projecting_forcing_relation_forcing_theorem} holds.
\end{enumerate}
\end{clm}

\begin{case}
 $\rho>\omega_1$.
 
 In this case the proofs of these are totally analogous to those when $n=0$, so we omit further discussion of this case, aside from one small remark. In \ref{dfn:CC_forcing_relations},
 in the definition of the $m$-good $\muSigma_{n+1}$ forcing relation $\sforces{\CC,\geq m}{\gamma\mu,n+1}$,
  in
clause \ref{item:rho_n>lambda_m-good_muSigma_n+1_forcing_rel},
although we demand
 $\supp(\vec{x})\leq d$,
 we make no such demand on $\supp(\varrho)$. Thus
 (like when $n=0$),
iteration maps with critical point $>\delta_d^N$ will fix the
 ``coarse'' part $\sigma$ of $\vec{x}$
 (which is in $\Mhat_{\lambda^N}$),
 and (also like when $n=0$) using this, one finds
many iteration maps  $j$ such that $j(\vec{x})=\vec{x}$.
Because we did not restrict $\supp(\varrho)$, however,
we might have $j(\varrho)\neq\varrho$.
But this does not matter; 
similarly (analogous to when $n=0$)
we might have $j(\beta)>\beta$,
but $j$ will preserve the forcing statements \ref{item:rho_n>lambda_m-good_muSigma_n+1_forcing_rel_clause_a}
and \ref{item:rho_n>lambda_m-good_muSigma_n+1_forcing_rel_clause_b}
in the definition (and fix $\sigma$ and certain ordinal parameters). Thus,
the arguments from the $n=0$ case are readily adapted.
\end{case}

\begin{case} $\rho=\omega_1$.

 In this case the proofs are a slight variant of those in the previous case.
 The key difference is that the role of the variable $\beta<\rho$
 (from the previous case, and the choice of which depends on the tuple $\vec{\delta}$ of Woodin cardinals) is replaced
 by the ``extra Woodin'' $\delta_{2k}$
 at the top of the tuple $\vec{\delta}'\in[\Delta_{\geq m}^N]^{2k+1}$. Thus, if $\vec{\delta}=(\delta_0,\ldots,\delta_{2k-1})$
 and $\delta_{2k}=\delta_i^N$,
 and
 \[ N[g]\sats\all^{\mathrm{gen}}_{\vec{\delta}}s\ \sforces{\CC_i^N}{}\exists y\in\HC\ \sforces{\CC_{\tail}}{}\varphi(y,s) \]
 where $\varphi$
 is a statement of the relevant form,
 then in the genericity iteration arguments, after iterating $\delta_0,\ldots,\delta_{2k-1}$ ``into'' the background model $N$ via $L[\es]$-constructions as before,
 and producing some tuple $s\in\Dd^k$
 in a small generic extension $N[g]$,
 the top Woodin $\delta_{2k}$
 can then be iterated ``into'' $N$
 higher up,
 thus finding some $i'<\om$ such that \[ N[g,s]\sats\ \sforces{\CC_{i'}^N}{}\exists y\in\HC\ \sforces{\CC_{\tail}}{}\varphi(y,s).\]
 With this kind of modification, the arguments from before are readily adapted. We leave the remaining details to the reader.
\end{case}

This completes the  inductive stage $n>0$,
and hence the proof of parts
\ref{item:rho_n=rho_n}--\ref{item:Hull_n+1_ordinals_match} (for $(\gamma,n)$ as there). \end{stage*}

As mentioned at the outset,
parts \ref{item:rho_n=rho_n_om_1}--\ref{item:CC_PP_forcing_theorems_om_1}
(for $n$ as there)
are established through similar but simpler proofs than those above, and
we already gave parts of this argument in \S\ref{subsec:the_generic_M(R*)}. We leave the remaining details to the reader.
\end{proof}

\begin{lem}\label{lem:n^*=n_0} We have:
\begin{enumerate}
 \item\label{item:n^*=n_0} If $\omega_1<\beta^*$ \tu{(}equivalently, $\lambda^{\Pg}<\OR^{\Pg}$
 \tu{)} then $n^*=n_0$.
 \item\label{item:n^*+1=n_0} If $\omega_1=\beta^*$ \tu{(}equivalently, $\lambda^{\Pg}=\OR^{\Pg}$\tu{)} then $n^*+1=n_0$.
\end{enumerate}
\end{lem}
\begin{proof}
Part \ref{item:n^*=n_0}: 
Suppose $\omega_1<\beta^*$.
Let $n=\min(n^*,n_0)$. If $\rho_1^{\M_{\beta^*}}=\om_1$ and $p_1^{\M_{\beta^*}}=\emptyset$ then let $q_1^{\M_{\beta^*}}=\{\omega_1\}$, and otherwise let $q_1^{\M_{\beta^*}}=\emptyset$.
For $k>1$ let $q_k^{\M_{\beta^*}}=p_k^{\M_{\beta^*}}$. 
Let $\vec{q}_{n+1}^{\M_{\beta^*}}=(q_{n+1}^{\M_{\beta^*}},\ldots,q_1^{\M_{\beta^*}})$. 

\begin{clm}For all $\rSigma_{n+1}$ formulas $\varphi$, we  have
 \[ \Pg\sats\varphi(\pvec_{n+1}^{\Pg})\iff\M_{\beta^*}\sats\emptyset\sforces{\PP^-}{\beta^*\mu,n+1}\varphi(\vec{q}_{n+1}^{\M_{\beta^*}}).\]
\end{clm}
\begin{proof}
Let $\vec{p}=\vec{q}^{\M_{\beta^*}}_{n+1}$.
Suppose $\M_{\beta^*}\sats\emptyset\sforces{\PP^-}{\beta^*\mu,n+1}\varphi(\vec{p})$.
So there is $k<\om$ and a $\mu$-cofinal set $X\sub\Dd^k$
such that $s\sforces{\PP^-}{\beta^*\mathrm{w},n+1}\varphi(\vec{p})$ for all $s\in X$. 
As in the proof of Lemma \ref{lem:suff_gen_it},
we can find a sufficiently Prikry generic iterate $N$ of $\Pg$
such that $q^{\beta^*}_{s,\varphi(\vec{p})}\in G_{N|\lambda^N}$.
So by the forcing theorem for $\sforces{\PP^-}{\beta^*\mathrm{w},n+1}$, we have $N\sats\varphi(\vec{p})$.
But $\OR^N=\beta^*$ and $\pvec_{n+1}^N=\vec{p}$,
and since $i_{\Pg N}:\Pg\to N$ is degree $n_0$ iteration map, we get $\Pg\sats\varphi(\pvec_{n+1}^{\Pg})$, as desired.

Conversely suppose $\M_{\beta^*}\not\sats\emptyset\sforces{\PP^-}{\beta^*\mu,n+1}\varphi(\vec{p})$. Then
we can fix $\left<X_k\right>_{k<\om}$ such that each $X_k\in\mu^k$,
and for each $k$ and $s\in X_k$, $s\not\sforces{\PP^-}{\beta^*\mathrm{w},n+1}\varphi(\vec{p})$.
We can then find a sufficiently Prikry generic iterate $N$,
as witnessed by $G$ going through all $X_k$,
and it follows that $N\sats\neg\varphi(\vec{p})$,
and like before, it follows that $\Pg\sats\neg\varphi(\vec{p}_{n+1}^{\Pg})$.
\end{proof}

By the claim,  $t=\Th_{\rSigma_{n+1}}^{\Pg}(\pvec_{n+1}^{\Pg})$ is $\muSigma_{n+1}^{\M_{\beta^*}}(\{\vec{p},\xg\})$,
and therefore $t\in\OD^{\beta^*,n+1}(\xg)$.
But then by choice of $(\beta^*,n^*)$, 
we have $n^*\leq n_0$. So $n=n^*$,
and since $\rho_{n^*+1}^{\M_{\beta^*}}=\om_1$,
we have $\rho_{n^*+1}^{\Pg}\leq\lambda^{\Pg}$.
Therefore we can refer to the projecting forcing relation for $\muSigma_{n^*+1}$, which we do in the next claim:

\begin{clm}
Let $N$ be a sufficiently Prikry generic $\Sigma_{\Pg}$-iterate
and $G\sub\CC^N$ witness that $N$ is an $\RR$-genericity iterate.
Let $x\in\RR$ and $i<\om$ be such that $x\in N[G\rest i]$.
Let $\varphi$ be a $\muSigma_{n^*+1}$ formula and let $\vec{q}\in(\beta^*)^{<\om}$. Then
\[ \M_{\beta^*}\sats\varphi(\vec{q},x)\iff N[G\rest i]\sats\emptyset\sforces{\CC_{\tail},\geq 0}{\OR\mu,n^*+1,\mathrm{proj}}\varphi(\vec{q},x).\]
\end{clm}
\begin{proof}
This is just by the corresponding forcing theorem;
see part \ref{item:projecting_forcing_relation_forcing_theorem}
of Lemma \ref{lem:fs_correspondence}.
\end{proof}

Now let $x_0,y_0\in\RR$ be such that
$y_0\in\OD^{\beta^*,n^*+1}_{\mu}(x_0)\cut\OD^{<\alphagap}(x_0)$. Fix a $\muSigma_{n^*+1}$ formula $\varphi_0$
and $\vec{q}\in(\beta^*)^{<\om}$ such that for all $m<\om$,
\[ m\in y_0\iff \M_{\beta^*}\sats\varphi_0(\vec{q},x_0,m).\]
Then with $G,i$ as above, for all $m<\om$,
\[ m\in y_0\iff N[G\rest i]\sats\emptyset\sforces{\CC_{\tail},\geq 0}{\OR\mu,n^*+1,\mathrm{proj}}\varphi_0(\vec{q},x_0,m). \]
Recall that $\sforces{\CC,\geq 0}{\bar{\beta}\mu,n^*+1,\mathrm{proj}}$
is $\rSigma_{n^*+1}^{N[G\rest i]}(\{\pvec_{n^*+1}^{N},G\rest i\})$.

Now suppose $n^*<n_0$. Let $\vec{r}$ be $(n^*+1)$-self-solid for $N[G\rest k]$ (see 
\cite[Lemma 3.1]{extmax} and
\cite[Definition 2.2, Lemma 2.3]{V=HODX_pub})
with $x_0,\vec{q},\vec{p}_{n^*+1}^{N}\in H$ where
\[ H[G\rest i]=\Hull_{n^*+1}^{N[G\rest i]}(\{\vec{r},G\rest i\}\cup\delta_i^N).\]
Let $C[G\rest i]$ be the transitive collapse of $H[G\rest i]$.
Then $C[G\rest i]$ is sound and since $\rho_{n^*+1}^N=\lambda^N$
and by condensation, $C[G\rest i]\pins N[G\rest i]$,
and note that $C$ (the natural ground) is sound
and $C\pins N$.
Letting $\pi:C[G\rest i]\to N[G\rest i]$ be the uncollapse,
\[ m\in y_0\iff C[G\rest i]\sats\emptyset\sforces{\CC,\geq 0}{\OR\mu,n^*+1,\mathrm{proj}}\varphi_0(\pi^{-1}(\vec{q}),x_0,m).\]

We have $C[G\rest i]\pins\Lp_{\Gammagap}((N|\delta_i^N,G))$, so $C[G\rest i]$ is above-$\delta_i^N$
iterable in $\SS_{\alphagap}$. Now arguing as in the proof of Lemma \ref{lem:Sigma_P-iterate_good}
part \ref{item:no_new_Sigma_1_truths},
in some generic extension of $V$,
we can iterate $C[G\rest i]$ (hence, above $\delta_i^N$)
with an $n^*$-maximal tree,
to form an $\RR$-genericity iterate $\bar{N}$
as witnessed by $H\sub\CC^{\bar{N}}$,
such that $\OR^{\bar{N}}<\alphagap$,
and $(\Mtilde^{\bar{N}})_G\in\J_{\alphagap}$.

Now define $\mathscr{T},\Uu_0,H_0,j_0,\alpha_0,\beta_0$
as in the proof of Lemma \ref{lem:Sigma_P-iterate_good} part \ref{item:no_new_Sigma_1_truths} (but starting with base mouse $C[G\rest i]$,
iterating above $\delta_i^N$). 
As there, $\beta_0<\alphagap$.
Let $\M_{\gamma}^{[\alpha_0,\beta_0]}$ be the levels
of the $\M$-hierarchy associated to the S-gap $[\alpha_0,\beta_0]$. Let $\beta_0^*$ be the end of that hierarchy
(analogous to $\beta^*$). So $\beta_0^*<\alphagap$.

Suppose $H'$ is also as in the proof of that lemma,
and  $G'$ be $(\J(H'),\CC^{H'})$-generic
(note we demand genericity not just over $H'$, but over $\J(H')$). Then as before, $\beta'=\OR^{H'}\leq\beta_0^*$
and
$(\Mtilde^{H'})_{G'}=\M_{\beta'}^{[\alpha_0,\beta_0]}$
So $\M_{\beta'}^{[\alpha_0,\beta_0]}\in \J_{\alphagap}$. 
Directly by first order properties (and Lemma \ref{lem:fs_correspondence})
it is straightforward to see that the forcing theorem holds for $(H',G')$, for the $\CC^{H'}$-forcing relations, for the $\mSigma_0$-elementary
forcing relation $\sforces{\CC^{H'}}{\beta'0}$,
and up to and including
$\sforces{\CC^{H'},\geq 0}{\beta'\mu,\ell+1}$
with $\lambda^{H'}<\rho_{\ell+1}^{H'}$ (this ensures
that we easily express having a winning strategy for the game associated to the quantifier $\all^{\mathrm{gen}}_{\geq\theta}s$). When $\rho_{\ell+1}^{H'}\leq\lambda^{H'}$ this is maybe not quite so clear (since the quantifiers
involved in expressing having such a winning strategy
are unbounded over $\lambda^{H'}$).
But with a little more care in selecting  $H'$ we can arrange that, letting $j:C[G\rest i]\to H'$ be the iteration map and $\vec{r}=j(\pi^{-1}(\vec{q}))$, we have:
\begin{enumerate}
 \item if
$\rho_{n^*}^{H_0}=\lambda^{H_0}$
then the forcing theorems  also
hold for the forcing relations
through $\sforces{\CC}{\beta'\mu,n^*}$, and
\item writing
\[ \varphi_0(\dot{\vec{q}},\dot{x},\dot{m})\iff\all^*_\mu s\ \psi_0(\dot{\vec{q}},\dot{x},\dot{m},s) \]
where $\psi_0$ is $\mSigma_{n^*+1}$,
then for all $m<\om$,
the following are equivalent
for all $k<\om$ and all $\vec{\delta}\in[\Delta^{H'}]^{2k}$:
\begin{enumerate}[label=\tu{(}\roman*\tu{)}]
\item\label{item:m_in_y_0} $m\in y_0$ 
\item $H'\sats\all^{\mathrm{gen}}_{\vec{\delta}}s\ \sforces{\CC_{\tail}}{\beta',\mathrm{w},n^*+1}\psi_0(\vec{r},x_0,m,s)$, 
\item\label{item:M_beta'_sats_phi(m)} $\M_{\beta'}^{[\alpha_0,\beta_0]}\sats\varphi_0(\vec{r},x_0,m)$.
\end{enumerate}
\end{enumerate}
Note that this yields a contradiction,
as the equivalence between \ref{item:m_in_y_0}
and \ref{item:M_beta'_sats_phi(m)} shows that $y_0\in\OD_{<\alphagap}(x_0)$.

To obtain $H'$ with these properties, we will enumerate in advance (preimages of) all possible violations, and eliminate them one by one during the $\RR$-genericity iteration, by iterating at tuples of Woodins into certain measure one sets, like in the arguments in the proof of Lemma \ref{lem:fs_correspondence}. let $\left<(\gamma_n,\ell_n,\varphi_n,\vec{a}_n,t_n,\nu_n,z_n)\right>_{n<\om}$ be a (generic) enumeration of all tuples $(\gamma,\ell,\varphi,\vec{a},t,\nu,z)$ such that:
\begin{enumerate}[label=--]
 \item 
$\gamma\leq\beta_0^*$,
\item $\ell\leq n^*$,
$\rho_{\ell+1}^{H_0}\leq\lambda^{H_0}$ and $\rho_{\ell+1}^{\M_{\gamma}^{[\alpha_0,\beta_0]}}=\om_1$,
\item $\varphi$ is a $\muSigma_{\ell+1}$ formula,
\item $\vec{a}\in(\M_{\gamma}^{[\alpha_0,\beta_0]})^{<\om}$,
\item $t$ is an $\rSigma_{n^*+1}$ min-term,
\item $\nu<\om_1$,
\item $z\in\RR$,  and
\item if $\ell=n^*$ then $\varphi=\varphi_0$ and
 $z=x_0$ and $\nu<\delta_i^{C[G\rest i]}$
 and for some $m<\om$, we have \[(\pi^{-1}(\vec{q}),m)=t^{C[G\rest i]}(\pvec_{n^*+1}^{C[G\rest i]},\nu,G\rest i),\]
 \end{enumerate}
 enumerated
 with infinitely many repetitions of each  tuple.
 Build the $\RR$-genericity iteration $\Tt_0\conc\Tt_1\conc\ldots$ of $H_0$, together
 with models $H_{n+1}$, $j_n<\om$ and generics $G_n$, as follows: recall that $\Uu_0$ was based on $C[G\rest i]|\delta_{j_0}^{C[G\rest i]}$,
 and $H_0=M^{\Uu_0}_\infty$. Let $G_0$ be $(H_0,\CC_{j_0}^{H_0})$-generic, extending $G\rest i$.
 Given $H_n$ and $G_n\sub\CC_{j_n}^{H_n}$,
 let $(\gamma,\ell,\varphi,\vec{a},t,\nu,z)=(\gamma_n,\ell_n,\varphi_n,\vec{a}_n,t_n,\nu_n,z_n)$.
 If $z\in H_n[G_n]$ and $\nu<\delta_{j_n}^{H_n}$
 and
 \begin{enumerate}[label=\tu{(}\roman*\tu{)}]\item\label{item:M_gamma_true} $\M_{\gamma}^{[\alpha_0,\beta_0]}\sats\ \varphi(\vec{a},z)$
 \end{enumerate}
 iff
 \begin{enumerate}[resume*]
  \item\label{item:H_n[G_n]_false} $\neg\Big(H_n[G_n]\sats\  \sforces{\CC_{\tail}}{\OR^{H_n}\mu,\ell+1}\varphi(\vec{r},z)\Big)$
 \end{enumerate}
    where
    \begin{enumerate}[label=--]
\item    $\vec{r}=t^{H_n}(\pvec_{n^*+1}^{H_n},\nu,G\rest i)$ if this is defined and is in $(\Mhat^{H_n[G_n]})^{<\om}$, and
\item $\vec{r}=o_0$ (the name for $0$) otherwise,
\end{enumerate}
then
\begin{enumerate}[label=--]
 \item if \ref{item:M_gamma_true} and \ref{item:H_n[G_n]_false} are both false,
 in particular
 \[H_n[G_n]\sats\ \sforces{\CC,> j_n}{\OR^{H_n}\mu,\ell+1}\varphi(\vec{r},z),\]
 then let $k<\om$ and $\vec{\delta}\in[\Delta_{\geq j_n}^{H_n}]^{2k}$
 witness this and iterate $H_n$ at $\vec{\delta}$
 so as to produce a witness for some
 $s\in\Dd^k$ which lies in the measure
 one set witnessing that $\M_\gamma^{[\alpha_0,\beta_0]}\sats\neg\varphi(\vec{a},z)$, and 
 \item if \ref{item:M_gamma_true} and \ref{item:H_n[G_n]_false} are both true,
 then proceed analogously but as in the proof
of  Lemma \ref{lem:fs_correspondence},
Case \ref{case:gamma=OR^N_and_lambda^N=rho_1^N},  Claim \ref{clm:muSigma_1_forcing_theorem_=_OR^N},\ref{item:>=all-gen>=theta_end} $\Rightarrow$ \ref{item:deltavec_bound_witness_end},
and its adaptation to the $n>0$ stage.
\end{enumerate}
Let $\Tt_n$ be the resulting tree, $H_{n+1}=M^{\Tt_n}_\infty$, $j_{n+1}$ be least such that $\Tt_n$ is based on $H_n|\delta_{j_{n+1}}^{H_n}$, and $G_{n+1}$ be $(H_{n+1},\CC_{j_{n+1}}^{H_{n+1}})$-generic
with $G_n\sub G_{n+1}$.

If instead $z\notin H_n[G_n]$, then let $\Tt_n$ be the $z$-genericity iteration at $\delta_{j_n+1}^{H_n}$, and $j_{n+1}=j_n+1$, and $G_{n+1}$ be $(H_{n+1},\CC_{j_{n+1}}^{H_{n+1}})$-generic with $G_n\sub G_{n+1}$ and $z\in H_{n+1}[G_{n+1}]$.

Otherwise let $\Tt_n$ be trivial and $H_{n+1}=H_n$ and $j_{n+1}=j_n$ and $G_{n+1}=G_n$.

This completes the construction. It is straightforward enough to see that it meets the requirements, producing the contradiction mentioned earlier.

Part \ref{item:n^*+1=n_0}: This is just a slight variant of the previous part, so we leave it to the reader.
\end{proof}

To be added in later installment: further analysis of relationship between $N$, $\mathscr{M}_{\beta^*}$ and $\mathbb{S}_{\betagap}$,
application to Rudominer-Steel conjecture,
and material on projective-like cases.

\section*{Acknowledgements}
Gef\"ordert durch die Deutsche Forschungsgemeinschaft (DFG) im Rahmen der Exzellenzstrategie des Bundes und der L\"ander EXC 2044--390685587, Mathematik M\"unster: Dynamik--Geometrie--Struktur.
Gef\"ordert durch die Deutsche Forschungsgemeinschaft (DFG) -- Projektnummer 445387776.
Funded by the Deutsche Forschungsgemeinschaft (DFG, German Research Foundation) -- project number 445387776.

\bibliographystyle{plain}
\bibliography{../bibliography/bibliography}
\end{document}